\pdfoutput=1
\documentclass[12pt,a4paper,twoside]{article}
\usepackage{verbatim,setspace}
\usepackage[normalem]{ulem}
\usepackage{natbib}
\usepackage{amsmath,amssymb,amsthm,tikz,enumitem,tocloft}
\usepackage[stable]{footmisc}
\usepackage{fancyhdr}
\newtheorem{thm}{Theorem}
\newtheorem{lem}[thm]{Lemma}
\newtheorem{definition}[thm]{Definition}
\newtheorem{assumption}[thm]{Assumption}
\newtheorem{rem}[thm]{Remark}

\setenumerate[1]{label={\rm\textbf{(\roman*)}},ref={\rm{(\roman*)}}}
\setenumerate[2]{label={\rm\textbf{(\alph*)}},ref={\rm{(\alph*)}}}

\newenvironment{AMS}%
{{\upshape\bfseries AMS subject classifications. }\ignorespaces}{}

\newenvironment{acks}[1][Acknowledgements]%
{%
  \par%
}

\newcommand{\arxivyesno}[2]{#1}
\newcommand{\bRplus}{{\mathbb R}_{>0}}
\newcommand{\bRgeq}{{\mathbb R}_{\geq 0}}
\newcommand{\RZ}{{\mathbb R} \slash {\mathbb Z}}
\newcommand{\bR}{{\mathbb R}}
\newcommand{\bS}{{\mathbb S}}
\newcommand{\bN}{{\mathbb N}}
\newcommand{\bZ}{\mathbb{Z}}
\newcommand{\bI}{\mathbb{I}}
\newcommand{\II}{{\rm I\!I}}
\newcommand{\tr}{\operatorname{tr}}
\newcommand{\spa}{\operatorname{span}}
\newcommand{\diag}{\operatorname{diag}}
\newcommand{\diam}{\operatorname{diam}}
\newcommand{\supp}{\operatorname{supp}}
\DeclareMathOperator*{\argmin}{arg\,min}
\newcommand{\Gauss}{{\mathcal{K}}}
\newcommand{\dH}[1]{\;{\rm d}{\mathcal{H}}^{#1}} 
\newcommand{\dL}[1]{\;{\rm d}{\mathcal{L}}^{#1}} 

\newcommand{\spont}{{\overline\varkappa}}
\newcommand{\bigchi}{\ensuremath{\mathrm{\mathcal{X}}}}
\newcommand{\charfcn}[1]{\bigchi_{#1}} 
\newcommand{\Whh}{V(\Gamma^h)}
\newcommand{\Whm}{V(\Gamma^m)}
\newcommand{\Whmp}{V(\Gamma^{m+1})}

\newcommand{\Vhh}{\underline{V}(\Gamma^h)}
\newcommand{\Vhm}{\underline{V}(\Gamma^m)}
\newcommand{\Vhmp}{\underline{V}(\Gamma^{m+1})}
\newcommand{\Vhz}{\underline{V}(\Gamma^0)}
\newcommand{\WhmD}{V_{\rm D}(\Gamma^m)}
\newcommand{\VhmD}{\underline{V}_{\rm D}(\Gamma^m)}
\newcommand{\matVhh}{\mat{V}(\Gamma^h)}
\newcommand{\matVhm}{\mat{V}(\Gamma^m)}
\newcommand{\VhI}{\underline{V}^h(\bI)}
\newcommand{\WhI}{V^h(\bI)}
\newcommand{\WhGhT}{V(\GhT)}
\newcommand{\VhGhT}{\underline{V}(\GhT)}
\newcommand{\WhTGhT}{V_T(\GhT)}
\newcommand{\VhTGhT}{\underline{V}_T(\GhT)}
\newcommand{\Wchh}{V_{\rm c}(\Gamma^h)}
\newcommand{\Wcht}{V_{\rm c}(\Gamma^h(t))}
\newcommand{\Vchh}{\underline V_{\rm c}(\Gamma^h)}
\newcommand{\matVchh}{\mat V_{\rm c}(\Gamma^h)}

\newcommand{\Vchm}{\underline V_{\rm c}(\Gamma^m)}
\newcommand{\matVchm}{\mat V_{\rm c}(\Gamma^m)}
\newcommand{\Vt}{[H^1(\Gamma(t))]^d}
\newcommand{\Wt}{H^1(\Gamma(t))}
\newcommand{\Vht}{\underline{V}(\Gamma^h(t))}
\newcommand{\Wht}{V(\Gamma^h(t))}
\newcommand{\sigmaO}{o}
\newcommand{\nabs}{\nabla_{\!s}}
\newcommand{\Id}{{\rm Id}}
\newcommand{\id}{{\rm id}}
\newcommand{\deldel}[1]{\frac{\delta}{{\delta}#1}}
\newcommand{\dd}[1]{\frac{\rm d}{{\rm d}#1}}
\newcommand{\ddt}{\dd{t}}
\newcommand{\ddeps}{\frac{\rm d}{{\rm d}\epsilon}}
\newcommand{\matpartu}{\partial_t^\bullet}

\newcommand{\matpartx}{\partial_t^\circ}
\newcommand{\matpartxh}{\partial_t^{\circ,h}}
\newcommand{\matpartxpar}[1]{\partial_{#1}^\circ}
\newcommand{\matpartxhpar}[1]{\partial_{#1}^{\circ,h}}
\newcommand{\matpartnpar}[1]{\partial_{#1}^{\text{\tiny$\square$}}}
\newcommand{\matpartn}{\matpartnpar{t}}
\newcommand{\ttau}{\Delta t}
\def\epsilon{\varepsilon} 
\def\conduct{\mathcal{K}}
\makeatletter
\renewcommand*{\uuline}{%
  \bgroup
  \UL@setULdepth
  \markoverwith{%
    \lower\ULdepth\hbox{%
      \kern-.03em%
      \vtop{%
        \hrule width.2em%
        \kern 0.6pt 
        \hrule
      }%
      \kern-.03em%
    }%
  }%
  \ULon
}
\makeatother
\newcommand{\mat}[1]{\smash{\uuline{#1}}}

\makeatletter
\newcommand*{\mint}[1]{%
  \mint@l{#1}{}%
}
\newcommand*{\mint@l}[2]{%
  \@ifnextchar\limits{%
    \mint@l{#1}%
  }{%
    \@ifnextchar\nolimits{%
      \mint@l{#1}%
    }{%
      \@ifnextchar\displaylimits{%
        \mint@l{#1}%
      }{%
        \mint@s{#2}{#1}%
      }%
    }%
  }%
}
\newcommand*{\mint@s}[2]{%
  \@ifnextchar_{%
    \mint@sub{#1}{#2}%
  }{%
    \@ifnextchar^{%
      \mint@sup{#1}{#2}%
    }{%
      \mint@{#1}{#2}{}{}%
    }%
  }%
}
\def\mint@sub#1#2_#3{%
  \@ifnextchar^{%
    \mint@sub@sup{#1}{#2}{#3}%
  }{%
    \mint@{#1}{#2}{#3}{}%
  }%
}
\def\mint@sup#1#2^#3{%
  \@ifnextchar_{%
    \mint@sup@sub{#1}{#2}{#3}%
  }{%
    \mint@{#1}{#2}{}{#3}%
  }%
}
\def\mint@sub@sup#1#2#3^#4{%
  \mint@{#1}{#2}{#3}{#4}%
}
\def\mint@sup@sub#1#2#3_#4{%
  \mint@{#1}{#2}{#4}{#3}%
}
\newcommand*{\mint@}[4]{%
  \mathop{}%
  \mkern-\thinmuskip
  \mathchoice{%
    \mint@@{#1}{#2}{#3}{#4}%
        \displaystyle\textstyle\scriptstyle
  }{%
    \mint@@{#1}{#2}{#3}{#4}%
        \textstyle\scriptstyle\scriptstyle
  }{%
    \mint@@{#1}{#2}{#3}{#4}%
        \scriptstyle\scriptscriptstyle\scriptscriptstyle
  }{%
    \mint@@{#1}{#2}{#3}{#4}%
        \scriptscriptstyle\scriptscriptstyle\scriptscriptstyle
  }%
  \mkern-\thinmuskip
  \int#1%
  \ifx\\#3\\\else_{#3}\fi
  \ifx\\#4\\\else^{#4}\fi  
}
\newcommand*{\mint@@}[7]{%
  \begingroup
    \sbox0{$#5\int\m@th$}%
    \sbox2{$#5\int_{}\m@th$}%
    \dimen2=\wd0 %
    \let\mint@limits=#1\relax
    \ifx\mint@limits\relax
      \sbox4{$#5\int_{\kern1sp}^{\kern1sp}\m@th$}%
      \ifdim\wd4>\wd2 %
        \let\mint@limits=\nolimits
      \else
        \let\mint@limits=\limits
      \fi
    \fi
    \ifx\mint@limits\displaylimits
      \ifx#5\displaystyle
        \let\mint@limits=\limits
      \fi
    \fi
    \ifx\mint@limits\limits
      \sbox0{$#7#3\m@th$}%
      \sbox2{$#7#4\m@th$}%
      \ifdim\wd0>\dimen2 %
        \dimen2=\wd0 %
      \fi
      \ifdim\wd2>\dimen2 %
        \dimen2=\wd2 %
      \fi
    \fi
    \rlap{%
      $#5%
        \vcenter{%
          \hbox to\dimen2{%
            \hss
            $#6{#2}\m@th$%
            \hss
          }%
        }%
      $%
    }%
  \endgroup
}
\makeatother
\newcommand{\Gmint}[1]{{\textstyle\mint{-}_{#1}}}
\newcommand{\vvv}{{\rm v}}

\newcommand{\uspace}{\mathbb{U}}
\newcommand{\pspace}{\mathbb{P}}
\newcommand{\XFEMGamma}{{\rm XFEM$_\Gamma$}}
\newcommand{\Psing}{P_{\rm sing}}
\newcommand{\NbulkT}{\vec N_{\Gamma^m,\Omega}^\transT}
\newcommand{\Nbulk}{\vec N_{\Gamma^m,\Omega}}

\newcommand{\SmD}{S^m_{\rm D}}
\newcommand{\Sml}{S^m_+}
\newcommand{\Smhoml}{S^m_{0,+}}
\newcommand{\Smhom}{S^m_0}
\newcommand{\SmDl}{S^m_{{\rm D},+}}
\newcommand{\MSrho}{\rho}                   
\newcommand{\Mbulk}{M_{\Gamma^m,\Omega}}
\newcommand{\NbulkTMS}{\vec N_{\Omega,\Gamma^m}^\transT}

\newcommand{\phasec}{\mathfrak{c}}
\newcommand{\chempot}{\mathfrak{m}}

\newcommand{\GT}{{\mathcal{G}_T}}           
\newcommand{\GhT}{{\mathcal{G}^h_T}}
\newcommand{\HT}{{\mathfrak{O}_T}}          
\newcommand{\HTi}[1]{{\mathfrak{O}_{#1,T}}} 
\newcommand{\subHT}{{\mathfrak{O}}}         
\newcommand{\SigmahTi}[1]{{\mathcal{S}^h_{#1,T}}}
\newcommand{\HhT}{{\mathfrak{O}^h_T}}      
\newcommand{\localpara}{\vec{\mathfrak X}} 
\newcommand{\conormal}{\vec\mu}        
\newcommand{\ek}{e}                    
\newcommand{\Wein}{{\rm W}}        
\newcommand{\normalspace}{{\rm N}}     
\newcommand{\tanspace}{{\rm T}}        
\newcommand{\subT}{{\rm T}}            
\makeatletter  
\newcommand*{\transT}{{\mathpalette\@transT{}}} 
\newcommand*{\@transT}[2]{
  \raisebox{0.75\height}{$\m@th#1\intercal$}%
}
\makeatother  
\newcommand{\Lapg}{{\Delta_s^{\!\tG}}}
\newcommand{\LapGh}{{\Delta_s^{\!\Gamma^h}}}
\newcommand{\LapGm}{{\Delta_s^{\!\Gamma^m}}}
\newcommand{\tG}{{\widetilde{G}}}
\newcommand{\curvegamma}{\mathfrak{g}} 

\arxivyesno{}{
\allowdisplaybreaks
\setlength{\parskip}{0pt}
}

\begin{document}

\arxivyesno{
\title{Parametric finite element approximations of curvature driven 
interface evolutions}
\author{John W. Barrett\footnotemark[2] \and 
        Harald Garcke\footnotemark[3]\ \and 
        Robert N\"urnberg\footnotemark[2]}

\renewcommand{\thefootnote}{\fnsymbol{footnote}}
\footnotetext[2]{Department of Mathematics, 
Imperial College London, London, SW7 2AZ, UK}
\footnotetext[3]{Fakult{\"a}t f{\"u}r Mathematik, Universit{\"a}t Regensburg, 
93040 Regensburg, Germany}

\date{}

\maketitle

\begin{abstract}
Parametric finite elements lead to very efficient numerical methods for surface
evolution equations. We introduce several computational techniques for 
curvature driven evolution equations based on a weak formulation for the 
mean curvature. The approaches discussed, in contrast to many other methods, 
have good mesh properties that avoid mesh coalescence and 
very non-uniform meshes. Mean curvature flow, surface diffusion, anisotropic
geometric flows, solidification, two-phase flow, Willmore and Helfrich flow 
as well as biomembranes are treated. 
We show stability results as well as 
results explaining the good mesh properties.
\end{abstract}

\begin{center}
\setlength\cftparskip{0pt}
\setlength\cftbeforesecskip{2pt}
\setlength\cftaftertoctitleskip{12pt}
\begin{minipage}{0.9\textwidth}

\setcounter{tocdepth}{1}
\begin{spacing}{0.9}
\renewcommand{\cftsecleader}{\cftdotfill{\cftdotsep}}
\tableofcontents
\end{spacing}
\end{minipage}
\end{center}

\pagestyle{fancyplain}
\lhead[\fancyplain{}{\textsf{\thepage}}]
{\fancyplain{}{\textsf{PFEA of curvature driven interface evolutions}}}
\rhead[\fancyplain{}{\textsf{J.~W.~Barrett, H.~Garcke, R.~N\"urnberg}}]
{\fancyplain{}{\textsf{\thepage}}}
\cfoot{}

}
{
\begin{frontmatter}
\chapter[PFEA of curvature driven interface evolutions]%
{Parametric finite element approximations of curvature driven 
interface evolutions
}

\begin{aug}
\author[addressrefs={ad1}]%
{\fnm{John W.} \snm{Barrett}} 
\author[addressrefs={ad2}]%
 {\fnm{Harald} \snm{Garcke}}%
\author[addressrefs={ad1}]%
{\fnm{Robert} \snm{N\"urnberg}} 
\address[id=ad1]%
{Department of Mathematics, 
Imperial College London, London, SW7 2AZ, UK}%
\address[id=ad2]%
{Fakult{\"a}t f{\"u}r Mathematik, Universit{\"a}t Regensburg, 
93040 Regensburg, Germany}%
\end{aug}

\setcounter{minitocdepth}{1}
\faketableofcontents
\minitoc
\end{frontmatter}
}

\section{Introduction} \label{sec:1}

Interfaces separating different regions in space frequently appear in
the natural and imaging sciences as well as in mathematics. 
In many situations these interfaces evolve in time and the
evolution laws contain curvature quantities.
Often an interface carries a surface energy and in the simplest
case the surface energy is given by the total surface area of the
interface. Frequently the surface energy decreases in time
and as the negative mean curvature is the first variation of the
surface area, the mean curvature naturally appears in the context of
evolving interfaces. In the simplest such situation the evolution of
the interface is the $L^2$--gradient flow of the surface area.
This leads to mean curvature flow, which will play a
prominent role in this \arxivyesno{work}{chapter}. 
Many recent applications involve
surface energies, which contain surface integrals of curvature
quantities. An example is the Willmore functional, which is the
integrated squared mean curvature. Evolution problems for surfaces
involving the first variation of the Willmore functional are
particularly challenging and will also be discussed in this work.

Often however, the evolution of the interface is given by laws which
couple quantities on the interface to quantities which are solutions
of partial differential equations in the bulk, i.e., in the surrounding 
regions. This happens for example in solidification phenomena in which
the temperature, which solves a diffusion equation in the liquid and
solid phases, influences the evolution of the interface. Other examples
are interfaces driven by a flow field, which is for example the
solution of a Stokes or Navier--Stokes system. These situations are
much more involved and lead to challenging problems in analysis and
computation.

This work deals with parametric finite element methods for evolving
interfaces. We first present the main ideas for simple geometric
evolution equations like the mean curvature flow. We will then couple
equations on the interface to bulk equations using an unfitted finite
element method, which uses a surface mesh that is independent from the
bulk mesh. As examples, solidification phenomena and problems arising in 
two-phase fluid flow will be discussed.
Upon introducing ideas on the approximation of the Willmore energy and the
associated Willmore flow, we end this \arxivyesno{work}{chapter}
with a discussion of the numerical approximation of evolution problems
for biomembranes. For other numerical approaches for geometric
partial differential equations, like level set methods, phase field methods
and other front tracking methods, we refer to the contributions
\arxivyesno{%
\cite{BanschS19handbook,Bartels19handbook,BonitoDN19handbook,DuF19handbook,%
SayeS19handbook,TurekO19handbook} in 
{\em Handbook of Numerical Analysis, Vol.~XXI \& XXII},}
{%
by B\"ansch and Schmidt; Bartels; Bonito, Demlow and Nochetto; Du and Feng; 
Saye and Sethian; Turek and Mierka in this handbook,}
and also to the review article \cite{DeckelnickDE05}.

\section{Geometry of surfaces} \label{sec:geometry}
In this section we review basic facts about surfaces, which will be 
necessary later on, when we develop numerical algorithms for curvature 
driven flows of hypersurfaces. We will define surfaces, differential 
operators on surfaces, important curvature quantities as well as a divergence 
theorem on hypersurfaces.
For a more detailed discussion of these themes we refer to
\cite{doCarmo76,DeckelnickDE05,Kuhnel15,Walker15}.
Readers familiar with these concepts may skip this section.

\subsection{Surfaces in $\bR^d$} \label{subsec:surfaces}

We will first define the term surface. Here, and throughout, let $d\geq2$.

\begin{definition} \label{def:hypersurf}
\rule{0pt}{0pt}
\begin{enumerate}
\item 
A subset $\Gamma \subset \bR^d$ is called an $n$-dimensional
$C^k$--surface, for $0\leq n \leq d$ and $k\geq 1$, 
if for every point $\vec p\in \Gamma$ there exists an open
neighbourhood $V\subset\bR^d$ of $\vec p$ and a
bijective $C^k$--map $\localpara:U \to V\cap\Gamma$, 
with $U\subset \bR^n$ open and connected, such that 
the Jacobian 
$\nabla\,\localpara = (\partial_1\,\localpara,\ldots,\partial_n\,\localpara)
: U\to\bR^{d\times n}$ of $\localpara$ has full rank in $U$.
\item The map $\localpara$ is called a local parameterization of $\Gamma$.
\item If $n=d-1$, then we call $\Gamma$ a hypersurface.
\end{enumerate}
\end{definition}

We now define what we mean by a mapping defined on a hypersurface being 
differentiable.

\begin{definition} \label{def:Ckf}
Let $\Gamma\subset \bR^d$ be an $n$-dimensional $C^k$--surface. 
A function $f:\Gamma \to \bR$ is a $C^k$--function, denoted by
$f\in C^k(\Gamma)$, 
if $f\circ \localpara :U \to \bR$ is of class $C^k$ for all 
$C^k$--parameterizations $\localpara : U\to \Gamma$.
\end{definition}

With the help of a local parameterization we can define a basis of the tangent space.

\begin{definition} \label{def:2.3}
Let $\Gamma$ be an $n$-dimensional $C^1$--surface with a local parameterization 
$\localpara:U \to \bR^d$. 
For $\vec p \in \Gamma$ let $\vec u\in U$ be such that 
$\vec p=\localpara (\vec u)$. The vectors
\[
\vec\partial_1=( \partial_1\,\localpara)(\vec u),\ldots, 
\vec\partial_n=(\partial_n\,\localpara) (\vec u)
\]
are linearly independent and span an $n$-dimensional space, which is called 
the tangent space and will be denoted by
\[ 
\tanspace_{\vec p}\,\Gamma = \spa \{\vec\partial_1, \ldots, \vec\partial_n\}\,.
\]
Its orthogonal complement 
$\normalspace_{\vec p}\,\Gamma =(\tanspace_{\vec p}\,\Gamma)^\perp$
in $\bR^d$ is called the normal space of $\Gamma$ at $\vec p$.
\end{definition}
Mainly oriented hypersurfaces will be of interest, which we define next.
\begin{definition}
A hypersurface $\Gamma$ is called orientable, if a continuous normal vector
field $\vec\nu :\Gamma \to \bS^{d-1}$, 
with $\vec\nu(\vec p) \in \normalspace_{\vec p}\,\Gamma$ for all 
$\vec p \in \Gamma$,
exists, where $\bS^{d-1}$ is the $(d-1)$-dimensional sphere in $\bR^d$.
\end{definition}

We will frequently use the following differential operators on a
surface $\Gamma$.

\begin{definition} \label{def:2.5}
Let $\Gamma\subset\bR^d$ be an $n$-dimensional $C^1$--surface, and let
$f:\Gamma\to\bR$, $\vec f:\Gamma\to\bR^d$ be $C^1$--functions.
\begin{enumerate}
\item \label{item:def2.5i}
For $\vec p\in\Gamma$ and $\vec\tau\in
\tanspace_{\vec p}\,\Gamma$ we define the directional derivative as
\begin{equation*}
(\partial_{\vec\tau}\, f)(\vec p) =
\lim_{\epsilon\to0} \frac{f(\vec y(\epsilon)) - f(\vec y(0))}{\epsilon}
= (f \circ \vec y)'(0)\,,
\end{equation*}
where $\vec y : (-1,1) \to \Gamma$ parameterizes a curve on $\Gamma$ 
with $\vec y(0)=\vec p$
and $\vec y\,'(0)=\vec\tau$.
\item \label{item:def2.5ii}
The surface gradient of $f$ on $\Gamma$ at the point
$\vec p\in\Gamma$ is defined as
\begin{equation*}
(\nabs\, f)(\vec p) = \sum_{i=1}^n (\partial_{\vec\tau_i}\,f)(\vec p)\, 
\vec\tau_i \,,
\end{equation*}
where $\{\vec\tau_1,\ldots,\vec\tau_n\}$ is an
orthonormal basis of $\tanspace_{\vec p}\,\Gamma$.
\item \label{item:def2.5iii}
The surface divergence of $\vec f$ at the point
$\vec p\in\Gamma$ is defined as
\begin{equation*}
(\nabs\,.\,\vec f)(\vec p) = \sum_{i=1}^n \vec\tau_i\,
.\, (\partial_{\vec\tau_i}\,\vec f)(\vec p)\,,
\end{equation*}
where $\{\vec\tau_1,\ldots,\vec\tau_n\}$ is again an
orthonormal basis of $\tanspace_{\vec p}\,\Gamma$.
\item \label{item:def2.5v}
The surface Jacobian of $\vec f$ at the point
$\vec p\in\Gamma$ is defined as
\[
(\nabs\,\vec f)(\vec p) = \sum_{i=1}^n
(\partial_{\vec\tau_i}\,\vec f)(\vec p) \otimes \vec\tau_i\,,
\]
where $\{\vec\tau_1,\ldots,\vec\tau_n\}$ is again an
orthonormal basis of $\tanspace_{\vec p}\,\Gamma$.
\item \label{item:def2.5LB}
If $\Gamma$ is a $C^2$--surface, and $f \in C^2(\Gamma)$, then
we define the Laplace--Beltrami operator and the surface Hessian
of $f$ on $\Gamma$ as
\begin{equation*}
\Delta_s\, f = \nabs\,.\,(\nabs\, f)\qquad\text{and}\qquad
\nabs^2\, f = \nabs\,(\nabs\, f)\,,
\end{equation*}
respectively.
\item \label{item:def2.5LBd}
The Laplace--Beltrami operator of $\vec f\in [C^2(\Gamma)]^d$,
for a $C^2$--surface $\Gamma$, is defined component-wise,
i.e.\ $(\Delta_s\,\vec f)\,.\,\vec\ek_i = \Delta_s\,(\vec f\,.\,\vec\ek_i)$,
$i = 1,\ldots,d$. 
Here, and throughout, we let $\vec\ek_i$ denote the $i$-th standard basis
vector in $\bR^d$.
\item \label{item:def2.5vi}
We define the generalized symmetrized surface Jacobian
\[
\mat D_s(\vec f) = \tfrac12\,\mat P_\Gamma
\,(\nabs\,\vec f + (\nabs\,\vec f)^\transT)\,\mat P_\Gamma 
\qquad\text{on }\Gamma\,,
\]
where $\mat P_\Gamma$ is the orthogonal projection onto the tangent space of 
$\Gamma$. I.e.\ 
\begin{equation} \label{eq:PGamma}
\mat P_\Gamma(\vec p) = \sum_{i=1}^n \vec\tau_i \otimes \vec\tau_i
\end{equation}
for an orthonormal basis $\{\vec\tau_1,\ldots,\vec\tau_n\}$ 
of $\tanspace_{\vec p}\,\Gamma$.
\item \label{item:def2.5tensor}
The surface divergence of a tensor $\mat f \in [C^1(\Gamma)]^{d\times d}$
is defined via \arxivyesno{}{\linebreak}%
$(\nabs\,.\,\mat f)\,.\,\vec\ek_i = \nabs\,.\,(\mat f^\transT\,\vec\ek_i)$
on $\Gamma$, for $i=1,\ldots,d$. 
\end{enumerate}
\end{definition}

\begin{rem} \label{rem:2.6}
\rule{0pt}{0pt}
\begin{enumerate}
\item \label{item:rem2.6i}
It is easy to check that {\rm Definition~\ref{def:2.5}\ref{item:def2.5i}}
does not depend on the choice of the parameterization $\vec y$.
Parameterizations $\vec y$ of curves, as $1$-dimensional surfaces,
are always assumed to be regular
in the sense that $\vec y$ is continuously differentiable and such
that $\vec y\,'\ne \vec 0$ holds everywhere, 
recall {\rm Definition~\ref{def:hypersurf}}. 
A curve is an equivalence class of regular parameterizations, 
see e.g.\ {\rm \citet[p.~8]{Kuhnel15}}, 
where the equivalence relation is given by parameter transformations. 
For simplicity, from now on, we will often also call a
parameterization $\vec y$ a curve. 
\item 
It is easy to check that {\rm Definitions~\ref{def:2.5}\ref{item:def2.5ii}},
{\rm \ref{def:2.5}\ref{item:def2.5iii}} and
{\rm \ref{def:2.5}\ref{item:def2.5v}} 
do not depend on the choice of the orthonormal basis 
of $\tanspace_{\vec p}\,\Gamma$.
\item \label{item:2.6vi} 
For $\vec p\in\Gamma$ it holds that 
\begin{alignat*}{2}
(\nabs\,\vec f)(\vec p)\,\vec\tau & = (\partial_{\vec\tau}\,\vec f)(\vec p) 
&& \quad\forall\ \vec\tau\in \tanspace_{\vec p}\,\Gamma \\
\quad\text{and}\quad
(\nabs\,\vec f)(\vec p)\,\vec\vvv & = \vec 0
&&\quad\forall\ \vec\vvv \in \normalspace_{\vec p}\,\Gamma\,.
\end{alignat*}
Moreover, it holds that the $i$-th row of $\nabs\,\vec f$ is
the surface gradient of the $i$-th component of $\vec f$,
i.e.\ $(\nabs\,\vec f)^\transT\,\vec\ek_i = \nabs\,(\vec f\,.\,\vec\ek_i)$,
$i = 1,\ldots,d$.
\item \label{item:Dsf}
We remark that the second projection $\mat P_\Gamma$ in 
{\rm Definition~\ref{def:2.5}\ref{item:def2.5vi}} ensures that
\[
(\mat D_s(\vec f))(\vec p)\,\vec\vvv = \vec 0 
\quad\forall\ \vec\vvv \in \normalspace_{\vec p}\,\Gamma\,,\ 
\vec p \in \Gamma\,, 
\]
similarly to \ref{item:2.6vi}.
The first projection then ensures that $(\mat D_s(\vec f))(\vec p)$ 
is symmetric.
\item \label{item:Pnabs}
It follows directly from {\rm Definitions~\ref{def:2.5}\ref{item:def2.5ii}},
{\rm \ref{def:2.5}\ref{item:def2.5v}} and \eqref{eq:PGamma} that
\[
\nabs\,f = \mat P_\Gamma\,\nabs\,f \quad\text{and}\quad
\nabs\,\vec f = (\nabs\, \vec f)\,\mat P_\Gamma \qquad\text{on }\Gamma\,.
\]
\item \label{item:PGamma}
We observe that if $\Gamma$ is a hypersurface, then
\[
\mat P_\Gamma = \mat\Id - \vec\nu\otimes\vec\nu \qquad\text{on }\Gamma\,,
\]
with $\vec\nu(p)$ denoting a unit vector in $\normalspace_{\vec p}\,\Gamma$
for $\vec p \in \Gamma$.
\end{enumerate}
\end{rem}

The following product rules and identities hold.
\begin{lem} \label{lem:productrules}
Let $\Gamma\subset\bR^d$ be an $n$-dimensional $C^1$--surface, and let
$f:\Gamma\to\bR$, $\vec f, \vec g:\Gamma\to\bR^d$
and $\mat f:\Gamma \to \bR^{d\times d}$ be $C^1$--functions.
Then it holds:
\rule{0pt}{0pt}
\begin{enumerate}
\item \label{item:pri}
\[
\nabs\,.\,(f\,\vec g) = \vec g\,.\,\nabs\,f + f\,\nabs\,.\,\vec g
\qquad\text{on } \Gamma\,,
\]
\item \label{item:prii}
\[
\nabs\,(\vec f\,.\,\vec g) = (\nabs\,\vec f)^\transT\,\vec g + (\nabs\,\vec
g)^\transT\,\vec f
\qquad\text{on } \Gamma\,,
\]
\item \label{item:priii}
\[
\nabs\,(f\,\vec g) = \vec g \otimes \nabs\,f+ f\,\nabs\,\vec g
\qquad\text{on } \Gamma\,,
\]
\item \label{item:priv}
\[
\nabs\,.\,(\mat f^\transT\,\vec g) = 
\vec g\,.\left(\nabs\,.\,\mat f\right) + \mat f : \nabs\,\vec g
\qquad\text{on } \Gamma\,,
\]
\item \label{item:trnabsf}
and
 \[
 \tr (\nabs\,\vec f) = \nabs\,.\,\vec f \qquad\text{on } \Gamma\,.
\]
\end{enumerate}
\begin{proof}
\ref{item:pri}--\ref{item:priv} are direct analogues of their flat
counterparts, and follow easily from Definition~\ref{def:2.5}.
\\
\ref{item:trnabsf}
Using the fact that the trace is invariant under basis changes, 
we obtain from {\rm Remark~\ref{rem:2.6}\ref{item:2.6vi}} and
Definition~\ref{def:2.5}\ref{item:def2.5iii} that for $\vec p \in \Gamma$
\[
\tr\left((\nabs\,\vec f)(\vec p)\right)
= \sum_{i=1}^n \vec\tau_i\,.\,(\nabs\,\vec f)(\vec p)\,\vec\tau_i
= \sum_{i=1}^n \vec\tau_i\,.\,(\partial_{\vec\tau_i}\,\vec f)(\vec p)
= (\nabs\,.\,\vec f)(\vec p)\,,
\]
where $\{\vec\tau_1,\ldots,\vec\tau_n\}$ is an orthonormal basis of 
$\tanspace_{\vec p}\,\Gamma$.
\end{proof}
\end{lem}

It is often convenient to consider alternative representations of the
differential operators introduced in Definition~\ref{def:2.5}.

\begin{rem} \label{rem:altforms}
\rule{0pt}{0pt}
\begin{enumerate}
\item \label{item:altformg}
For a local parameterization $\localpara : U\to\Gamma$ one can
define the first fundamental form, or metric tensor, 
$(g_{ij})_{i,j=1,\ldots,n}$ as
\begin{equation*}
g_{ij} = \vec\partial_i\,.\,\vec\partial_j =
 \partial_i\,\localpara\,.\,\partial_j\,\localpara\,,
\end{equation*}
and a little linear algebra shows that
\begin{align*}
(\nabs\, f) \circ \localpara
&= \sum^n_{i,j=1} g^{ij}\,\partial_i\,(f\circ\localpara)\,\vec\partial_j
\qquad\text{in } U\,,\\
(\nabs\,.\,\vec f)\circ\localpara
&= \sum^n_{i,j=1}
g^{ij}\,\partial_i\,(\vec f\circ\localpara)\,.\,\vec\partial_j
\qquad\text{in } U\,,
\end{align*}
where $(g^{ij})_{i,j=1,\ldots,n}$ is the inverse of
$(g_{ij})_{i,j=1,\ldots,n}$. Moreover, it holds that
\begin{equation*}
(\Delta_s\, f) \circ \localpara = \frac{1}{\sqrt{g}} \sum^n_{i,j=1} \partial_i\,
(\sqrt{g}\, g^{ij}\,\partial_j\,(f \circ \localpara) )
\qquad\text{in } U\,,
\end{equation*}
where
\begin{equation} \label{eq:detg}
g=\det\left((g_{ij})_{i,j=1,\ldots,n}\right)
\end{equation}
is the square of the local area element, see {\rm\citet[\S{}XI.6]{AmannE09}}.
\item \label{item:partialsd}
Denoting by 
$\nabs = (\partial_{s_1}, \ldots, \partial_{s_d})^\transT$ the surface 
gradient on $\Gamma$, it holds that 
$\nabs\,.\,\vec f = \sum_{i=1}^d \partial_{s_i}\,(\vec f\,.\,\vec\ek_i)$,
$\nabs\,\vec f = ( \partial_{s_j}\,\vec f\,.\,\vec\ek_i )_{i,j=1}^d$ and 
$\nabs^2\,f = ( \partial_{s_j}\partial_{s_i}\,f$ $)_{i,j=1}^d$.
Moreover, it follows from {\rm Lemma~\ref{lem:productrules}\ref{item:trnabsf}}
that the Laplace--Bel\-trami operator satisfies 
$\Delta_s = \sum_{i=1}^d \partial_{s_i}^2$.
\item \label{item:eq:grad}
If $\Gamma$ is an orientable hypersurface with normal vector field
$\vec\nu$, we obtain
for an extension of the functions $f$ and $\vec f$ to an open
neighbourhood of $\Gamma$ the formulas
\begin{equation*} 
\nabs\,f= \nabla\, f - (\nabla\,f\,.\,\vec\nu)\,\vec\nu
\quad\text{and}\quad
\nabs\, .\,\vec f =
\nabla\, .\,\vec f - ((\nabla\,\vec f)\,\vec\nu)\,.\,\vec\nu
\qquad\text{on }\Gamma\,,
\end{equation*}
where $\nabla$ and $\nabla\,.$ are the gradient and the divergence in
$\bR^d$. These identities follow directly from
{\rm Definition~\ref{def:2.5}\ref{item:def2.5ii}},\ref{item:def2.5iii} 
and the representation of $\nabla\, f$ and
$\nabla\,.\,\vec f$ with the help of an orthonormal basis.
\end{enumerate}
\end{rem}

\begin{lem} \label{lem:nabsid}
Let $\Gamma$ be a $C^1$--hypersurface.
Then the identity map $\vec\id$ satisfies the following.
\begin{enumerate}
\item \label{item:nabsid}
\begin{equation*}
\nabs\,\vec\id = \mat P_\Gamma 
\qquad\text{on } \Gamma\,.
\end{equation*}
\item \label{item:nabsidnabsf}
For $\vec f \in [C^1(\Gamma)]^d$ it holds that
\[
\nabs\,\vec\id : \nabs\,\vec f = \nabs\,.\,\vec f 
\qquad\text{on } \Gamma\,,
\]
where $\mat A:\mat B=\tr (\mat A^\transT\,\mat B)$ 
denotes the Hilbert--Schmidt inner product for matrices $\mat A$ and $\mat B$.
\item \label{item:Deltasid}
If $\Gamma$ is an orientable $C^2$--surface
with normal vector field $\vec\nu$, then
\[
\Delta_s\,\vec\id = - \,(\nabs\,.\,\vec\nu)\, \vec\nu 
\qquad\text{on } \Gamma\,.
\]
\end{enumerate}
\end{lem}
\begin{proof}
\ref{item:nabsid}
The result follows from 
Definitions~\ref{def:2.5}\ref{item:def2.5v}, \ref{def:2.5}\ref{item:def2.5i} %
and \eqref{eq:PGamma}.
\\
\ref{item:nabsidnabsf}
It follows from \ref{item:nabsid}, Remark~\ref{rem:2.6}\ref{item:2.6vi}
and Lemma~\ref{lem:productrules}\ref{item:trnabsf} that
$\nabs\,\vec\id : \nabs\,\vec f = \tr (\nabs\,\vec f) = \nabs\,.\,\vec f$.
\\
\ref{item:Deltasid}
Using Remark~\ref{rem:altforms}\ref{item:eq:grad}, and on recalling 
{\rm Definition~\ref{def:2.5}\ref{item:def2.5LBd}},\ref{item:def2.5LB},
we compute for $i=1,\ldots,d$ 
\begin{align*}
(\Delta_s\,\vec\id)\,.\,\vec\ek_i & =
\Delta_s\,(\vec\id\,.\,\vec\ek_i) = 
\nabs\,.\left[\nabs\,(\vec\id\,.\,\vec\ek_i)\right] = \nabs\,.\left[
\vec\ek_i-(\vec\ek_i\,.\,\vec\nu)\,\vec\nu\right] 
\nonumber \\ &
= -\nabs\,.\left[(\vec\nu\,.\,\vec\ek_i)\,\vec\nu\right]
= -(\vec\nu\,.\,\vec\ek_i)\,\nabs\,.\,\vec\nu 
\qquad\text{on } \Gamma\,,
\end{align*}
where we have noted that $\nabs\,(\vec\nu\,.\,\vec\ek_i)$ 
is orthogonal to $\vec\nu$.
\end{proof}

\subsection{Curvature} \label{subsec:curvature}

We now define the fundamental curvature quantities
needed for the material presented in the remainder of this 
\arxivyesno{work}{chapter}.

\newcommand{\onbtau}{\vec{\mathfrak t}}
\begin{definition} \label{def:Wp}
Let $\Gamma$ be an orientable $C^2$--hypersurface
with normal vector field $\vec\nu$.
\begin{enumerate}
\item \label{item:Wpi}
The Weingarten map at $\vec p\in\Gamma$ is defined through
\begin{equation*}
\Wein_{\vec p} : \tanspace_{\vec p}\,\Gamma \to \tanspace_{\vec p}\,\Gamma\,,\quad
\Wein_{\vec p}(\vec\tau) = -(\partial_{\vec\tau}\,\vec\nu)(\vec p)\,.
\end{equation*}
\item 
The corresponding bilinear form is called the second
fundamental form and is given by
\begin{equation*}
\II_{\vec p} (\vec\tau_1,\vec\tau_2) = \Wein_{\vec p}(\vec\tau_1)\,.\,\vec\tau_2 
= -(\partial_{\vec\tau_1}\,\vec\nu)(\vec p)\,.\,\vec\tau_2
\end{equation*}
for all $\vec\tau_1,\vec\tau_2 \in \tanspace_{\vec p}\,\Gamma$,
$\vec p\in\Gamma$.
It can be shown that $\Wein_{\vec p}$ is self-adjoint, which implies
that $\II_{\vec p}$ is symmetric, see e.g.\ {\rm\citet[\S3.9]{Kuhnel15}}.
Hence there exists an orthonormal basis
$\{\onbtau_1,\ldots,\onbtau_{d-1}\}$ of $\tanspace_{\vec p}\,\Gamma$, 
consisting of eigenvectors of $\Wein_{\vec p}$ with corresponding eigenvalues
$\varkappa_1,\ldots,\varkappa_{d-1}$.
\end{enumerate}
\end{definition}

\begin{definition} \label{def:2.8}
\rule{0pt}{0pt}
\begin{enumerate}
\item The eigenvalues $\varkappa_1,\ldots,\varkappa_{d-1}$ of
$\Wein_{\vec p}$ are called the principal curvatures of $\Gamma$ at $\vec p$.
\item \label{item:def2.8ii}
The mean curvature $\varkappa$ of $\Gamma$ at $\vec p$ is
defined to be the trace of $\Wein_{\vec p}$, i.e.\
\begin{equation*}
\varkappa(\vec p) = \tr \Wein_{\vec p}=\varkappa_1+\ldots+\varkappa_{d-1}\,.
\end{equation*}
\item \label{item:def2.8iii}
The mean curvature vector $\vec\varkappa$ of $\Gamma$ at $\vec p$
is defined as
\begin{equation*}
\vec\varkappa(\vec p)=\varkappa(\vec p)\,\vec\nu(\vec p)\,.
\end{equation*}
\item \label{item:def2.8iv}
For $d=3$ the Gaussian curvature $\Gauss$ at $\vec p$ is the
determinant of $\Wein_{\vec p}$, i.e.\ the product of the principal
curvatures. This means we set
\begin{equation*}
\Gauss(\vec p) = \det \Wein_{\vec p} = \varkappa_1\,\varkappa_2\,.
\end{equation*}
\end{enumerate}
\end{definition}

\begin{lem} \label{lem:nabsnu}
Let $\Gamma$ be an orientable $C^2$--hypersurface with normal vector field
$\vec\nu$.
The Jacobian $(\nabs\,\vec\nu)(\vec p)$, for $\vec p \in \Gamma$, 
induces a self-adjoint linear map from $\bR^d$ to $\bR^d$ that collapses to
$-\Wein_{\vec p}$ on $\tanspace_{\vec p}\,\Gamma$ and that maps 
$\normalspace_{\vec p}\,\Gamma$ to zero. As a consequence, the following hold:
\begin{enumerate}
\item \label{item:nabsnunu}
\[
(\nabs\,\vec\nu)^\transT\,\vec\nu = (\nabs\,\vec\nu)\,\vec\nu = \vec 0
 \qquad\text{on }\Gamma\,,
\]
\item \label{item:nabsnuT}
\[
(\nabs\,\vec\nu)^\transT = \nabs\,\vec\nu  \qquad\text{on }\Gamma\,,
\]
\item \label{item:trnabsnu}
\[
\varkappa = -\tr(\nabs\,\vec\nu)
\qquad\text{on }\Gamma\,,
\]
\item \label{item:nabsnu2}
and
\[
|\nabs\,\vec\nu|^2 = \nabs\,\vec\nu:\nabs\,\vec\nu =
\varkappa_1^2 + \cdots + \varkappa_{d-1}^2
\qquad\text{on }\Gamma\,.
\]
\end{enumerate}
\end{lem}
\begin{proof}
Let $\vec p \in \Gamma$.
It follows from {\rm Remark~\ref{rem:2.6}\ref{item:2.6vi}} and 
{\rm Definition~\ref{def:Wp}} that
$\Wein_{\vec p}(\vec\tau) = $ \linebreak $- (\nabs\,\vec\nu)(\vec p)\,\vec\tau
= -\left((\nabs\,\vec\nu)(\vec p)\right)^\transT \vec\tau$ for all 
$\vec\tau \in \tanspace_{\vec p}\,\Gamma$.
Moreover, on 
denoting $\vec\nu = (\nu_1,\ldots,\nu_d)^\transT$, 
and on noting from Definition~\ref{def:2.5}\ref{item:def2.5v}
that the $i$-th column of $(\nabs\,\vec\nu)^\transT$ is 
$\nabs\,\nu_i$, we have 
that
\[
(\nabs\,\vec\nu)^\transT\,\vec\nu
= \sum_{i=1}^d \nu_i\,\nabs\,\nu_i 
= \tfrac12\,\nabs\,|\vec\nu|^2 = 0\,.
\]
Combining this with Remark~\ref{rem:2.6}\ref{item:2.6vi}
yields that 
\arxivyesno{$(\nabs\,\vec\nu)(\vec p)\,\vec\nu = 
\left((\nabs\,\vec\nu)(\vec p)\right)^\transT \vec\nu = \vec 0$,}
{$(\nabs\,\vec\nu)(\vec p)\,\vec\nu = 
\left((\nabs\,\vec\nu)(\vec p)\right)^\transT \vec\nu $ $= \vec 0$,}
and so the linear map induced by $(\nabs\,\vec\nu)(\vec p)$ is self-adjoint.
This proves \ref{item:nabsnunu} and \ref{item:nabsnuT}. Furthermore, if
$\{\onbtau_1,\ldots,\onbtau_{d-1}\}$ 
is an orthonormal basis of $\tanspace_{\vec p}\,\Gamma$, consisting of
eigenvectors of $\Wein_{\vec p}$ with corresponding eigenvalues
$\varkappa_1,\ldots,\varkappa_{d-1}$, then
$\{\onbtau_1,\ldots,\onbtau_{d-1},\vec\nu(\vec p)\}$ is an orthonormal basis of
$\bR^d$, consisting of
eigenvectors of $(\nabs\,\vec\nu)(\vec p)$ with corresponding eigenvalues
$\varkappa_1,\ldots,\varkappa_{d-1},0$. This implies \ref{item:trnabsnu},
on recalling Definition~\ref{def:2.8}\ref{item:def2.8ii}, 
and \ref{item:nabsnu2}.
\end{proof}

The next lemma shows that while the sign of the mean curvature $\varkappa$
depends on the choice of the  unit normal $\vec\nu$, the mean curvature vector 
$\vec\varkappa$ is an invariant under the change of the sign of the normal.

\begin{lem} \label{lem:varkappa}
Let $\Gamma$ be an orientable $C^2$--hypersurface with normal vector field
$\vec\nu$. The following formulas for the mean curvature and the mean
curvature vector hold true.
\begin{enumerate}
\item \label{item:varkappa}
For the mean curvature it holds that
\begin{equation*}
\varkappa = -\nabs\,.\,\vec\nu \qquad\text{on } \Gamma\,.
\end{equation*}
\item \label{item:vecvarkappa}
For the mean curvature vector it holds that
\begin{equation*}
\vec\varkappa = \varkappa\,\vec\nu =\Delta_s\,\vec\id
\qquad\text{on } \Gamma\,.
\end{equation*}
\end{enumerate}
\end{lem}

\begin{proof}
\ref{item:varkappa} This follows from Lemma~\ref{lem:nabsnu}\ref{item:trnabsnu} and Lemma~\ref{lem:productrules}\ref{item:trnabsf}.
\\
\ref{item:vecvarkappa}
The result follows directly from \ref{item:varkappa},
Definition~\ref{def:2.8}\ref{item:def2.8iii} and
\arxivyesno{Lemma~\ref{lem:nabsid}\ref{item:Deltasid}.}
{Lemma\linebreak \ref{lem:nabsid}\ref{item:Deltasid}.}
\end{proof}

\begin{lem} \label{lem:divf} 
Let $\Gamma$ be an orientable $C^2$--hypersurface with normal vector field
$\vec\nu$. Let $\vec f \in [C^1(\Gamma)]^d$ and set
$\vec f_\subT = \mat P_\Gamma\,\vec f$ on $\Gamma$. Then it holds that
\begin{enumerate}
\item \label{item:divfi}
\begin{equation*}
\nabs\,.\,\vec f
= -(\vec f\,.\,\vec\nu)\,\varkappa + \nabs\,.\,
\vec f_\subT\qquad\text{on }\Gamma\,,
\end{equation*}
\item \label{item:divfii}
\begin{equation*} 
\nabs\,\vec f = 
(\vec f\,.\,\vec\nu)\,\nabs\,\vec\nu
+\vec\nu\otimes\nabs\,(\vec f\,.\,\vec\nu)
+\nabs\,\vec f_\subT\qquad\text{on } \Gamma\,,
\end{equation*}
\item \label{item:divfiii}
\begin{equation*}
 \mat D_s(\vec f) = (\vec f\,.\,\vec\nu)\,\nabs\,\vec\nu
+\tfrac12\, (\mat P_\Gamma\,\nabs\,\vec f_\subT
+(\nabs\,\vec f_\subT)^\transT \mat P_\Gamma)
\qquad\text{on } \Gamma\,.
\end{equation*}
\end{enumerate}
\end{lem}
\begin{proof}
\ref{item:divfi}
Using Lemma~\ref{lem:productrules}\ref{item:pri},
Definition~\ref{def:2.5}\ref{item:def2.5ii} and
Lemma~\ref{lem:varkappa}\ref{item:varkappa}, we compute that
\[
\nabs\,.\,\vec f
=\nabs\,.\left((\vec f\,.\,\vec\nu)\,\vec\nu\right) 
+ \nabs\,.\,\vec f_\subT = (\vec f\,.\,\vec\nu)
\,\nabs\,.\,\vec\nu +\nabs\,.\,\vec f_\subT 
= - (\vec f\,.\,\vec\nu)\,\varkappa + \nabs\,.\,\vec f_\subT\,.
\]
\ref{item:divfii} 
Similarly, it follows from Lemma~\ref{lem:productrules}\ref{item:priii} that
\[
\nabs\,\vec f
=\nabs\left((\vec f\,.\,\vec\nu)\,\vec\nu\right) 
+ \nabs\,\vec f_\subT 
= (\vec f\,.\,\vec\nu)\,\nabs\,\vec\nu
+\vec\nu\otimes\nabs\,(\vec f\,.\,\vec\nu)
+\nabs\,\vec f_\subT\,.
\]
\ref{item:divfiii} 
The desired result follows immediately from \ref{item:divvii},
{\rm Definitions~\ref{def:2.5}\ref{item:def2.5v}}, \arxivyesno{}{\linebreak}%
\ref{def:2.5}\ref{item:def2.5vi} and the fact that
$\nabs\,\vec\nu$ is a symmetric mapping that maps tangent
vectors to tangent vectors, recall Lemma~\ref{lem:nabsnu}.
\end{proof}

In the derivation of some relevant formulas it is sometimes helpful to
be able to extend functions defined on a hypersurface $\Gamma$ to a
neighbourhood of $\Gamma$. This then frequently allows us to use the
calculus in $\bR^d$ to compute for quantities on the surface,
recall Remark~\ref{rem:altforms}\ref{item:eq:grad}. 
Let $\Gamma\subset\bR^d$ be a compact orientable $C^2$--hypersurface 
without boundary and with normal vector field $\vec\nu$.
For $\delta>0$ we define a tubular neighbourhood
\begin{equation*} 
\mathcal{N}_\delta = \{\vec z\in\bR^d : \vec z = \vec p +
\eta\,\vec\nu(\vec p),\ \vec p\in\Gamma,\ |\eta|<\delta\}
\end{equation*}
of $\Gamma$. In \citet[Appendix 14.6]{GilbargT83} it is shown that
there is a bijective relation between $\vec z\in \mathcal{N}_\delta$ and
$(\vec p,\eta)\in \Gamma\times(-\delta,\delta)$, provided that $\delta$
is small enough. We can hence define the functions 
$\vec\Pi_\Gamma(\vec z) = \vec p$ 
and $d_{\Gamma}(\vec z) = \eta$, 
and it turns out that $|d_\Gamma(\vec z)|$ is the distance of
$\vec z$ to $\Gamma$, where $d_\Gamma(\vec z)$ is positive if $\vec z$ lies
on the side towards which $\vec\nu$ is pointing, and negative
otherwise. We call $d_\Gamma$ the signed distance function to $\Gamma$,
and $\vec\Pi_\Gamma(\vec z)$ is the projection of $\vec z$ onto $\Gamma$, 
i.e.\ 
\begin{equation*}
\vec\Pi_\Gamma(\vec z) = \argmin_{\vec y\in\Gamma} |\vec y-\vec z|\,.
\end{equation*}
For later use, we recall from 
\citet[Appendix 14.6]{GilbargT83} that, for $\delta$ sufficiently small, 
\begin{subequations} \label{eq:dGamma}
\begin{equation} \label{eq:dGammaa}
d_\Gamma \in C^2(\mathcal{N}_\delta)\,.
\end{equation}
It also holds that
\begin{equation} \label{eq:graddist}
\nabla\,d_\Gamma = \vec\nu \circ \vec\Pi_\Gamma \quad 
\Rightarrow \quad\left|\nabla\,d_\Gamma\right|=1 
\quad\text{in } \mathcal{N}_\delta\,,
\end{equation}
\end{subequations}
which can be shown as follows. It clearly holds that
\[
\vec\id =\vec\Pi_\Gamma + d_\Gamma\,\vec\nu \circ \vec\Pi_\Gamma
\quad\text{and} \quad
d_\Gamma = \left(\vec\id - \vec\Pi_\Gamma\right) 
.\left(\vec\nu \circ \vec\Pi_\Gamma\right) 
\quad\text{in } \mathcal{N}_\delta\,.
\]
Differentiating the second identity, and observing the first, yields 
for $k=1,\ldots,d$
\begin{align*}
\partial_k\, d_\Gamma & = \left(\vec\ek_k - \partial_k\,\vec\Pi_\Gamma
\right) . \left(\vec\nu \circ \vec\Pi_\Gamma\right)
- \left(\vec\id - \vec\Pi_\Gamma\right) .\,
 \partial_k \left(\vec\nu \circ \vec\Pi_\Gamma\right) \nonumber \\ & 
= \vec\ek_k\,.\,\vec\nu \circ \vec\Pi_\Gamma
- d_\Gamma\left(\vec\nu \circ \vec\Pi_\Gamma\right) . \,
 \partial_k \left(\vec\nu \circ \vec\Pi_\Gamma\right) \nonumber \\ & 
= \vec\ek_k\,.\,\vec\nu \circ \vec\Pi_\Gamma
- d_\Gamma\,\tfrac12\,\partial_k \left|\vec\nu \circ \vec\Pi_\Gamma\right|^2
= \vec\ek_k\,.\,\vec\nu \circ \vec\Pi_\Gamma
\quad\text{in } \mathcal{N}_\delta\,,
\end{align*}
where we have noted that $\partial_k\,\vec\Pi_\Gamma$
is tangential in $\mathcal{N}_\delta$.
This proves (\ref{eq:graddist}).

We now extend a function $f$ defined on $\Gamma$ to $\mathcal{N}_\delta$, 
for $\delta$ sufficiently small, via 
\begin{equation} \label{eq:extendf}
f = f \circ \vec\Pi_\Gamma
\quad\text{in } \mathcal{N}_\delta\,.
\end{equation}
This means that we extend $f$ constantly in the normal direction, and so
we obtain from Remark~\ref{rem:altforms}\ref{item:eq:grad} that
\begin{equation} \label{eq:nablanabs}
\nabla\,f \,.\,\vec\nu = 0\,, \quad\text{and hence}\quad
\nabla\, f = \nabs\, f \qquad\text{on } \Gamma\,.
\end{equation}

In the flat case the Schwarz theorem yields that the
Hessian is symmetric. In the curved case, however, 
this is no longer the case. In fact, the following result holds.

\begin{lem} \label{lem:surfhess}
Let $\Gamma$ be an orientable $C^2$--hypersurface with normal vector field
$\vec\nu$, and let $f \in C^2(\Gamma)$.
For the surface Hessian it holds that
\begin{equation*}
\nabs^2\,f - (\nabs^2\,f)^\transT =
[(\nabs\,\vec\nu)\,\nabs\,f]\otimes
\vec\nu-\vec\nu\otimes[(\nabs\,\vec\nu)\,\nabs\,f] 
\qquad\text{on } \Gamma\,.
\end{equation*}
\end{lem}
\begin{proof}
Recalling from Remark~\ref{rem:altforms}\ref{item:partialsd}
the notation $\nabs = (\partial_{s_1}, \ldots, \partial_{s_d})^\transT$, 
and denoting $\vec\nu = (\nu_1,\ldots,\nu_d)^\transT$, 
the claim can be equivalently written as 
\begin{equation} \label{eq:partialsij}
\partial_{s_j}\,\partial_{s_i}\,f - 
\partial_{s_i}\,\partial_{s_j}\,f =
[(\nabs\,\vec\nu)\,\nabs\,f]_i\,\nu_j
- [(\nabs\,\vec\nu)\,\nabs\,f]_j\,\nu_i \qquad\text{on } \Gamma\,,
\end{equation}
for all $i,j \in\{1,\ldots,d\}$.
In order to prove this,
we extend $f$ to a neighbourhood of $\Gamma$ as in (\ref{eq:extendf}).
It follows from (\ref{eq:nablanabs}) that 
$\partial_{s_i}(\nabla\, f \,.\,\vec\nu)=0$ on $\Gamma$,
and so we compute, using Remark~\ref{rem:altforms}\ref{item:eq:grad},
for all $i,j=1,\ldots,d$ that
\[
\partial_{s_i}\,\partial_{s_j}\,f = \partial_{s_i}\, 
(\partial_j\,f - (\nabla\,f \,.\,\vec\nu)\,\nu_j)
= \partial_{s_i}\,\partial_j\,f
= \partial_i\,\partial_j\,f - (\nabla\,\partial_j\,f\,.\,\vec\nu)\,\nu_i \,.
\]
In addition we have, on using the Schwarz theorem
and on extending $\vec\nu$ to the neighbourhood of $\Gamma$ similarly to
(\ref{eq:extendf}), 
and so (\ref{eq:nablanabs}) yields 
$\partial_j\,\vec\nu = \partial_{s_j}\,\vec\nu$,
that
\begin{align*}
0&= \partial_{s_j}\,(\nabla\,f\,.\,\vec\nu) 
=\partial_j\, (\nabla\,f \,.\,\vec\nu)
- \left( \nabla\, (\nabla\,f \,.\,\vec\nu)\,.\,\vec\nu\right) \nu_j \\
&= \nabla\,\partial_j\,f\,.\,\vec\nu  + \nabla\,f\,.\,\partial_j\,\vec\nu
- \left( \nabla\, (\nabla\,f \,.\,\vec\nu)\,.\,\vec\nu\right) \nu_j 
= \nabla\,\partial_j\,f\,.\,\vec\nu  + \nabs\,f\,.\,\partial_{s_j}\,\vec\nu
 \,.
\end{align*}
Combining the above with the Schwarz theorem and 
Lemma~\ref{lem:nabsnu}\ref{item:nabsnuT} yields that
\begin{align*}
\partial_{s_j}\,\partial_{s_i}\,f - \partial_{s_i}\,\partial_{s_j}\,f & 
= (\nabs\,f\,.\,\partial_{s_i}\,\vec\nu)\,\nu_j 
- (\nabs\,f\,.\,\partial_{s_j}\,\vec\nu)\,\nu_i \\ &
= (\nabs\,f\,.\,\nabs\,\nu_i)\,\nu_j - (\nabs\,f\,.\,\nabs\,\nu_j)\,\nu_i \\
& = [ (\nabs\,\vec\nu)\,\nabs\, f ]_i\,\nu_j -
[ (\nabs\,\vec\nu)\,\nabs\, f ]_j\,\nu_i\,.
\end{align*}
This proves (\ref{eq:partialsij}).
\end{proof}

\begin{lem} \label{lem:Deltasnu}
Let $\Gamma$ be an orientable $C^3$--hypersurface with normal vector field
$\vec\nu$. It holds that
\[
\nabs\,\varkappa = - \Delta_s\,\vec\nu - |\nabs\,\vec\nu|^2\,\vec\nu
\qquad\text{on } \Gamma\,.
\]
\end{lem}
\begin{proof}
Using the same notation as in the proof of Lemma~\ref{lem:surfhess},
it follows from 
Lemma \ref{lem:nabsnu}\ref{item:nabsnuT}, Lemma~\ref{lem:surfhess},
Lemma~\ref{lem:varkappa}\ref{item:varkappa} and 
Lemma~\ref{lem:nabsnu}\ref{item:nabsnunu} that, for $j=1,\ldots,d$,
\begin{align*}
\Delta_s \,\nu_j & = \sum_{i=1}^d \partial_{s_i}\,\partial_{s_i}\,\nu_j
= \sum_{i=1}^d \partial_{s_i}\,\partial_{s_j}\,\nu_i \nonumber \\ &
= \sum_{i=1}^d \left(\partial_{s_j}\,\partial_{s_i}\,\nu_i
- [(\nabs\,\vec\nu)\,\nabs\,\nu_i]_i\,\nu_j
+ [(\nabs\,\vec\nu)\,\nabs\,\nu_i]_j\,\nu_i\right)  \nonumber \\ &
= \partial_{s_j}\,(\nabs\,.\,\vec\nu)
- \nu_j\,\sum_{i=1}^d |\nabs\,\nu_i|^2
+ (\nabs\,\nu_j)\,.\,\sum_{i=1}^d (\nabs\,\nu_i)\,\nu_i \nonumber \\ &
= \partial_{s_j}\,(\nabs\,.\,\vec\nu) - \nu_j\,|\nabs\,\vec\nu|^2
+ (\nabs\,\nu_j)\,.\,(\nabs\,\vec\nu)^\transT\,\vec\nu
= - \partial_{s_j}\,\varkappa - \nu_j\,|\nabs\,\vec\nu|^2\,.
\end{align*}
This yields the claim.
\end{proof}
  
\subsection{The divergence theorem}

\begin{definition} \label{def:boundary}
\rule{0pt}{0pt}
\begin{enumerate}
\item \label{item:def2.16i}
A subset $\Gamma\subset\bR^d$ is called an
$n$-dimensional $C^k$--surface with boundary,
for $1\leq n \leq d$ and $k\geq1$, 
if for each point $\vec p\in\Gamma$ one of the following conditions 
is satisfied.
\begin{enumerate}
\item \label{item:2.10i}
There exist an open neighbourhood $V\subset \bR^d$ of $\vec p$ and
a bijective $C^k$--map $\localpara:U\to V\cap\Gamma$, with
$U\subset\bR^n$ open and connected, such that the Jacobian
$\nabla\,\localpara : U \to \bR^{d\times n}$ of $\localpara$ has full rank
in $U$.
\item \label{item:2.10ii}
There exist open and connected sets $U\subset\bR^n$,
$V\subset\bR^d$, with $\vec0 \in U$ and $\vec p\in V$, and an injective
$C^k$--map $\localpara: U\to\bR^d$, such that $\localpara(\vec0) = \vec p$ 
and
\begin{equation*}
\localpara(U\cap (\bR^{n-1} \times \bRgeq)) = V \cap \Gamma \,.
\end{equation*}
\end{enumerate}
\item \label{item:2.10iii}
The maps $\localpara:U\to\bR^d$ in \ref{item:2.10i} and \ref{item:2.10ii} 
are called local parameterizations of $\Gamma$. In the latter case, we call
$\localpara$ a local boundary parameterization of $\Gamma$.
\item \label{item:2.10iv}
If $n=d-1$, then we call $\Gamma$ a hypersurface with boundary.
\end{enumerate}
\end{definition}

\begin{rem} 
\rule{0pt}{0pt}
\begin{enumerate}
\item 
A boundary point of $\Gamma$ is a point on
$\Gamma$ for which the second condition in the definition is
fulfilled. The set of all boundary points is called
the boundary of $\Gamma$ and is denoted by $\partial\Gamma$.
\item 
If $\Gamma$ is an $n$-dimensional surface with boundary,
its boundary $\partial\Gamma$ is either empty or an
$(n-1)$-dimensional surface without boundary, see
{\rm\citet[\S3.1]{AgricolaF02}} for details.
\item
If $\partial\Gamma$ is empty, we say that $\Gamma$ is 
a surface without boundary. If, in addition, $\Gamma$ is compact, then
we call $\Gamma$ a closed surface. Here we note that
{\rm Definition~\ref{def:boundary}} 
implies that any bounded hypersurface is compact.
\item
All the definitions and results in {\rm \S\ref{subsec:surfaces}} and
{\rm \S\ref{subsec:curvature}} have been stated for $\vec p \in \Gamma
\setminus\partial\Gamma$. However, they easily generalize to
$\vec p \in \partial\Gamma$ for a surface with boundary. The only required
changes are as follows. In {\rm Definition~\ref{def:Ckf}}, $U$ is replaced by
$U \cap (\bR^{n-1}\times \bRgeq)$
for local boundary parameterizations $\localpara$, and similarly in
{\rm Remark~\ref{rem:altforms}\ref{item:altformg}}. 
In addition, in {\rm Definition~\ref{def:2.5}\ref{item:def2.5i}}, 
we choose a curve $\vec y : (-1,0] \to \Gamma$, or
$\vec y : [0,1) \to \Gamma$, and take the natural one-sided limit
in the definition of $(\partial_{\vec\tau}\, f)(\vec p)$.
\end{enumerate}
\end{rem}

\begin{definition} \label{def:int}
Let $\Gamma$ be an $n$-dimensional $C^1$--surface in $\bR^d$,
$f:\Gamma \to\bR$ a function and let $\localpara : U \to \Gamma$ be a
local parameterization of $\Gamma\setminus\partial\Gamma$. 
Then, on recalling \eqref{eq:detg}, we define
\begin{equation}\label{eq:int}
\int_{\localpara(U)} f \dH{n} = \int_U f\circ\localpara\,\sqrt{g} \dL{n}\,,
\end{equation}
for all functions $f$ such that $f\circ\localpara$ is
integrable, where $\mathcal{L}^n$ is the $n$-dimensional Lebesgue measure.
For a local boundary parameterization $\localpara$ of $\Gamma$, 
we replace $U$ with
$U \cap (\bR^{n-1}\times \bRgeq)$ in {\rm (\ref{eq:int})}.
The integral $\int_\Gamma f \dH{n}$ is defined using a partition of unity and 
the definition \eqref{eq:int}, see {\rm\citet[\S{}XII.1]{AmannE09}} for details.
The induced measure, defined by $\mathcal{H}^{n}(\Gamma) = \int_\Gamma
1\dH{n}$, is called the $n$-dimensional Hausdorff measure in $\bR^d$.
\end{definition}

\begin{definition} \label{def:conormal}
For a $C^1$--hypersurface $\Gamma$ with boundary it
holds that for $\vec p\in\partial\Gamma$ the tangent space
$\tanspace_{\vec p}\,\partial\Gamma$ is $(d-2)$-dimensional,
$\tanspace_{\vec p}\,\Gamma$ is $(d-1)$-dimensional and
$\tanspace_{\vec p}\,\partial\Gamma\subset \tanspace_{\vec p}\,\Gamma$. 
We can hence choose a unique vector $\conormal(\vec p)
\in \tanspace_{\vec p}\,\Gamma$, 
which we call the outer unit conormal, such that
\begin{enumerate}
\item 
$|\conormal(\vec p)|=1$,
\item \label{item:conormalii}
$\conormal(\vec p) \in \normalspace_{\vec p}\,\partial\Gamma$,
\item 
there exists a curve $\vec y:(-1,0]\to\Gamma$ on $\Gamma$ 
with $\vec y(0)=\vec p$ and $\vec y\,'(0)=\conormal(\vec p)$.
\end{enumerate}
\end{definition}

We will frequently use the following generalization of the divergence
theorem on hypersurfaces.

\begin{thm} \label{thm:div}
Let $\Gamma$ be a compact orientable $C^2$--hypersurface 
with normal vector field $\vec\nu$, and let $\vec f \in [C^1(\Gamma)]^d$. 
Then it holds that
\begin{equation*}
\int_\Gamma \nabs\,.\,\vec f+\varkappa\,\vec f\, .\,\vec\nu 
\dH{d-1} = \int_{\partial\Gamma} \vec f\,.\,\conormal \dH{d-2}\,,
\end{equation*}
where $\conormal$ is the outer unit conormal to $\partial\Gamma$.
\end{thm}
\begin{proof} 
The result is well-known for tangential vector fields $\vec f$, for which
$\varkappa\,\vec f\,.\,\vec\nu$ is not present. 
This special case can be shown on
surfaces similarly to the case of flat domains in $\bR^d$,
see e.g.\ \citet[\S{}XII.3]{AmannE09} or \citet[\S3.8]{AgricolaF02}. 
In the case of a nontangential vector field $\vec f$, we define the 
tangential vector field
\begin{equation*}
  \vec h = \vec f -(\vec f\,.\,\vec\nu)\,\vec\nu
\end{equation*}
and compute, on using Lemma~\ref{lem:productrules}\ref{item:pri}, 
$\nabs\,(\vec f\,.\,\vec\nu)$ being orthogonal to $\vec\nu$ and
Lemma~\ref{lem:varkappa}\ref{item:varkappa}, that
\[
  \nabs\, .\,\vec h =
     \nabs\,.\,\vec f - (\vec f\,.\,\vec\nu)\,\nabs\,.\,\vec\nu
  = \nabs\,.\,\vec f + \varkappa\,\vec f\,.\,\vec\nu\,.
\]
Using the divergence theorem for $\vec h$, taking into account that
$\vec h\,.\,\conormal = \vec f\,.\,\conormal$, 
we obtain the assertion of the theorem.
\end{proof}      

\begin{rem} \label{rem:ibp}
\rule{0pt}{0pt}
\begin{enumerate}
\item \label{item:ibp}
The following integration by parts rule is a direct consequence of 
{\rm Theorem~\ref{thm:div}}, 
{\rm Lemma~\ref{lem:productrules}\ref{item:pri}}
and {\rm Definition~\ref{def:2.5}\ref{item:def2.5ii}}.
Let $f \in C^2(\Gamma)$ and $\eta \in C^1(\Gamma)$. Then it holds that
\begin{equation*}
\int_\Gamma \eta\,\Delta_s\,f + \nabs\,f\,.\,\nabs\,\eta \dH{d-1} 
= \int_{\partial\Gamma} \eta\,\nabs\,f\,\,.\,\conormal \dH{d-2}\,.
\end{equation*}
\item \label{item:GG}
Moreover, for $f \in C^1(\Gamma)$ we obtain the Gauss--Green formula
\begin{equation*}
\int_\Gamma \nabs\,f + f\,\varkappa\,\vec\nu \dH{d-1} 
= \int_{\partial\Gamma} f\,\conormal \dH{d-2}\,,
\end{equation*}
by choosing $\vec f = f\,\vec\ek_i$, $i=1,\ldots,d$, in 
{\rm Theorem~\ref{thm:div}}, and applying 
{\rm Lemma~\ref{lem:productrules}\ref{item:pri}}. 
\item \label{item:ibpvec}
Let $\vec f \in [C^2(\Gamma)]^d$ and $\vec\eta \in [C^1(\Gamma)]^d$. 
Then it follows from \ref{item:ibp} 
and {\rm Remark~\ref{rem:2.6}\ref{item:2.6vi}} that
\[
\int_\Gamma \vec\eta\,.\,\Delta_s\,\vec f + \nabs\,\vec f : \nabs\,\vec\eta
 \dH{d-1} 
= \int_{\partial\Gamma} \vec\eta\,.\,(\nabs\,\vec f)\,\conormal \dH{d-2}\,.
\]
On noting that for symmetric matrices $\mat A \in \bR^{d \times d}$ it holds 
that $\mat P_\Gamma\,\mat A\,\mat P_\Gamma : \mat B = 
\mat P_\Gamma\,\mat A\,\mat P_\Gamma : 
\tfrac12\,\mat P_\Gamma\,(\mat B + \mat B^\transT)\,
\mat P_\Gamma$ on $\Gamma$ for all $\mat B \in \bR^{d \times d}$, we obtain
furthermore that
\[
\int_\Gamma \vec\eta\,.\left(\nabs\,.\,\mat D_s(\vec f)\right)
+ \mat D_s(\vec f) : \mat D_s(\vec\eta) \dH{d-1} 
= \int_{\partial\Gamma} \vec\eta\,.\,\mat D_s(\vec f)\,\conormal \dH{d-2}\,,
\]
where we have recalled {\rm Definition~\ref{def:2.5}\ref{item:def2.5vi}}
and used {\rm Theorem~\ref{thm:div}} with
{\rm Lemma~\ref{lem:productrules}\ref{item:priv}} and
{\rm Remark~\ref{rem:2.6}\ref{item:Dsf}}.
\item \label{item:weakvarkappa}
Of fundamental importance in the development of numerical approximations
for curvature driven evolution equations is the following identity.
Let $\vec\eta \in [C^1(\Gamma)]^d$. Then it follows from \ref{item:ibpvec},
{\rm Lemma~\ref{lem:varkappa}\ref{item:vecvarkappa}}, 
{\rm Lemma~\ref{lem:nabsid}\ref{item:nabsid}}
and {\rm Definition~\ref{def:conormal}} that
\begin{equation*}
\int_\Gamma \varkappa\,\vec\eta\,.\,\vec\nu + \nabs\,\vec\id: \nabs\,\vec\eta
 \dH{d-1} 
= \int_{\partial\Gamma} \vec\eta\,.\,\conormal \dH{d-2}\,.
\end{equation*}
\item \label{item:weakweingarten}
For numerical approximations of the Weingarten map, 
recall {\rm Definition \ref{def:Wp}\ref{item:Wpi}},
the following identity is of use. Let $\mat\eta \in [C^1(\Gamma)]^{d\times d}$.
Then 
it follows from {\rm Theorem~\ref{thm:div}},
with $\vec f = \mat\eta^\transT\,\vec\nu$, on recalling 
{\rm Lemma \ref{lem:productrules}\ref{item:priv}}, that
\[
\int_\Gamma \nabs\,\vec\nu : \mat\eta + 
\vec\nu\,.\,[\varkappa\,\mat\eta\,\vec\nu + 
\nabs\,.\,\mat\eta] \dH{d-1} 
= \int_{\partial\Gamma} \vec\nu\,.\,\mat\eta\,\conormal \dH{d-2}\, .
\]
\end{enumerate}
\end{rem}

\subsection{Evolving surfaces and transport theorems} \label{subsec:evolve}

We are mainly interested in evolving hypersurfaces. Hence we now
consider hypersurfaces which evolve in a time interval $[0,T]$ and
define the term evolving hypersurface.

\begin{definition} \label{def:GT}
\rule{0pt}{0pt}
\begin{enumerate}
\item 
Let $(\Gamma(t))_{t\in [0,T]}$ be a family of
$C^k$--hypersurfaces (with or without boundary), for $k\geq1$. The set
\begin{equation*}
\GT = \bigcup_{t\in[0,T]} (\Gamma(t)\times\{t\})
\end{equation*}
is called a $C^k$--evolving hypersurface if it is a $C^k$--hypersurface
with boundary in $\bR^{d+1}$, such that
$\tanspace_{(\vec p,t)}\,\GT \neq \bR^d\times\{0\}$ for all
$(\vec p,t) \in \GT$. We will often identify
$\GT$ with $(\Gamma(t))_{t\in [0,T]}$, and call the latter also a
$C^k$--evolving hypersurface.
\item \label{item:defV}
Let $\GT$ be a $C^1$--evolving hypersurface and
$(\vec p_0,t_0)\in \GT$. We assume that $\GT$ allows for a
continuous vector field $\vec\nu :\GT\to\bR^d$, such that
$\vec\nu(\cdot,t)$ is a unit normal to $\Gamma(t)$. Furthermore, let
$\vec y:(t_0-\delta,t_0+\delta)\to\bR^d$, for some $\delta>0$, with
$\vec y(t)\in\Gamma(t)$ and $\vec y(t_0)=\vec p_0$ be a 
smooth curve in $\bR^d$. Then the normal velocity of $\Gamma(t_0)$ at 
$\vec p_0$ is defined as
\begin{equation*}
\mathcal{V}(\vec p_0,t_0) = \vec\nu(\vec p_0,t_0)\,.\,\vec y\,'(t_0)\,.
\end{equation*}
\item \label{item:orientGT}
Let $\GT$ be a $C^k$--evolving hypersurface satisfying the assumptions in
\ref{item:defV}. Then we call $\GT$ a $C^k$--evolving orientable hypersurface.
\end{enumerate}
\end{definition}

\begin{rem} \label{rem:2.16}
Let $\GT$ be a $C^1$--evolving orientable hypersurface.
\begin{enumerate}
\item 
The condition $\tanspace_{(\vec p,t)}\,\GT \neq\bR^d\times\{0\}$, 
for all $(\vec p,t) \in \GT$, guarantees 
the existence of a curve $\vec y$ in 
{\rm Definition~\ref{def:GT}\ref{item:defV}}, recall
{\rm Remark~\ref{rem:2.6}\ref{item:rem2.6i}}.

\item \label{item:2.16iii} 
It is easy to show that $\mathcal{V}$ does not depend on the 
curve $\vec y$. 
To see this, we note that 
$(\vec y'(t_0),1) \in \tanspace_{(\vec p_0,t_0)}\,\GT$.
Hence, on letting $\{\vec\tau_1,\ldots,\vec\tau_{d-1}\}$ denote a
basis of $\tanspace_{\vec p_0}\,\Gamma(t_0)$, we have that
$\{(\vec\tau_1,0),\ldots,(\vec\tau_{d-1},0),(\vec y'(t_0),1)\}$
is a basis of $\tanspace_{(\vec p_0,t_0)}\,\GT$.
Then it follows from {\rm Definition~\ref{def:GT}\ref{item:defV}} that
$\vec y'(t_0)= \mathcal{V}(\vec p_0,t_0)\,\vec\nu(\vec p_0,t_0) 
+ \sum_{i=1}^{d-1} \alpha_i \,\vec\tau_{i}$
for some $\alpha_i \in \bR$, $i=1,\ldots,d-1$, and so
$(\mathcal{V}(\vec p_0,t_0)\,\vec\nu(\vec p_0,t_0),1) \in 
\tanspace_{(\vec p_0,t_0)}\,\GT$.
In addition, we observe that there exists a unique vector 
$(\vec\omega,1) \in \tanspace_{(\vec p_0,t_0)}\,\GT$
with $\vec\omega$ parallel to $\vec\nu(\vec p_0,t_0)$. 
Therefore we have that $\vec\omega = \mathcal{V}(\vec p_0,t_0)\,
\vec\nu(\vec p_0,t_0)$, and so $\mathcal{V}(\vec p_0,t_0)$ 
does not depend on $\vec y$.
\item 
Moreover, we observe that 
$(1+\mathcal{V}^2)^{-\frac12}\,(\vec\nu, -\mathcal{V})$ 
is a continuous normal vector field on $\GT$, and so $\GT$ is an
orientable hypersurface in $\bR^{d+1}$.
\item \label{item:signeddist}
If $\GT$ is a $C^2$--evolving hypersurface, then 
it is even easier to show that $\mathcal{V}$ does not depend on the 
curve $\vec y$. 
In order to do so, we choose a $C^1$--function $f$ defined in a small
neighbourhood of $\GT$ such that $f=0$ on $\GT$, 
and such that $(\nabla\,f)(\vec z,t)\neq \vec 0$ for all 
$(\vec z,t) \in \GT$. A possible choice for $f(\cdot,t)$ is the 
signed distance function $d_\Gamma(\cdot,t)$ to $\Gamma (t)$, recall
\eqref{eq:dGamma}.
We then compute for a curve $\vec y$ as in 
{\rm Definition~\ref{def:GT}\ref{item:defV}} that
\[ 
0= \ddt\, f(\vec y(t),t) = (\nabla\, f)(\vec y(t),t)\,.\,\vec y^\prime (t)
+ (\partial_t\, f) (\vec y(t),t)\,.
\]
As $\vec\nu = \nabla\,f / |\nabla\,f|$, we obtain
\[
\mathcal{V}(\vec p_0,t_0) =  \vec\nu (\vec p_0,t_0)\,.\, \vec y^\prime (t_0)=
- \left( \frac{\partial_t\,f}{|\nabla\, f|}\right) (\vec p_0,t_0)\,, 
\]
where the right hand side does not depend on  $\vec y$.
\end{enumerate}
\end{rem}

Typically we will consider evolving hypersurfaces that are given by a
global parameterization as follows.

\begin{definition} \label{def:globalx}
Let $\GT$ be a $C^k$--evolving hypersurface, and let 
$\Upsilon$ be a $C^k$--hypersurface in $\bR^d$, with $k\geq 1$.
\begin{enumerate}
\item \label{item:globalx}
A $C^k$--map $\vec x : \Upsilon\times [0,T] \to \bR^d$
such that $\vec x(\cdot,t)$ is a diffeomorphism from
$\Upsilon$ to $\Gamma(t)$ for all $t\in [0,T]$ is called a global
parameterization of $\GT$.
\item \label{item:vecV}
The full velocity of $\Gamma(t)$ on $\GT$, 
induced by the parameterization $\vec x$, is defined by
\begin{equation*} 
\vec{\mathcal{V}}(\vec x(\vec q,t),t) = (\partial_t\,\vec x) (\vec q,t) 
\quad\forall\ (\vec q,t) \in \Upsilon \times [0,T]\,.
\end{equation*}
\item \label{item:VT}
The tangential velocity of $\Gamma(t)$ on $\GT$, 
induced by the parameterization $\vec x$, is defined by
\[ 
\vec{\mathcal{V}}_\subT = \mat P_\Gamma\,\vec{\mathcal{V}} 
\qquad\text{on } \Gamma(t)\,,
\]
recall \eqref{eq:PGamma}.
\item \label{item:DsV}
We define the rate of deformation tensor of $\Gamma(t)$ by 
\[ 
\mat D_s(\vec{\mathcal{V}}) \qquad\text{on } \Gamma(t)\,,
\]
recall {\rm Definition~\ref{def:2.5}\ref{item:def2.5vi}}.
\end{enumerate}
\end{definition}

\begin{rem} \label{rem:vecV}
\rule{0pt}{0pt}
\begin{enumerate}
\item \label{item:VVnu}
For the normal velocity of an evolving orientable hypersurface $\GT$ 
defined in {\rm Definition~\ref{def:GT}\ref{item:defV}} it holds that
\[
\mathcal{V} = \vec{\mathcal{V}}\,.\,\vec\nu \qquad\text{on } \GT\,.
\]
Hence we have that
\begin{equation*}
\vec{\mathcal{V}} = \mathcal{V}\,\vec\nu + \vec{\mathcal{V}}_\subT
\qquad\text{on } \GT\,.
\end{equation*} 
Here we stress that $\mathcal{V}$ does not depend on the 
parameterization $\vec x$, while $\vec{\mathcal{V}}_\subT$ clearly does.
\item
The expression ``rate of deformation tensor'' for $\mat D_s(\vec{\mathcal{V}})$
is justified, because it encodes how $\Gamma(t)$ is locally deformed due to the
motion induced by $\vec x$. This will be made rigorous in
{\rm Lemma~\ref{lem:dtgij}}, below.
\end{enumerate}
\end{rem}

For the velocity field $\vec{\mathcal{V}}$ on $\GT$ we have the following
properties. Here and throughout, for notational convenience, 
we often identify $\Gamma(t) \times \{t\}$ with $\Gamma(t)$.

\begin{lem} \label{lem:divv} 
Let $\GT$ be a $C^2$--evolving orientable hypersurface with 
a global parameterization leading to the velocity field 
$\vec{\mathcal{V}}$.
\begin{enumerate}
\item \label{item:divvi}
It holds that
\begin{equation*}
\nabs\,.\,\vec{\mathcal{V}}
= -\mathcal{V}\,\varkappa + \nabs\,.\,\vec{\mathcal{V}}_\subT
\qquad\text{on }\Gamma(t)\,.
\end{equation*}
\item \label{item:divvii}
Moreover, it holds that
\begin{equation*} 
\nabs\,\vec{\mathcal{V}} = 
\mathcal{V}\,\nabs\,\vec\nu+\vec\nu\otimes\nabs\,\mathcal{V}
+\nabs\,\vec{\mathcal{V}}_\subT
\qquad\text{on } \Gamma(t)\,.
\end{equation*}
\item \label{item:divviii}
Finally, for the rate of deformation tensor 
it holds that
\begin{equation*}
 \mat D_s(\vec{\mathcal{V}}) = \mathcal{V}\,\nabs\,\vec\nu
+\tfrac12\, (\mat P_\Gamma\,\nabs\,\vec{\mathcal{V}}_\subT
+(\nabs\,\vec{\mathcal{V}}_\subT)^\transT \mat P_\Gamma)
\qquad\text{on } \Gamma(t)\,.
\end{equation*}
\end{enumerate}
\end{lem}
\begin{proof}
The desired results follow immediately from Lemma~\ref{lem:divf}, on noting
$\vec{\mathcal{V}}(\cdot,t) \in $ \linebreak $[C^1(\Gamma(t))]^d$ and
Remark~\ref{rem:vecV}\ref{item:VVnu}.
\end{proof}

It is often convenient to also consider local parameterizations of $\Gamma(t)$,
as defined in Definition~\ref{def:boundary}. 
To this end, let
$\vec\varphi:U\to\bR^d$, $U\subset\bR^{d-1}$ open and connected, 
be a local parameterization of $\Upsilon$. Then
\begin{equation} \label{eq:localpara}
\localpara(t) = \vec x(\cdot,t) \circ \vec\varphi \qquad\text{in } U\,,
\quad t \in [0,T]\,,
\end{equation}
defines a local parameterization of $\Gamma(t)$. 

We now define the time derivative of a function
$f:\GT\to\bR$. We cannot differentiate $f(\vec p,t)$
directly with respect to $t$ due to the fact that $\vec p$ might not
lie on $\Gamma(t)$ for different times $t$. When differentiating with
respect to $t$, we need to move the point $\vec p$. There is hence
some ambiguity in defining the time derivative. 

\begin{definition}\label{def:2.18}
Let $\GT$ be a $C^1$--evolving hypersurface with a
global parameterization $\vec x:\Upsilon \times [0,T]\to\bR^d$,
and let $f\in C^1(\GT)$.
\begin{enumerate}
\item \label{item:def2.18i}
The expression
\begin{equation*}
(\matpartx\,f)(\vec x(\vec q,t),t) = \ddt\, f(\vec x(\vec q,t),t) 
\quad\forall\ (\vec q,t) \in \Upsilon \times [0,T]
\end{equation*}
is the time derivative following the parameterizations $\vec x(\cdot,t)$
of $f$ on $\Gamma(t)$. It is also called the material time derivative
induced by $\vec x$.
\item \label{item:def2.18ii}
The normal time derivative of $f$ on $\Gamma(t)$ is defined as
\[
  \matpartn\,f = \matpartx\,f -\vec{\mathcal{V}}_\subT\,.\,\nabs\,f\,.
\]
\end{enumerate}
\end{definition}

\begin{rem}\label{rem:timeder} 
\rule{0pt}{0pt}
\begin{enumerate}
\item \label{item:rem2.19i}
The quantity $\matpartx\,f$ depends on the parameterization
$\vec x$. Moreover, it holds that
$\vec{\mathcal{V}} = \matpartx\,\vec\id$ on $\Gamma(t)$ and
$\vec{\mathcal{V}} \circ \localpara = \partial_t\,\localpara$ in $U$.
\end{enumerate}
For the following observations, we assume that 
$f$ is extended to a neighbourhood of $\GT$.
\begin{enumerate}
\setcounter{enumi}{1}
\item \label{item:rem2.19ii}
It holds that
\begin{equation*}
\matpartx\,f = \partial_t\,f +\vec{\mathcal{V}}\,.\,\nabla\,f
\qquad\text{on } \Gamma(t)\,,
\end{equation*}
where $\nabla\,f$ denotes the gradient in $\bR^d$ of the extension $f$.
\item \label{item:rem2.19iii}
Moreover, we have that
\begin{equation*}
\matpartn\,f = \partial_t\, f + \mathcal{V}\,\vec\nu\,.\,\nabla\, f
\qquad\text{on } \Gamma(t)\,.
\end{equation*}
Taking {\rm Remark~\ref{rem:2.16}\ref{item:2.16iii}} into account, 
we observe, in particular, that $\matpartn\,f$ does not depend on $\vec x$. 
In fact, $\matpartn\,f$ is the derivative of $f$ in the direction 
$(\mathcal{V}\,\vec\nu, 1)$, where the vector
$(\mathcal{V}\,\vec\nu, 1)$ is a space time tangential vector of the
evolving surface $\GT$. I.e.\
\[
\matpartn\,f = \partial_{(\mathcal{V}\,\vec\nu, 1)}\,f 
\qquad\text{on } \GT\,,
\]
recall {\rm Definition~\ref{def:2.5}\ref{item:def2.5i}}.
\item
Taking a curve $t\mapsto\vec y(t)\in\bR^d$ such that 
$\vec y(t) \in \Gamma(t)$ and
$\vec y\,'(t)=(\mathcal{V}\,\vec\nu) \circ \vec y$, 
we obtain from \ref{item:rem2.19iii} that
\begin{equation*}
(\matpartn\,f)(\vec y(t),t) = \ddt\,f(\vec y(t),t)\,.
\end{equation*}
\end{enumerate}
\end{rem}

For the time-dependent metric tensor $(g_{ij}(t))_{i,j=1,\ldots,d-1}$, 
which, similarly to Remark \ref{rem:altforms}\ref{item:altformg}, 
is defined via
\begin{equation} \label{eq:gijt}
g_{ij}(t) = \partial_i\,\localpara(t)\,.\,\partial_j\,\localpara(t)
\qquad\text{in } U\,,
\end{equation}
for $\localpara(t) = \vec x(\cdot,t) \circ \vec\varphi$,
recall (\ref{eq:localpara}),
we obtain the following lemma. Here and throughout, for notational convenience,
we often omit the dependence on $t$.

\begin{lem} \label{lem:dtgij}
Let $\GT$ be a $C^2$--evolving hypersurface with a global parameterization 
$\vec x$ leading to the velocity field 
$\vec{\mathcal{V}}$, and let the metric tensor be defined by \eqref{eq:gijt}.
Then it holds that
\begin{equation*} 
\partial_t\,g_{ij} 
= 2\,((\mat D_s(\vec{\mathcal{V}})\circ\localpara)
\,\partial_i\,\localpara)\,.\,\partial_j\localpara
\qquad\text{in } U \,,
\end{equation*}
where $\mat D_s(\vec{\mathcal{V}})$ is the rate of deformation tensor,
recall {\rm Definition~\ref{def:globalx}\ref{item:DsV}}. 
\end{lem}
\begin{proof}
We introduce the shorthand notation
$\vec\partial_i = \partial_i\,\localpara$, $i=1,\ldots,d-1$, and
recall from Remark~\ref{rem:timeder}\ref{item:rem2.19i} that
$\vec{\mathcal{V}}\circ\localpara = \partial_t\,\localpara$ in $U$.
Then we compute, using Remark~\ref{rem:2.6}\ref{item:2.6vi},
\begin{align*}
\partial_t\, g_{ij} & = 
\partial_t\, (\partial_i\,\localpara\,.\,\partial_j\,\localpara)
= \partial_i\,(\vec{\mathcal{V}}\circ\localpara)\,.\,\vec\partial_j 
+ \vec\partial_i\,.\,\partial_j\,(\vec{\mathcal{V}}\circ\localpara)
 \nonumber \\ &
=((\nabs\,\vec{\mathcal{V}})\circ\localpara\,\vec\partial_i)\,.\,\vec\partial_j
+\vec\partial_i\,.\,(\nabs\,\vec{\mathcal{V}})\circ\localpara\,\vec\partial_j
 \nonumber \\ &
= (\nabs\,\vec{\mathcal{V}}+(\nabs\,\vec{\mathcal{V}})^\transT)\circ\localpara
\,\vec\partial_i\,.\,\vec\partial_j\,.
\end{align*}
As $\vec\partial_i$ and $\vec\partial_j$ are tangential, the claim follows.
\end{proof}

In order to compute the first variation of area, it is crucial to
know how the area element $\sqrt{g}$ evolves in time, recall
Definition~\ref{def:int}. This is studied in the next lemma.

\begin{lem} \label{lem:derive-g} 
Let $\GT$ be a $C^2$--evolving orientable hypersurface with a global 
parameterization $\vec x$ leading to the velocity field 
$\vec{\mathcal{V}}$, and let the metric tensor be defined by \eqref{eq:gijt}.
It holds that
\begin{equation*}
\partial_t \sqrt{g} = (\nabs\,.\,\vec{\mathcal{V}})\circ\localpara
\,\sqrt{g} 
= (-\mathcal{V}\,\varkappa +\nabs\,.\,\vec{\mathcal{V}}_\subT)\circ\localpara
\,\sqrt{g} \qquad\text{in } U\,,
\end{equation*}
where
\[
g(t)=\det\left((g_{ij}(t))_{i,j=1,\ldots,d-1}\right).
\]
\end{lem}
\begin{proof} 
\newcommand{\matg}{{\mat G}}
Jacobi's formula for the derivative of the
determinant, see e.g.\ \citet[Lemma 5.3]{EckGK17}, for
$\matg(t)=(g_{ij}(t))_{i,j=1,\ldots,d-1}$, gives
\begin{equation*}
\partial_t \det\matg(t) = \det\matg(t)
\tr\left(\matg^{-1}(t)\,\partial_t\,\matg(t)\right).
\end{equation*}
As $g(t)=\det\matg(t)$, and since $\matg(t)$ is symmetric, we obtain from
the proof of Lemma~\ref{lem:dtgij} that
\begin{align*}
\partial_t \sqrt{g} &= \partial_t \sqrt{\det\matg} 
= \tfrac12\,\sqrt{g}\,\tr\left(\matg^{-1}\,\partial_t\,\matg\right)
= \tfrac12\,\sqrt{g}\,\,\matg^{-1} : \partial_t\,\matg
\nonumber \\ &
=\tfrac12\,\sqrt{g}\,\sum^{d-1}_{i,j=1} g^{ij}\,\partial_t\,g_{ij} 
=\sqrt{g}\,\sum^{d-1}_{i,j=1}
g^{ij}\,\partial_i\,(\vec{\mathcal{V}}\circ\localpara)
\,.\,\partial_j\,\localpara 
\nonumber \\ &
= \sqrt{g}\,(\nabs\,.\,\vec{\mathcal{V}}) \circ\localpara\,,
\end{align*}
where in the last step we have recalled
Remark~\ref{rem:altforms}\ref{item:altformg}.
The second identity is then a direct consequence of 
Lemma~\ref{lem:divv}\ref{item:divvi}.
\end{proof}

We can now prove a transport theorem for evolving hypersurfaces, which
will be crucial for many arguments that follow.

\begin{thm} \label{thm:trans} 
Let $\GT$ be a compact $C^2$--evolving orientable
hypersurface with a global parameterization leading to the
velocity field $\vec{\mathcal{V}}$, and let $f \in C^1(\GT)$. 
Then it holds that
\begin{align*}
\ddt\,\int_{\Gamma(t)} f \dH{d-1}
&=\int_{\Gamma(t)} (\matpartx\,f+f\,\nabs\,.\,\vec{\mathcal{V}}) \dH{d-1}\\
&=\int_{\Gamma(t)} (\matpartx\,f+f\,\nabs\,.\,\vec{\mathcal{V}}_\subT
-f\,\mathcal{V}\,\varkappa) \dH{d-1}\\
&=\int_{\Gamma(t)} (\matpartn\,f-f\,\mathcal{V}\,\varkappa)\dH{d-1}
+ \int_{\partial\Gamma(t)} f\,\vec{\mathcal{V}}\,.\,\conormal \dH{d-2}\,,
\end{align*}
where $\conormal$ is the outer unit conormal to $\partial\Gamma$.
\end{thm}
\begin{proof} 
We first consider an $f$ with support in the image of a single 
time-depen\-dent local parameterization 
$\localpara(t) = \vec x(\cdot,t)\circ\vec\varphi$, recall (\ref{eq:localpara}). 
In this case we can compute from Definition~\ref{def:int},
Definition~\ref{def:2.18}\ref{item:def2.18i}
and Lemma~\ref{lem:derive-g} that
\begin{align*}
\ddt\,\int_{\Gamma(t)} f \dH{d-1} & =
\ddt\,\int_{\vec x (\vec\varphi(U),t)} f \dH{d-1} \\
& =\ddt\,\int_U
f(\vec x(\vec\varphi(\vec u),t),t)\,\sqrt{g} \dL{d-1}\\ &
= \int_U (\matpartx\,f+f\,\nabs\,.\,\vec{\mathcal{V}})\circ\localpara
\,\sqrt{g} \dL{d-1}\\&
= \int_{\Gamma(t)} (\matpartx\,f + f\,\nabs\,.\,\vec{\mathcal{V}}) \dH{d-1}\,.
\end{align*}
Using a partition of unity argument now proves the first identity in the
claim. Lemma \ref{lem:divv}\ref{item:divvi} yields the second
identity, and Definition~\ref{def:2.18}\ref{item:def2.18ii} 
and \arxivyesno{}{\linebreak}%
Lemma~\ref{lem:productrules}\ref{item:pri}
then give the last identity, on noting from Theorem~\ref{thm:div} 
and Definition~\ref{def:conormal} that
\begin{align*} 
\int_{\Gamma(t)} \nabs\,.\,(f\,\vec{\mathcal{V}_\subT}) \dH{d-1} &
= - \int_{\Gamma(t)} \varkappa\,f\,\vec{\mathcal{V}}_\subT\,.\,\vec\nu\dH{d-1}
+ \int_{\partial\Gamma(t)} f\,\vec{\mathcal{V}}_\subT\,.\,\conormal \dH{d-2}
\nonumber \\ &
= \int_{\partial\Gamma(t)} f\,\vec{\mathcal{V}}\,.\,\conormal \dH{d-2}\,.
\end{align*}
\end{proof}

We also have the following transport theorem for moving domains.

\begin{thm} \label{thm:transvol}
Let $\GT$ be a compact $C^2$--evolving orientable hypersurface,
such that $\Gamma(t)$ is bounding a domain $\Omega(t)\subset\bR^d$,
for $t\in [0,T]$. We assume that $\vec\nu(t)$ is the outer unit normal 
to $\Omega(t)$ on $\Gamma(t)$, and that $f\in C^1(\overline\HT)$, where
\begin{equation*}
\HT = \bigcup_{t\in[0,T]} (\Omega(t)\times\{t\})\,.
\end{equation*}
Then it holds that
\begin{equation}\label{eq:transvol}
\ddt\,\int_{\Omega(t)} f \dL{d} =
\int_{\Omega(t)} \partial_t\,f \dL{d} 
+ \int_{\Gamma(t)} f\,\mathcal{V} \dH{d-1}\,.
\end{equation}
\end{thm}
\begin{proof} 
A proof can be found in \citet[Chapter~2]{PrussS16}. If
the moving domain is transported with a velocity field
$\vec{\mathcal{V}}_\subHT : \HT \to \bR^d$, 
with $\vec{\mathcal{V}}_\subHT\,.\,\vec\nu = \mathcal{V}$ on $\GT$,
then the theorem can be shown similarly to Theorem~\ref{thm:trans}. 
An alternative proof would integrate \eqref{eq:transvol}, 
and use the divergence theorem in space-time, see e.g.\
\citet[\S7.3]{EckGK17}.
\end{proof}

\begin{rem} \label{rem:reynolds}
If the moving domain $\HT$ is transported with a velocity field 
$\vec{\mathcal{V}}_\subHT : \HT \to \bR^d$, we obtain the Reynolds transport 
theorem
\begin{align*}
\ddt\,\int_{\Omega(t)} f\dL{d} &
= \int_{\Omega(t)} \partial_t\,f \dL{d} + 
\int_{\Gamma(t)} f\,\vec{\mathcal{V}}_\subHT\,.\,\vec\nu \dH{d-1} \\ &
= \int_{\Omega(t)} \partial_t\,f + \nabla\,.\,(f\,\vec{\mathcal{V}}_\subHT)
\dL{d}\,,
\end{align*}
where we have used the divergence theorem in $\bR^d$,
and where we once again assumed that $\vec\nu(t)$ is the outer unit normal
to $\Omega(t)$ on $\Gamma(t)$.
\end{rem}

Using Theorem~\ref{thm:transvol}, one can also show a
transport theorem for two-phase moving domains. To this end,
let $\Omega\subset\bR^d$ be a fixed, bounded domain. Suppose
that $\GT$ is a compact $C^2$--evolving orientable hypersurface with
$\Gamma(t)\subset\Omega$, such that $\Gamma(t)$ encloses a region
$\Omega_-(t)\subset\Omega$, with $\Gamma(t)=\partial\Omega_-(t)$,
for all $t \in [0,T]$. 
Let $\Omega_+(t)= \Omega\setminus\overline{\Omega_-(t)}$,
and let $\vec\nu(t)$ denote the outer normal to $\Omega_-(t)$ on $\Gamma(t)$,
see Figure~\ref{fig:sketch}.
\begin{figure}
\begin{center}
\arxivyesno{
\begin{tikzpicture}
}{
\begin{tikzpicture}[thick,scale=0.6, every node/.style={transform shape}]
}
\draw (-1,-3) -- (4,-3) -- (4,2) -- (-1,2) -- (-1,-3);

\draw[thick] plot [smooth cycle, tension=0.6]
coordinates {(-0.5,-0.25) (0.25,0.5) (1,0.75) (2,0.5) (2.5,-0.5) (2,-1.25) 
(1.25,-1) (0.5,-1.5) (-0.25, -1)};

\draw (2.5,0.5) node {$\Gamma(t)$};
\draw (1.2,-0.35) node {$\Omega_-(t)$};
\draw[-latex] (2.5, -0.5) to (3.1, -0.5) node[above] {$\vec\nu(t)$};
\draw (3,-2) node {$\Omega_+(t)$};
\end{tikzpicture}
\end{center}
\caption{A sketch of the domain $\Omega$ in the case $d=2$.}
\label{fig:sketch}
\end{figure}
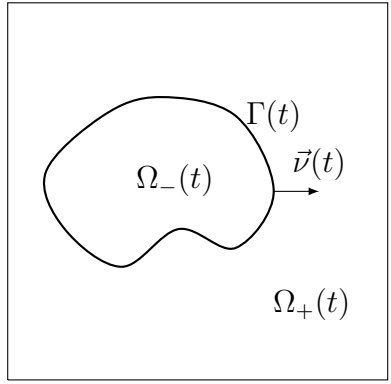%
We have  
\[ 
\overline{\Omega} = \overline{\Omega_-(t)} \cup \overline{\Omega_+(t)}\,,
\quad
\overline{\Omega_-(t)} \cap \overline{\Omega_+(t)} = \Gamma(t)\,, 
\quad
\partial\Omega_+(t) = \Gamma(t)\cup\partial\Omega\,.
\]
Now let
\begin{equation*}
\HTi{\pm} = \bigcup_{t\in[0,T]} (\Omega_\pm(t)\times\{t\})\,,
\end{equation*}
and let $f_\pm: \HTi{\pm} \to\bR$ be given such that each
$f_\pm$ has a continuous extension to $\overline{\HTi{\pm}}$.
Defining $f:\HTi{-} \cup \HTi{+} \to\bR$ as
\begin{equation} \label{eq:fpm}
f(\cdot,t) = f_-(\cdot,t)\,\charfcn{\Omega_-(t)} + 
f_+(\cdot,t)\,\charfcn{\Omega_+(t)} \quad\forall\ t\in[0,T]\,,
\end{equation}
where, here and throughout, $\charfcn{\mathfrak{A}}$ defines the characteristic 
function for a set $\mathfrak{A}$, we let
\begin{equation} \label{eq:fjump}
[f]^+_-(\vec z,t) = \lim_{\vec y\to \vec z\atop \vec y\in\Omega_+(t)} 
f(\vec y,t) - \lim_{\vec y\to \vec z\atop \vec y\in\Omega_-(t)} f(\vec y,t)
\quad\forall\ (\vec z,t) \in \GT\,.
\end{equation}
Then we obtain that following result.

\begin{thm} \label{thm:2phasetransport}
Let $\GT \subset \Omega \times [0,T]$ 
be a compact $C^2$--evolving orientable hypersurface,
such that $\Gamma(t)$ is bounding a domain $\Omega_-(t)\subset\bR^d$,
for $t\in [0,T]$. We assume that $\vec\nu(t)$ is the outer unit normal 
to $\Omega_-(t)$ on $\Gamma(t)$, and that 
$f_\pm\in C^1(\overline{\HTi{\pm}})$.
Then, for $f$ as defined in \eqref{eq:fpm}, it holds that
\begin{equation*}
\ddt\,\int_\Omega f \dL{d} = \int_{\Omega_-(t)} \partial_t\,f \dL{d} 
+ \int_{\Omega_+(t)} \partial_t\, f \dL{d} 
- \int_{\Gamma(t)} [f]^+_-\,\mathcal{V} \dH{d-1}\,.
\end{equation*}
\end{thm}
\begin{proof}
The claim follows directly from Theorem~\ref{thm:transvol}.
\end{proof}

\begin{rem} \label{rem:dnufjump}
We naturally extend \eqref{eq:fjump} to vector-valued quantities. For example,
if $f$, as defined in \eqref{eq:fpm}, has a continuous extension
to $\overline\Omega \times [0,T]$, and such that each $\nabla\,f_\pm$ 
has a continuous extension to $\overline{\HTi{\pm}}$, then we define
\begin{align} \label{eq:dnufjump}
& [\partial_{\vec\nu}\,f]^+_-(\vec z,t) = 
\lim_{\vec y\to \vec z\atop \vec y\in\Omega_+(t)} 
(\nabla\,f_+)(\vec y,t) \,.\,\vec\nu(\vec z,t)
- \lim_{\vec y\to \vec z\atop \vec y\in\Omega_-(t)} 
(\nabla\,f_-)(\vec y,t) \,.\,\vec\nu(\vec z,t) \nonumber \\ & \hspace{7cm}
\quad\forall\ (\vec z,t) \in \GT\,.
\end{align}
For such an $f$, integration by parts in $\bR^d$ immediately yields that
\begin{align*}
& \int_{\Omega_-(t)\cup\Omega_+(t)} \eta\,\Delta\,f \dL{d}
\nonumber \\ & \quad
= \int_{\partial\Omega} \eta\,\partial_{\vec\nu_\Omega}\,f \dH{d-1}
 - \int_{\Gamma(t)} \eta\,[\partial_{\vec\nu}\,f]^+_- \dH{d-1}
- \int_\Omega \nabla\,f\,.\,\nabla\,\eta \dL{d} 
\end{align*}
for all $\eta \in C^1(\overline\Omega)$,
where $\vec\nu_\Omega$ denotes the outer unit normal to $\Omega$ on 
$\partial\Omega$.
\end{rem}

\subsection{Time derivatives of the normal}

We also frequently need time derivatives of the normal. The
relevant results are stated in the following lemma.

\begin{lem} \label{lem:dtnu}
Let $\GT$ be a $C^2$--evolving orientable hypersurface.  
\begin{enumerate}
\item \label{item:dtnu}
Let $\vec{\mathcal{V}}$ be the velocity field induced by a  
a global parameterization of $\GT$. Then it holds that
\begin{equation*}
\matpartx\,\vec\nu = -(\nabs\,\vec{\mathcal{V}})^\transT\,\vec\nu
\qquad\text{on } \Gamma(t)\,.
\end{equation*}
\item \label{item:normdtnu}
The normal time derivative of $\vec\nu$ satisfies
\begin{equation*}
\matpartn\,\vec\nu = -\nabs\,\mathcal{V} \qquad\text{on } \Gamma(t)\,.
\end{equation*}
\end{enumerate}
\end{lem}
\begin{proof} 
\ref{item:dtnu}
For $\vec p \in \Gamma(t)$, with $\vec p = \vec x(\vec\varphi(\vec u,t),t)
= \localpara(\vec u,t)$, recall \eqref{eq:localpara}, 
we define a basis $\{\vec\tau_1,\ldots,\vec\tau_{d-1}\}$ of the tangent space 
$\tanspace_{\vec p}\,\Gamma(t)$ via 
$\vec\tau_i(\vec p) = \partial_i\,\localpara(\vec u, t)$, $i=1\,\ldots,d-1$.
We also recall from Remark~\ref{rem:timeder}\ref{item:rem2.19i} that 
$\vec{\mathcal{V}} \circ \localpara = \partial_t\,\localpara$ and hence note 
that
\begin{align*}
\matpartx\,\vec\tau_i &
= \partial_t\,[(\partial_i\,\localpara)\circ\localpara^{-1}]
= [\partial_i\,(\partial_t\,\localpara)]\circ\localpara^{-1}
= [\partial_i\,(\vec{\mathcal{V}} \circ \localpara)] \circ\localpara^{-1}
= \partial_{\vec\tau_i}\,\vec{\mathcal{V}} \\ &
= (\nabs\,\vec{\mathcal{V}})\,\vec\tau_i
\quad\text{on } \Gamma(t)\,,
\end{align*}
where for the last step we have recalled 
Remark~\ref{rem:2.6}\ref{item:2.6vi}.
As $\vec\nu\,.\,\vec\tau_i = 0$, it follows for $i=1,\ldots,d-1$ that
\begin{equation*}
(\matpartx\,\vec\nu)\,.\,\vec\tau_i 
= -\vec\nu\,.\,\matpartx\,\vec\tau_i 
= -\vec\nu\,.\,((\nabs\,\vec{\mathcal{V}})\,\vec\tau_i) 
= -((\nabs\,\vec{\mathcal{V}})^\transT\,\vec\nu)\,.\,\vec\tau_i
\quad\text{on } \Gamma(t)\,.
\end{equation*}
Since $\{\vec\tau_1,\ldots,\vec\tau_{d-1}\}$ is a basis of the tangent space
$\tanspace_{\vec p}\,\Gamma(t)$, we obtain, 
on using $(\matpartx\,\vec\nu)\,.\,\vec\nu =
\frac12\,\matpartx\,|\vec\nu|^2 = 0$ and the fact that
$\vec\nu\,.\,(\nabs\,\vec{\mathcal{V}})^\transT\,\vec\nu =
(\nabs\,\vec{\mathcal{V}})\,\vec\nu\,.\,\vec\nu = 0$, recall
Remark~\ref{rem:2.6}\ref{item:2.6vi}, that 
\begin{equation*}
\matpartx\,\vec\nu = -(\nabs\,\vec{\mathcal{V}})^\transT\,\vec\nu
\qquad\text{on } \Gamma(t)\,.
\end{equation*}
\ref{item:normdtnu}
Let $\vec x$ be an arbitrary global parameterization of $\GT$, with induced
velocity field $\vec{\mathcal{V}}$.
Using Definition~\ref{def:2.18}\ref{item:def2.18ii},
Lemma~\ref{lem:nabsnu}, the result from \ref{item:dtnu}
and Lemma~\ref{lem:productrules}\ref{item:prii}, 
we compute
\begin{align*}
\matpartn\,\vec\nu &
= \matpartx\,\vec\nu - (\nabs\,\vec\nu)\,\vec{\mathcal{V}}_\subT
= \matpartx\,\vec\nu - (\nabs\,\vec\nu)^\transT\,\vec{\mathcal{V}}_\subT
= -(\nabs\,\vec{\mathcal{V}})^\transT \,\vec\nu - (\nabs\,\vec\nu )^\transT\,
\vec{\mathcal{V}}_\subT \nonumber \\ &
= -(\nabs\,\vec{\mathcal{V}})^\transT \,\vec\nu - (\nabs\,\vec\nu )^\transT\,
\vec{\mathcal{V}}
= -\nabs\,(\vec{\mathcal{V}}\,.\,\vec\nu)
= -\nabs\,{\mathcal{V}}
\qquad\text{on }\Gamma(t)\,.
\end{align*}
\end{proof}

\subsection{Time derivatives of the mean curvature}
In order to be able to compute the first variation of energies that depend on
the mean curvature, for example with the help of Theorem~\ref{thm:trans},
we need expressions for the time derivatives of curvature.

We begin with the following commutator rule for time derivatives and
surface differential operators.

\begin{lem} \label{lem:commdiff} 
Let $\GT$ be a $C^2$--evolving orientable hypersurface with
a global parameterization leading to the
velocity field $\vec{\mathcal{V}}$, 
and let $f : \GT \to \bR$, $\vec f : \GT \to \bR^d$ be $C^1$--functions. 
Then we have the following results:
\begin{enumerate}
\item \label{item:commdiffi}
\begin{equation*}
\matpartx\,\nabs\, f - \nabs\,\matpartx\, f =
[\nabs\,\vec{\mathcal{V}} - 2\,\mat D_s(\vec{\mathcal{V}})]\,\nabs\,f
\qquad\text{on } \Gamma(t)\,.
\end{equation*}
\item \label{item:commdiffii}
\begin{equation*}
\matpartx\,\nabs\,\vec f - \nabs\,\matpartx\,\vec f = 
(\nabs\,\vec f)\,
[\nabs\,\vec{\mathcal{V}} - 2\,\mat D_s(\vec{\mathcal{V}})]^\transT
\qquad\text{on } \Gamma(t)\,.
\end{equation*}
\item \label{item:commdiffiii}
\begin{equation*}
\matpartx\,(\nabs\,.\,\vec f) - \nabs\,.\,(\matpartx\,\vec f) =
[\nabs\,\vec{\mathcal{V}}-2\,\mat D_s(\vec{\mathcal{V}})] : \nabs\,\vec f
\qquad\text{on } \Gamma(t)\,.
\end{equation*}
\end{enumerate}
\end{lem}
\begin{proof} 
\ref{item:commdiffi}
We extend $f$ to a neighbourhood of $\GT$, such that the extension from
$\Gamma(t)$ is constant in the $\vec\nu$-direction, recall
(\ref{eq:extendf}). 
On noting (\ref{eq:nablanabs})
and Lemma~\ref{lem:dtnu}\ref{item:dtnu}, it holds on $\Gamma(t)$ that
\begin{align*}
0&=\matpartx\,(\nabla\,f\,.\,\vec\nu) = \matpartx\,(\nabs\,f\,.\,\vec\nu) 
= (\matpartx\,\nabs\,f)\,.\,\vec\nu + \nabs\,f\,.\,\matpartx\,\vec\nu\\
&= (\matpartx\,\nabs\,f)\,.\,\vec\nu - (\nabs\,f)\,.\,
(\nabs\,\vec{\mathcal{V}})^\transT\,\vec\nu 
=(\matpartx\,\nabs\,f)\,.\,\vec\nu - ((\nabs\,\vec{\mathcal{V}})\,
\nabs\,f)\,.\,\vec\nu\,.
\end{align*}
As $\nabs\,\matpartx\,f$ is tangential, we obtain
\[
(\matpartx\,\nabs\,f-\nabs\,\matpartx\,f)\,.\,\vec\nu=
((\nabs\,\vec{\mathcal{V}})\,\nabs\,f)\,.\,\vec\nu\,.
\]
We now identify the tangential part of 
$\matpartx\,\nabs\,f-\nabs\,\matpartx\,f$, recall \eqref{eq:PGamma}.
We compute, on noting (\ref{eq:nablanabs}),
Remark~\ref{rem:2.6}\ref{item:Pnabs}, 
Remark~\ref{rem:timeder}\ref{item:rem2.19ii},
Remark~\ref{rem:altforms}\ref{item:eq:grad},
Lemma~\ref{lem:productrules}\ref{item:prii} and
Remark~\ref{rem:2.6}\ref{item:2.6vi} that
\begin{align*}
\mat P_\Gamma\,(\matpartx\,\nabs\,f-\nabs\,\matpartx\,f)
&= \mat P_\Gamma\,\matpartx\,\nabla\,f- \nabs\,\matpartx\,f\\
&= \mat P_\Gamma\,(\partial_t\,\nabla\,f + (\nabla^2\,f)\,\vec{\mathcal{V}})
-\nabs\,(\partial_t\,f+\vec{\mathcal{V}}\,.\,\nabla\,f)\\
&= \mat P_\Gamma\,(\partial_t\,\nabla\,f + (\nabla^2\,f)\,\vec{\mathcal{V}})
- \mat P_\Gamma\,\nabla\,\partial_t\,f
-\nabs\,(\vec{\mathcal{V}}\,.\,\nabla\,f)\\
&= \mat P_\Gamma\,(\nabla^2\,f)\,\vec{\mathcal{V}} 
-(\nabs\,\vec{\mathcal{V}})^\transT\,\nabla\,f
- \mat P_\Gamma\,\sum_{i=1}^d \vec{\mathcal{V}}\,.\,\vec\ek_i\,\nabla\,\partial_i\,f 
\\ &
= -(\nabs\,\vec{\mathcal{V}})^\transT\,\nabs\,f\,.
\end{align*}
Combining the above, on recalling Definition~\ref{def:2.5}\ref{item:def2.5vi}
and Remark~\ref{rem:2.6}\ref{item:Pnabs},
yields that
\begin{align*}
& \matpartx\,\nabs\,f-\nabs\,\matpartx\,f
 = \mat P_\Gamma\,(\matpartx\,\nabs\,f-\nabs\,\matpartx\,f)
+ (\mat\Id - \mat P_\Gamma)\,(\matpartx\,\nabs\,f-\nabs\,\matpartx\,f) 
\\ & \quad
= -(\nabs\,\vec{\mathcal{V}})^\transT\,\nabs\,f 
+ (\mat\Id - \mat P_\Gamma)\,(\nabs\,\vec{\mathcal{V}}) \,\nabs\,f 
= [\nabs\,\vec{\mathcal{V}} - 2\,\mat D_s(\vec{\mathcal{V}})]\,\nabs\,f\,,
\end{align*}
which shows the claim.
\\
\ref{item:commdiffii}
It follows from \ref{item:commdiffi} and Remark~\ref{rem:2.6}\ref{item:2.6vi}
that
\begin{align*}
(\matpartx\,\nabs\,\vec f - \nabs\,\matpartx\,\vec f)^\transT\,\vec\ek_i &
= (\matpartx\,\nabs - \nabs\,\matpartx)\,(\vec f\,.\,\vec\ek_i) \\ &
= [\nabs\,\vec{\mathcal{V}} - 2\,\mat D_s(\vec{\mathcal{V}})]\,(\nabs\,
\vec f)^\transT\,\vec\ek_i\,,
\end{align*}
for $i=1,\ldots,d$, and this proves the desired result.
\\
\ref{item:commdiffiii}
Using \ref{item:commdiffii} and Lemma~\ref{lem:productrules}\ref{item:trnabsf},
we compute
\begin{align*}
& \matpartx\,(\nabs\,.\,\vec f) - \nabs\,.\,(\matpartx\,\vec f) =
\tr\left(\matpartx\,(\nabs\,\vec f) - \nabs\,(\matpartx\,\vec f)\right) 
\\ & \qquad
= \tr\left((\nabs\,\vec f)\,
[\nabs\,\vec{\mathcal{V}} - 2\,\mat D_s(\vec{\mathcal{V}})]^\transT
\right) 
= [\nabs\,\vec{\mathcal{V}} - 2\,\mat D_s(\vec{\mathcal{V}})] : 
\nabs\,\vec f\,,
\end{align*}
which yields the desired result.
\end{proof}

We now obtain formulas for the time derivatives of the mean curvature.
\begin{lem} \label{lem:derkappa} 
Let $\GT$ be a $C^3$--evolving orientable hypersurface.
\begin{enumerate}
\item \label{item:lem10.3i}
Let $\vec{\mathcal{V}}_\subT$ be the tangential velocity field induced by a  
a global parameterization of $\GT$. Then it holds that
\[
\matpartx\,\varkappa
= \Delta_s\,{\mathcal{V}}+ {\mathcal{V}}\,|\nabs\,\vec\nu|^2  +
\vec{\mathcal{V}}_\subT\,.\,\nabs\,\varkappa
\qquad\text{on }\Gamma(t)\,.
\]
\item \label{item:lem10.3ii}
The normal time derivative of the mean curvature satisfies
\begin{equation*}
\matpartn\,\varkappa = \Delta_s\,{\mathcal{V}}+{\mathcal{V}}\,|\nabs\,\vec\nu|^2
\qquad\text{on } \Gamma(t)\,.
\end{equation*}
\end{enumerate}
\end{lem}
\begin{proof} 
\ref{item:lem10.3i}
It follows from 
Lemma~\ref{lem:varkappa}\ref{item:varkappa},
Lemma~\ref{lem:commdiff}\ref{item:commdiffiii}, 
\arxivyesno{Lemma~\ref{lem:dtnu}\ref{item:dtnu},}
{Lemma \linebreak \ref{lem:dtnu}\ref{item:dtnu},}
Lemma~\ref{lem:divv}\ref{item:divvii}, \ref{item:divviii} and
Lemma~\ref{lem:nabsnu} that
\begin{align*}
\matpartx\,\varkappa
&= -\matpartx\,(\nabs\,.\,\vec\nu)
= -\nabs\,.\,(\matpartx\,\vec\nu)-[\nabs\,\vec{\mathcal{V}}-
2\,\mat D_s(\vec{\mathcal{V}})]:\nabs\,\vec\nu\\
&=\nabs\,.\,((\nabs\,\vec{\mathcal{V}})^\transT\,\vec\nu) +
\mathcal{V}\,|\nabs\,\vec\nu|^2 
+ \nabs\,\vec{\mathcal{V}}_\subT:\nabs\,\vec\nu\\
&= \nabs\,.\,(\nabs\,{\mathcal{V}} 
+ (\nabs\,\vec{\mathcal{V}}_\subT)^\transT\,\vec\nu) + 
\mathcal{V}\,|\nabs\,\vec\nu|^2 
+ \nabs\,\vec{\mathcal{V}}_\subT:\nabs\,\vec\nu\,.
\end{align*}
As $\vec{\mathcal{V}}_\subT\,.\,\vec\nu=0$, we obtain
from Lemma~\ref{lem:productrules}\ref{item:prii} and 
Lemma~\ref{lem:nabsnu}\ref{item:nabsnuT} that
$(\nabs\,\vec{\mathcal{V}}_\subT)^\transT\,\!\vec\nu 
=\! -(\nabs\,\vec\nu)\,\vec{\mathcal{V}}_\subT$, and hence
Lemma~\ref{lem:productrules}\ref{item:pri} implies that
\begin{align*}
\matpartx\,\varkappa & =
\Delta_s\,{\mathcal{V}}
- \nabs\,.\,((\nabs\,\vec\nu)\,\vec{\mathcal{V}}_\subT)
+ {\mathcal{V}}\,|\nabs\,\vec\nu|^2 
+ \nabs\,\vec{\mathcal{V}}_\subT:\nabs\,\vec\nu \\ & 
= \Delta_s\,{\mathcal{V}}
 - (\Delta_s\,\vec\nu) \,.\,\vec{\mathcal{V}}_\subT
- \nabs\,\vec\nu:\nabs\,\vec{\mathcal{V}}_\subT
+ {\mathcal{V}}\,|\nabs\,\vec\nu|^2 
+ \nabs\,\vec{\mathcal{V}}_\subT:\nabs\,\vec\nu \\ &
= \Delta_s\,{\mathcal{V}}
 - (\Delta_s\,\vec\nu) \,.\,\vec{\mathcal{V}}_\subT
+ {\mathcal{V}}\,|\nabs\,\vec\nu|^2 \,.
\end{align*}
Combining this with Lemma~\ref{lem:Deltasnu} yields the 
desired result.
\\
\ref{item:lem10.3ii}
On choosing an arbitrary global parameterization of $\GT$, with induced
tangential velocity field $\vec{\mathcal{V}}_\subT$,
the claim follows from \ref{item:lem10.3i}
and Definition~\ref{def:2.18}\ref{item:def2.18ii}.
\end{proof}

\subsection{Gauss--Bonnet theorem}
In the case $d=3$, we can consider curves $\curvegamma$ on a hypersurface
$\Gamma$. For any curve in $\bR^3$, we can define the curvature vector
\begin{equation}\label{eq:identgamma}
\vec\varkappa_\curvegamma = \vec\id_{ss} \qquad\text{on } \curvegamma\,,
\end{equation}
where $\partial_s^2$ denotes the second derivative with respect to arclength
on $\gamma$. We note that \eqref{eq:identgamma} is invariant under a change of
parameterization of the curve.

Later we will need the Gauss--Bonnet theorem, which uses \eqref{eq:identgamma}
for the special case $\curvegamma = \partial\Gamma \subset \Gamma$.

\begin{thm}[Gauss--Bonnet] \label{thm:GB}
Let $\Gamma$ be a compact orientable $C^2$--hyper\-surface in $\bR^3$. 
Then it holds that
\begin{equation*}
\int_\Gamma \Gauss \dH2 = 2\,\pi\,{\rm m}(\Gamma) + \int_{\partial\Gamma}
\vec\varkappa_{\partial\Gamma}\,.\,\conormal \dH1\,,
\end{equation*}
where $\conormal$ is the outer unit conormal to $\partial\Gamma$, and where
where ${\rm m}(\Gamma)\in\bZ$ is the Euler characteristic of $\Gamma$. 
\end{thm}

For a definition of the Euler characteristic and a
proof of the Gauss--Bonnet formula we refer to \citet[\S4F]{Kuhnel15}.

\section{Parametric finite elements} \label{sec:PFEA}
In this section we discuss the main concepts that are necessary for the 
numerical approximation of curvature driven evolution equations with the help
of parametric finite elements.
Readers familiar with these concepts, and readers more interested in the actual
finite element approximations discussed in this \arxivyesno{work}{chapter}, 
may skip this section and go directly to the appropriate sections of interest,
below.

\subsection{Polyhedral surfaces} \label{subsec:polysurf}
In order to approximate a smooth surface we use polyhedral surfaces as
follows. The idea of using polyhedral surfaces to
approximate the curvature driven evolution of hypersurfaces
goes back to \cite{Dziuk91}.

\begin{definition} \label{def:polysurf}
\rule{0pt}{0pt}
\begin{enumerate}
\item \label{item:polysurf}
A subset $\Gamma^h\subset\bR^d$ is called an $n$-dimensional polyhedral 
surface for $1 \leq n \leq d$, 
with or without boundary, if it is the finite union of closed,
nondegenerate $n$-simplices, 
where the intersection of any two simplices is either empty or a common 
$k$-simplex, $0 \leq k < n$.
\item If $n=d-1$, then we call $\Gamma^h$ a polyhedral hypersurface.
\item If $n=1$, then we call $\Gamma^h$ a polygonal curve.
\end{enumerate}
\end{definition}

\begin{rem} \label{rem:polysurf}
\rule{0pt}{0pt}
\begin{enumerate}
\item
The boundary of a polyhedral surface $\Gamma^h$, as a
$C^0$--surface, is defined as in {\rm Definition~\ref{def:boundary}},
and is denoted by $\partial\Gamma^h$. Of course, $\partial\Gamma^h$ is given
as the union of all $(n-1)$-simplices, that form part of the boundary of 
exactly one of the $n$-simplices that make up $\Gamma^h$.
If $\partial\Gamma^h$ is empty, then we call $\Gamma^h$ a 
closed polyhedral surface.
\item
For a polyhedral hypersurface $\Gamma^h$ with boundary, 
the outer conormal to $\Gamma^h$ is well-defined almost everywhere
on $\partial\Gamma^h$, where locally we use the definition
{\rm Definition~\ref{def:conormal}}.
\item \label{item:triangulation}
In order to define geometric quantities for, and finite element spaces on
polyhedral surfaces, it is often convenient to define $\Gamma^h$ in
terms of a triangulation. To this end, from now on, we let
an $n$-dimensional polyhedral surface be given by
\[
\Gamma^h = \bigcup_{j=1}^J \overline{\sigma_j}\,,
\]
where $\{\sigma_j\}_{j=1}^J$ is a family of disjoint, 
(relatively) open $n$-simplices, such that
$\overline{\sigma_i}\cap\overline{\sigma_j}$ for $i\not=j$ is either empty or
a common $k$-simplex of $\overline{\sigma_i}$ and $\overline{\sigma_j}$, 
$0 \leq k < n$.
For later use, we denote the vertices of $\Gamma^h$
by $\{\vec q_k\}_{k=1}^K$, and assume that the vertices of $\sigma_j$
are given by $\{\vec q_{j,k}\}_{k=1}^{n+1}$, $j=1,\ldots,J$.
\end{enumerate}
\end{rem}

\begin{definition} \label{def:Vh}
Let $\Gamma^h=\bigcup_{j=1}^J \overline{\sigma_j}$ be an $n$-dimensional
polyhedral surface, as described in 
{\rm Remark~\ref{rem:polysurf}\ref{item:triangulation}},
with vertices $\{\vec q_k\}_{k=1}^K$.
\begin{enumerate}
\item \label{item:Vh}
We define the finite element spaces of continuous
piecewise linear functions on $\Gamma^h$ via
\begin{align*}
\Whh &= \{\chi \in C(\Gamma^h) : \chi_{\mid_{\sigma_j}}
\text{ is affine for $j=1, \ldots, J$} \}\,,\\
\Vhh & = [\Whh]^d\,,\quad\matVhh = [\Whh]^{d\times d} \,.
\end{align*}
We let $\{\phi^{\Gamma^h}_k\}_{k=1}^K$ denote
the standard basis of $\Whh$, i.e.\ 
\[
\phi^{\Gamma^h}_i(\vec q_j) = \delta_{ij}\,, \qquad i,j = 1,\ldots,K\,.
\]
Moreover, we let $\pi_{\Gamma^h}:C(\Gamma^h) \to \Whh$ be the standard 
interpolation operator, i.e.\
\[ \pi_{\Gamma^h}\,\eta = \sum_{k=1}^K \eta(\vec q_k)\,\phi^{\Gamma^h}_k 
\quad\forall\ \eta \in C(\Gamma^h)\,,
\]
and similarly $\vec\pi_{\Gamma^h}:[C(\Gamma^h)]^d \to \Vhh$.
\item \label{item:Vch}
We define the spaces of piecewise constant functions on $\Gamma^h$ via
\begin{align*}
\Wchh &= \{\chi \in L^\infty(\Gamma^h) : \chi_{\mid_{\sigma_j}}
\text{ is constant for $j=1, \ldots, J$} \}\,,\\
\Vchh & = [\Wchh]^d\,,\quad\matVchh = [\Wchh]^{d\times d} \,.
\end{align*}
\item \label{item:ip0}
We let $\left\langle \cdot, \cdot \right\rangle_{\Gamma^h}$
denote the $L^2$--inner product on $\Gamma^h$,
with \text{$|\cdot|_{\Gamma^h}$} the associated $L^2$--norm,
and we extend these definitions to any $n$-dimensional piecewise 
$C^1$--surface $\Gamma$.
For piecewise continuous functions, 
$u,\vvv\in L^\infty(\Gamma^h)$, with possible jumps
across the edges of $\{\sigma_j\}_{j=1}^J$,
we introduce the mass lumped inner product
$\left\langle\cdot,\cdot\right\rangle^h_{\Gamma^h}$ as
\begin{equation}
\left\langle u, \vvv \right\rangle^h_{\Gamma^h} =
\frac1{n+1}\,\sum_{j=1}^J \mathcal{H}^{n}(\sigma_j)\,
\sum_{k=1}^{n+1} (u\,\vvv)((\vec q_{j,k})^-)\,,
\label{eq:ip0}
\end{equation}
where $u((\vec q)^-) =
\underset{\sigma_j\ni \vec p\to \vec q}{\lim}\, u(\vec p)$.
The definition \eqref{eq:ip0}
is naturally extended to vector- and tensor-valued functions. We also let
\begin{equation} \label{eq:ip0norm}
\left|u\right|_{\Gamma^h}^h = \left( \left\langle u, u 
\right\rangle^h_{\Gamma^h}\right)^\frac12 ,
\end{equation}
which on $\Whh$ defines the norm induced by \eqref{eq:ip0}. We extend
\eqref{eq:ip0norm} to vector-valued functions to obtain a norm on $\Vhh$.
\end{enumerate}
\end{definition}

\begin{rem} \label{rem:ip0}
It follows from {\rm Definition~\ref{def:Vh}} that
\[
\left\langle \eta, 1 \right\rangle^h_{\Gamma^h} = 
\left\langle  \pi_{\Gamma^h}\,\eta, 1 \right\rangle_{\Gamma^h} \qquad
\forall\ \eta \in C(\Gamma^h)\,.
\]
\end{rem}

\subsubsection{Orientation}
In order to discuss the orientation of a polyhedral hypersurface, we begin with
the definition of the wedge product in $\bR^d$.

\begin{definition} \label{def:wedge}
Let $\vec\vvv_1, \ldots, \vec\vvv_{d-1} \in \bR^d$. Then the wedge product 
\[
 \vec z = \vec\vvv_1\land\cdots\land\vec\vvv_{d-1}
\]
is the unique vector $\vec z \in \bR^d$ such 
that $\vec b\,.\,\vec z= \det(\vec\vvv_1,\ldots,\vec\vvv_{d-1}, \vec b)$
for all $\vec b\in\bR^d$. 
\end{definition}

\begin{rem} \label{rem:wedge}
\rule{0pt}{0pt}
\begin{enumerate}
\item \label{item:wedge2d3d}
The wedge product is the usual cross product
of two vectors in $\bR^3$, and the anti-clockwise rotation 
through $\frac\pi2$ of a vector in $\bR^2$.
\item \label{item:wedge}
The wedge product $\vec\vvv_1\land\cdots\land\vec\vvv_{d-1}$
is perpendicular to each of the 
$(d-1)$ vectors $\vec\vvv_1,\ldots,\vec\vvv_{d-1}$, and has length equal to 
the volume of the parallelotope spanned by them.
\item \label{item:measure}
The measure of a $(d-1)$-simplex $\sigma$, with vertices
$\{\vec q_{k}\}_{k=1}^{d}$, can be computed via 
\[
\mathcal{H}^{d-1}(\sigma) = 
\frac{1}{d-1}\,|(\vec q_{2}-\vec q_{1}) \land \cdots \land
(\vec q_{d}-\vec q_{1})|\,.
\]
\end{enumerate}
\end{rem}

We recall the following definition of orientable polyhedral hypersurface 
\arxivyesno{}{\linebreak}%
from \citet[p.~20]{Matveev06}.

\begin{definition} \label{def:orientable}
We say that the polyhedral hypersurface 
$\Gamma^h=\bigcup_{j=1}^J \overline{\sigma_j}$ is
orientable, if it is possible to consistently 
orientate the simplices $\{\sigma_j\}_{j=1}^J$, e.g.\ by choosing the order
$\{\vec q_{j,k}\}_{k=1}^d$ for the vertices of $\sigma_j$,
$j=1,\ldots,J$, in such a way, that on nonempty intersections 
$\overline{\sigma_i} \cap \overline{\sigma_j}$ that form a $(d-2)$-simplex,
the two orientations induced by $\sigma_i$ and $\sigma_j$ are
opposite to each other.
\end{definition}

\begin{rem} \label{rem:orientable}
For a polyhedral hypersurface in $\bR^3$,
each triangle is oriented by choosing a direction 
around the boundary of the triangle. On each triangle, this gives a direction 
to every edge of the triangle. 
If this can be done in such a way, that two neighbouring edges are always
pointing in the opposite direction, then the surface is orientable. An example
for a non-orientable polyhedral hypersurface in $\bR^3$ is a triangulation of
the M\"obius strip.
Of course, for $d=2$ we are dealing with polygonal curves,
and they are always orientable.
\end{rem}

\begin{definition} \label{def:nuh}
Let $\Gamma^h=\bigcup_{j=1}^J \overline{\sigma_j}$ be an
orientable polyhedral hypersurface $\Gamma^h$, with a consistent
ordering of the vertices $\{\vec q_{j,k}\}_{k=1}^d$ for each
$\sigma_j$, $j=1,\ldots,J$.
Then we define the consistent, piecewise constant unit normal 
$\vec\nu^h \in \Vchh$ 
via
\begin{equation*} 
 \vec\nu^h =
\frac{(\vec q_{j,2}-\vec q_{j,1}) \land \cdots \land
(\vec q_{j,d}-\vec q_{j,1})}{|(\vec q_{j,2}-\vec q_{j,1}) \land
\cdots \land (\vec q_{j,d}-\vec q_{j,1})|}
\qquad\text{on } \sigma_j\,,
\end{equation*}
$j = 1,\ldots,J$, recall {\rm Definition~\ref{def:wedge}}.
\end{definition}

\begin{rem} \label{rem:nuh}
Of course, changing the orientation of $\Gamma^h$ will change the
sign of $\vec\nu^h$. For the majority of the approximations introduced in this
\arxivyesno{work}{chapter}, 
the choice of normal is not important. We will clearly state 
the choice of the sign of $\vec\nu^h$ in situations where it is critical. 
\end{rem}

\begin{definition} \label{def:omegah}
Let $\Gamma^h$ be an orientable polyhedral hypersurface with unit normal
$\vec\nu^h$. Then we define the vertex normal vector 
$\vec\omega^h \in \Vhh$ to be the mass-lumped $L^2$--projection 
of $\vec\nu^h$ onto $\Vhh$, i.e.\
\begin{equation} \label{eq:omegah}
\left\langle \vec\omega^h, \vec\varphi \right\rangle_{\Gamma^h}^h 
= \left\langle \vec\nu^h, \vec\varphi \right\rangle_{\Gamma^h}
\quad\forall\ \vec\varphi\in\Vhh\,.
\end{equation}
\end{definition}

\begin{rem} \label{rem:omegah}
It is easy to see that, for $k=1,\ldots,K$,
\begin{equation} \label{eq:omegahk}
\vec\omega^h(\vec q_k) = 
\frac{1}{\mathcal{H}^{d-1}(\Lambda_k)}
\sum_{\sigma_j\in\mathcal{T}_k} 
\mathcal{H}^{d-1}(\sigma_j)\,\vec\nu^h_{\mid_{\sigma_j}}
\,,
\end{equation}
where
\[
\Lambda_k = \bigcup_{\sigma_j \in \mathcal{T}_k} \overline{\sigma_j}
\quad\text{ and } \quad
\mathcal{T}_k= \{\sigma_j : \vec q_k \in \overline{\sigma_j}\}
\,.
\]
In particular, we note that one can interpret 
$\vec\omega^h(\vec q_k)$ as a weighted normal at the
vertex $\vec q_k$ of $\Gamma^h$.
It follows from \eqref{eq:ip0} and \eqref{eq:omegah} that 
\begin{equation} \label{eq:omegahnuh}
\left\langle \chi\,\vec\omega^h, \vec\varphi \right\rangle_{\Gamma^h}^h 
= \left\langle \chi\,\vec\nu^h, \vec\varphi \right\rangle_{\Gamma^h}^h
\quad\forall\ \chi\in\Whh\,,\ \vec\varphi\in\Vhh\,.
\end{equation}
Combining \eqref{eq:omegahnuh} and \eqref{eq:omegah} yields that
\begin{equation} \label{eq:intnuh}
\left\langle \vec\omega^h, \vec\varphi \right\rangle_{\Gamma^h}^h 
= \left\langle \vec\nu^h, \vec\varphi \right\rangle_{\Gamma^h}^h
= \left\langle \vec\nu^h, \vec\varphi \right\rangle_{\Gamma^h}
\quad\forall\ \vec\varphi\in\Vhh\,.
\end{equation}
\end{rem}

\subsubsection{Polygonal curves} \label{subsubsec:polycurves}
Most of the above definitions simplify dramatically when $\Gamma^h$ is a
polygonal curve. Given a closed polygonal curve 
$\Gamma^h=\bigcup_{j=1}^J \overline{\sigma_j}$, we can parameterize
$\Gamma^h$ with the help of a finite element function defined on the
periodic unit interval $\bI = \RZ$.

\begin{definition} \label{def:VhI}
Let $\bI=\bigcup^J_{j=1} I_j$ be decomposed into the intervals
$I_j=[q_{j-1},$ $q_j]$, given by the nodes $q_j=j\,h$, $h=J^{-1}$, for 
$j=0,\ldots,J$. We make use of the periodicity of $\RZ$, i.e.\ 
$q_J= q_0$, $q_{J+1} = q_1$ and so on.
\begin{enumerate}
\item
We define the finite element spaces of periodic, continuous piecewise linear
functions in $\bI$ via
\begin{align*}
\WhI & = \{\chi\in C(\bI) : \chi_{\mid_{I_j}}
\text{ is affine for $j=1,\ldots,J$} \} \,, \\ 
\VhI & = [\WhI]^d\,.
\end{align*}
\item
We let $\left\langle\cdot,\cdot\right\rangle_{\bI}$ denote the $L^2$--inner
product on $\bI$. 
For piecewise continuous functions, $u,\vvv\in L^\infty(\bI)$, 
with possible jumps at the nodes $\{q_j\}_{j=1}^J$, we introduce the mass
lumped inner product on $\bI$ as
\begin{equation*}
\left\langle u, \vvv \right\rangle_\bI^h = \tfrac12\,h\,\sum_{j=1}^J 
\left[(u\,\vvv)(q_j^-) + (u\,\vvv)(q_{j-1}^+)\right]. 
\end{equation*}
\end{enumerate}
\end{definition}

\begin{rem} \label{rem:VhI}
Given a closed polygonal curve $\Gamma^h=\bigcup_{j=1}^J \overline{\sigma_j}$,
we can now find a function $\vec X^h \in \VhI$ such that 
$\Gamma^h = \vec X^h(\bI)$. Then the following hold.
\begin{enumerate}
\item
It follows from {\rm Remark~\ref{rem:altforms}\ref{item:altformg}} that
\begin{alignat*}{2}
(\nabs\, f) \circ \vec X^h
&=  \partial_s\,(f\circ \vec X^h)\,\vec X^h_s
\qquad &&\text{in } I_j\,,\\
(\nabs\,.\,\vec f)\circ \vec X^h
&= \partial_s\,(\vec f\circ \vec X^h)\,.\,\vec X^h_s
\qquad &&\text{in } I_j\,,
\end{alignat*}
for $j=1,\ldots,J$,
where $\vec X^h_s = \partial_s\,\vec X^h$ and
$\partial_s = |\partial_1\,\vec X^h|^{-1}\,\partial_1$ denotes
differentiation with respect to the arclength of $\Gamma^h$.
{From} now on we define the shorthand notation 
$\vec X^h_\rho = \partial_1\,\vec X^h$, i.e.\ $\rho \in \bI$ plays the role
of the parameterization variable.
\item
We have that
\[
\left\langle f, 1 \right\rangle_{\Gamma^h} =
\left\langle f \circ \vec X^h, |\vec X^h_\rho| \right\rangle_{\bI} 
\quad\text{and}\quad
\left\langle f, 1 \right\rangle_{\Gamma^h}^h =
\left\langle f \circ \vec X^h, |\vec X^h_\rho| \right\rangle_{\bI}^h. 
\]
\item \label{item:nuh}
For the normal $\vec\nu^h$ on $\Gamma^h$ defined as in 
{\rm Definition~\ref{def:nuh}} and 
\arxivyesno{{\rm Remark~\ref{rem:wedge}\ref{item:wedge2d3d}},}
{{\rm Remark \ref{rem:wedge}\ref{item:wedge2d3d}},}
it holds that
\begin{equation} \label{eq:nuh2d}
\vec\nu^h \circ \vec X^h = -(\vec X^h_s)^\perp
\qquad\text{in } I_j\,,
\end{equation}
if $\overline{\sigma_j} = [\vec q^h_{j,1}, \vec q^h_{j,2}] =
[\vec X^h(q_{j-1}), \vec X^h(q_{j})]$,
for $j=1,\ldots,J$.
Here $\cdot^\perp$, acting on $\bR^2$, denotes clockwise rotation by
$\frac{\pi}{2}$.
\item \label{item:omegah}
In order to find a simple expression for
the vertex normal $\vec\omega^h$ on $\Gamma^h$ defined as in 
{\rm Definition~\ref{def:omegah}}, we let
$\vec h_j = \vec X^h(q_j) -\vec X^h(q_{j-1})$,
$j=1,\ldots,J+1$, which according to 
{\rm Definition~\ref{def:polysurf}\ref{item:polysurf}} are nonzero.
Then \eqref{eq:nuh2d} reduces to
\[
(\vec\nu^h \circ \vec X^h)_{\mid_{I_j}}
= \vec\nu^h_j = - \frac{\vec h_j^\perp}
{|\vec h_j|} \qquad j = 1,\ldots,J+1\,.
\]
Hence, on recalling \eqref{eq:omegahk}, we obtain the weighted vertex normals
\begin{align}
\vec\omega^h(\vec X^h(q_j)) & = \frac{|\vec h_{j}|\,\vec\nu^h_{j} +
|\vec h_{j+1}|\,\vec\nu^h_{j+1}}{|\vec h_{j}| + |\vec h_{j+1}|}
 = - \frac{\left(\vec h_{j} + \vec h_{j+1} \right)^\perp}
{|\vec h_{j}| + |\vec h_{j+1}|} \nonumber \\ &
 = - \frac{\left(\vec X^h(q_{j+1}) - \vec X^h(q_{j-1}) \right)^\perp}
{|\vec h_{j}| + |\vec h_{j+1}|} \quad
 j = 1,\ldots,J\,.
\label{eq:omegah2d}
\end{align}
\end{enumerate}
\end{rem}

\subsection{Stability estimates} \label{subsec:stabest}

\begin{lem}\label{lem:stab3d}
Let $\Gamma^h=\bigcup^J_{j=1}\overline{\sigma_j}$ be a
two-dimensional polyhedral surface. 
Then we have for $j=1,\ldots, J$ that
\begin{equation}\label{eq:DNorm}
\tfrac12\,\int_{\sigma_j}|\nabs\,\vec X|^2\dH{2} \geq 
\mathcal{H}^2(\vec X(\sigma_j)) \qquad 
\forall\ \vec X \in \Vhh
\end{equation}
with equality for $\vec X= \vec\id_{\mid_{\Gamma^h}} 
\in \Vhh$.
\end{lem}
\begin{proof}
We note that the integrands in \eqref{eq:DNorm} are constant. In
particular, we recall from Definition~\ref{def:2.5}\ref{item:def2.5v} that,
for $\vec X \in \Vhh$,
\begin{subequations} \label{eq:tander}
\begin{equation}\label{eq:tanderY}
\nabs\,\vec X = \sum^2_{i=1} (\partial_{\vec\tau_i}\,\vec X)
\otimes\vec\tau_i\quad\text{and}\quad
|\nabs\,\vec X|^2 = \sum^2_{i=1} |\partial_{\vec\tau_i}\,\vec X
|^2 \qquad\text{on } \sigma_j\,,
\end{equation}
and so
\begin{equation}\label{eq:tanderid}
\nabs\,\vec\id = \sum^2_{i=1} \vec\tau_i\otimes\vec\tau_i
\quad\text{and}\quad
|\nabs\,\vec\id|^2 = 2 \qquad\text{on } \sigma_j\,,
\end{equation}
\end{subequations}
where $\{\vec\tau_1,\vec\tau_2\}$ is an orthonormal basis for the tangent plane
of $\sigma_j$. Moreover, it holds that
\begin{equation} \label{eq:sigarea}
\mathcal{H}^2(\vec X(\sigma_j))=\int_{\sigma_j} \sqrt{g} \dH2\,,
\end{equation}
where, similarly to (\ref{eq:int}) and (\ref{eq:detg}), 
\[
g=\det\left( \partial_{\vec\tau_i}\,\vec X\,.\,\partial_{\vec\tau_j}\,\vec X 
\right)_{i,j=1,2}
= |\partial_{\vec\tau_1}\,\vec X|^2\,|\partial_{\vec\tau_2}\,\vec X|^2
-\left(\partial_{\vec\tau_1}\,\vec X\,.\,\partial_{\vec\tau_2}\,\vec X\right)^2.
\]
Next, we note that
\begin{equation} \label{eq:vecineq}
\sqrt{g} \leq |\partial_{\vec\tau_1}\,\vec X|\,|\partial_{\vec\tau_2}\,\vec X|
\leq \tfrac12 
\left(|\partial_{\vec\tau_1}\,\vec X|^2 + 
|\partial_{\vec\tau_2}\,\vec X|^2 \right) ,
\end{equation}
with equality if and only if 
$\partial_{\vec\tau_1}\,\vec X\,.\,\partial_{\vec\tau_2}\,\vec X=0$ and 
$|\partial_{\vec\tau_1}\,\vec X|=|\partial_{\vec\tau_2}\,\vec X|$.
The desired results (\ref{eq:DNorm}) then follow immediately
on combining (\ref{eq:tander}), (\ref{eq:sigarea}) and (\ref{eq:vecineq}).
\end{proof}

\begin{rem}
A result like that in {\rm Lemma~\ref{lem:stab3d}} is not true 
for an $n$-dimen\-sional polyhedral surface 
with $n \neq 2$. In this case 
\[
\mathcal{H}^{n}(\vec X(\sigma_j))=\int_{\sigma_j} \sqrt{g} \dH{n}\,,
\]
where $g=\det\left( \partial_{\vec\tau_i}\,\vec X\,.\,
\partial_{\vec\tau_j}\,\vec X \right)_{i,j=1,\ldots,n}$,
with $\{\vec\tau_1,\ldots,\vec\tau_n\}$ being an orthonormal basis for 
the tangent space of $\sigma_j$,
scales with respect to $\vec X$ with the power $n$. Whereas
\[
|\nabs\, \vec X|^2 = \sum^{n}_{i=1} \left| \partial_{\vec\tau_i}\,\vec X
\right|^2 
\]
scales to the power two. Hence a simple scaling argument shows that
there can be no constant $c_0$ such that 
$$
c_0\,\int_{\sigma_j}|\nabs\,\vec X|^2 \dH{n} \ge 
\mathcal{H}^{n}(\vec X(\sigma_j))\quad\forall\ \vec X\in \Vhh\,.
$$
This shows that the estimate in
{\rm Lemma~\ref{lem:stab3d}} 
can only be used for 2-dimensio\-nal polyhedral surfaces.
\end{rem}

Despite the above remark, we are able to prove the following crucial 
stability bound, which will be extensively used in later sections, 
for $n=1$ as well as for $n=2$. 

\begin{lem} \label{lem:stab2d3d}
Let $\Gamma^h=\bigcup^J_{j=1}\overline{\sigma_j}$ be an $n$-dimensional 
polyhedral surface, and let $\vec X \in \Vhh$.
Then it holds, in the case $n=1$, that
\begin{equation*} 
\left\langle\nabs\,\vec X, \nabs\,(\vec X-\vec\id)
\right\rangle_{\Gamma^h}
\geq \mathcal{H}^{1}(\vec X(\Gamma^h)) - 
\mathcal{H}^{1}(\Gamma^h) + 
\left| |\nabs\,\vec X| - 1 \right|^2_{\Gamma^h} .
\end{equation*}
Moreover, in the case $n=2$, we have that
\begin{equation*} 
\left\langle\nabs\,\vec X, \nabs\,(\vec X-\vec\id)
\right\rangle_{\Gamma^h}
\geq \mathcal{H}^{2}(\vec X(\Gamma^h)) - 
\mathcal{H}^{2}(\Gamma^h) + 
 \tfrac12\left| \nabs\,(\vec X-\vec\id)\right|^2_{\Gamma^h} .
\end{equation*}
\end{lem}
\begin{proof}
For $n=2$ it follows from Lemma~\ref{lem:stab3d} that
\begin{align*}
&\left\langle\nabs\,\vec X,\nabs\,(\vec X-\vec\id)\right\rangle_{\Gamma^h} 
\nonumber \\ & \qquad
 = \tfrac12 \left[ \left|\nabs\,\vec X\right|^2_{\Gamma^h} 
- \left|\nabs\,\vec\id\right|^2_{\Gamma^h} 
+ \left|\nabs\,(\vec X-\vec\id)\right|^2_{\Gamma^h} \right]
\nonumber \\ & \qquad
\geq \mathcal{H}^{2}(\vec X(\Gamma^h)) - \mathcal{H}^{2}(\Gamma^h)
+ \tfrac12\left|\nabs\,(\vec X-\vec\id)\right|^2_{\Gamma^h} . 
\end{align*}
For $n=1$, we let $\vec h_j = \vec q_{j,2} - \vec q_{j,1}$,
and similarly $\vec h_j^{\vec X} = \vec X(\vec q_{j,2}) - 
\vec X(\vec q_{j,1})$, for $j=1,\ldots,J$. 
Then, on using ideas from \citet[Theorem~2]{Dziuk99b}, 
it follows from the Cauchy--Schwarz inequality that
\begin{align}
& \left\langle\nabs\,\vec X, \nabs\,(\vec X-\vec\id)
\right\rangle_{\Gamma^h}
= \sum_{j=1}^J \left[\frac{|\vec h_j^{\vec X}|^2 -
\vec h_j^{\vec X}\,.\,\vec h_j}{|\vec h_j|}\right] 
\nonumber \\ & \qquad\qquad
= \sum_{j=1}^J \left[\frac{(|\vec h_j^{\vec X}| - |\vec h_j|)^2 +
|\vec h_j^{\vec X}|\,|\vec h_j| -
\vec h_j^{\vec X}\,.\,\vec h_j}{|\vec h_j|}
+ |\vec h_j^{\vec X}| - |\vec h_j| \right] 
\nonumber \\ & \qquad\qquad
\geq \sum_{j=1}^J \left[|\vec h_j^{\vec X}| - |\vec h_j|
+ \left(\frac{| \vec h_j^{\vec X} |}{|\vec h_j|} -\frac{ | \vec h_j
|}{|\vec h_j|} \right)^2\, |\vec h_j| \right] 
\nonumber \\&\qquad\qquad
= \mathcal{H}^1(\vec X(\Gamma^h)) - \mathcal{H}^1(\Gamma^h) + 
\left| | \nabs\,\vec X| - 1 \right|^2_{\Gamma^h} . \label{eq:stab2d}
\end{align}
\end{proof}

The result in Lemma~\ref{lem:stab2d3d} is relevant for semi-implicit time
discretizations. For fully implicit discretizations we need the following
result.

\begin{lem} \label{lem:fdfistab2d}
Let $\Gamma^h=\bigcup^J_{j=1}\overline{\sigma_j}$ be a 
polygonal curve, and let $\vec X \in \Vhh$. 
Then it holds that
\begin{equation*} 
\left\langle\nabs\,\vec\id, \nabs\,(\vec\id-\vec X)
\right\rangle_{\Gamma^h}
\geq \mathcal{H}^{1}(\Gamma^h) - \mathcal{H}^{1}(\vec X(\Gamma^h)) \,.
\end{equation*}
\end{lem}
\begin{proof}
Using the same notation as in the proof of Lemma~\ref{lem:stab2d3d}, we have
that
\begin{align*}
\left\langle\nabs\,\vec\id, \nabs\,(\vec\id-\vec X)
\right\rangle_{\Gamma^h} & 
= \sum_{j=1}^J \left[\frac{|\vec h_j|^2 -
\vec h_j\,.\,\vec h_j^{\vec X}}{|\vec h_j|}\right] 
\nonumber \\ & 
\geq \sum_{j=1}^J \left[|\vec h_j| - |\vec h_j^{\vec X}| \right]
= \mathcal{H}^1(\Gamma^h) - \mathcal{H}^1(\vec X(\Gamma^h)) \,.
\end{align*}
\end{proof}

\begin{rem} \label{rem:fdfistab2d}
A result analogous to {\rm Lemma~\ref{lem:fdfistab2d}} for $n$-dimensional 
polyhedral surfaces with $n > 1$ is not true. To see this, we construct the
following counterexample.
Let $\Gamma^h$ be given by a
single $n$-simplex, and let $\{\vec\tau_1,\ldots,\vec\tau_n\}$ 
be an orthonormal basis for the tangent space of $\Gamma^h$. Now choose 
$\vec X \in \Vhh$ such that
$\partial_{\vec\tau_1}\,\vec X = \alpha\,\vec\tau_1$ and
$\partial_{\vec\tau_i}\,\vec X = \epsilon\,\vec\tau_i$, $i=2,\ldots,n$,
for $\alpha,\epsilon\in\bRplus$.
Then it holds that
$\mathcal{H}^n(\vec X(\Gamma^h)) = \int_{\Gamma^h}\,
\sqrt{\det\left( \partial_{\vec\tau_i}\,\vec X\,.\,
\partial_{\vec\tau_j}\,\vec X \right)_{i,j=1,\ldots,n}}\dH{n}
=\alpha\,\epsilon^{n-1}\, \mathcal{H}^n(\Gamma^h)$.
Moreover, $|\nabs\,\vec\id|^2 = n$ and
$\nabs\,\vec\id : \nabs\,\vec X = \alpha + (n-1)\,\epsilon$
on $\Gamma^h$, and so 
\[
\left\langle\nabs\,\vec\id, \nabs\,(\vec\id-\vec X) \right\rangle_{\Gamma^h}
\geq \mathcal{H}^{n}(\Gamma^h) - \mathcal{H}^{n}(\vec X(\Gamma^h)) 
\]
is equivalent to
$n - (\alpha + (n-1)\,\epsilon) \geq 1 - \alpha\,\epsilon^{n-1}$, and hence to
$(n-1)\,(1 - \epsilon) \geq \alpha\,(1 - \epsilon^{n-1})$. 
Choosing $\epsilon \in (0,1)$ and
$\alpha > (n-1)\,\frac{1-\epsilon}{1-\epsilon^{n-1}}$ 
yields a contradiction.
\end{rem}

\subsection{Curvature approximations}
Given a polyhedral hypersurface
$\Gamma^h=\bigcup_{j=1}^J \overline{\sigma_j}$, it is clear from
Definition~\ref{def:2.5} that first order differential operators
are well-defined almost everywhere on $\Gamma^h$, for example for functions
in $\Whh$ or $\Vhh$. However, second order operators are not.

That means that discrete curvature approximations need to be defined in a 
suitable way. For everything that follows we assume that $\Gamma^h$ is a 
closed hypersurface.

One way is to define the discrete Laplace--Beltrami operator
$\LapGh : \Whh \to \Whh$ via
\begin{equation} \label{eq:discreteLB}
\left\langle \LapGh \chi, \zeta \right\rangle_{\Gamma^h}^h
= - \left\langle \nabs\,\chi, \nabs\,\zeta \right\rangle_{\Gamma^h}
\quad\forall\ \zeta \in \Whh\,,
\end{equation}
which is a discrete analogue of Remark~\ref{rem:ibp}\ref{item:ibp}.
As usual, for $\vec\chi \in \Vhh$, we define $\LapGh \vec\chi$
component-wise. Then a possible approximation to the curvature vector,
recall Lemma~\ref{lem:varkappa}\ref{item:vecvarkappa}, is
$\vec\kappa^h = \LapGh \vec\id$, i.e.\
$\vec\kappa^h \in \Vhh$ is the unique solution to
\begin{equation} \label{eq:veckappah}
\left\langle \vec\kappa^h, \vec\eta\right\rangle^h_{\Gamma^h}
= - \left\langle \nabs\,\vec\id, \nabs\,\vec\eta \right\rangle_{\Gamma^h}
\quad\forall\ \vec\eta \in \Vhh\,.
\end{equation}
Of course, $\vec\kappa^h$ gives both an approximation to the mean curvature, as
well as a notion of a vertex normal direction, which in general will be
different to the direction defined by Definition~\ref{def:omegah}.

Some special polyhedral hypersurfaces 
allow an alternative definition of mean curvature. 

\begin{definition} \label{def:conformal}
A closed orientable polyhedral hypersurface $\Gamma^h$,
with unit normal $\vec\nu^h$, is called a conformal polyhedral
hypersurface, if there exists a $\kappa^h \in \Whh$ such that
\begin{equation} \label{eq:conformal}
\left\langle \kappa^h\,\vec\nu^h, \vec\eta\right\rangle^h_{\Gamma^h}
= - \left\langle \nabs\,\vec\id, \nabs\,\vec\eta \right\rangle_{\Gamma^h}
\quad\forall\ \vec\eta \in \Vhh\,.
\end{equation}
\end{definition}

\begin{rem} \label{rem:conformal}
\rule{0pt}{0pt}
\begin{enumerate}
\item \label{item:parallel}
For a conformal polyhedral hypersurface $\Gamma^h$, 
the two vertex normal directions
defined by $\vec\kappa^h$ in \eqref{eq:veckappah} and $\vec\omega^h$ in
{\rm Definition~\ref{def:omegah}} agree, i.e.\ the two vectors are parallel at
each vertex of $\Gamma^h$. In particular, on recalling \eqref{eq:omegahnuh},
it holds that
\[
\vec\pi_{\Gamma^h}\left[ \kappa^h\,\vec\omega^h \right] 
= \vec\kappa^h\,.
\]
\item
It is discussed in {\rm \citet[\S4.1]{gflows3d}} that for $d=3$ the geometric
property from \ref{item:parallel} means that the triangulation of $\Gamma^h$
is characterized by a good mesh quality. 
In the case $d=2$ it holds that any conformal polygonal curve is weakly
equidistributed, see the following theorem.
\end{enumerate}
\end{rem}

\begin{thm}\label{thm:equid}
Let $\Gamma^h$ be a closed conformal polygonal curve in $\bR^2$,
as defined in {\rm Definition~\ref{def:conformal}}.
Then any two neighbouring elements on $\Gamma^h$ either have equal
length, or they are parallel.
\end{thm}
\begin{proof}
We choose a $\vec X^h \in \VhI$ with $\Gamma^h = \vec X^h(\bI)$. Then it
follows from Remark~\ref{rem:VhI} and Definition~\ref{def:conformal} that
there exists a $\kappa^h \in \WhI$ such that
\begin{equation} \label{eq:conformalI}
\left\langle \kappa^h\,\vec\omega^h\circ\vec X^h, \vec\eta\,
|\vec X^h_\rho| \right\rangle^h_{\bI}
= - \left\langle \vec X^h_s, \vec\eta_s\, |\vec X^h_\rho| 
\right\rangle_{\bI} \quad\forall\ \vec\eta \in \VhI\,.
\end{equation}
On using the notation from Remark~\ref{rem:VhI}\ref{item:omegah},
we fix a $j \in \{1,\ldots,J\}$, and then need to show that
\begin{equation} \label{eq:TMh}
|\vec h_{j}|= |\vec h_{j+1}| 
\quad\mbox{if} \quad\vec h_{j} \nparallel \vec h_{j+1}\,.
\end{equation}
We recall from Definition~\ref{def:polysurf}\ref{item:polysurf}
that $\vec h_{j}$ and $\vec h_{j+1}$ are nonzero.
If $\vec h_{j} + \vec h_{j+1} = \vec0$,
then (\ref{eq:TMh}) directly follows. 
Otherwise, we observe from \eqref{eq:omegah2d} that
\[
\vec\omega^h(\vec X^h(q_j)) = 
- \frac{\left(\vec X^h(q_{j+1}) - \vec X^h(q_{j-1}) \right)^\perp}
{|\vec h_{j}| + |\vec h_{j+1}|} \,.
\]
Hence choosing an $\vec\eta \in \VhI$ in \eqref{eq:conformalI} 
with 
\[
\vec\eta(q_i) = \delta_{ij}\left(\vec X^h(q_{j+1}) - \vec X^h(q_{j-1})\right)
= \delta_{ij}\left(\vec h_{j} + \vec h_{j+1}\right),
\] 
for $i=1,\ldots,J$, we obtain
\begin{equation*} 
0 = \left(\frac{\vec h_{j+1}}{|\vec h_{j+1}|}
- \frac{\vec h_{j}}{|\vec h_{j}|}  \right) .
\left( \vec h_{j} + \vec h_{j+1} \right) 
=
\frac{|\vec h_{j+1}|-|\vec h_{j}|}
{|\vec h_{j}|\,|\vec h_{j+1}|} \left(
|\vec h_{j}|\,|\vec h_{j+1}|
- \vec h_{j}\,.\,\vec h_{j+1} \right).
\end{equation*}
The Cauchy--Schwarz inequality now implies that 
$|\vec h_{j}|= |\vec h_{j+1}|$ if $\vec h_{j}$
and $\vec h_{j+1}$ are not parallel.
\end{proof}

For polyhedral hypersurfaces that do not satisfy the special property
in Definition~\ref{def:conformal}, we can introduce a discrete mean curvature
as follows. 
Find $(\vec X, \kappa^h) \in \Vhh\times\Whh$ such that
\begin{subequations} \label{eq:BGNcurvature}
\begin{align}
 \left\langle \vec X - \vec\id,\chi\,\vec\nu^h\right\rangle^h_{\Gamma^h}
& = 0 \quad\forall\ \chi \in \Whh\,, \label{eq:BGNcurvaturea} \\
\left\langle \kappa^h\,\vec\nu^h, \vec\eta\right\rangle^h_{\Gamma^h}
+ \left\langle \nabs\,\vec X, \nabs\,\vec\eta \right\rangle_{\Gamma^h} & = 0
\quad\forall\ \vec\eta \in \Vhh\,.
\label{eq:BGNcurvatureb}
\end{align}
\end{subequations}

\begin{rem} \label{rem:BGNcurvature}
\rule{0pt}{0pt}
\begin{enumerate}
\item 
The system \eqref{eq:BGNcurvature} can be viewed as a linearization of
\eqref{eq:conformal}, where in some sense
we allow $\Gamma^h$ to deform slightly, by moving
vertices tangentially. 
\item
If $(\vec\id_{\mid_{\Gamma^h}}, \kappa^h)\in \Vhh\times\Whh$ 
solves \eqref{eq:BGNcurvature}, then $\kappa^h$ solves \eqref{eq:conformal}, 
and so $\Gamma^h$ is a conformal polyhedral hypersurface.
\end{enumerate}
\end{rem}

Under a mild assumption, there exists a unique solution to the system
\eqref{eq:BGNcurvature}. 

\begin{assumption} \label{ass:A}
Let $\Gamma^h$ be an orientable polyhedral hypersurface with unit normal
$\vec\nu^h$ and vertex normal vector $\vec\omega^h \in \Vhh$.
\begin{enumerate}
\item \label{item:assA1}
Let $\dim \spa\{\vec\omega^h(\vec q_k)\}_{k=1}^K = d$.
\item \label{item:assA2}
Let $\vec\omega^h(\vec q_k) \not= \vec 0$, $k=1,\ldots,K$.
\end{enumerate}
\end{assumption}

\begin{rem} \label{rem:assA}
{\rm Assumption~\ref{ass:A}\ref{item:assA1}} means that the discrete 
vertex normals of $\Gamma^h$ span the whole space $\bR^d$. 
On recalling \eqref{eq:omegah} we observe that
{\rm Assumption~\ref{ass:A}\ref{item:assA1}} is equivalent to
$\dim \left\{ \int_{\Gamma^h} \chi\,\vec\nu^h \dH{d-1} : \chi \in \Whh \right\}
= d$.
Clearly, {\rm Assumption~\ref{ass:A}} is only violated in very rare occasions. 
For example, it always holds for surfaces $\Gamma^h$ without 
self-intersections.
\end{rem}

\begin{lem} \label{lem:exXk}
Let $\Gamma^h$ satisfy {\rm Assumption~\ref{ass:A}}. Then there exists a unique
solution $(\vec X, \kappa^h) \in \Vhh\times\Whh$ to \eqref{eq:BGNcurvature}. 
\end{lem}
\begin{proof}
As (\ref{eq:BGNcurvature}) is a linear system, where the
number of unknowns equals the number of equations,
it is enough to show uniqueness.
We hence consider the homogeneous system and assume that
$(\vec X_0, \kappa_0) \in \Vhh\times\Whh$ is such that
\begin{subequations} 
\begin{align}
 \left\langle \vec X_0 ,\chi\,\vec\nu^h\right\rangle^h_{\Gamma^h}
& = 0 \quad\forall\ \chi \in \Whh\,, \label{eq:BGNexa} \\
\left\langle \kappa_0\,\vec\nu^h, \vec\eta\right\rangle^h_{\Gamma^h}
+ \left\langle \nabs\,\vec X_0, \nabs\,\vec\eta \right\rangle_{\Gamma^h} & = 0
\quad\forall\ \vec\eta \in \Vhh\,.
\label{eq:BGNexb}
\end{align}
\end{subequations}
Choosing $\chi = \kappa_0\in\Whh$ in \eqref{eq:BGNexa} and 
$\vec\eta = \vec X_0 \in \Vhh$ in \eqref{eq:BGNexb} yields that 
$|\nabs\,\vec X_0|_{\Gamma^h}^2 = 0$, and so
$\vec X_0$ is constant, i.e.\ $\vec X_0 = \vec X^c$ 
on $\Gamma^h$ for $\vec X^c \in \bR^d$.
In particular, choosing $\vec\eta =
\vec\pi_{\Gamma^h}\,[\kappa_0\,\vec\omega^h]$ in \eqref{eq:BGNexb} yields,
on recalling \eqref{eq:omegahnuh} and \eqref{eq:ip0norm}, that
\[
0 = \left\langle \kappa_0\,\vec\omega^h, \kappa_0\,\vec\omega^h
\right\rangle^h_{\Gamma^h} = 
\left(\left| \kappa_0\,\vec\omega^h \right|_{\Gamma^h}^h\right)^2
,
\]
and so $\vec\pi_{\Gamma^h}\,[\kappa_0\,\vec\omega^h] = \vec 0$.
Now Assumption~\ref{ass:A}\ref{item:assA2} yields that $\kappa_0 = 0$.
Moreover, it follows from \eqref{eq:BGNexa}, on recalling \eqref{eq:intnuh}, 
that
\begin{equation*}
0 = \left\langle\vec X^c, \chi\,\vec\nu^h\right\rangle_{\Gamma^h}^h
= \left\langle\vec X^c, \chi\,\vec\nu^h\right\rangle_{\Gamma^h}
= \vec X^c\,.\,\int_{\Gamma^h} \chi\,\vec\nu^h \dH{d-1}  
\quad\forall\ \chi \in \Whh\,, 
\end{equation*}
and so Assumption~\ref{ass:A}\ref{item:assA1}, recall Remark~\ref{rem:assA},
implies that $\vec X^c = \vec0$.
Hence we have shown that there exists a unique solution
$(\vec X,\kappa^h) \in \Vhh\times \Whh$ to (\ref{eq:BGNcurvature}).
\end{proof}

Further discrete curvature approximations can be obtained with the help of
approximations to the Weingarten map, recall Definition~\ref{def:Wp} and
Lemma \ref{lem:nabsnu}. In particular, using
Remark~\ref{rem:ibp}\ref{item:weakweingarten} leads to the following 
discretization of $\nabs\,\vec\nu$, which goes back to \citet[(3.2)]{Heine04}.
Given a closed polyhedral hypersurface $\Gamma^h$ 
and a curvature vector approximation
$\vec\kappa^h\in \Vhh$, find $\mat W^h\in\matVhh$ such that
\begin{subequations} \label{eq:matWh}
\begin{equation} 
\left\langle \mat W^h, \mat\chi \right\rangle_{\Gamma^h} =
- \left\langle \vec\kappa^h, \mat\chi\,\vec\nu^h 
\right\rangle_{\Gamma^h}
- \left\langle \vec\nu^h , \nabs\,.\,\mat\chi \right\rangle_{\Gamma^h} \qquad
\forall\ \mat\chi \in \matVhh\,.
\label{eq:Heine04}
\end{equation}
For example, $\vec\kappa^h$ can be defined via \eqref{eq:veckappah},
or via \eqref{eq:veckappah} without mass lumping, which
corresponds to the choice in \citet[(3.1)]{Heine04}.
We note that $\mat W^h$ is not necessarily symmetric, whereas 
$\nabs\,\vec\nu$ is, recall Lemma \ref{lem:nabsnu}\ref{item:nabsnuT}. 
An alternative approximation of 
$\nabs\,\vec\nu$ replaces \eqref{eq:Heine04} with
\begin{equation} 
\left\langle \mat W^h, \mat\chi \right\rangle_{\Gamma^h}^h =
- \tfrac12 \left\langle \vec\nu^h , (\mat\chi + \mat\chi^\transT)\,\vec\kappa^h 
+ \nabs\,.\,
(\mat\chi + \mat\chi^\transT) \right\rangle_{\Gamma^h}^h \qquad
\forall\ \mat\chi \in \matVhh\,,
\label{eq:Whsym}
\end{equation}
which yields $(\mat W^h)^\transT = \mat W^h$, and which
was considered, for example, in \citet[(4.12b)]{nsns2phase}.

A slightly modified version of \eqref{eq:Heine04} has been utilized in
\cite{willmore}, and is given as follows. 
Given $\Gamma^h$ and a mean curvature approximation
$\kappa^h\in \Whh$, find $\mat W^h\in\matVhh$ such that
\begin{equation} 
\left\langle \mat W^h, \mat\chi \right\rangle_{\Gamma^h}^h =
- \left\langle \kappa^h\,\vec\nu^h , \mat\chi\,\vec\nu^h 
\right\rangle_{\Gamma^h}^h
- \left\langle \vec\nu^h , \nabs\,.\,\mat\chi \right\rangle_{\Gamma^h} \qquad
\forall\ \mat\chi \in \matVhh\,.
\label{eq:Wh}
\end{equation}

Finally, piecewise constant approximations to $\nabs\,\vec\nu$ can be defined 
by
\begin{equation} \label{eq:nabsomega}
\nabs\,\vec\omega^h \in \matVchh \quad\text{and}\quad
\nabs\left(\vec\pi_{\Gamma^h}\frac{\vec\omega^h}{|\vec\omega^h|} 
\right) \in \matVchh\,,
\end{equation}
\end{subequations}
the latter of which clearly needs Assumption~\ref{ass:A}\ref{item:assA2} to
hold, and has been employed in e.g.\ 
\cite{willmore}. We note that the two approximations in \eqref{eq:nabsomega} 
are in general not symmetric. 

\begin{rem} \label{rem:inconsistent}
For the case of curves, $d=2$, and adopting the notation 
of {\rm \S\ref{subsubsec:polycurves}}, we set $\Gamma^h = \vec X^h(\bI)$,
where $\vec X^h \in \VhI$ interpolates $\vec x$ with $\Gamma=\vec x(\bI)$.
Then it is shown in {\rm \citet[Lemma~2.2]{DeckelnickD09}} 
that
the approximation $\vec\kappa^h$ from \eqref{eq:veckappah} 
approximates the true curvature vector, $\vec\varkappa$, of $\Gamma$ 
with order $O(h)$ in $[L^2(\bI)]^2$ for smooth $\Gamma$.
Unfortunately, for the case $d=3$ it is shown in {\rm \cite{Heine04}}
that \eqref{eq:veckappah} and \eqref{eq:Heine04} are not convergent on 
general meshes, see also \cite{HildebrandtPW06}.
Similar conclusions can be drawn from the numerical experiments in
{\rm \citet[\S4.2.1]{willmore}}, and also apply to \eqref{eq:Whsym}
and \eqref{eq:Wh}. 
However, we note that the approximations \eqref{eq:nabsomega} and
\eqref{eq:BGNcurvature} behave better 
in practice, see Tables~2 and 5 and Tables~3 and 6 in {\rm \cite{willmore}},
respectively, for closely related approximations.
Moreover, one can prove convergence for higher order piecewise polynomial
approximations of $\Gamma$ and $\vec\varkappa$, see {\rm \cite{Heine04}}.
Even though \eqref{eq:veckappah} may not be convergent for continuous 
piecewise linears,
it turns out that the use of such an approximation does lead to convergence
in approximating geometric flows;
see e.g.\ {\rm \S\ref{subsubsec:Dziuk}} below. 
\end{rem}

\subsection{Evolving polyhedral surfaces and transport theorems} 
\label{subsec:ESFEM}

We now define discrete analogues to 
evolving hypersurfaces, their velocity fields and material time derivatives.
For more details we refer to \citet[\S5.4]{DziukE13}.

\begin{definition} \label{def:GhT}
\rule{0pt}{0pt}
\begin{enumerate}
\item \label{item:GhT}
Let $(\Gamma^h(t))_{t\in [0,T]}$ be a family of polyhedral
hypersurfaces, such that each $\Gamma^h(t)$ admits a triangulation of the form
{\rm Remark~\ref{rem:polysurf}\ref{item:triangulation}} for fixed $J$ and $K$, 
and such that the position of each vertex $\vec q_k$, 
$k=1,\ldots,K$, is a $C^1$-function in time. Then the set
\begin{equation*}
\GhT = \bigcup_{t\in[0,T]} (\Gamma^h(t)\times\{t\})
\end{equation*}
is called an evolving polyhedral hypersurface.
We will often identify $\GhT$ with \arxivyesno{\linebreak}{}%
$(\Gamma^h(t))_{t\in [0,T]}$, 
and call the latter also an evolving polyhedral hypersurface.
\item \label{item:vecVh}
The velocity of $\Gamma^h(t)$ on $\GhT$ is defined by
\begin{equation*}
\vec{\mathcal{V}}^h(\vec z, t) = \sum_{k=1}^{K}
\left[\ddt\,\vec q_k(t)\right] \phi^{\Gamma^h(t)}_k(\vec z) 
\quad\forall\ (\vec z,t) \in \GhT\,,
\end{equation*}
where we have recalled the notation from 
{\rm Definition~\ref{def:Vh}\ref{item:Vh}}.
\item
We define the finite element spaces
\[
\WhGhT = \{ \chi \in C(\GhT) : 
\chi(\cdot, t) \in \Wht \quad\forall\ t \in [0,T] \}
\]
and $\VhGhT = [\WhGhT]^d$.
\item \label{item:matpartxh}
Let $f \in L^\infty(\GhT)$, with 
$f \in C^1(\SigmahTi{j})$ for $j = 1,\ldots,J$, where
\[
\SigmahTi{j}= \bigcup_{t\in[0,T]} (\overline{\sigma_j(t)}\times\{t\})
\]
is a $C^1$--evolving hypersurface. 
For $j \in \{1,\ldots,J\}$, let 
$\vec x : \Upsilon\times[0,T] \to \bR^d$ be a global parameterization of
$\SigmahTi{j}$ such that $\vec x(\cdot,t) : \Upsilon \to
\overline{\sigma_j(t)}$ is an affine function.
Then we define the discrete time derivative of $f$ by
\begin{equation*} 
\matpartxh\, f = \matpartx\, f 
\qquad\text{on } \SigmahTi{j}\,,
\end{equation*}
recall {\rm Definition~\ref{def:2.18}\ref{item:def2.18i}}.
\item
We define the finite element spaces
\[
\WhTGhT = \{ \chi \in \WhGhT : \matpartxh\,\chi \in C(\GhT) \}
\]
and $\VhTGhT = [\WhTGhT]^d$ of finite element functions on $\GhT$
with a continuous material derivative.
\end{enumerate}
\end{definition}

\begin{rem} \label{rem:GhT}
\rule{0pt}{0pt}
\begin{enumerate}
\item \label{item:phihk}
On introducing the short hand notation $\phi_k^h(\cdot, t) = 
\phi_k^{\Gamma^h(t)}$, it holds that
\[
\matpartxh\,\phi_k^h = 0 \qquad\text{on } \GhT\,,\qquad k = 1,\ldots,K\,.
\]
\item \label{item:matpartxhfh}
In general the discrete material derivative $\matpartxh\,f$ is only defined
piecewise on $\GhT$. But 
a direct consequence of \ref{item:phihk} is that for $\chi \in \WhGhT$,
with $\chi(\vec q_k(\cdot), \cdot) \in C^1([0,T])$, $k=1,\ldots,K$,
it holds that
\begin{equation*}
(\matpartxh\,\chi) (\vec z, t) = \sum_{k=1}^{K}
\left[\ddt\,\chi(\vec q_k(t),t)\right] \phi^{\Gamma^h(t)}_k(\vec z) 
\quad\forall\ (\vec z,t) \in \GhT\,,
\end{equation*}
i.e.\ we can choose a continuous representation of $\matpartxh\,\chi$, and 
hence $\chi \in \WhTGhT$.
\item \label{item:VhinVhGhT}
We have that $\vec{\mathcal{V}}^h \in \VhGhT$ with
$\vec{\mathcal{V}}^h = \matpartxh\,\vec\id$ on $\GhT$.
\item 
On extending $f$ to a neighbourhood of $\SigmahTi{j}$, $j=1,\ldots,J$,
it holds that
\begin{equation*} 
\matpartxh\, f = \partial_t\,f + \vec{\mathcal{V}}^h\,.\,\nabla\,f
\qquad\text{on } \SigmahTi{j}\,,
\end{equation*}
recall {\rm Remark~\ref{rem:timeder}\ref{item:rem2.19ii}}.
\end{enumerate}
\end{rem}

\begin{thm} \label{thm:disctrans}
Let $\GhT$ be an evolving polyhedral hypersurface, and let 
$\eta,\zeta \in \WhTGhT$. 
\begin{enumerate}
\item \label{item:disctrans}
It holds that
\begin{equation*}
\ddt \left\langle \eta, \zeta \right\rangle_{\Gamma^h(t)}
 = \left\langle \matpartxh\,\eta, \zeta \right\rangle_{\Gamma^h(t)}
 + \left\langle \eta, \matpartxh\,\zeta \right\rangle_{\Gamma^h(t)}
+ \left\langle \eta\,\zeta, \nabs\,.\,\vec{\mathcal{V}}^h 
\right\rangle_{\Gamma^h(t)} .
\end{equation*}
\item \label{item:disctransh}
It holds that
\begin{equation*}
\ddt \left\langle \eta, \zeta \right\rangle_{\Gamma^h(t)}^h
 = \left\langle \matpartxh\,\eta, \zeta \right\rangle_{\Gamma^h(t)}^h
 + \left\langle \eta, \matpartxh\,\zeta \right\rangle_{\Gamma^h(t)}^h
+ \left\langle \eta\,\zeta, \nabs\,.\,\vec{\mathcal{V}}^h 
\right\rangle_{\Gamma^h(t)}^h .
\end{equation*}
\end{enumerate}
\end{thm}
\begin{proof}
\ref{item:disctrans}
Using the transport theorem, Theorem~\ref{thm:trans}, on each evolving
simplex $\sigma_j(t)$ of $\Gamma^h(t)$, and using the assumptions
$\eta,\zeta \in \WhTGhT$, leads to
\begin{align*}
\ddt \left\langle \eta, \zeta \right\rangle_{\Gamma^h(t)} &
= \ddt\,\sum_{j=1}^J \int_{\sigma_j(t)} \eta\,\zeta \dH{d-1}  \\ &
= \sum_{j=1}^J \int_{\sigma_j(t)} \matpartxh\,(\eta\,\zeta) + \eta\,\zeta\,
\nabs\,.\,\vec{\mathcal{V}}^h \dH{d-1} \nonumber \\ &
= \left\langle \matpartxh\,\eta, \zeta \right\rangle_{\Gamma^h(t)}
 + \left\langle \eta, \matpartxh\,\zeta \right\rangle_{\Gamma^h(t)}
+ \left\langle \eta\,\zeta, \nabs\,.\,\vec{\mathcal{V}}^h 
\right\rangle_{\Gamma^h(t)} .
\end{align*}
\ref{item:disctransh}
This proof is analogous to \ref{item:disctrans} and can be found in
\citet[Lemma~3.1]{tpfs}.
\end{proof}

\begin{thm} \label{thm:disctransvol}
Let $\GhT$ be an evolving polyhedral hypersurface,
such that $\Gamma^h(t)$ is bounding a domain $\Omega^h(t)\subset\bR^d$,
for $t\in [0,T]$. We assume that $\vec\nu^h(t)$ is the outer unit normal 
to $\Omega^h(t)$ on $\Gamma^h(t)$, and that $f\in C^1(\overline\HhT)$, where
\begin{equation*}
\HhT = \bigcup_{t\in[0,T]} (\Omega^h(t)\times\{t\})\,.
\end{equation*}
Then it holds that
\begin{equation*} 
\ddt\,\int_{\Omega^h(t)} f \dL{d} =
\int_{\Omega^h(t)} \partial_t\,f \dL{d} 
+ \left\langle f , \vec{\mathcal{V}}^h\,.\,\vec\nu^h
\right\rangle_{\Gamma^h(t)}.
\end{equation*}
\end{thm}
\begin{proof} 
This follows as in \citet[\S7.3]{EckGK17} using a variant of
the divergence theorem for Lipschitz domains.
\end{proof}

\subsection{Further results for evolving polyhedral surfaces}
\label{subsec:further}

We state discrete analogues of Lemma~\ref{lem:dtnu} and 
Lemma~\ref{lem:commdiff}.

\begin{lem} \label{lem:dtnuh}
Let $\GhT$ be an evolving polyhedral hypersurface. Then it holds that
\begin{equation*}
\matpartxh\,\vec\nu^h = -(\nabs\,\vec{\mathcal{V}}^h)^\transT\,\vec\nu^h
\qquad\text{a.e.\ on } \Gamma^h(t)\,.
\end{equation*}
\end{lem}
\begin{proof}
Similarly to the proof of Theorem~\ref{thm:disctrans}\ref{item:disctrans},
we appeal to 
\arxivyesno{Lemma~\ref{lem:dtnu}\ref{item:dtnu}}
{Lemma\linebreak\ref{lem:dtnu}\ref{item:dtnu}} 
on each evolving simplex $\sigma_j(t)$ of $\Gamma^h(t)$.
\end{proof}

\begin{lem} \label{lem:commdiffh} 
Let $\GhT$ be an evolving polyhedral hypersurface, and let 
$\eta \in \WhTGhT$, $\vec\eta \in \VhTGhT$. 
Then we have the following results, where we recall from
{\rm Definition~\ref{def:2.5}\ref{item:def2.5vi}} that
$\mat D_s(\vec{\mathcal{V}}^h) = \frac12\, 
\mat P_{\Gamma^h}\,
(\nabs\,\vec{\mathcal{V}}^h + (\nabs\,\vec{\mathcal{V}}^h)^\transT)\,
\mat P_{\Gamma ^h}$ almost everywhere on $\Gamma^h(t)$.
\begin{enumerate}
\item \label{item:commdiffhi}
\begin{equation*}
\matpartxh\,\nabs\,\eta - \nabs\,\matpartxh\,\eta =
[\nabs\,\vec{\mathcal{V}}^h - 2\,\mat D_s(\vec{\mathcal{V}}^h)]\,\nabs\,\eta
\qquad\text{a.e.\ on } \Gamma^h(t)\,.
\end{equation*}
\item \label{item:commdiffhii}
\begin{equation*}
\matpartxh\,\nabs\,\vec\eta - \nabs\,\matpartxh\,\vec\eta = 
(\nabs\,\vec\eta)\,
[\nabs\,\vec{\mathcal{V}}^h - 2\,\mat D_s(\vec{\mathcal{V}}^h)]^\transT
\qquad\text{a.e.\ on } \Gamma^h(t)\,.
\end{equation*}
\item \label{item:commdiffhiii}
\begin{equation*}
\matpartxh\,(\nabs\,.\,\vec\eta) - \nabs\,.\,(\matpartxh\,\vec\eta) =
[\nabs\,\vec{\mathcal{V}}^h-2\,\mat D_s(\vec{\mathcal{V}}^h)] 
: \nabs\,\vec\eta \quad\text{a.e.\ on } \Gamma^h(t)\,.
\end{equation*}
\end{enumerate}
\end{lem}
\begin{proof}
Similarly to the proof of Theorem~\ref{thm:disctrans}\ref{item:disctrans},
we appeal to Lemma~\ref{lem:commdiff} on each evolving simplex
$\sigma_j(t)$ of $\Gamma^h(t)$.
\end{proof}

\section{Mean curvature flow} \label{sec:mc}
The main ideas needed to numerically solve curvature driven evolution 
equations are most easily introduced with the help of the mean curvature flow.
For a family of closed evolving hypersurfaces $(\Gamma(t))_{t\in [0,T]}$
in $\bR^d$, $d\geq2$, we consider at each time the total surface area 
$|\Gamma(t)| = \mathcal{H}^{d-1} (\Gamma(t))$.
In this section we will only consider closed surfaces $\Gamma(t)$, i.e.\
surfaces which are compact and without boundary. 
The transport theorem, Theorem~\ref{thm:trans}, gives
\begin{equation}\label{eq:firstvar}
\ddt\left|\Gamma(t)\right| = - \left\langle \varkappa, \mathcal{V} 
\right\rangle_{\Gamma(t)} ,
\end{equation}
where $\varkappa$ is the mean curvature of $\Gamma(t)$, $\mathcal V$ is its
normal velocity, and where we recall that
$\left\langle \cdot, \cdot \right\rangle_{\Gamma(t)}$
denotes the $L^2$--inner product on $\Gamma(t)$. 
Here we used that $\Gamma(t)$ has no boundary.
We hence obtain that the geometric evolution law for an evolving hypersurface
\begin{equation} \label{eq:MCFgradflow}
\mathcal V = \varkappa \qquad\text{on } \Gamma(t)
\end{equation}
most efficiently decreases the surface area, 
and hence it is also called the $L^2$--gradient flow of $|\Gamma(t)|$,
see e.g.\ \cite{Mantegazza11,Garcke13} for details.
This law is the most fundamental curvature driven evolution law and
has been studied in detail both analytically and numerically, we refer to
\cite{Huisken84,GageH86,Giga06,Mantegazza11,Garcke13} and 
\cite{DeckelnickDE05}, and the references therein, for details.

\subsection{Weak formulation}
In this section we want to introduce, for mean curvature flow, 
ideas introduced by the present authors to approximate curvature 
driven evolution laws for hypersurfaces.
These ideas lead to stable approximations, which, in addition, are such that 
the quality of the mesh approximating the evolving surface in general
remains good. In fact, the latter property is crucial,
as many parametric approaches suffer from mesh degeneracies during the 
evolution. These degeneracies may even lead to situations, where 
the resulting algorithms break down during the evolution. 
The ideas presented in this section
will then be the basis for more complex evolution laws studied later on.
We will also compare our approach to other methods in the 
literature dealing with mean curvature flow.

The basis of our approach is a weak formulation, which we introduce next. 
The goal is to write the evolution law $\mathcal V= \varkappa$ in a weak form.
To this end, we firstly note that the evolution law, 
for a global parameterization $\vec x :\Upsilon \times [0,T] \to \bR^d$, 
and corresponding orientable hypersurfaces $\Gamma(t)= \vec x (\Upsilon,t)$, 
recall Definition~\ref{def:globalx}, can be written as 
\begin{equation} \label{eq:BGNweakmcf}
\vec{\mathcal{V}}\,.\,\vec\nu =\varkappa\,,\qquad
\varkappa\,\vec\nu = \Delta_s\,\vec\id
\qquad\text{on } \Gamma(t)\,,
\end{equation}
where we have noted Remark~\ref{rem:vecV}\ref{item:VVnu}
and Lemma~\ref{lem:varkappa}\ref{item:vecvarkappa}.
On recalling from Remark \ref{rem:ibp}\ref{item:weakvarkappa} the weak
formulation of the second identity, we propose the following weak formulation
for mean curvature flow.
Given a closed hypersurface $\Gamma(0)$, find an evolving hypersurface 
$(\Gamma(t))_{t\in[0,T]}$
with a global parameterization and induced velocity field
$\vec{\mathcal{V}}$, and $\varkappa \in L^2(\GT)$ as follows. 
For almost all $t \in (0,T)$, find
$(\vec{\mathcal V}(\cdot,t),\varkappa(\cdot,t))\in 
[L^2(\Gamma(t))]^d \times L^2(\Gamma(t))$ such that 
\begin{subequations} \label{eq:weakmc}
\begin{align}
 \left\langle \vec{\mathcal{V}}, \chi\,\vec\nu \right\rangle_{\Gamma(t)} 
- \left\langle \varkappa,\chi \right\rangle_{\Gamma(t)} & = 0
\quad\forall\ \chi \in L^2(\Gamma(t))\,, \label{eq:weakmca} \\
 \left\langle \varkappa\,\vec\nu,\vec\eta \right\rangle_{\Gamma(t)} 
+ \left\langle \nabs\,\vec\id , \nabs\,\vec\eta \right\rangle_{\Gamma(t)} & = 0
\quad\forall\ \vec\eta \in \Vt\,. \label{eq:weakmcb} 
\end{align}
\end{subequations}
We note here that we consider closed surfaces and hence no boundary term 
appears in \eqref{eq:weakmcb}.
Using the weak formulation (\ref{eq:weakmcb}) of 
$\varkappa\,\vec\nu = \Delta_s\,\vec\id$ was the fundamental idea of 
\cite{Dziuk91}, which made it possible to approximate smooth surfaces
and their mean curvature by piecewise smooth surfaces.
Finally we remark that $H^1(\Gamma(t))$ denotes the usual Sobolev space
of square integrable functions on $\Gamma(t)$ with square integrable surface
gradient, and we refer to \citet[I\,\S4]{Wloka87} for an introduction to 
Sobolev spaces on 
surfaces. 

\subsection{Finite element approximation}\label{subsec:mcFEA}

Given an initial polyhedral hypersurface $\Gamma^0$ the plan is to construct 
polyhedral hypersurfaces $\Gamma^m$, $m=1,\ldots,M$, which approximate the
true continuous solution $\Gamma(t_m)$ to the mean curvature flow at times 
$0 = t_0 < t_1<\ldots<t_M=T$, 
which form a partition of a time interval $[0,T]$
with time steps 
\begin{equation} \label{eq:ttaum}
\ttau_m= t_{m+1}-t_m\,,\qquad m=0,\ldots,M-1\,.
\end{equation}
An idea going back to \cite{Dziuk91} 
is to parameterize $\Gamma^{m+1}$ over $\Gamma^m$ with the help of 
parameterizations $\vec X^{m+1}:\Gamma^m \to \bR^d$. 
We recall the definitions of the finite element spaces and inner products 
from Definition~\ref{def:Vh}, and then recall the following finite element 
approximation of \eqref{eq:weakmc} for mean curvature flow from
\cite{triplejMC,gflows3d}.

Let the closed polyhedral hypersurface $\Gamma^0$ be an approximation of
$\Gamma(0)$. Then, for $m=0,\ldots,M-1$, find 
$(\vec X^{m+1},\kappa^{m+1}) \in \Vhm\times\Whm$ such that
\begin{subequations} \label{eq:MCIP}
\begin{align}
\left\langle \frac{\vec X^{m+1}-\vec\id}{\ttau_m}, 
\chi\,\vec\nu^m\right\rangle^h_{\Gamma^m}
- \left\langle \kappa^{m+1}, \chi \right\rangle_{\Gamma^m}^h & = 0
\quad\forall\ \chi \in \Whm\,, \label{eq:MCIPa} \\
\left\langle \kappa^{m+1}\,\vec\nu^m, \vec\eta\right\rangle^h_{\Gamma^m}
+ \left\langle \nabs\,\vec X^{m+1}, \nabs\,\vec\eta \right\rangle_{\Gamma^m} & = 0
\quad\forall\ \vec\eta \in \Vhm \label{eq:MCIPb}
\end{align}
\end{subequations}
and set
$\Gamma^{m+1} = \vec X^{m+1}(\Gamma^m).$

Here we recall from Definition~\ref{def:globalx}\ref{item:vecV} that 
$\frac{\vec X^{m+1}-\vec\id}{\ttau_m}$ on $\Gamma^m$
is a natural approximation of $\vec{\mathcal{V}}$ on $\Gamma(t_m)$.
We also remark that although the original problem was highly nonlinear, 
the system (\ref{eq:MCIP}) is linear and easy to solve. The main reason for
this is that the geometry, which enters the weak formulations via the area
element, the normal vector and the surface gradients, is taken explicitly.

\begin{rem}[Surfaces with boundary] \label{rem:boundaries}
We recall that \eqref{eq:MCFgradflow} was derived as the $L^2$--gra\-dient flow
for $|\Gamma(t)|$ from \eqref{eq:firstvar} for an evolving hypersurface without
boundary. If we allow $\partial\Gamma(t)$ to be nonempty, on the other hand,
then \eqref{eq:firstvar} needs to be adapted to
\begin{equation} \label{eq:boundaryfirstvar}
\ddt\left|\Gamma(t)\right| = 
- \left\langle \varkappa\,\vec\nu, \vec{\mathcal{V}}
\right\rangle_{\Gamma(t)}
+ \left\langle \vec{\mathcal{V}} , \conormal \right\rangle_{\partial\Gamma(t)},
\end{equation}
where $\vec{\mathcal{V}}$ is the velocity field induced by a global
parameterization of the evolving hypersurface, and $\conormal(t)$ denotes the
outer unit conormal on $\partial\Gamma(t)$,
recall {\rm Theorem~\ref{thm:trans}}.
Now choosing a boundary condition that makes the last term in
\eqref{eq:boundaryfirstvar} vanish will ensure that \eqref{eq:MCFgradflow} is
still the  $L^2$--gradient flow for $|\Gamma(t)|$. The simplest such boundary
condition fixes
\begin{equation} \label{eq:fixedbc}
\partial\Gamma(t) = \partial\Gamma(0) \quad\forall\ t \in (0,T]\,.
\end{equation}
The approximation \eqref{eq:MCIP} can be easily generalized to this situation,
by replacing the space $\Vhm$ in \eqref{eq:MCIPb} with
\begin{equation} \label{eq:VhmD}
\VhmD = \left\{ 
\vec\eta \in \Vhm : \vec\eta = \vec 0 \text{ on } \partial\Gamma^m
\right\},
\end{equation}
and by seeking $\vec X^{m+1}$ such that $\vec X^{m+1} -
\vec\id_{\mid_{\Gamma^m}} \in \VhmD$. More complicated boundary conditions
can be handled in a similar way, for example the case of prescribed contact
angles when the boundary $\partial\Gamma(t)$ is allowed to move along the
boundary of a fixed given domain.
A further related example 
is the evolution of a cluster of hypersurfaces, where the 
boundaries of a number of surfaces are required to remain attached to each
other. Typically triple junction points in the plane, and triple junction lines
in $\bR^3$ are of interest, and these situations can be naturally approximated
with the methods presented in this \arxivyesno{work}{chapter} 
both for mean curvature flow, as
well as for more general geometric evolution equations. 
We refer the interested reader to the series of 
papers {\rm \cite{triplej,triplejMC,triplejANI,clust3d,ejam3d}}.
\end{rem}

\begin{rem}[Implementation] \label{rem:meshgen}
We note that implementing the system \eqref{eq:MCIP} is not difficult.
For $d=2$ an equivalent finite difference formulation can be derived, see
{\rm \S\ref{subsubsec:curvesMC}} below.
For $d=3$ the scheme \eqref{eq:MCIP} can either be implemented directly 
in a high level computing environment like MATLAB,
or within a finite element toolbox that allows the approximation of PDEs on
two-dimensional hypersurfaces in $\bR^3$.
Examples for such toolboxes are 
ALBERTA, {\rm\cite{Alberta}};
Dune, {\rm\cite{dunefem}};
and FELICITY, {\rm\cite{Walker18}}.
An advantage of these toolboxes is that they allow a nearly
dimension-independent implementation of the scheme, and so the cases $d=2$ and
$d=3$ can be treated together.
For the piecewise linear approximation \eqref{eq:MCIP}, the only difference 
to standard problems on flat, stationary domains in $\bR^{d-1}$ is then, 
that the vertices of the initial mesh, on which the PDE is to be approximated, 
have $d$ coordinates, rather than $d-1$, and that the vertices of the mesh are
moved after each time step.

The initial mesh can either be created with the help of a simple, coarse
macro-trian\-gulation, that is then refined and transformed with the help of the
capabilities of the chosen finite element toolbox. Or it can be created
by using sophisticated 3D volume mesh generators, that allow to extract the
surface mesh of the generated 3D volume mesh. Examples for such volume mesh
generators are 
Gmsh, {\rm\cite{GeuzaineR09}};
CGAL, {\rm\cite{RineauY19}};
TIGER, {\rm\cite{Walker13}};
Cleaver, {\rm\cite{Cleaver}};
and NETGEN, {\rm\cite{Schoberl97}}.
\end{rem}

\subsection{Discrete linear systems}\label{subsec:3.3}
We now describe the linear systems arising from (\ref{eq:MCIP}).
We introduce the matrices $\vec N_{\Gamma^m} \in (\bR^d)^{K\times K}$,
$M_{\Gamma^m}, A_{\Gamma^m}\in \bR^{K\times K}$ and $\mat A_{\Gamma^m}\in
(\bR^{d\times d})^{K\times K}$ with entries
\begin{align}
\left[M_{\Gamma^m}\right]_{kl} & 
=\left\langle\phi_k^{\Gamma^m},\phi_l^{\Gamma^m}
\right\rangle_{\Gamma^m}^h,\quad
\left[\vec N_{\Gamma^m}\right]_{kl}
=\int_{\Gamma^{m}} \pi_{\Gamma^m}\left[\phi_k^{\Gamma^m}
\,\phi_l^{\Gamma^m}\right] \vec\nu^m \dH{d-1}\,, \quad\nonumber \\
\left[A_{\Gamma^m}\right]_{kl} & = \left\langle\nabs\,\phi_k^{\Gamma^m},
\nabs\,\phi_l^{\Gamma^m}\right\rangle_{\Gamma^m} 
\label{eq:mat0}
\end{align}
and $\left[\mat A_{\Gamma^m}\right]_{kl} = \left[A_{\Gamma^m}\right]_{kl}\,
\mat\Id$, where we have recalled Definition~\ref{def:Vh}\ref{item:Vh}.
Assembling these matrices is similar to the situation of finite element
methods for domains in $\bR^{d-1}$. In addition to computing the volume
of a simplex, one has to compute its normal and
the surface gradients of the basis functions.
The latter can be computed using the formula
\begin{equation*}
\nabs\,\phi_k^{\Gamma^m} 
= \sum_{i=1}^{d-1} \left(\partial_{\vec\tau_i}\,\phi_k^{\Gamma^m}\right)
\vec\tau_i\,,
\end{equation*}
where $\{\vec\tau_1,\ldots,\vec\tau_{d-1}\}$ is an orthonormal basis of the
tangent space to the simplex, recall 
Definition~\ref{def:2.5}\ref{item:def2.5ii}.
We can then formulate (\ref{eq:MCIP}) as: Find 
\arxivyesno{$(\delta\vec X^{m+1},\kappa^{m+1}) \in (\bR^d)^K \times \bR^K$}
{$(\delta\vec X^{m+1},$ $\kappa^{m+1}) \in (\bR^d)^K \times \bR^K$}
such that
\begin{equation}
\begin{pmatrix} \ttau_m\,M_{\Gamma^m} & -\vec N^\transT_{\Gamma^m} \\ 
\vec N_{\Gamma^m} & \mat A_{\Gamma^m}
\end{pmatrix}\begin{pmatrix} \kappa^{m+1} \\ \delta\vec X^{m+1}\end{pmatrix}
= \begin{pmatrix} 0 \\ -\mat A_{\Gamma^m}\,\vec X^m \end{pmatrix},
\label{eq:MCvec}
\end{equation}
where, with the obvious abuse of notation, $\delta\vec X^{m+1}
=(\delta\vec X^{m+1}_1,\ldots,\delta\vec X^{m+1}_K)^\transT$, 
$\kappa^{m+1}=(\kappa^{m+1}_1,\ldots,\kappa^{m+1}_K)^\transT$,
and $\vec X^m =(\vec X^m_1,\ldots,\vec X^m_K)^\transT$
are the vectors of coefficients with respect to the standard
basis for $\vec X^{m+1}- \vec\id_{\mid_{\Gamma^m}}$, $\kappa^{m+1}$
and $\vec\id_{\mid_{\Gamma^m}}$, respectively.
In the following section we will show that the above system, under very mild
assumptions, is invertible. Hence it can be solved, for example, 
with a sparse direct solution method like UMFPACK, see \cite{Davis04}, 
in an efficient way.
It is also possible to solve the system iteratively, by first using a Schur
complement approach to (\ref{eq:MCvec}), in order to eliminate $\kappa^{m+1}$, 
and then use a (precondioned) conjugate gradient solver, 
see e.g.\ \cite{gflows3d} for details. In fact, the Schur complement 
approach essentially boils down to the system (\ref{eq:elimkappa}), below, 
which we derive next.  

It is also possible to rewrite the system \eqref{eq:MCIP} 
as an equation for $\vec X^{m+1}$ as the only unknown. 
To this end, let $\vec\omega^m \in \Vhm$ be the vertex normal to $\Gamma^m$,
recall Definition~\ref{def:omegah}.
Then it follows from \eqref{eq:omegahnuh} that
we can compute $\kappa^{m+1}$ from \eqref{eq:MCIPa} via
\begin{equation*}
\kappa^{m+1} = \frac1{\ttau_m}\,\pi_{\Gamma^m} \left[ \left(\vec X^{m+1} - 
\vec\id_{\mid_{\Gamma^m}}\right) .\, \vec\omega^m\right] , 
\end{equation*}
recall Definition~\ref{def:Vh}\ref{item:Vh}.
Hence we can rewrite \eqref{eq:MCIP} as:
Find $\vec X^{m+1} \in \Vhm$ such that
\begin{equation}
\left\langle \frac{\vec X^{m+1}-\vec\id}{\ttau_m}\,.\,
\vec\omega^m, \vec{\eta}\,.\,\vec\omega^m
\right\rangle_{\Gamma^m}^h +
\left\langle \nabs\,\vec X^{m+1},\nabs\,\vec{\eta}\right\rangle_{\Gamma^m}
= 0
\quad\forall\ \vec{\eta} \in \Vhm\,.
\label{eq:elimkappa}
\end{equation}

\subsubsection{Curves in the plane} \label{subsubsec:curvesMC}
The system \eqref{eq:MCIP} is particularly simple in the case of closed curves.
On recalling Definition~\ref{def:VhI}, we can reformulate it as follows.

Let $\vec X^0 \in \VhI$ be such that $\Gamma^0 = \vec X^0(\bI)$ is a polygonal
approximation of $\Gamma(0)$. Then, for $m=0,\ldots,M-1$,  
find $(\vec X^{m+1},\kappa^{m+1}) \in \VhI\times \WhI$ such that
\begin{subequations} \label{eq:MCIP2d}
\begin{align}
& \left\langle \frac{\vec X^{m+1}-\vec X^m}{\ttau_m}, 
\chi\,\vec\nu^m\circ\vec X^m\,
|\vec X^m_\rho| \right\rangle^h_{\bI}
- \left\langle \kappa^{m+1}, \chi \,|\vec X^m_\rho| 
\right\rangle_{\bI}^h = 0
\quad\forall\ \chi \in \WhI\,, \label{eq:MCIP2da} \\
& \left\langle \kappa^{m+1}\,\vec\nu^m \circ \vec X^m, 
\vec\eta\,|\vec X^m_\rho| \right\rangle^h_{\bI}
+ \left\langle \vec X^{m+1}_\rho, \vec\eta_\rho\,|\vec X^m_\rho|^{-1} 
\right\rangle_{\bI} = 0
\quad\forall\ \vec\eta \in \VhI \label{eq:MCIP2db}
\end{align}
\end{subequations}
and set $\Gamma^{m+1} = \vec X^{m+1}(\bI)$.

It is now straightforward to rewrite \eqref{eq:MCIP2d} as a finite 
difference scheme. Let $\kappa_j^{m+1}$, $\vec X_j^{m+1}$, $\vec X_j^{m}$
and $\vec\omega^m_j$ be the values of 
$\kappa^{m+1}$, $\vec X^{m+1}$, $\vec X^{m}$ and $\vec\omega^m \circ \vec X^m$
at the node $q_j$, for $j=0,\ldots,J+1$. 
Let $\vec h^m_j = \vec X^m_j -\vec X^m_{j-1}$, $j=1,\ldots,J+1$.
Then, on recalling \eqref{eq:omegah2d}, we obtain the weighted vertex normals
\begin{equation}
\vec\omega^m_j = - \frac{\left(\vec X^m_{j+1} - \vec X^m_{j-1} \right)^\perp}
{|\vec h^m_{j}| + |\vec h^m_{j+1}|}
\qquad j = 1,\ldots,J\,.
\label{eq:omega2d}
\end{equation}
Equation (\ref{eq:MCIP2da}) can hence be 
rewritten in the following finite difference form
\begin{subequations} \label{eq:FD}
\begin{equation}
 \frac1{\ttau_m} \left(\vec X^{m+1}_j -\vec X^m_j\right) .\, \vec\omega^m_j
= \kappa^{m+1}_j\,,\quad j = 1,\ldots,J\,.
\label{eq:FD1}
\end{equation}
Testing (\ref{eq:MCIP2db}) with the standard basis functions 
gives the finite difference type equations
\begin{equation} 
 \kappa^{m+1}_j\,\vec\omega^m_j =
\frac2{|\vec h^m_{j}|+|\vec h^m_{j+1}|} 
\left( \frac{ \vec X^{m+1}_{j+1} -
\vec X^{m+1}_j}{|\vec h^m_{j+1}| } -
\frac{\vec X^{m+1}_{j} - \vec X^{m+1}_{j-1}}{|\vec h^m_{j}|}\right)
\quad j = 1,\ldots,J\,.
\label{eq:FD2}
\end{equation}
\end{subequations}
Of course, it is now also possible to use (\ref{eq:FD1}) in order to 
reduce (\ref{eq:FD2}) to an equation for the nodal values of $\vec X^{m+1}$,
in order to obtain a finite difference analogue of \eqref{eq:elimkappa}.

\subsection{Existence and uniqueness}\label{subsec:3.4}
In order to show existence and uniqueness of solutions to \eqref{eq:MCIP},
we recall Assumption~\ref{ass:A}.

\begin{thm}\label{thm:uniqueMC}
Let $\Gamma^m$ satisfy {\rm Assumption~\ref{ass:A}\ref{item:assA1}}.
Then there exists a unique solution 
$(\vec X^{m+1},\kappa^{m+1}) \in \Vhm\times \Whm$
to the system \eqref{eq:MCIP}.
\end{thm}
\begin{proof}
Similarly to the proof of Lemma~\ref{lem:exXk}, existence follows from
uniqueness, and so we 
consider the homogeneous system
\begin{subequations}
\begin{align}
\left\langle \vec X, \chi\,\vec\nu^m\right\rangle^h_{\Gamma^m}
- \ttau_m\left\langle \kappa, \chi \right\rangle_{\Gamma^m}^h & = 0
\quad\forall\ \chi \in \Whm\,, \label{u:eq:MCIPa} \\
\left\langle \kappa\,\vec\nu^m, \vec\eta\right\rangle^h_{\Gamma^m}
+ \left\langle \nabs\,\vec X, \nabs\,\vec\eta \right\rangle_{\Gamma^m} & = 0
\quad\forall\ \vec\eta \in \Vhm\,, \label{u:eq:MCIPb}
\end{align}
\end{subequations}
for unknowns $(\vec X,\kappa) \in \Vhm\times \Whm$.
Choosing $\chi = \kappa\in \Whm$ in (\ref{u:eq:MCIPa}) and
$\vec\eta = \vec X\in\Vhm$ in (\ref{u:eq:MCIPb}) yields that
\begin{equation} \label{eq:unique0}
\left|\nabs\,\vec X\right|_{\Gamma^m}^2 
+\ttau_m\left(\left|\kappa\right|_{\Gamma^m}^h\right)^2 =0\,,
\end{equation}
where we have recalled \eqref{eq:ip0norm}. 
It follows from (\ref{eq:unique0}) that $\kappa = 0$ and
$\vec X = \vec X^c$ on $\Gamma^m$, for $\vec X^c\in\bR^d$. 
Together with \eqref{u:eq:MCIPa} this implies,
on recalling \eqref{eq:intnuh}, that
\begin{equation}
\left\langle\vec X^c, \chi\,\vec\nu^m\right\rangle_{\Gamma^m} = 0 
\quad\forall\ \chi \in \Whm\,. \label{eq:unique1}
\end{equation}
As in the proof of Lemma~\ref{lem:exXk} it follows from (\ref{eq:unique1}) 
and Assumption \ref{ass:A}\ref{item:assA1} that $\vec X^c = \vec0$.
Hence we have shown that there exists a unique solution
$(\vec X^{m+1},\kappa^{m+1}) \in \Vhm\times \Whm$ to (\ref{eq:MCIP}).
\end{proof}

\subsection{Stability}\label{subsec:3.5}
For $d=2$ and $d=3$ it can be shown that the scheme 
(\ref{eq:MCIP}) is unconditionally stable.

\begin{thm} \label{thm:stabMC}
Let $d=2$ or $d=3$. 
Let $(\vec X^{m+1},\kappa^{m+1}) \in \Vhm \times \Whm$ be a solution to
\eqref{eq:MCIP}. Then we have that
\begin{equation*} 
\left|\Gamma^{m+1}\right| 
+ \ttau_m\left(\left|\kappa^{m+1}\right|^h_{\Gamma^m}\right)^2 
\leq \left|\Gamma^m\right|\,.
\end{equation*}
\end{thm}
\begin{proof}
Choosing  $\chi = \kappa^{m+1}\in \Whm$ in (\ref{eq:MCIPa}) and
$\vec\eta = \frac1{\ttau_m}\,(\vec X^{m+1}-\vec\id_{\mid_{\Gamma^m}})\in\Vhm$
in (\ref{eq:MCIPb}) yields that
\begin{equation} \label{eq:stabMC0}
\left\langle\nabs\,\vec X^{m+1}, 
\nabs\left(\vec X^{m+1}-\vec\id\right)\right\rangle_{\Gamma^m}
+\ttau_m\left(\left|\kappa^{m+1}\right|^h_{\Gamma^m}\right)^2 = 0\,.
\end{equation}
In addition, we have from Lemma~\ref{lem:stab2d3d} that
\begin{equation} \label{eq:Xidstab}
\left\langle\nabs\,\vec X^{m+1}, \nabs\left(\vec X^{m+1}-\vec\id\right)
\right\rangle_{\Gamma^m}
\geq \left|\Gamma^{m+1}\right| - \left|\Gamma^m\right|\,.
\end{equation}
Combining (\ref{eq:stabMC0}) and \eqref{eq:Xidstab} then yields the claim. 
\end{proof}

\begin{rem} \label{rem:stabMC}
Clearly, if $\{(\vec X^{m+1},\kappa^{m+1})\}_{m=0}^{M-1}$ denote solutions to
\eqref{eq:MCIP} for 
\arxivyesno{$m=0,\ldots,$ $M-1$,}{$m=0,\ldots,M-1$,}
then the above theorem immediately yields the energy bound
\[
\left|\Gamma^k\right| + \sum_{m=0}^{k-1} 
\ttau_m\left(\left|\kappa^{m+1}\right|^h_{\Gamma^m}\right)^2 
\leq \left|\Gamma^0\right| .
\]
for $k=1,\ldots,M$.
\end{rem}
 
\subsection{Equipartition property} \label{subsec:equi}
Many numerical methods that use polyhedral surfaces to approximate
an evolving hypersurface suffer from the fact that the meshes
will have very inhomogeneous properties during the evolution.
For example, mesh points can come very close to each other during the 
evolution, or some angles in the mesh can become rather small or rather 
large. This leads to very unstable and inaccurate situations. 
In addition, the condition number of the linear systems to be solved at each time level will become very large. In some situations,
the computations cannot even be continued after some time.
It is an important property of the approach discussed in this section so far, 
that the mesh typically behaves very well. 
We will discuss this below for mean curvature flow, 
which is the simplest possible situation. 
However, similar results hold true for the extension of this scheme to 
more complicated flows discussed in the remaining sections of this
\arxivyesno{work}{chapter}.

The behavior of the mesh is best explained with the help of a semidiscrete
version of the scheme (\ref{eq:MCIP}), 
where we recall Definition~\ref{def:GhT}.
Given the closed polyhedral hypersurface $\Gamma^h(0)$,
find an evolving polyhedral hypersurface $\GhT$,
with induced velocity $\vec{\mathcal{V}}^h \in \VhGhT$,
and $\kappa^h \in \WhGhT$,
i.e.\ $\vec{\mathcal{V}}^h(\cdot, t) \in \Vht$ and $\kappa^h(\cdot,t)\in \Wht$
for all $t \in [0,T]$, such that, for all $t \in (0,T]$,
\begin{subequations} \label{eq:semidis}
\begin{align}
\left\langle \vec{\mathcal{V}}^h, \chi\,\vec\nu^h
\right\rangle^h_{\Gamma^h(t)} -
\left\langle \kappa^h , \chi\right\rangle_{\Gamma^h(t)} &= 0 \qquad
\forall\ \chi \in \Wht\,,\label{eq:semidisa}\\
\left\langle\kappa^h\,\vec\nu^h,\vec\eta\right\rangle^h_{\Gamma^h(t)} 
+ \left\langle \nabs\,\vec\id,
\nabs\,\vec\eta\right\rangle_{\Gamma^h(t)} &= 0\qquad
\forall\ \vec\eta\in \Vht\,.
\label{eq:semidisb}
\end{align}
\end{subequations}

We can prove that the scheme \eqref{eq:semidis} is stable,
i.e.\ a semidiscrete analogue of Theorem~\ref{thm:stabMC} holds, and that
its tangential motion ensures that any solution satisfies the property
in Definition~\ref{def:conformal}.

\begin{thm}\label{thm:semidis}
Let $(\GhT, \kappa^h)$ be a solution of \eqref{eq:semidis}. 
\begin{enumerate}
\item \label{item:semidisstab}
It holds that
\[
\ddt\left|\Gamma^h(t)\right| + 
\left(\left|\kappa^h\right|^h_{\Gamma^h(t)}\right)^2
= 0\,.
\]
\item \label{item:semidisTM}
For any $t \in (0,T]$, it holds that
$\Gamma^h(t)$ is a conformal polyhedral surface.
In particular, for $d=2$, any two neighbouring elements of the curve 
$\Gamma^h(t)$ either have equal length, or they are parallel.
\end{enumerate}
\end{thm}
\begin{proof}
\ref{item:semidisstab}
Choosing $\chi = \kappa^h(\cdot,t)\in \Wht$ in \eqref{eq:semidisa} and
$\vec\eta = \vec{\mathcal{V}}^h(\cdot, t) \in \Vht$ in \eqref{eq:semidisb} 
gives, on recalling Theorem~\ref{thm:disctrans} and
Lemma~\ref{lem:nabsid}\ref{item:nabsidnabsf}, that
\[
\ddt\left|\Gamma^h(t)\right| = \left\langle 1, \nabs\,.\,\vec{\mathcal{V}}^h
\right\rangle_{\Gamma^h(t)} =
 \left\langle \nabs\,\vec\id, \nabs\,\vec{\mathcal{V}}^h 
\right\rangle_{\Gamma^h(t)} 
= - \left\langle \kappa^h,\kappa^h \right\rangle_{\Gamma^h(t)}^h ,
\]
which is the claim. \\
\ref{item:semidisTM}
This follows directly from Definition~\ref{def:conformal} and
Theorem~\ref{thm:equid}. 
\end{proof}

\begin{rem} \label{rem:semidisTM}
In general it is not clear whether solutions to 
\eqref{eq:semidis} exist. For example, for $d\geq3$ the topology of the
triangulation fixed by $\Gamma^h(0)$ may be such that no solutions satisfying
both \eqref{eq:semidisa} and \eqref{eq:semidisb} exist. In those situations,
the fully discrete scheme \eqref{eq:MCIP}, even for very small time step sizes
$\ttau_m$, may exhibit mesh defects that would
not be expected for surfaces satisfying {\rm Definition~\ref{def:conformal}},
recall {\rm Remark~\ref{rem:conformal}}.
An example for such defects was recently presented in
{\rm \citet[Fig.~27]{ElliottF17}}.
\end{rem}

In the case of curves, i.e.\ $d=2$, the scheme \eqref{eq:semidis} can be
easily interpreted as a differential-algebraic system of equations.
To this end, we adopt the notation from \S\ref{subsubsec:curvesMC} for
$\vec X^h(t) \in \VhI$, and let
$\vec h_{j}(t) = \vec X^h_j(t) - \vec X^h_{j-1}(t)$, $j=1,\ldots,J+1$.
Then we obtain the obvious analogue of \eqref{eq:omega2d} 
for the vertex normals $\vec\omega^h_j(t)$ and,
similarly to (\ref{eq:FD}), the
system \eqref{eq:semidis} can then be written as
the following differential-algebraic system of equations (DAEs).
Given $\vec X^h(0) \in \VhI$, for $t \in (0,T]$ find
$\vec X^h(t) \in \VhI$ and $\kappa^h(t) \in \WhI$ such that
\begin{subequations} 
\begin{align}
& - \left( \ddt\,\vec X_j^h \right) . \left( 
\frac{\left( \vec X_{j+1}^h - \vec X_{j-1}^h \right)^\perp}
{\left|\vec h_{j}\right| + \left|\vec h_{j+1}\right| }\right) = \kappa_j^h \,,
\label{eq:FD12} \\
& -\kappa_j^h \,
\frac{\left( \vec X_{j+1}^h - \vec X_{j-1}^h \right)^\perp}
{\left|\vec h_{j}\right| + \left|\vec h_{j+1}\right|} = 
{\frac2 {\left|\vec h_{j}\right| +\left|\vec h_{j+1}\right|}} 
\left( \frac{ \vec X_{j+1}^h -
\vec X_j^h}{\left|\vec h_{j+1}\right| } -
\frac{\vec X_{j}^h -\vec X_{j-1}^h}{\left|\vec h_{j}\right| }\right) ,
\label{eq:FD22}
\end{align}
\end{subequations}
for $j = 1,\ldots,J$.

\begin{rem}
Equation \eqref{eq:FD12} only specifies one direction of 
$\ddt\,\vec X_j^h$, while the evolution of the remaining direction is
enforced via the algebraic conditions in \eqref{eq:FD22}.
Hence the overall system is
a highly nonlinear and degenerate differential-algebraic system of
equations.
The defining equations become singular where vertices coalesce. 
However, as is shown in {\rm Theorem~\ref{thm:semidis}\ref{item:semidisTM}}, 
in regions where the curve is not straight, this does not happen.
\end{rem}

So far we have discussed the linear fully discrete scheme \eqref{eq:MCIP},
which is obtained as a possible time discretization of \eqref{eq:semidis}.
The fully discrete scheme \eqref{eq:MCIP} will not inherit the conformality
property of Theorem~\ref{thm:semidis}\ref{item:semidisTM}, but solutions to
\eqref{eq:MCIP} still exhibit a tangential motion that leads to a good
distribution of vertices.
We now discuss a fully implicit time discretization of \eqref{eq:semidis}, 
where the solutions are conformal polyhedral surfaces at every time step.
These ideas were first presented in \cite{fdfi},
in the case of curves.

Let the closed polyhedral hypersurface $\Gamma^0$ be an approximation of
$\Gamma(0)$. Then, for $m=0,\ldots, M-1$, 
find a polyhedral hypersurface $\Gamma^{m+1}$,
and $(\vec X^{m},$ $\kappa^{m+1}) \in \Vhmp\times\Whmp$ 
with $\Gamma^m = \vec X^m(\Gamma^{m+1})$, such that
\begin{subequations} \label{eq:MCfdfi}
\begin{align}
\left\langle \frac{\vec\id-\vec X^m}{\ttau_m}, 
\chi\,\vec\nu^{m+1}\right\rangle^h_{\Gamma^{m+1}}
- \left\langle \kappa^{m+1}, \chi \right\rangle_{\Gamma^{m+1}}^h & = 0
\quad\forall\ \chi \in \Whmp\,, \label{eq:MCfdfia} \\
\left\langle \kappa^{m+1}\,\vec\nu^{m+1}, \vec\eta\right\rangle^h_{\Gamma^{m+1}}
+ \left\langle \nabs\,\vec\id, \nabs\,\vec\eta \right\rangle_{\Gamma^{m+1}} 
& = 0 \quad\forall\ \vec\eta \in \Vhmp\,. \label{eq:MCfdfib}
\end{align}
\end{subequations}

\begin{thm}\label{thm:fdfiTM}
Let $\Gamma^{m+1}$ and $(\vec X^m, \kappa^{m+1}) \in \Vhmp\times\Whmp$
be a solution to \eqref{eq:MCfdfi}. 
\begin{enumerate}
\item \label{item:fdfiTM}
Then it holds that
$\Gamma^{m+1}$ is a conformal polyhedral surface.
In particular, for $d=2$, any two neighbouring elements of the curve 
$\Gamma^{m+1}$ either have equal length, or they are parallel.
\item \label{item:fdfistab}
In the case $d=2$ it holds that
\begin{equation}
\left|\Gamma^{m+1}\right| + 
\ttau_m\left(\left|\kappa^{m+1}\right|^h_{\Gamma^{m+1}}\right)^2 
\leq \left|\Gamma^m\right|\,.
\label{eq:stabMCfdfi}
\end{equation}
\end{enumerate}
\end{thm}
\begin{proof}
\ref{item:fdfiTM}
This follows directly from Definition~\ref{def:conformal} and
Theorem~\ref{thm:equid}. 
\\
\ref{item:fdfistab}
Choosing $\chi = \kappa^{m+1}\in \Whmp$ in (\ref{eq:MCfdfia}) and
$\vec\eta = \frac1{\ttau_m}\,(\vec\id_{\mid_{\Gamma^{m+1}}} - \vec X^m)
\in\Vhmp$ in (\ref{eq:MCfdfib}) yields that
\begin{equation} \label{eq:stabMCfdfi0}
\left\langle\nabs\,\vec\id, 
\nabs\left(\vec\id-\vec X^m\right)\right\rangle_{\Gamma^{m+1}}
+\ttau_m\left(\left|\kappa^{m+1}\right|^h_{\Gamma^{m+1}}\right)^2 = 0\,.
\end{equation}
In addition, we have from Lemma~\ref{lem:fdfistab2d} that
\begin{equation} \label{eq:fdfiXidstab}
\left\langle\nabs\,\vec\id, \nabs\left(\vec\id-\vec X^m\right)
\right\rangle_{\Gamma^{m+1}} \geq \left|\Gamma^{m+1}\right| - 
\left|\Gamma^m\right|\,.
\end{equation}
Combining (\ref{eq:stabMCfdfi0}) and \eqref{eq:fdfiXidstab} 
then yields the claim.
\end{proof}

\begin{rem} \label{rem:fdfiTM}
Similarly to {\rm Remark~\ref{rem:semidisTM}}, in general it is not clear 
whether solutions to \eqref{eq:MCfdfi} exist. However, for $d=2$ solutions in
general appear to exist, and a small adaptation of \eqref{eq:MCfdfi} 
discussed below yields an unconditionally stable scheme that
equidistributes.
\end{rem}

For the case of curves, $d=2$, and building on \eqref{eq:MCfdfi}, 
in \cite{fdfi} the present authors introduced fully discrete
parametric finite element discretizations for $\mathcal{V} =
\varkappa$ that are unconditionally stable and that intrinsically
equidistribute the vertices at each time level. Using the notation
from \S\ref{subsubsec:curvesMC}, see also Remark~\ref{rem:VhI}, we can
reformulate \eqref{eq:MCfdfi} in the case of curves as follows.

Let $\vec X^0 \in \VhI$ be such that $\Gamma^0 = \vec X^0(\bI)$ is a polygonal
approximation of $\Gamma(0)$. Then, for $m=0,\ldots,M-1$,  
find $\Gamma^{m+1} = \vec X^{m+1}(\bI)$,
with $\vec X^{m+1}\in \VhI$, and $\kappa^{m+1} \in\WhI$ such that
for all $\chi\in \WhI$, $\vec\eta \in \VhI$
\begin{subequations} \label{eq:2425}
\begin{align}\label{eq:24a}
  \left\langle \frac{\vec X^{m+1}-\vec X^m}{\ttau_m} ,
  \chi\left(\vec X^{m+1}_\rho\right)^\perp \right\rangle^h_{\bI} 
+ \left\langle
  \kappa^{m+1},\chi\left|\vec X^{m+1}_\rho\right| \right\rangle^h_{\bI}
  &= 0 
\,,\\
  \left\langle \kappa^{m+1}\left(\vec X^{m+1}_\rho\right)^\perp,\vec\eta
  \right\rangle^h_{\bI} -
  \left\langle \vec X^{m+1}_\rho, \vec\eta_\rho\,
\left|\vec X^{m+1}_\rho\right|^{-1}\right\rangle_{\bI}
  &= 0 
\,.\label{eq:25}
\end{align}
\end{subequations}
We recall from Theorem~\ref{thm:fdfiTM} that solutions to \eqref{eq:2425} 
are equidistributed,
at least where elements of the curve $\Gamma^{m+1}$ are not locally parallel. 
This result can be sharpened to full equidistribution by
considering the following adapted version of \eqref{eq:2425}. 
For $m=0,\ldots,M-1$, 
find $\Gamma^{m+1} = \vec X^{m+1}(\bI)$, with $\vec X^{m+1} \in \VhI$,
and $\kappa^{m+1} \in \WhI$ such that
for all $\chi\in \WhI$, $\vec\eta \in \VhI$
\begin{subequations} \label{eq:21617}
\begin{align}\label{eq:216a}
  \left\langle \frac{\vec X^{m+1} -\vec X^m}{\ttau_m}, \chi
  \left(\vec X^{m+1}_\rho\right)^\perp\right\rangle^h_{\bI} 
+ \left|\Gamma^{m+1}\right| \left\langle
  \kappa^{m+1},\chi\right\rangle^h_{\bI}
  &= 0 
\,,\\
  \left\langle\kappa^{m+1}\left(\vec X^{m+1}_\rho\right)^\perp,
\vec\eta\right\rangle^h_{\bI} -
  |\Gamma^{m+1}|^{-1}\left\langle\vec X^{m+1}_\rho, \vec\eta_\rho\right\rangle_{\bI}
  &= 0
\,.\label{eq:217}
\end{align}
\end{subequations}

\begin{thm} \label{thm:stabfdfi}
Let $(\vec X^{m+1},\kappa^{m+1}) \in \VhI \times \WhI$ 
be a solution of \eqref{eq:21617}. 
Then it holds that
\begin{equation} \label{eq:fdfiTM}
\left|\vec X^{m+1}_\rho\right| = \left|\Gamma^{m+1}\right|
\quad\text{in } I_j\,,\qquad j=1,\ldots,J.
\end{equation}
Moreover, $(\vec X^{m+1},\kappa^{m+1})$ solves
\eqref{eq:2425}, and satisfies the stability estimate
\begin{equation} \label{eq:stabfdfi}
\left|\Gamma^{m+1}\right| + \ttau_m\left|\Gamma^{m+1}\right| 
\left\langle\kappa^{m+1},\kappa^{m+1}
\right\rangle^h_{\bI} \leq \left|\Gamma^m\right|.
\end{equation}
\end{thm}
\begin{proof}
If $(\vec X^{m+1},\kappa^{m+1}) \in \VhI \times \WhI$ 
is a solution of \eqref{eq:21617}, then on using the notation from 
Remark~\ref{rem:VhI}\ref{item:omegah}, and similarly to the proof
of Theorem~\ref{thm:equid}, we fix $j \in \{1,\ldots,J\}$ and
choose $\vec\eta \in \VhI$ in \eqref{eq:MCfdfib} with 
\[
\vec\eta(q_i) = 
\delta_{ij}\left(\vec X^{m+1}(q_{j+1}) - \vec X^{m+1}(q_{j-1})\right) 
= \delta_{ij}\left(\vec h_{j}^{m+1} + \vec h_{j+1}^{m+1}\right),
\]
for $i=1,\ldots,J$, in order to obtain
\begin{equation}
0 = \left(\vec h^{m+1}_{j+1} - \vec h^{m+1}_{j} \right) .
\left( \vec h^{m+1}_{j+1} + \vec h^{m+1}_{j} \right) = 
\left|\vec h^{m+1}_{j+1}\right|^2 - \left|\vec h^{m+1}_{j}\right|^2 \,.
\label{eq:eqTM0}
\end{equation}
Since \eqref{eq:eqTM0} holds for $j=1,\ldots,J$, we obtain 
\eqref{eq:fdfiTM}.
It follows from \eqref{eq:fdfiTM} that 
\arxivyesno{\linebreak}{}%
$(\vec X^{m+1},\kappa^{m+1}) \in \VhI \times \WhI$ also solves
\eqref{eq:2425}, and 
$(\vec X^m \circ (\vec X^{m+1})^{-1},
\kappa^{m+1} \circ (\vec X^{m+1})^{-1})
\in \Vhmp \times \Whmp$ is a solution to \eqref{eq:MCfdfi},
satisfying the stability bound \eqref{eq:stabMCfdfi}. Hence we obtain
\eqref{eq:stabfdfi}. 
\end{proof}

\begin{rem}
\rule{0pt}{0pt}
\begin{enumerate}
\item
Clearly, \eqref{eq:fdfiTM} means
that any solution to \eqref{eq:21617} is truly equidistributed.
\item
The system \eqref{eq:21617} is nonlinear, and so it can be solved either by
a Newton method or with the following iteration. 
Given $\Gamma^{m+1,0} = \vec X^{m+1,0}(\bI)$, with
$\vec X^{m+1,0}\in\VhI$, we
seek for $i\ge 0$ solutions $(\vec X^{m+1,i+1},
\kappa^{m+1,i+1})\in \VhI \times\WhI$ such that for all
$\chi\in \WhI$, $\vec\eta\in\VhI$
\begin{align*}
& \left\langle \frac{\vec X^{m+1,i+1}-\vec X^m}{\ttau_m}, 
\chi\left(\vec X^{m+1,i}_\rho\right)^\perp\right\rangle^h_{\bI} +
\left|\Gamma^{m+i,i}\right|
\left\langle\kappa^{m+1,i+1}, \chi\right\rangle^h_{\bI}
= 0 
\,,\\
& \left\langle\kappa^{m+1,i+1}\left(\vec X^{m+1,i}_\rho\right)^\perp,
\vec\eta\right\rangle^h_{\bI} 
- \left|\Gamma^{m+1,i}\right|^{-1}\left\langle\vec
X^{m+1,i+1}_\rho, \vec\eta_\rho\right\rangle
= 0 
\,,
\end{align*}
and set $\Gamma^{m+1,i+1} = \vec X^{m+1,i+1}(\bI)$.
This system has a unique solution and can be solved similarly to the
discussion in {\rm \S\ref{subsec:3.3}}. We refer to 
{\rm \cite{fdfi}} for details.
\end{enumerate}
\end{rem}

As an example, we show an evolution of curve shortening flow, i.e.\ mean
curvature flow in the case $d=2$, computed with the scheme \eqref{eq:MCIP} 
in Figure~\ref{fig:mcf}.
\begin{figure}
\center
\arxivyesno{
\newcommand\lheight{2.5cm} 
\newcommand\llheight{1.8cm} 
}{
\newcommand\lheight{2.0cm} 
\newcommand\llheight{1.5cm} 
}
\includegraphics[angle=-90,totalheight=\lheight]{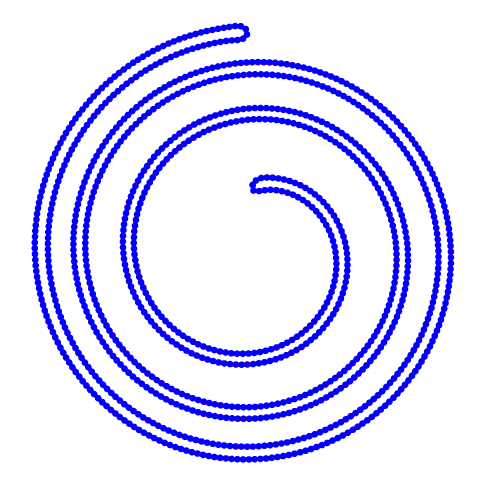}
\includegraphics[angle=-90,totalheight=\lheight]{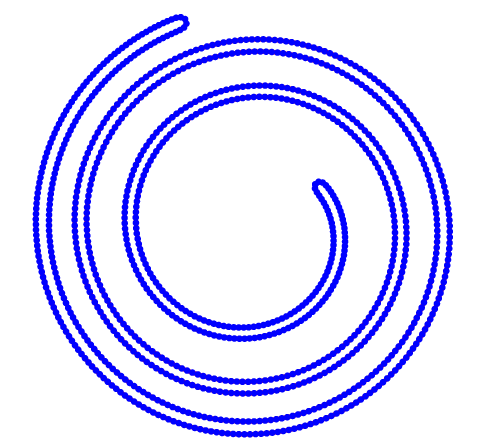}
\includegraphics[angle=-90,totalheight=\lheight]{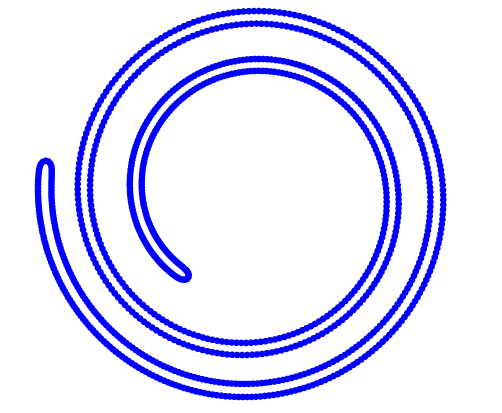}
\includegraphics[angle=-90,totalheight=\lheight]{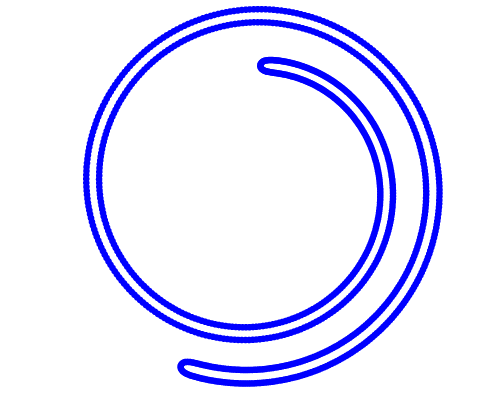}
\includegraphics[angle=-90,totalheight=\lheight]{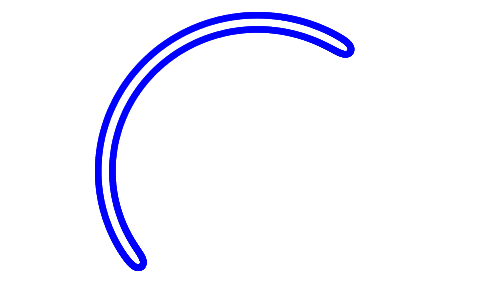}
\includegraphics[angle=-90,totalheight=\lheight]{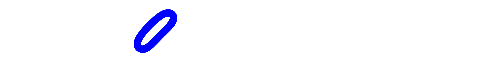} \\
\includegraphics[angle=-90,totalheight=\llheight]{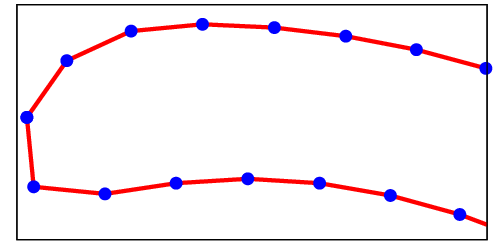}
\includegraphics[angle=-90,totalheight=\llheight]{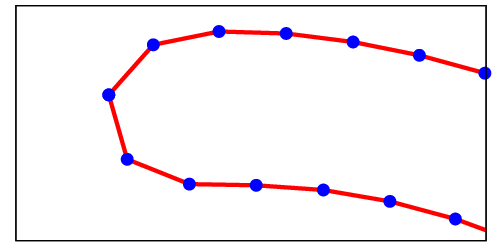}
\includegraphics[angle=-90,totalheight=\llheight]{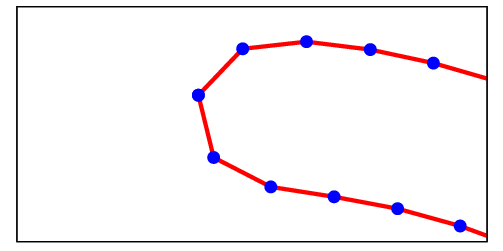}
\includegraphics[angle=-90,totalheight=\llheight]{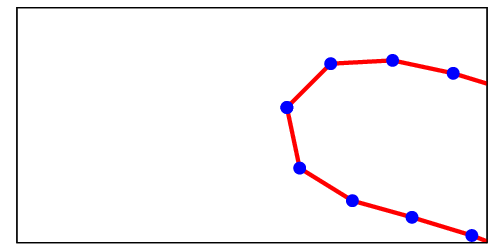}
\includegraphics[angle=-90,totalheight=\llheight]{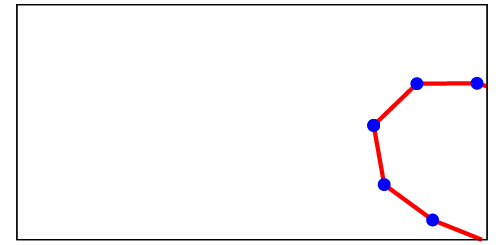}
\qquad\qquad
\includegraphics[angle=-90,totalheight=\llheight]{figures/spiral_Qtheta0c_t0d}
\includegraphics[angle=-90,totalheight=\llheight]{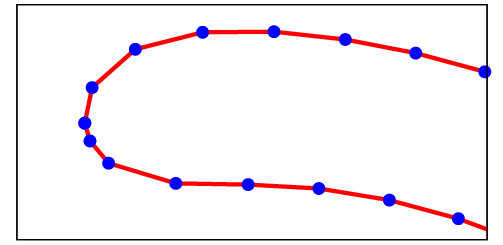}
\includegraphics[angle=-90,totalheight=\llheight]{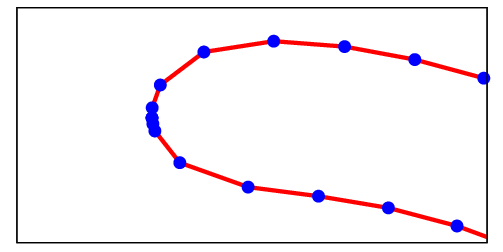}
\includegraphics[angle=-90,totalheight=\llheight]{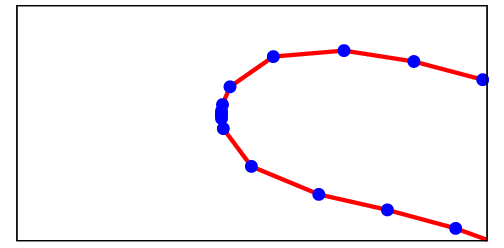}
\includegraphics[angle=-90,totalheight=\llheight]{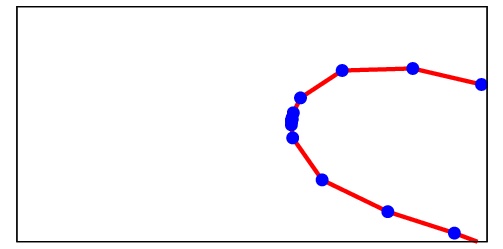}
\caption{Mean curvature flow for a spiral. Numerical solution
for the scheme \eqref{eq:MCIP} with $J = 1024$ and 
$\ttau_m=\ttau = 10^{-7}$, $m=0,\ldots,M-1$. The discrete solution is shown at
times $t = 0,\ 0.001,\ 0.005,\ 0.01,\ 0.02,\ 0.024$.
Below we show details of the vertex distributions for the scheme
\eqref{eq:MCIP} close to the inner tip (left),
and compare that with the classical Dziuk scheme \eqref{eq:DziukMC} (right). 
Coalescence for the latter scheme means that the simulation breaks down
and cannot be continued.
}
\label{fig:mcf}
\end{figure}%

\subsection{Alternative parametric methods} \label{subsec:alternativeMC}
The discussion so far was focussed on contributions of the present
authors. We will now discuss other numerical methods for the numerical 
approximation of mean curvature flow, focussing on parametric finite element
methods.

\subsubsection{The classical Dziuk approach} \label{subsubsec:Dziuk}
The approach introduced by \cite{Dziuk91} is based on a weak formulation of
\begin{equation} \label{eq:Dziukweakmcf}
\vec{\mathcal{V}} =\vec\varkappa\,,\quad
\vec\varkappa = \Delta_s\,\vec\id \quad\text{on } \Gamma(t)
\qquad\Rightarrow\qquad
\vec{\mathcal{V}} = \Delta_s\,\vec\id 
\quad\text{on }\Gamma(t)\,,
\end{equation}
rather than \eqref{eq:BGNweakmcf}. 
Since $\Delta_s\,\vec\id = \vec\varkappa = \varkappa\,\vec\nu$ 
is collinear to  the normal, 
recall Lemma~\ref{lem:varkappa}\ref{item:vecvarkappa},
the evolving surface always moves in a direction collinear to the normal. 
The system to be solved, using the notation of \S\ref{subsec:mcFEA},
is then, given the closed polyhedral hypersurface $\Gamma^m$,
to find $\vec X^{m+1} \in \Vhm$ such that
\begin{equation} \label{eq:DziukMC}
  \left\langle\frac{\vec X^{m+1}-\vec\id}{\ttau_m},
    \vec\eta\right\rangle_{\Gamma^m} + \left\langle\nabs\,\vec X^{m+1},
  \nabs\,\vec\eta\right\rangle_{\Gamma^m} = 0 
\quad\forall\ \vec\eta\in\Vhm\,.
\end{equation}
Existence, uniqueness and stability of this fully discrete linear
system can be shown similarly to the results in \S\ref{subsec:3.4} and
\S\ref{subsec:3.5}. For the case $d=2$ an error estimate 
for a semidiscrete continuous-in-time variant of \eqref{eq:DziukMC} 
with mass lumping is shown in \cite{Dziuk94}, 
on assuming that the approximated solution is sufficiently smooth.
Even though for the case $d=3$ no convergence results have been shown for
either \eqref{eq:DziukMC} or \eqref{eq:MCIP}, in practice both approximations
appear to be convergent, see e.g.\ \citet[Table~1]{gflows3d}.

\subsubsection{The convergent finite element algorithm of 
Kov{\'a}cs, Li, Lubich}
Very recently, a first proof of convergence for
a numerical method for mean curvature flow for $d=3$ was given by
\cite{KovacsLL19}. They prove error estimates for semi- as well as 
full-discretizations of mean curvature flow, again 
assuming that the approximated solution is sufficiently smooth.
They use the system
\[
\vec{\mathcal{V}} = \varkappa\,\vec\nu\,,\quad
\matpartx\,\vec\nu = \Delta_s\,\vec\nu + \left|\nabs\,\vec\nu\right|^2
\vec\nu\,,\quad
\matpartx\,\varkappa = \Delta_s\,\varkappa + \left|\nabs\,\vec\nu\right|^2
\varkappa
\qquad\text{on } \Gamma(t)\,.
\]
The last two equations can be derived using basic geometric analysis and
the law $\vec{\mathcal{V}}=\varkappa\,\vec\nu$; see
Lemma~\ref{lem:dtnu}\ref{item:dtnu}, Lemma~\ref{lem:Deltasnu}, 
Lemma~\ref{lem:productrules}\ref{item:priii},
Lemma \ref{lem:nabsnu}\ref{item:nabsnunu}
and Lemma~\ref{lem:derkappa}\ref{item:lem10.3i}.
Compared to \eqref{eq:MCIP} 
the approach has the disadvantages that meshes can still deteriorate,
as the approximated flow is normal, similarly to the classical Dziuk approach,
and that an additional system has to be solved. Of course, a clear advantage 
is that it is presently the only method, for which convergence can be shown 
for $d=3$.

\subsubsection{Alternative numerical methods that equidistribute} 
\label{subsubsec:equid}
For the mean curvature flow of curves in the plane, the scheme \eqref{eq:21617}
equidistributes the vertices of the polygonal approximation of the curve at
every time step. In this section we discuss some alternative numerical methods
that achieve the same goal.

As discussed
before, it is desirable to control the tangential distribution of mesh
points, and so to prevent numerical singularities such as coalescence
or swallow tails, i.e.\ numerically induced self-intersections. In
practice, many authors use mesh smoothing methods, see e.g.\
\cite{Sethian85,DziukKS02,BanschMN05}. However, mesh smoothing has
some undesirable features, i.e.\ they might smoothen the solution
excessively and stability results that hold for the original
approximation are in general lost. For geometric evolution laws that
have some conservation properties, for example flows where the enclosed
area is conserved, great care must be taken to maintain the conservation 
properties after mesh smoothing.

In the case of an evolving curve $(\Gamma(t))_{t\in[0,T]}$, we consider
parameterizations $\vec x : \bI\times[0,T]\to\bR^2$, where $\bI= \bR/\bZ$ 
is the periodic interval $[0,1]$, such that $\Gamma(t) = \vec x(\bI,t)$,
$t \in [0,T]$ are solutions of the curvature flow equation 
$\mathcal{V} = \varkappa$. 
With $s$ we denote the arclength of $\Gamma(t)$, so that the
unit tangent is given by 
$\vec x_s(\rho,t) = \frac{\vec x_\rho(\rho,t)}{|\vec x_\rho(\rho,t)|}$, 
$\rho\in \bI$, and a possible normal on $\Gamma(t)$ 
is $\vec\nu \circ \vec x = - \vec x_s^\perp$.
The basic idea is to only consider parameterizations that satisfy
\begin{equation}\label{eq:equi}
\left|\vec x_\rho(\rho,t)\right|=\left|\Gamma(t)\right|
\quad\forall\rho\in \bI\,,\ t\in [0,T] \,,
\end{equation}
which means that $\rho$, up to a constant, is arclength. 
First such approaches are due to \cite{KesslerKL84}, and
were analyzed and modified in \cite{Strain89,HouLS94,MikulaS01}. 
It is also
possible to discretize an evolution law for the position vector $\vec x$ 
in the form
\begin{equation}\label{eq:tangterm}
  \partial_t\,\vec x = (\mathcal{V}\,\vec\nu) \circ \vec x + \alpha\,\vec x_s
\qquad\text{in } \bI\,,
\end{equation}
where the tangential velocity is chosen such that the vertices of the
polygonal approximation of $\Gamma(t)$ remain close to being equidistributed. 
This idea has been used in \cite{Kimura94}, and \cite{MikulaS01} show
that choosing $\alpha$ such that
\begin{equation*}
\alpha_s = (\mathcal{V}\,\varkappa) \circ \vec x
-\frac1{\left|\Gamma(t)\right|}
\left\langle \mathcal{V},\varkappa\right\rangle_{\Gamma(t)} ,
\end{equation*}
leads to the property \eqref{eq:equi}, if \eqref{eq:equi} is fulfilled
at the initial time. 
Also \cite{DeckelnickD95} use \eqref{eq:tangterm} with 
\begin{equation} \label{eq:TMDD95}
\alpha = \frac{\vec x_{\rho\rho}\,.\,\vec x_\rho}{|\vec x_\rho|^3} 
\qquad\text{in } \bI\,,
\end{equation}
and observe well distributed polygonal meshes in practice. 
In addition, an error analysis can be performed for a semidiscrete 
continuous-in-time finite element approximation.
The resulting fully discrete scheme is particularly simple, and 
can be formulated as follows. Given $\vec X^m \in \VhI$, 
find $\vec X^{m+1} \in \VhI$ such that
\begin{equation}
\left\langle \frac{\vec X^{m+1}-\vec X^m}{\ttau_m}, \vec\eta
\left|\vec X^m_\rho\right|^2 \right\rangle_{\bI}
+ \left\langle \vec X^{m+1}_\rho , \vec\eta_\rho 
\right\rangle_{\bI} = 0 
\quad\forall\ \vec\eta\in \VhI\,,
\label{eq:DD95}
\end{equation}
where we observe that the numerical computations presented in
\cite{DeckelnickD95} use mass lumping for the first term in \eqref{eq:DD95}.
The authors in \cite{PanW08,Pan08}, on the other hand, 
solve $\mathcal{V}=\varkappa$ with the side constraint
\begin{equation*}
  \vec x_\rho(\rho,t)\,.\,\vec x_{\rho\rho} (\rho,t) = 0 \quad
  \forall\ \rho\in \bI\,,\ t\in [0,T]\,,
\end{equation*}
which leads to $(\frac12|\vec x_\rho|^2)_\rho = \vec x_\rho \,.\,\vec
x_{\rho\rho} = 0$ and hence to equidistributed parameterizations. However
this approach leads to very nonlinear fully discrete problems.

\subsubsection{Using the DeTurck trick to obtain good mesh properties}
Mean curvature flow $\mathcal{V}=\varkappa$, as a parabolic equation,
has certain degeneracies, which basically stem from the fact that a
reparameterization of a solution leads to another solution. It was an
idea by \cite{DeTurck83}
to use solutions of the harmonic map heat flow, in
order to reparameterize the evolution of curvature flows in such a way
that the reparameterized equations are strongly parabolic. 
\cite{ElliottF17} used the DeTurck trick to solve
the mean curvature flow numerically. 
The resulting algorithm has the remarkable property, that provided the
initial surface is approximated by a triangulation with a good mesh quality,
the meshes will remain good during the evolution. 
This, roughly speaking, is due to the fact that,
under appropriate assumptions, harmonic maps are conformal maps, which
hence preserve angles.

The algorithm of \cite{ElliottF17}, for the case $d=3$, may at times lead 
to better meshes than the method discussed in \S\ref{subsec:mcFEA}. 
However, the systems
to be solved are more complicated, and the approach also seems to
be less flexible, e.g.\ when it comes to anisotropy, triple junctions
and the coupling of the interface evolution to other fields. 
We refer to \cite{ElliottF17} for more details about this approach and for
an error analysis in the case $d=2$ for a semidiscrete continuous-in-time 
variant. This error analysis is an extension of that in \cite{DeckelnickD95}
for the variant with the specific tangential velocity \eqref{eq:TMDD95}.
The error analysis in \cite{ElliottF17} was extended to the fully discrete
case in \cite{BarrettDS17}. 

\subsubsection{Other numerical approaches}
In some situations the surface to be computed can be written as a
graph. Finite element approximations in this case have been analyzed in 
\cite{Dziuk99b} and \cite{DeckelnickD99}.
For the case of axisymmetric surfaces moving by mean curvature flow,
numerical schemes have been presented in 
\cite{MayerS02,aximcf}.
We also mention \cite{MikulaRSS14}, 
who proposed methods for a tangential redistribution of points on evolving 
surfaces. They use volume-oriented and length-oriented tangential 
redistribution methods, which can be applied to
manifolds in any dimension. Completely different
approaches are the level set method and the phase field method, which are
\arxivyesno{discussed, for example, 
in {\em Handbook of Numerical Analysis, Vol.~XXI \& XXII}}
{discussed in other chapters of this handbook.}

\section{Surface diffusion and other flows } \label{sec:surfdiff}

In this section we extend the ideas from Section~\ref{sec:mc} to more
general evolution laws for evolving hypersurfaces in
$\bR^d$, $d\geq2$.

\subsection{Properties of the surface diffusion flow}
Fourth order geometric evolution equations for hypersurfaces have played
an important role in the last twenty years. A basic
example, which can be formulated in a very simple way, is motion by surface
diffusion
\begin{equation*} 
  \mathcal{V} = -\Delta_s\,\varkappa
\qquad\text{on }\Gamma(t)\,.
\end{equation*}
This flow was introduced by \cite{Mullins57}
to describe diffusion at interfaces in alloys. The flow
has the remarkable property that for closed surfaces it is surface
area decreasing and the enclosed volume is preserved. These two
properties follow from the transport theorems as follows. If we define
$\Omega(t)$ to be the domain enclosed by the closed hypersurface $\Gamma(t)$,
with $\vec\nu(t)$ denoting the outer unit normal to $\Omega(t)$ on $\Gamma(t)$,
then we obtain with the help of Theorem~\ref{thm:transvol} that
\begin{equation} \label{eq:ddtvol}
\ddt\,\mathcal{L}^d(\Omega(t)) = \ddt\,\int_{\Omega(t)} 1\dL{d} 
= \left\langle 1,\mathcal{V} \right\rangle_{\Gamma(t)}
= - \left\langle 1,\Delta_s\,\varkappa \right\rangle_{\Gamma(t)} = 0 \,,
\end{equation}
where the last identity follows from the divergence theorem on
manifolds, recall Remark~\ref{rem:ibp}\ref{item:ibp}. Furthermore, for the
total surface area, $|\Gamma(t)|$, we obtain from Theorem~\ref{thm:trans} that
\begin{equation} \label{eq:ddtareaSD}
\ddt\,|\Gamma(t)| 
= -\left\langle \varkappa,\mathcal{V} \right\rangle_{\Gamma(t)}
= \left\langle \varkappa,\Delta_s\,\varkappa \right\rangle_{\Gamma(t)}
= - \left| \nabs\,\varkappa \right|_{\Gamma(t)}^2 \leq 0\,, 
\end{equation}
where for the last identity we recall once more 
Remark~\ref{rem:ibp}\ref{item:ibp}.
In fact, the properties \eqref{eq:ddtvol} and \eqref{eq:ddtareaSD} are closely
related to the fact that mathematically surface diffusion can be interpreted 
as the $H^{-1}$--gradient flow of $|\Gamma(t)|$, see e.g.\
\cite{TaylorC94}.

Using the formulation
\begin{equation} \label{eq:strongSD}
\vec{\mathcal{V}}\,.\,\vec\nu = -\Delta_s\,\varkappa\,, \quad
\varkappa\,\vec\nu = \Delta_s\,\vec\id \qquad\text{on } \Gamma(t)\,,
\end{equation}
it is now possible to formulate a finite element approximation.

\subsection{Finite element approximation for surface diffusion} 
\label{subsec:sdFEA}
Using the same notation as in \S\ref{subsec:mcFEA},
we recall the following finite element approximation 
of \eqref{eq:strongSD} for surface diffusion from \cite{triplej,gflows3d}.
Let the closed polyhedral hypersurface $\Gamma^0$ be an approximation of
$\Gamma(0)$. Then, for $m=0,\ldots,M-1$, find 
$(\vec X^{m+1},\kappa^{m+1}) \in \Vhm\times\Whm$ such that
\begin{subequations} \label{eq:sdfea}
\begin{align}
\left\langle \frac{\vec X^{m+1}-\vec\id}{\ttau_m}, 
\chi\,\vec\nu^m\right\rangle^h_{\Gamma^m}
- \left\langle \nabs\, \kappa^{m+1}, \nabs\, \chi \right\rangle_{\Gamma^m} & = 0
\quad\forall\ \chi \in \Whm\,, \label{eq:26a} \\
\left\langle \kappa^{m+1}\,\vec\nu^m, \vec\eta\right\rangle^h_{\Gamma^m}
+ \left\langle \nabs\, \vec X^{m+1}, \nabs\, \vec\eta 
\right\rangle_{\Gamma^m} & = 0 
\quad\forall\ \vec\eta \in \Vhm \label{eq:26b}
\end{align}
\end{subequations}
and set $\Gamma^{m+1} = \vec X^{m+1}(\Gamma^m)$.
Existence and uniqueness of discrete solutions for this scheme can be
shown as in \S\ref{subsec:mcFEA}.

\begin{thm} \label{thm:uniqueSD}
Let $\Gamma^m$ satisfy {\rm Assumption~\ref{ass:A}\ref{item:assA1}}.
Then there exists a unique solution $(\vec X^{m+1}, \kappa^{m+1})\in
\Vhm\times \Whm$ to the system \eqref{eq:sdfea}.
\end{thm}
\begin{proof} 
The proof is very similar to the proof of Theorem~\ref{thm:uniqueMC}.
For a solution $(\vec X,\kappa)\in\Vhm\times \Whm$ of the homogeneous system
\begin{subequations} 
\begin{align}
\left\langle \vec X, \chi\,\vec\nu^m\right\rangle^h_{\Gamma^m}
- \ttau_m\left\langle \nabs\,\kappa, \nabs\,\chi \right\rangle_{\Gamma^m} & = 0
\quad\forall\ \chi \in \Whm\,, \label{eq:28a} \\
\left\langle \kappa\,\vec\nu^m, \vec\eta\right\rangle^h_{\Gamma^m}
+ \left\langle \nabs\,\vec X, \nabs\,\vec\eta \right\rangle_{\Gamma^m} & = 0
\quad\forall\ \vec\eta \in \Vhm \label{eq:28b}
\end{align}
\end{subequations}
we first prove $\kappa = 0$, and then proceed as in the proof of
Lemma~\ref{lem:exXk}.
Choosing $\chi = \kappa\in \Whm$ in (\ref{eq:28a}) and
$\vec\eta = \vec X\in\Vhm$ in (\ref{eq:28b}) yields that
\begin{equation*} 
\left|\nabs\,\vec X\right|_{\Gamma^m}^2 
+\ttau_m\left|\nabs\,\kappa\right|_{\Gamma^m}^2 =0\,.
\end{equation*}
We hence obtain that $\kappa = \kappa^c$ and 
$\vec X = \vec X^c$ on $\Gamma^m$, for $\kappa^c\in\bR$ and 
$\vec X^c\in \bR^d$, and so it follows from \eqref{eq:28b} 
and \eqref{eq:omegahnuh} that
\begin{equation}
0 = \left\langle\kappa^c\,\vec\nu^m, \vec\eta\right\rangle_{\Gamma^m}^h 
 = \kappa^c\left\langle\vec\omega^m, \vec\eta\right\rangle_{\Gamma^m}^h 
\quad\forall\ \vec\eta \in \Vhm\,. \label{eq:unique11}
\end{equation}
If we assume that $\kappa^c\not=0$, then 
choosing $\vec\eta = \vec\omega^m\in\Vhm$ in (\ref{eq:unique11}) 
yields, on recalling (\ref{eq:ip0norm}), that 
$|\vec\omega^m|_{\Gamma_m}^h = 0$, and so $\vec\omega^m = \vec 0$.
However, that clearly contradicts Assumption~\ref{ass:A}\ref{item:assA1}
and hence we have that $\kappa^c=0$. Moreover, as in the proof of
Lemma~\ref{lem:exXk}, we obtain from \eqref{eq:28a} and
Assumption~\ref{ass:A}\ref{item:assA1} that $\vec X^c=0$. 
Hence we have shown that there exists a unique solution
$(\vec X^{m+1},\kappa^{m+1})$ $ \in \Vhm\times \Whm$ to \eqref{eq:sdfea}.
\end{proof} 

Similarly to \S\ref{subsec:3.5}, it can be shown that the scheme 
(\ref{eq:sdfea}) is unconditionally stable for $d=2$ and $d=3$.

\begin{thm} \label{thm:stabSD}
Let $d=2$ or $d=3$. 
Let $(\vec X^{m+1},\kappa^{m+1}) \in \Vhm \times \Whm$ be a solution to
\eqref{eq:sdfea}. Then we have that
\begin{equation*} 
\left|\Gamma^{m+1}\right| 
+ \ttau_m\left|\nabs\,\kappa^{m+1}\right|_{\Gamma^m}^2 
\leq \left|\Gamma^m\right|.
\end{equation*}
\end{thm}
\begin{proof}
Choosing $\chi = \kappa^{m+1}\in \Whm$ in (\ref{eq:26a}) and
$\vec\eta = \frac1{\ttau_m}\,(\vec X^{m+1}-\vec\id_{\mid_{\Gamma^m}})\in\Vhm$
in (\ref{eq:26b}) yields that
\begin{equation} \label{eq:stabSD0}
\left\langle\nabs\,\vec X^{m+1}, 
\nabs\left(\vec X^{m+1}-\vec\id\right)\right\rangle_{\Gamma^m}
+\ttau_m \left|\nabs\,\kappa^{m+1}\right|_{\Gamma^m}^2 = 0\,.
\end{equation}
Combining (\ref{eq:stabSD0}) and Lemma~\ref{lem:stab2d3d} yields the claim. 
\end{proof}

\subsection{Volume conservation for the semidiscrete scheme} \label{subsec:vol}
An important aspect of discretizations for geometric evolution equations
is to mimic decay and conservation properties of the evolution laws
on the discrete level. In this section we discuss the volume conservation
properties of the scheme \eqref{eq:sdfea} for surface diffusion,
recall \eqref{eq:ddtvol}. 

On recalling Definition~\ref{def:GhT}, we consider the following 
semidiscrete variant of \eqref{eq:sdfea}. 
Given the closed polyhedral hypersurface $\Gamma^h(0)$,
find an evolving polyhedral hypersurface $\GhT$ 
with induced velocity $\vec{\mathcal{V}}^h \in \VhGhT$,
and $\kappa^h \in \WhGhT$,
i.e.\ $\vec{\mathcal{V}}^h(\cdot, t) \in \Vht$ and $\kappa^h(\cdot,t)\in \Wht$
for all $t \in [0,T]$, such that, for all $t \in (0,T]$,
\begin{subequations} \label{eq:SDSD}
\begin{align}
\left\langle \vec{\mathcal{V}}^h, \chi\,\vec\nu^h\right\rangle^h_{\Gamma^h(t)}
- \left\langle \nabs\,\kappa^h ,\nabs\,\chi\right\rangle_{\Gamma^h(t)} 
&= 0 \quad\forall\ \chi \in \Wht\,,\label{eq:SDSD1}\\
\left\langle\kappa^h\,\vec\nu^h,\vec\eta\right\rangle^h_{\Gamma^h(t)} 
+ \left\langle \nabs\,\vec\id,\nabs\,\vec\eta\right\rangle_{\Gamma^h(t)} 
&= 0\quad\forall\ \vec\eta\in \Vht\,. \label{eq:SDSD2}
\end{align}
\end{subequations}

\begin{thm} \label{thm:SDSD}
Let $(\GhT, \kappa^h)$ be a solution of \eqref{eq:SDSD}, 
and let $\Omega^h(t)$ be the domain enclosed by $\Gamma^h(t)$, for 
$t \in [0,T]$. 
\begin{enumerate}
\item \label{item:SDSDstab}
It holds that
\[
\ddt\left|\Gamma^h(t)\right| +
\left|\nabs\,\kappa^h \right|_{\Gamma^h(t)}^2 = 0\,.
\]
\item \label{item:SDSDvol}
It holds that
\[
\ddt\,\mathcal{L}^d(\Omega^h(t)) = 0\,.
\] 
\item \label{item:SDSDTM}
For any $t \in (0,T]$, it holds that
$\Gamma^h(t)$ is a conformal polyhedral surface.
In particular, for $d=2$, any two neighbouring elements of the curve 
$\Gamma^h(t)$ either have equal length, or they are parallel.
\end{enumerate}
\end{thm}
\begin{proof} 
\ref{item:SDSDstab} 
Similarly to the proof of Theorem~\ref{thm:semidis}\ref{item:semidisstab},
choosing $\chi = \kappa^h(\cdot,t)\in \Wht$ in \eqref{eq:SDSD1} and
$\vec\eta = \vec{\mathcal{V}}^h(\cdot, t) \in \Vht$ in \eqref{eq:SDSD2} gives
\[
\ddt\left|\Gamma^h(t)\right| 
= \left\langle \nabs\,\vec\id, \nabs\,\vec{\mathcal{V}}^h 
 \right\rangle_{\Gamma^h(t)} 
= - \left|\nabs\,\kappa^h \right|_{\Gamma^h(t)}^2 ,
\]
which is the claim. \\
\ref{item:SDSDvol} 
Choosing $\chi = 1$ in \eqref{eq:SDSD1} yields,
on recalling Theorem~\ref{thm:disctransvol},
$\vec{\mathcal{V}}^h(\cdot, t) \in \Vht$ and \eqref{eq:intnuh}, that
\begin{equation*}
\ddt\,\mathcal{L}^d(\Omega^h(t)) 
= \pm 
\left\langle \vec{\mathcal{V}}^h,\vec\nu^h\right\rangle_{\Gamma^h(t)} 
= \pm 
\left\langle \vec{\mathcal{V}}^h,\vec\nu^h\right\rangle_{\Gamma^h(t)}^h 
= 0\,,
\end{equation*}
where the sign $\pm$ depends on whether $\vec\nu^h$ here is the outer/inner 
normal of $\Omega^h(t)$ on $\Gamma^h(t)$.\\
\ref{item:SDSDTM}
This follows directly from Definition~\ref{def:conformal} and
Theorem~\ref{thm:equid}. 
\end{proof}

\begin{rem} \label{rem:SDSD}
\rule{0pt}{0pt}
\begin{enumerate}
\item \label{item:remSDSDi}
The result in {\rm Theorem~\ref{thm:SDSD}\ref{item:SDSDstab}} is the
semidiscrete analogue of the stability result 
{\rm Theorem~\ref{thm:stabSD}} for the fully discrete scheme \eqref{eq:sdfea},
and it mimics the energy law \eqref{eq:ddtareaSD} on the discrete level.
\item \label{item:remSDSDii}
The result in {\rm Theorem~\ref{thm:SDSD}\ref{item:SDSDvol}} mimics the 
volume conservation property \eqref{eq:ddtvol} on the discrete level.
Moreover, while the fully discrete scheme \eqref{eq:sdfea} does not conserve 
the enclosed volume exactly, it does satisfy
\begin{equation*} 
\left\langle \vec X^{m+1} - \vec\id, \vec\nu^m \right\rangle_{\Gamma^m}
= \left\langle \vec X^{m+1} - \vec\id, \vec\nu^m \right\rangle_{\Gamma^m}^h
= 0\,,
\end{equation*}
recall \eqref{eq:intnuh}. 
On noting {\rm Theorem~\ref{thm:disctransvol}}
this can be interpreted as a fully discrete analogue of
{\rm Theorem~\ref{thm:SDSD}\ref{item:SDSDvol}}. In fact, in practice the scheme
\eqref{eq:sdfea} exhibits excellent volume conservation properties, with the
relative enclosed volume loss tending to zero as the time step size goes to
zero.
\end{enumerate}
\end{rem}

\subsection{Generalizations to other flows} \label{subsec:otherflows}
The finite element approximation \eqref{eq:sdfea} for surface diffusion can
easily be adapted to more general situations. The resultant discretizations
will, under certain circumstances, again satisfy a stability result.

Here we consider flows of the form
\begin{equation}\label{eq:Fgeomev}
  \mathcal{V} = \mathfrak{F} (\varkappa)\qquad\text{on } \qquad\Gamma(t)\,,
\end{equation}
where $\mathfrak{F}$ maps functions on the closed hypersurface $\Gamma(t)$ 
to functions on $\Gamma(t)$. For mean curvature flow we have 
$\mathfrak{F}(\chi) = \chi$, while surface diffusion
is obtained with $\mathfrak{F}(\chi) = -\Delta_s\,\chi$. 
If the map $\mathfrak{F}$ is such that
\begin{equation}\label{eq:Fineq}
 \left\langle \mathfrak{F}(\varkappa),\varkappa \right\rangle_{\Gamma(t)}
\geq 0\,,
\end{equation}
then we have, on recalling Theorem~\ref{thm:trans}, that
\begin{equation}\label{eq:Fineq2}
\ddt\left|\Gamma(t)\right| 
= - \left\langle \mathcal{V},\varkappa\right\rangle_{\Gamma(t)}
= - \left\langle \mathfrak{F}(\varkappa),\varkappa\right\rangle_{\Gamma(t)}
\leq 0\,.
\end{equation}
The inequality \eqref{eq:Fineq} clearly is true for mean curvature flow, 
where we obtain
\begin{equation*}
\left\langle \mathfrak{F}(\varkappa),\varkappa\right\rangle_{\Gamma(t)}
= \left| \varkappa \right|_{\Gamma(t)}^2 \geq 0\,,
\end{equation*}
and also for surface diffusion, where
\begin{equation*}
\left\langle \mathfrak{F}(\varkappa),\varkappa\right\rangle_{\Gamma(t)}
= \left| \nabs\,\varkappa \right|_{\Gamma(t)}^2 \geq 0\,,
\end{equation*}
recall \eqref{eq:ddtareaSD}. 
Our goal now is to obtain a stability result like that in 
Theorem~\ref{thm:stabSD} for discretizations of flows of the form 
\eqref{eq:Fgeomev} that satisfy the energy bound \eqref{eq:Fineq2}. 
As a consistency condition for a discrete formulation, we require that there 
is a map $\mathfrak{F}^m: \Whm \to \Whm$ which
approximates $\mathfrak{F}$ and is such that
\begin{equation} \label{eq:Fconsist}
\left\langle\mathfrak{F}^m(\chi),\chi\right\rangle^h_{\Gamma^m} \geq 0
\quad\forall\ \chi \in \Whm\,.
\end{equation}
We consider now the following finite element approximation of 
(\ref{eq:Fgeomev}). 
Let the closed polyhedral hypersurface $\Gamma^0$ be an approximation of
$\Gamma(0)$. Then, for $m=0,\ldots,M-1$, find 
$(\vec X^{m+1},$ $\kappa^{m+1}) \in \Vhm\times\Whm$ such that
\begin{subequations} \label{eq:F12}
\begin{align}
\left\langle \frac{\vec X^{m+1}-\vec\id}{\ttau_m}, \chi\,\vec\nu^m
\right\rangle^h_{\Gamma^m}
-\left\langle \mathfrak{F}^m(\kappa^{m+1}), \chi
\right\rangle^h_{\Gamma^m} & = 0 \quad\forall\ \chi \in \Whm\,, 
\label{eq:F1} \\
\left\langle \kappa^{m+1}\,\vec\nu^m,\vec\eta\right\rangle^h_{\Gamma^m}
+ \left\langle \nabs\,\vec X^{m+1}, \nabs\,\vec\eta \right\rangle_{\Gamma^m} 
& = 0 \quad\forall\ \vec\eta \in \Vhm\, \label{eq:F2}
\end{align}
\end{subequations}
and set $\Gamma^{m+1} = \vec X^{m+1}(\Gamma^m)$.
We obtain the following stability theorem, together with an existence and
uniqueness result in the linear case.

\begin{thm} \label{thm:other}
\rule{0pt}{0pt}
\begin{enumerate}
\item \label{item:otherex}
Let $\Gamma^m$ satisfy {\rm Assumption~\ref{ass:A}}, and let
$\mathfrak{F}^m : \Whm \to \Whm$ be a linear map such that 
\eqref{eq:Fconsist} holds.
Then there exists a unique solution $(\vec X^{m+1}, \kappa^{m+1})\in
\Vhm\times \Whm$ to the system \eqref{eq:F12}.
\item \label{item:otherstab}
Let $d=2$ or $d=3$. 
Let $(\vec X^{m+1},\kappa^{m+1}) \in \Vhm \times \Whm$ be a solution to
\eqref{eq:F12}. Then we have that
\begin{equation*}
\left|\Gamma^{m+1}\right| + \ttau_m 
\left\langle \mathfrak{F}^m(\kappa^{m+1}), \kappa^{m+1} 
\right\rangle^h_{\Gamma^m} \leq \left|\Gamma^m\right|,
\end{equation*}
where both terms on the left hand side are nonnegative if
$\mathfrak{F}^m$ satisfies \eqref{eq:Fconsist}.
\end{enumerate}
\end{thm}
\begin{proof}
\ref{item:otherex}
The proof is analogous to the proof of Lemma~\ref{lem:exXk}.
\\
\ref{item:otherstab}
The proof is analogous to the proof of Theorem~\ref{thm:stabSD}.
\end{proof}

\begin{rem}[Examples]  \label{rem:other}
\rule{0pt}{0pt}
\begin{enumerate}
\item 
Choosing $\mathfrak{F}(\varkappa) = -\Delta_s\,\varkappa$
and $\mathfrak{F}^m(\chi) = - \LapGm \chi$, recall \eqref{eq:discreteLB}, 
we recover the 
results in {\rm Theorem~\ref{thm:uniqueSD}} and {\rm Theorem~\ref{thm:stabSD}}
for surface diffusion. Similarly, for 
$\mathfrak{F}(\varkappa) = \varkappa$ and $\mathfrak{F}^m(\chi) = \chi$
we obtain {\rm Theorem~\ref{thm:uniqueMC}} and {\rm Theorem~\ref{thm:stabMC}}
for mean curvature flow.
\item 
Of interest are also nonlinear flows
$\mathcal{V}=f(\varkappa)$, i.e.\ $\mathfrak{F}(\varkappa) =
f(\varkappa)$, where $f:(a,b)\to\bR$ with $-\infty\le a <
b\le\infty$ is a strictly monotonically increasing continuous
function such that $f(0)=0$. An example is $f(r)=|r|^{\beta-1}\,r$,
with $\beta \in \bRplus$, which has applications in image analysis, see
{\rm \cite{AlvarezGLM93,SapiroT94,MikulaS01}}. Here we define
$\mathfrak{F}^m (\chi) = \pi_{\Gamma^m}\,[f(\chi)]$, 
recall {\rm Definition~\ref{def:Vh}\ref{item:Vh}}.
For the resulting nonlinear scheme \eqref{eq:F12} we obtain the stability
result {\rm Theorem~\ref{thm:other}\ref{item:otherstab}}.
Existence and uniqueness results for \eqref{eq:F12} are shown in
{\rm \citet[Theorem~2.1]{gflows3d}}.
\item 
Also the volume conserving flow
\begin{equation}\label{eq:conservedMC}
\mathcal{V} = \mathfrak{F}(\varkappa) =\varkappa -
\Gmint{\Gamma(t)}\,\varkappa
\qquad\text{on } \Gamma(t)\,,
\end{equation}
where we define the average
\begin{equation} \label{eq:mint}
\Gmint{\Gamma(t)}\,\varkappa = \frac1{|\Gamma(t)|}\,
\int_{\Gamma(t)}\varkappa\dH{d-1}\,,
\end{equation}
falls into the class of functions for which stability can be shown. 
Here, for all $\chi \in \Whm$, we choose 
$\mathfrak{F}^m(\chi) = \chi - \Gmint{\Gamma^m}\,\chi \in \Whm$ and obtain
\begin{equation*}
\left\langle \mathfrak{F}^m(\chi), \chi \right\rangle^h_{\Gamma^m} = 
\left(\left| \chi - \Gmint{\Gamma^m}\,\chi \right|^h_{\Gamma^m}\right)^2
\geq 0\,,
\end{equation*}
and so both results in {\rm Theorem~\ref{thm:other}} hold.
\item 
In materials science \eqref{eq:conservedMC} is called surface
attachment limited kinetics (SALK). A flow interpolating between
SALK and surface diffusion is
\begin{equation*}
\mathcal{V} 
= -\Delta_s\left(\frac1\alpha -
\frac1\xi\,\Delta_s\right)^{-1}\varkappa
\qquad\text{on } \Gamma(t)\,,
\end{equation*}
where $\alpha,\,\xi \in \bRplus$.
This flow was proposed by {\rm \cite{TaylorC94}} and analyzed by
{\rm \cite{ElliottG97,EscherGI01}} as an evolution law for interfaces in
materials science. 
Introducing the new variable $y$, this flow can be restated as 
\begin{equation*}
\vec{\mathcal{V}}\,.\,\vec\nu = -\Delta_s\, y,\quad
\left(\frac1\alpha-\frac1\xi\,\Delta_s\right)y=\varkappa\,,
\qquad\varkappa\,\vec\nu = \Delta_s\,\vec\id
\qquad\text{on } \Gamma(t)\,.
\end{equation*}
Hence, for any $\chi \in \Whm$, the function $\mathfrak{F}^m(\chi) \in \Whm$ 
is defined via
\begin{equation}\label{eq:inter1}
\left\langle \mathfrak{F}^m(\chi), \eta \right\rangle^h_{\Gamma^m} = 
\left\langle\nabs\,Y(\chi),\nabs\,\eta\right\rangle_{\Gamma^m}
\quad\forall\ \eta\in \Whm\,,
\end{equation}
where $Y \in \Whm$ solves
\begin{equation}\label{eq:inter2}
\frac1\xi\left\langle\nabs\,Y,\nabs\, \varphi \right\rangle_{\Gamma^m}
+\frac1\alpha\left\langle Y,\varphi\right\rangle^h_{\Gamma^m} =
\left\langle\chi,\varphi \right\rangle^h_{\Gamma^m}
\quad\forall\ \varphi\in \Whm\,.
\end{equation}
Choosing $\eta=\chi$ in \eqref{eq:inter1} and
$\varphi=\chi-\frac1\alpha Y$ in \eqref{eq:inter2} then 
yields that
\begin{equation*}
\left\langle \mathfrak{F}^m(\chi),\chi \right\rangle_{\Gamma^m}^h  
= \left\langle\nabs\,Y,\nabs\, \chi \right\rangle_{\Gamma^m}
=\frac1\alpha \left|\nabs\,Y \right|^2_{\Gamma^m} +
\xi \left(\left|\chi-\frac1\alpha\,Y\right|^h_{\Gamma^m}\right)^2
\geq 0 \,,
\end{equation*}
and hence we obtain a stable discretization,
satisfying the two results in {\rm Theorem~\ref{thm:other}}, also in this case.
\end{enumerate}
\end{rem}

\begin{rem} [Semidiscrete schemes and tangential motion] \label{rem:othersd}
\rule{0pt}{0pt}
\begin{enumerate}
\item \label{item:othersd}
Similarly to {\rm \S\ref{subsec:vol}}, semidiscrete variants of the schemes
discussed in {\rm Remark~\ref{rem:other}} can also be considered.
In each case they will satisfy {\rm Theorem~\ref{thm:SDSD}\ref{item:SDSDTM}}, 
the appropriate analogue of {\rm Theorem~\ref{thm:SDSD}\ref{item:SDSDstab}}, 
as well as {\rm Theorem~\ref{thm:SDSD}\ref{item:SDSDvol}} when the
approximated flow is volume preserving,
provided that the semidiscrete analogue of $\mathfrak F^m$ satisfies 
$\left\langle \mathfrak F^h(\chi), 1 \right\rangle^h_{\Gamma^h(t)} = 0$
for all $\chi \in \Wht$.
\item \label{item:otherfdfi}
In the case of curves, $d=2$, for all the situations in 
{\rm Remark \ref{rem:other}}, fully discrete fully implicit schemes along the
lines of \eqref{eq:21617} can be considered. In each case these schemes will
inherit the stability properties from {\rm Remark~\ref{rem:other}} and, in
addition, they will satisfy the strong equidistribution property
\eqref{eq:fdfiTM}. For further details we refer to {\rm \cite{fdfi}}.
\end{enumerate}
\end{rem}

\subsection{Approximations with reduced or induced tangential motion} 
\label{subsec:TM}

In this section we consider an alternative to \eqref{eq:F2},
that allows to either reduce the tangential motion or encourage tangential 
motion in selected directions.
To this end, we assume that $\Gamma^m$ satisfies
Assumption~\ref{ass:A}\ref{item:assA2} and let
$\vec\tau^m_i \in \Vhm$, for $i=1,\ldots,d-1$, be such that
$\left\{\frac{\vec\omega^m(\vec q^m_k)}{|\vec\omega^m(\vec q^m_k)|},
\vec\tau^m_1(\vec q^m_k),\ldots,\vec\tau^m_{d-1}(\vec q^m_k)\right\}$
is an orthonormal basis of $\bR^d$ for $k=1,\ldots,K$.
Moreover, we choose coefficients $0\leq\alpha^m_i,\delta^m_i 
\in \Whm$, $i=1,\ldots, d-1$, and forcing terms
$c^m_i\in \Whm$, $i=1,\ldots, d-1$. Then, in place of
\eqref{eq:F12} we consider the following approximation.

Let the closed polyhedral hypersurface $\Gamma^0$ be an approximation of
$\Gamma(0)$. Then, for $m=0,\ldots,M-1$, find 
$(\vec X^{m+1},\kappa^{m+1}) \in \Vhm\times\Whm$ 
and $\beta^{m+1}_i \in \Whm$, $i=1,\ldots,d-1$, such that
\begin{subequations} \label{eq:F12beta}
\begin{align}
& \left\langle \frac{\vec X^{m+1}-\vec\id}{\ttau_m}, \chi\,\vec\nu^m
\right\rangle^h_{\Gamma^m}
-\left\langle \mathfrak{F}^m(\kappa^{m+1}), \chi
\right\rangle^h_{\Gamma^m} = 0 \quad\forall\ \chi \in \Whm\,, 
\label{eq:F1beta} \\
& \left\langle \alpha^m_i\,\frac{\vec X^{m+1}-\vec\id}{\ttau_m},
\xi\,\vec\tau^m_i\right\rangle_{\Gamma^m}^h
- \left\langle \alpha^m_i\left[\delta^m_i\,\beta^{m+1}_i + c^m_i\right], 
\xi \right\rangle_{\Gamma^m}^h = 0 \quad\forall\ \xi \in \Whm\,,
\nonumber \\ & \hspace{7cm} i=1,\ldots,d-1 \,, \label{eq:Betaa} \\
& \left\langle \kappa^{m+1}\,\vec\nu^m 
+ \sum_{i=1}^{d-1} \alpha^m_i\,\beta^{m+1}_i\,\vec\tau^m_i,
\vec\eta\right\rangle_{\Gamma^m}^h
+ \left\langle \nabs\, \vec X^{m+1}, \nabs\, \vec\eta \right\rangle_{\Gamma^m}
 = 0 \quad\forall\ \vec\eta \in \Vhm  \label{eq:Betab}
\end{align}
\end{subequations}
and set $\Gamma^{m+1} = \vec X^{m+1}(\Gamma^m)$.

\begin{thm}\label{thm:uniqueTM}
\rule{0pt}{0pt}
\begin{enumerate}
\item \label{item:uniqueTM}
Let $\Gamma^m$ satisfy {\rm Assumption~\ref{ass:A}} and let
$\mathfrak{F}^m : \Whm \to \Whm$ be a linear map such that 
\eqref{eq:Fconsist} holds. 
Then there exists a solution 
$(\vec X^{m+1},$ $\kappa^{m+1},\beta_1^{m+1},\ldots,\beta_{d-1}^{m+1})\in
\Vhm\times [\Whm]^d$ to the system \eqref{eq:F12beta}, with 
\arxivyesno{
$(\vec X^{m+1}, \kappa^{m+1}, 
\pi_{\Gamma^m}\,[\alpha^m_1\,\beta_1^{m+1}],\ldots,$ \linebreak $
\pi_{\Gamma^m}\,[\alpha^m_{d-1}\,\beta_{d-1}^{m+1}])$}
{$(\vec X^{m+1}, \kappa^{m+1}, 
\pi_{\Gamma^m}\,[\alpha^m_1\,\beta_1^{m+1}],\ldots,
\pi_{\Gamma^m}\,[\alpha^m_{d-1}\,\beta_{d-1}^{m+1}])$}
being unique.
\item \label{item:stabTM}
Let $d=2$ or $d=3$. 
Let $(\vec X^{m+1},\kappa^{m+1},\beta_1^{m+1},\ldots,\beta_{d-1}^{m+1}) 
\in \Vhm \times [\Whm]^d$ be a solution to \eqref{eq:F12beta}. 
Then we have that
\begin{align}
& |\Gamma^{m+1}| + 
\ttau_m \left\langle \mathfrak{F}^m(\kappa^{m+1}), \kappa^{m+1} 
\right\rangle^h_{\Gamma^m}
\nonumber \\ & \qquad
+\ttau_m \left\langle\alpha^m_i\left[\delta^m_i\,\beta^{m+1}_i + c^m_i\right],
\beta^{m+1}_i \right\rangle_{\Gamma^m}^h 
\leq |\Gamma^m|\,.
\label{eq:stabk}
\end{align}
\end{enumerate}
\end{thm}
\begin{proof}
\ref{item:uniqueTM}
Let
\[
B^m_i(\Gamma^m) = \left\{ \xi \in \Whm : \xi(\vec q^m_k) = 0 \ \text{ if }
\alpha^m_i(\vec q^m_k) = 0\,, k = 1,\ldots,K \right\} \,,
\]
for $i = 1,\ldots,d-1$. We now prove the desired results by showing existence
of a unique solution to the system \eqref{eq:F12beta} with
$\beta^{m+1}_i \in \Whm$ replaced by $\beta^{m+1}_i \in B^m_i(\Gamma^m)$, 
and with $\Whm$ in \eqref{eq:Betaa} replaced by $B^m_i(\Gamma^m)$. 
As usual,
existence to this adapted linear system follows from uniqueness, and so we
let $(\vec X,\kappa,b_1,\ldots,b_{d-1}) \in \Vhm\times \Whm \times
B^m_1(\Gamma^m) \times \cdots \times B^m_{d-1}(\Gamma^m)$ be such that
\begin{subequations} \label{eq:F1betaIP}
\begin{align}
& \left\langle \vec X, \chi\,\vec\omega^m\right\rangle_{\Gamma^m}^h
- \ttau_m \left\langle \mathfrak{F}^m( \kappa), \chi \right\rangle_{\Gamma^m}^h 
 = 0 \quad\forall\ \chi \in \Whm\,, \label{eq:F1betaIPa} \\
& \left\langle \alpha^m_i\,\vec X, \xi\,\vec\tau^m_i\right\rangle_{\Gamma^m}^h
- \ttau_m\left\langle \alpha^m_i\,\delta^m_i\,b_i, \xi 
\right\rangle_{\Gamma^m}^h 
= 0 \quad\forall\ \xi \in B^m_i(\Gamma^m)\,, 
\nonumber \\ & \hspace{7cm} i=1,\ldots,d-1 \,, \label{eq:F1betaIPb} \\
& \left\langle \kappa\,\vec\omega^m + \sum_{i=1}^{d-1}
\alpha^m_i\,b_i\,\vec\tau^m_i, \vec\eta\right\rangle_{\Gamma^m}^h
+ \left\langle \nabs\,\vec X, \nabs\,\vec\eta \right\rangle_{\Gamma^m} = 0
\quad\forall\ \vec\eta \in \Vhm\,, \label{eq:F1betaIPc}
\end{align}
\end{subequations}
where we have observed \eqref{eq:omegahnuh}. 
Choosing $\chi = \kappa\in \Whm$ in (\ref{eq:F1betaIPa}),
$\xi = b_i \in B^m_i(\Gamma^m)$ in (\ref{eq:F1betaIPb}) for $i=1,\ldots,d-1$,
and $\vec\eta = \vec X\in\Vhm$ in (\ref{eq:F1betaIPc}) yields that
\begin{equation} \label{eq:unique00}
\ttau_m \left\langle \mathfrak{F}^m( \kappa), \kappa \right\rangle_{\Gamma^m}^h
+\ttau_m \left\langle\alpha^m_i\,\delta^m_i\,b_i , b_i 
\right\rangle_{\Gamma^m}^h
+ \left|\nabs\, \vec X\right|_{\Gamma^m}^2  =0\,.
\end{equation}
It immediately follows from \eqref{eq:unique00}, \eqref{eq:Fconsist} and
our assumptions on $\alpha^m_i$ and $\delta^m_i$ 
that $\vec X$ is constant on $\Gamma^m$.
Hence choosing $\vec\eta = \vec\tau^m_i$ in \eqref{eq:F1betaIPc} yields
that $\left| \pi_{\Gamma^m}[\alpha^m_i\,b_i] \right|_{\Gamma^m}^h = 0$, and so
$\pi_{\Gamma^m}[\alpha^m_i\,b_i] = 0$, which implies that
$b_i = 0 \in B^m_i(\Gamma^m)$, for $i=1,\ldots,d-1$.
The remainder of the proof proceeds as the proof
of Lemma~\ref{lem:exXk} to show that $\kappa = 0$ and 
$\vec X = \vec 0$. 
This proves the uniqueness of a solution 
to \eqref{eq:F1betaIP}, and hence the existence of a solution
$(\vec X^{m+1},\kappa^{m+1},\beta_1^{m+1},\ldots,\beta_{d-1}^{m+1}) 
\in \Vhm\times \Whm \times 
B^m_1(\Gamma^m) \times \cdots \times B^m_{d-1}(\Gamma^m)$ to the modified
system discussed at the beginning of the proof. Clearly, this solution also
solves the original system \eqref{eq:F1beta}, with only the values
of $\beta_i^{m+1}$ arbitrary where $\alpha_i^{m}$ vanishes.
\\
\ref{item:stabTM}
Choosing $\chi = \kappa^{m+1} \in \Whm$ in \eqref{eq:F1beta},
$\xi = \beta_i^{m+1} \in \Whm$ in \eqref{eq:Betaa} for $i=1,\ldots,d-1$ and
$\vec\eta = \frac1{\ttau_m}\,(\vec X^{m+1}-\vec\id_{\mid_{\Gamma^m}})\in\Vhm$
in \eqref{eq:Betab} yields the desired result on 
recalling Lemma~\ref{lem:stab2d3d}.
\end{proof}

\begin{rem} \label{rem:F12beta}
\rule{0pt}{0pt}
\begin{enumerate}
\item
Clearly, \eqref{eq:F12beta} with $\alpha^m_i=0$, for $i=1,\ldots,d-1$,
reduces to the original scheme \eqref{eq:F12}.
\item
In the case $d=2$ or $d=3$, if $\mathfrak{F}^m$ satisfies \eqref{eq:Fconsist}, 
and if $c^m_i=0$ for $i=1,\ldots,d-1$, then \eqref{eq:stabk} provides a 
stability estimate for \eqref{eq:F12beta}.
\item \label{item:strategies}
Choosing $\delta^m_i(\vec q^m_k)=1$ and $c^m_i(\vec q^m_k)=0$ 
for $i=1,\ldots,d-1$, 
it follows intuitively from \eqref{eq:stabk} that tangential motion of 
$\vec q^m_k$ in the direction of $\vec\tau^m_i(\vec q^m_k)$ will be 
suppressed if $\alpha^m_i(\vec q^m_k)$ is large. 
Conversely, it is also clear from \eqref{eq:Betaa} that
choosing $\delta^m_i(\vec q^m_k)=0$ and $\alpha^m_i(\vec q^m_k)>0$ allows 
us to fix the tangential motion. These observations form an
ansatz to control tangential movement in the discrete evolution of
geometric flows.
In particular, in {\rm \cite{willmore}} the following strategies have been 
proposed and employed.
\begin{enumerate}[label={{\rm (S\arabic*)}}]
\item \label{item:S1}
$\alpha^m_i = \alpha \in \bRgeq\,,\ \delta^m_i = 1\,,\ c^m_i = 0\,;$
\item \label{item:S2}
$\alpha^m_i = \alpha \in \bRplus\,,\ \delta^m_i = \delta\in\bRplus\,,$ \\ 
$c^m_i(\vec q^m_k) = \frac1{\ttau_m}\left(\vec z^m_k-\vec q^m_k\right)
.\,\vec\tau^m_{i}(\vec q^m_k)\,,\ k=1,\ldots,K\,;$
\item \label{item:S3}
$\alpha^m_i = 1\,,\ \delta^m_i = 0\,,$ \\ 
$c^m_i(\vec q^m_k) = \frac1{\ttau_m}\left(\vec z^m_k-\vec q^m_k\right)
.\,\vec\tau^m_{i}(\vec q^m_k)\,,\ k=1,\ldots,K\,;$
\end{enumerate}
for $i=1,\ldots,d-1$, 
where $\vec z^m_k$ is the average of the neighbouring nodes of $\vec q^m_k$.
The effect of these strategies can be summarized as follows. With increasing
$\alpha>0$, \ref{item:S1} leads to smaller $|\beta^{m+1}_i|$, $i=1\to d-1$, 
and hence to less tangential motion, see \eqref{eq:Betab}. Strategy 
\ref{item:S2}, on the other hand, is intended to induce a tangential movement 
towards the ``barycentres'' $\vec z^m_k$.
Lastly, \ref{item:S3} completely fixes the tangential motion, so that 
after the time step, each vertex $\vec X^{m+1}(\vec q^m_k)$ has the same 
tangential components as $\vec z^m_k$.
\item \label{item:F12betasd}
Similarly to {\rm \S\ref{subsec:vol}}, a semidiscrete variant of the scheme
\eqref{eq:F12beta} can also be considered.
It will satisfy the appropriate analogue of 
{\rm Theorem~\ref{thm:SDSD}\ref{item:SDSDstab}}, 
as well as {\rm Theorem~\ref{thm:SDSD}\ref{item:SDSDvol}} when the
approximated flow is volume preserving.
\end{enumerate}
\end{rem}

\begin{rem} \label{rem:Qtheta}
\rule{0pt}{0pt}
\begin{enumerate}
\item \label{item:normalize}
In place of \eqref{eq:F12} and  \eqref{eq:F12beta} it is possible to consider 
two closely related variants, where in \eqref{eq:F12} and \eqref{eq:F12beta} 
the terms $\vec\nu^m$ are replaced by $\frac{\vec\omega^m}{|\vec\omega^m|}$. 
On recalling \eqref{eq:omegahnuh} this is a minor change to the original
schemes, where at each vertex $\vec\omega^m$ is now normalized.
For these variants all the results in 
{\rm Theorem~\ref{thm:other}} and {\rm Theorem~\ref{thm:uniqueTM}} still hold
true, as well as most of the results in 
{\rm Remark~\ref{rem:other}\ref{item:othersd}} and
{\rm Remark~\ref{rem:F12beta}}. The only property in the latter two remarks
that no longer holds is the volume preservation. As can be seen from the proof
of {\rm Theorem~\ref{thm:SDSD}\ref{item:SDSDvol}}, volume conservation requires 
$\vec\nu^m$ to be present in \eqref{eq:F12} and \eqref{eq:F12beta}, and it is
for this reason that in general we prefer these schemes in their original form.
\item \label{item:Qtheta}
In the case of mean curvature flow, when $\mathfrak F^m(\chi) = \chi$, 
the solution $\vec X^{m+1}\in \Vhm$
to \eqref{eq:F12beta} for the strategy \ref{item:S1} from
{\rm Remark~\ref{rem:F12beta}\ref{item:strategies}} satisfies
\begin{align} 
& \frac1{\ttau_m}
\left\langle \left[\left(\vec X^{m+1}-\vec\id\right).\,\vec\omega^m\right]
\vec\omega^m
+ \alpha\,\sum_{i=1}^{d-1} \left[\left(\vec X^{m+1}-\vec\id\right).\,
\vec\tau^m_i\right]
\vec\tau^m_i, \vec\eta\right\rangle^h_{\Gamma^m} 
\nonumber \\ & \hspace{2cm}
+ \left\langle \nabs\,\vec X^{m+1}, \nabs\,\vec\eta \right\rangle_{\Gamma^m}
 = 0 \quad\forall\ \vec\eta \in \Vhm\,. \label{eq:MCalpha0}
\end{align}
A closely related approximation is obtained by normalizing $\vec\omega^m$ in
\arxivyesno{}{\linebreak}%
\eqref{eq:MCalpha0} to yield
\begin{align} 
& \frac1{\ttau_m}
\left\langle \left[\left(\vec X^{m+1}-\vec\id\right).\,\vec\omega^m\right]
\frac{\vec\omega^m}{\left|\vec\omega^m\right|^2}
+ \alpha\,\sum_{i=1}^{d-1} \left[\left(\vec X^{m+1}-\vec\id\right).\,
\vec\tau^m_i\right]\vec\tau^m_i, \vec\eta\right\rangle^h_{\Gamma^m} 
\nonumber \\ & \hspace{2cm}
+ \left\langle \nabs\,\vec X^{m+1}, \nabs\,\vec\eta \right\rangle_{\Gamma^m}
 = 0 \quad\forall\ \vec\eta \in \Vhm\,, \label{eq:MCalpha}
\end{align}
and this corresponds to the variant of \eqref{eq:F12beta} discussed in
\ref{item:normalize}.
On introducing $\mat Q^m_\theta \in \matVhm$, for $\theta \in\bRgeq$, with 
\begin{equation} \label{eq:Qm}
\mat Q^m_\theta(\vec q^m_k) = 
\theta\,\mat\Id + (1-\theta)\,
\frac{\vec\omega^m(\vec q^m_k) \otimes \vec\omega^m(\vec q^m_k)}
{\left|\vec\omega^m(\vec q^m_k)\right|^2}\,,
\quad k= 1 ,\ldots, K\,,
\end{equation}
it immediately follows from the fact that 
\arxivyesno{$\{\frac{\vec\omega^m(\vec q^m_k)}{|\vec\omega^m(\vec q^m_k)|},
\vec\tau^m_1(\vec q^m_k),\ldots,\vec\tau^m_{d-1}(\vec q^m_k)\}$}
{$\{\frac{\vec\omega^m(\vec q^m_k)}{|\vec\omega^m(\vec q^m_k)|},
\vec\tau^m_1(\vec q^m_k),\ldots,$\linebreak 
$\vec\tau^m_{d-1}(\vec q^m_k)\}$}
is an orthonormal basis, that for every $k=1,\ldots,K$ and for every 
$\vec z \in \bR^d$
\begin{align*} 
& \left[\vec z \,.\,\vec\omega^m(\vec q^m_k)\right]
\frac{\vec\omega^m(\vec q^m_k)}{\left|\vec\omega^m(\vec q^m_k)\right|^2}
+ \alpha\,\sum_{i=1}^{d-1} \left[\vec z\,.\,\vec\tau^m_i(\vec q^m_k)\right]
\vec\tau^m_i(\vec q^m_k)  \nonumber \\ & \hspace{1cm}
= (1-\alpha)\,\vec z \,.\,\vec\omega^m(\vec q^m_k)\,
\frac{\vec\omega^m(\vec q^m_k)}{\left|\vec\omega^m(\vec q^m_k)\right|^2}
+ \alpha\,\vec z 
= \mat Q^m_\alpha(\vec q^m_k)\,\vec z\,.
\end{align*}
Hence, for $\alpha = \theta$, \eqref{eq:MCalpha} is equivalent to
\begin{equation} \label{eq:MCtheta}
\left\langle \mat Q^m_\theta\,\frac{\vec X^{m+1}-\vec\id}{\ttau_m}, 
\vec\eta\right\rangle^h_{\Gamma^m}
+ \left\langle \nabs\,\vec X^{m+1}, \nabs\,\vec\eta \right\rangle_{\Gamma^m}
 = 0 \quad\forall\ \vec\eta \in \Vhm\,. 
\end{equation}
Clearly, $\theta = 1$ collapses to \eqref{eq:DziukMC} with mass lumping,
while $\theta = 0$ is \eqref{eq:elimkappa} with normalized $\vec\omega^m$,
recall \ref{item:normalize}. For $\theta \in (0,1)$ we obtain some
interpolation between the two, with the tangential motion as described in
\ref{item:S1} from {\rm Remark~\ref{rem:F12beta}\ref{item:strategies}}
for $\alpha = \theta$.
We also observe that the main idea in {\rm \cite{ElliottF17}}, 
for the case $d=2$, 
is to introduce $\theta\,\mat\Id + (1-\theta)\,\vec\nu^m \otimes \vec\nu^m$
into the first term in \eqref{eq:DD95}, similarly to how \eqref{eq:MCtheta}
relates to \eqref{eq:DziukMC}.
\end{enumerate}
\end{rem}

\begin{rem}[Discrete linear systems] \label{rem:disc15}
It is a simple matter to adapt the definitions and techniques in 
{\rm \S\ref{subsec:3.3}} in order to derive the discrete linear systems that
need to be solved at each time level for the approximations
\eqref{eq:sdfea}, \eqref{eq:F12} and \eqref{eq:F12beta}. 
The most efficient way is employing a sparse direct solution method such
as UMFPACK, see {\rm \cite{Davis04}}.
\end{rem}

\subsection{Alternative parametric methods} \label{subsec:alternativeSD}

An alternative parametric finite element approximation of surface diffusion is
given in \cite{BanschMN05}. Their scheme is based on a discretization of 
the formulation
\begin{equation*}
\vec{\mathcal{V}} = \mathcal{V}\,\vec\nu\,,\quad
\mathcal{V} = -\Delta_s\,\varkappa\,, \quad
\varkappa = \vec\varkappa\,.\,\vec\nu\,,\quad
\vec\varkappa = \Delta_s\,\vec\id\,,
\qquad\text{on } \Gamma(t)\,,
\end{equation*}
in contrast to \eqref{eq:strongSD}. Compared to \eqref{eq:sdfea} there are two
more variables, and as the surface always moves in a direction collinear to the
normal, the surface meshes will in general deteriorate.

We refer also to \cite{BanschMN04}, where, on assuming a sufficiently
smooth solution, an error analysis is presented
for a semidiscrete finite element approximation by continuous
piecewise polynomials of degree $k\geq1$ for a graph formulation of 
surface diffusion.  
A semidiscrete finite element approximation of axisymmetric surface diffusion 
by continuous piecewise linears 
in a graph formulation has been considered in \cite{ColemanFM96}.
The corresponding error analysis, on assuming a sufficiently 
smooth solution, is presented in \cite{DeckelnickDE03}.
In addition, a parametric finite element approximation of axisymmetric surface 
diffusion has been considered in \cite{axisd}.

\section{Anisotropic flows} \label{sec:ani}
\subsection{Derivation of the governing equations} \label{subsec:ani1}
In many interface problems the energy density is directionally
dependent. This can result, for example, from a material's directional
dependence of a physical property. This appears, for example, in a
crystal where the energy of an interface depends on how its direction
is related to the crystal lattice orientations. A typical anisotropic
surface energy has the form
\begin{equation}\label{eq:11}
  |\Gamma|_\gamma = \int_\Gamma \gamma(\vec\nu) \dH{d-1}\,,
\end{equation}
where $\Gamma$ is a closed orientable $C^1$-hypersurface 
in $\bR^d$, $d\geq2$, with a unit normal field $\vec\nu$,
and $\gamma:\bS^{d-1} \to \bRplus$ is a given anisotropic energy density. 
It is mathematically convenient to extend $\gamma$ to a function on $\bR^d$.
In particular, denoting the extension again with $\gamma$, 
we assume that it is absolutely homogeneous of degree one, i.e.\
\begin{equation*}
  \gamma(\lambda\,\vec p) = |\lambda|\,\gamma(\vec p) \quad 
\forall\ \vec p\in \bR^d,\,\,\lambda\in \bR\,.
\end{equation*}
Assuming from now on that $\gamma \in C^2(\bR^d\setminus\{\vec0\})$,
we can differentiate this identity with respect to $\lambda$ to obtain
\begin{equation} \label{eq:gammadash}
  \gamma'(\vec p)\,.\,\vec p = \gamma(\vec p) \quad\forall\ 
\vec p\in\bR^d\setminus\{\vec 0\}\,,
\end{equation}
where $\gamma'$ is the gradient of $\gamma$. 
In the isotropic case, $\gamma(\vec p)=|\vec p|$, and so $\gamma'(\vec p)= \frac{\vec p}{|\vec p|}$.
We refer to \cite{TaylorCH92,BellettiniNP99,Giga06} for more details on
anisotropic energies in materials science and geometry.

\begin{lem} \label{lem:ani}
Let $(\Gamma(t))_{t\in[0,T]}$ be a $C^2$--evolving closed orientable 
hypersurface with normal vector field $\vec\nu$.
Then we have the anisotropic version of {\rm (\ref{eq:firstvar})} 
\begin{equation*}
\ddt\,|\Gamma(t)|_\gamma 
= -\left\langle \varkappa_\gamma,\mathcal{V} \right\rangle_{\Gamma(t)},
\end{equation*}
where $\varkappa_\gamma = -\nabs\,.\,\vec\nu_\gamma$ is the
weighted mean curvature and $\vec\nu_\gamma = \gamma'(\vec\nu)$ is
the Cahn--Hoffmann vector, the anisotropic version of 
{\rm Lemma \,\ref{lem:varkappa}\ref{item:varkappa}}.  
\end{lem}
\begin{proof} 
It follows from the transport theorem, Theorem~\ref{thm:trans}, 
on noting (\ref{eq:11}), \eqref{eq:gammadash}, 
Lemma \ref{lem:dtnu}\ref{item:normdtnu}, Theorem~\ref{thm:div} and 
Lemma~\ref{lem:productrules}\ref{item:pri}, that
\begin{align*}
\ddt\,|\Gamma(t)|_\gamma &
= \ddt \left\langle 1, \gamma(\vec\nu) \right\rangle_{\Gamma(t)}
= \left\langle 1, \matpartn\,\gamma(\vec\nu)
-\gamma(\vec\nu)\,\mathcal{V}\,\varkappa \right\rangle_{\Gamma(t)}\\ &
= \left\langle 1, \gamma'(\vec\nu)\,.\,\matpartn\,\vec\nu - 
\gamma'(\vec\nu)\,.\,\vec\nu\,\mathcal{V}\,\varkappa \right\rangle_{\Gamma(t)}
\\ &
= \left\langle 1, -\gamma'(\vec\nu)\,.\,\nabs\,\mathcal{V}
+\nabs\,.\,(\gamma'(\vec\nu)\,\mathcal{V}) \right\rangle_{\Gamma(t)}\\ &
= \left\langle\nabs\,.\,\gamma'(\vec\nu),\mathcal{V}\right\rangle_{\Gamma(t)}
= -\left\langle \varkappa_\gamma,\mathcal{V} \right\rangle_{\Gamma(t)}.
\end{align*}
\end{proof}

We can hence define anisotropic versions of mean curvature flow and
surface diffusion as
\begin{equation}\label{eq:aniflows}
\mbox{(a)} \quad
\mathcal{V} = \beta(\vec\nu)\,\varkappa_\gamma
\qquad\mbox{and} \qquad
\mbox{(b)} \quad
\mathcal{V} =
-\nabs\,.\left(\beta(\vec\nu)\,\nabs\, \varkappa_\gamma\right),
\end{equation}                                                  
where $\beta: \bS^{d-1} \to\bRplus$ is a smooth kinetic coefficient.
We refer to
\cite{TaylorCH92,CahnT94,TaylorC94,Giga06} for a derivation and more
information about these evolution laws.

For anisotropic mean curvature flow we obtain from Lemma \ref{lem:ani} and 
\arxivyesno{}{\linebreak}%
(\ref{eq:aniflows}) that
\begin{equation}\label{eq:anisoMCin}
\ddt\,|\Gamma(t)|_\gamma 
= -\left\langle \varkappa_\gamma,\mathcal{V} \right\rangle_{\Gamma(t)}
= -\left\langle \beta(\vec\nu) , \varkappa_\gamma^2 \right\rangle_{\Gamma(t)}
\leq 0\,,
\end{equation}
and for anisotropic surface diffusion it holds with the help of 
Lemma \ref{lem:productrules}\ref{item:pri}, Definition 
\ref{def:2.5}\ref{item:def2.5ii} 
and the
divergence theorem, Theorem \ref{thm:div}, that
\begin{equation}\label{eq:anisoSDin}
\ddt\,|\Gamma(t)|_\gamma 
= -\left\langle \varkappa_\gamma,\mathcal{V} \right\rangle_{\Gamma(t)}
= -\left\langle \beta(\vec\nu),\left|\nabs\,\varkappa_\gamma\right|^2 
\right\rangle_{\Gamma(t)}
\leq 0\,.
\end{equation}

Of course, also for the surface energy $|\Gamma|_\gamma$, an
isoperimetric problem can be formulated. Here one wants to find the
shape, which minimizes $|\Gamma|_\gamma$ under all shapes with a given 
enclosed volume. In order to do so, one defines the dual function
\begin{equation*}
  \gamma^\ast (\vec q) = \sup_{\vec p\in\bR^d\setminus\{\vec 0\}}
  \frac{\vec p\,.\,\vec q}{\gamma(\vec p)} \quad\forall\ \vec q \in \bR^d\,.
\end{equation*}
Then the solution of the isoperimetric problem is, up to a scaling,
the Wulff shape
\begin{equation*}
\mathcal{W} = \{\vec q\in\bR^d : \gamma^\ast(\vec q) \le 1\}\,.
\end{equation*}
This is the $1$-ball of $\gamma^\ast$ and we also define the $1$-ball
of $\gamma$
\begin{equation*}
\mathcal{F} = \{\vec p \in\bR^d : \gamma(\vec p) \le 1\}\,,
\end{equation*}
which is called Frank diagram. We refer to Figure~\ref{fig:wulfffrank}
for examples. 
Also surface energies for which the Frank diagram or the Wulff
shape have flat parts, edges and corners are of interest. These
are called crystalline surface energies, and we will be able to
approximate these also in a stable manner.
\begin{figure}
\center
\arxivyesno{
\newcommand\lheight{3cm} 
}{
\newcommand\lheight{2.5cm} 
}
\includegraphics[angle=0,totalheight=\lheight]{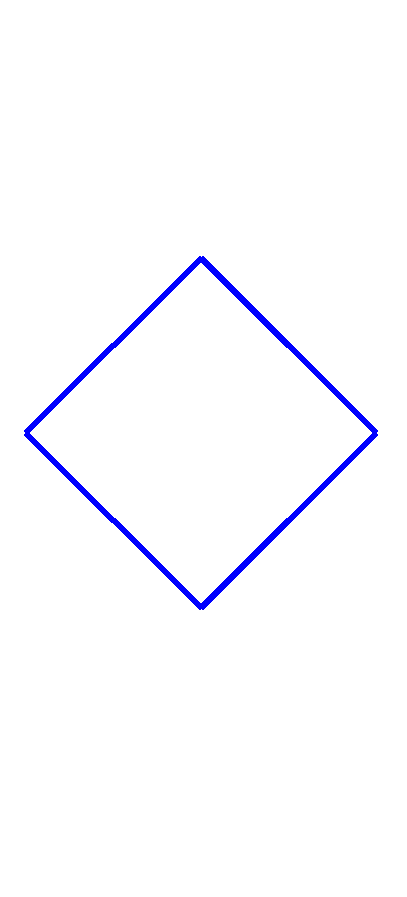} \quad
\includegraphics[angle=0,totalheight=\lheight]{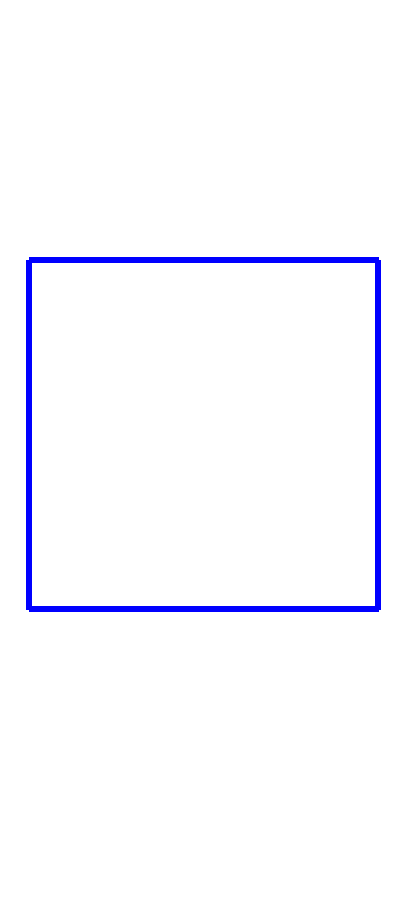} \qquad\qquad\qquad
\includegraphics[angle=0,totalheight=\lheight]{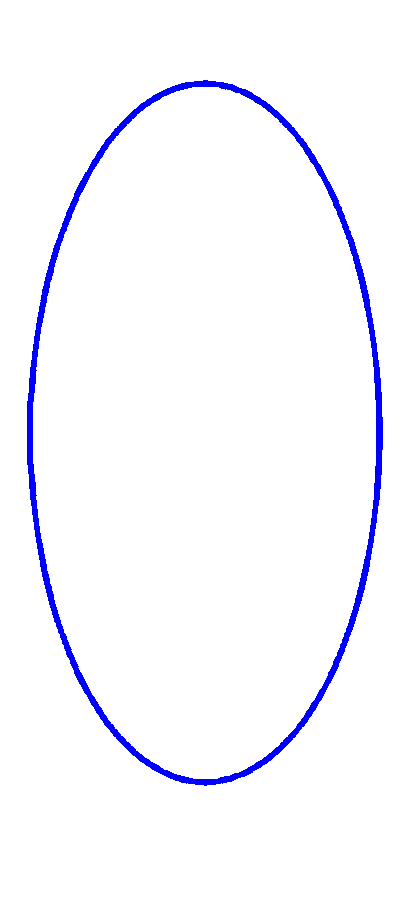} \quad
\includegraphics[angle=0,totalheight=\lheight]{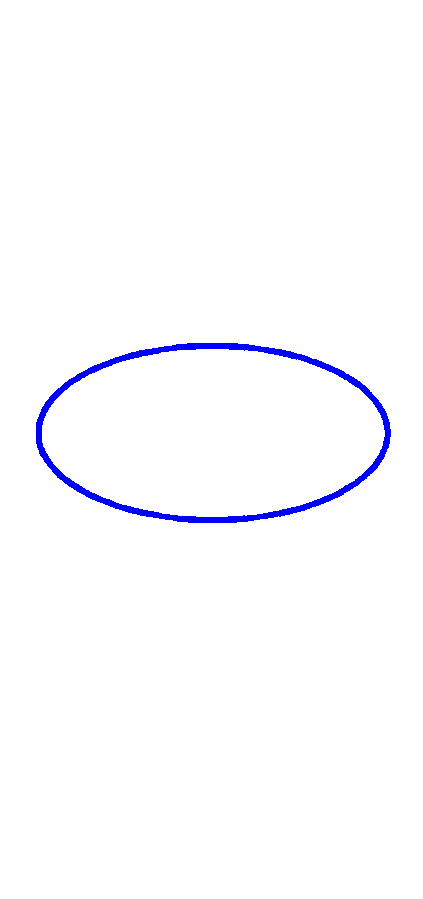}
\caption{Frank diagram and Wulff shape in $\bR^2$ for the
$l^1$-norm, $\gamma(\vec p) = \sum_{i=1}^2 |p_i|$, left,
and for the weighted norm $\gamma(\vec p) = 
\left(\vec p\,.\,\protect\mat{G}\,\vec p\right)^\frac12$, 
$\protect\mat{G} = \frac14\,\binom{4\ 0}{0\ 1}$, right.
}
\label{fig:wulfffrank}
\end{figure}%

\subsection{Suitable weak formulations}
Coming up with stable discretizations is a very difficult task as the flows
\begin{equation*}
\mathcal{V} = \beta(\vec\nu)\,\varkappa_\gamma \quad\text{and}\quad 
\mathcal{V} = 
-\nabs\,.\left(\beta(\vec\nu)\,\nabs\,\varkappa_\gamma\right) 
\qquad\text{on }\Gamma(t)
\end{equation*}
are even more nonlinear than their isotropic counterparts. A major
difficulty is that an analogue of the identity
\begin{equation}\label{eq:dzkappa}
\vec\varkappa = \varkappa\,\vec\nu = \Delta_s\,\vec\id\qquad\mbox{on }\Gamma(t)\,,
\end{equation}
recall Lemma \ref{lem:varkappa}\ref{item:vecvarkappa}, 
is no longer true in general.
However, in many practical situations the present authors were able to
come up with an anisotropic version of (\ref{eq:dzkappa}).

The main observation is that (\ref{eq:dzkappa}) remains true
if we replace the standard Euclidean inner product in $\bR^d$ by an
inner product 
\begin{equation} \label{eq:anisopro}
\left(\vec u,\vec\vvv\right)_\tG = \vec u\,.\,\mat\tG\,\vec\vvv \quad\forall\
\vec u, \vec\vvv\in\bR^d\,,
\end{equation} where $\mat\tG \in \bR^{d\times d}$
is symmetric and positive definite.
One only has to replace the mean curvature
vector $\vec\varkappa$ and the Laplace--Beltrami operator $\Delta_s$
by versions which are appropriate for this new inner product. 
In fact, one just has to consider the canonical Laplace--Beltrami 
operator on $\Gamma(t)$ with respect to the
Riemannian metric given by the new inner product.
However, the surface energy density related to the
new inner product needs to be identified. It can be shown, see
\citet[Lemma~2.1]{ani3d}, that
\begin{equation}\label{eq:111}
  \gamma(\vec\nu) = \sqrt{\vec\nu\,.\, \mat G\,\vec\nu}
\end{equation}
leads to $\mat\tG = [\det\mat G]^{\frac{1}{d-1}} \mat G^{-1}$ and
a weighted mean curvature $\varkappa_\gamma$ 
for which we have a relationship similar to
that in the isotropic case, recall (\ref{eq:dzkappa}). 
In fact, it is shown that
\begin{equation*}
  \varkappa_\gamma\, \vec\nu = \gamma(\vec\nu)\,
  \mat\tG\,\Lapg\,\vec\id  
  \qquad\mbox{on }\Gamma(t)\,,
\end{equation*}
see \citet[(2.33), (2.37)]{ani3d},
where $\Lapg = \nabs^{\tG}\,.\,\nabs^{\tG}$
is the Laplace--Beltrami operator induced by the inner product
\eqref{eq:dzkappa} with $(\nabs^{\tG}\,.)$ and $\nabs^{\tG}$
the associated tangential divergence and tangential gradient,
respectively. 
A suitable generalized divergence theorem on manifolds then
allows one to introduce a weak formulation of
$\gamma(\vec\nu)\,\mat\tG\, \Lapg \,\vec\id$.
Unfortunately, simple anisotropies of the form (\ref{eq:111}) 
only lead to ellipsoidal Wulff shapes as on the right hand side of
Figure~\ref{fig:wulfffrank}, and of course we would like to handle more
general situations.

We now consider a larger class of surface energy densities, which
are given as suitable norms of the ellipsoidal anisotropies. In particular,
we choose
\begin{equation} \label{eq:g}
\gamma(\vec p) = \left(\sum_{\ell=1}^L\,
\left[\gamma_{\ell}(\vec p)\right]^r\right)^{\frac1r}, \qquad
\gamma_\ell(\vec p)= \sqrt{\vec p\,.\,\mat G_{\ell}\,\vec p}\,,
\quad\ell=1,\ldots, L\,,
\end{equation}
so that
\[
\gamma'(\vec p) = \left[\gamma(\vec p)\right]^{1-r}\sum_{\ell=1}^L\,
\left[\gamma_{\ell}(\vec p)\right]^{r-1} \gamma'_\ell(\vec p)\,. 
\]
Here $r\in [1,\infty)$ and $\mat G_{\ell} \in \bR^{d\times d}$, 
$\ell=1,\ldots, L$, are symmetric and positive definite. It turns
out that most energies of relevance can be approximated by the above
class of energies. In particular, hexagonal and cubic anisotropies can be
modelled with appropriate choices of $r$, $L$ and 
$\{\mat G_{\ell}\}_{\ell=1}^L$, see Figure~\ref{fig:BGNwulff}.
\begin{figure}
\center
\arxivyesno{
\newcommand\lheight{3cm} 
}
{
\newcommand\lheight{2.5cm} 
}
\includegraphics[angle=0,totalheight=\lheight]{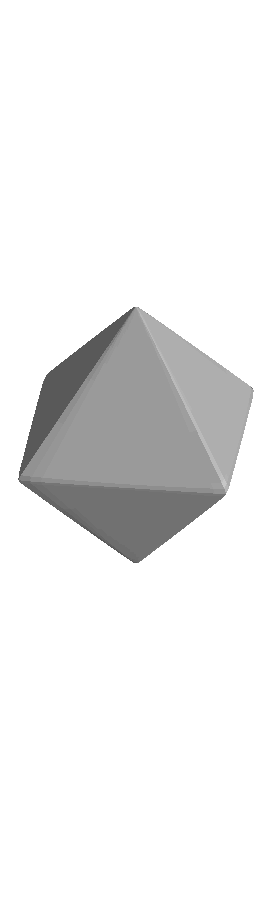} 
\includegraphics[angle=0,totalheight=\lheight]{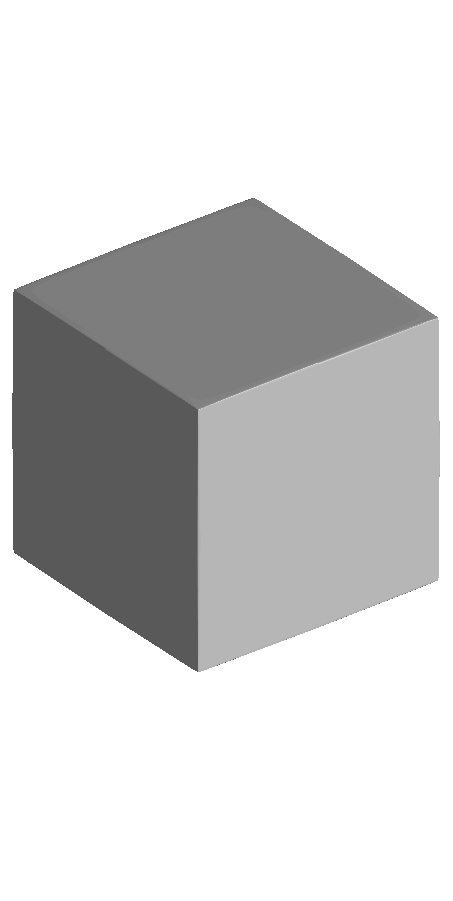} \qquad\quad
\includegraphics[angle=0,totalheight=\lheight]{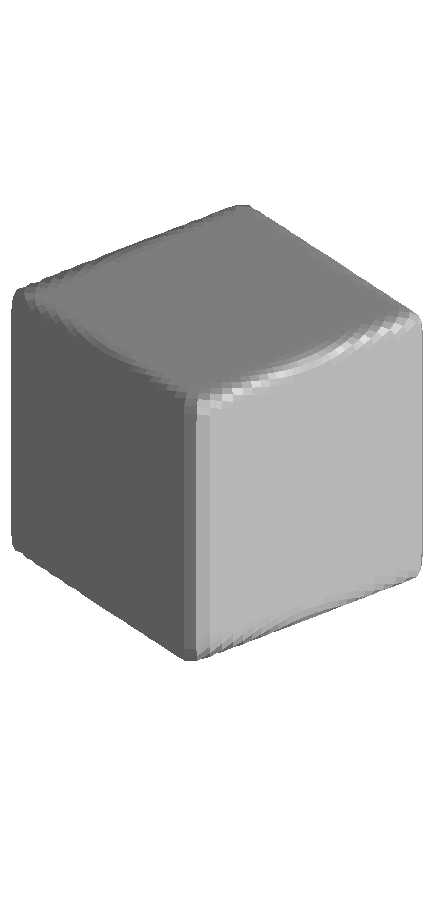} \
\includegraphics[angle=0,totalheight=\lheight]{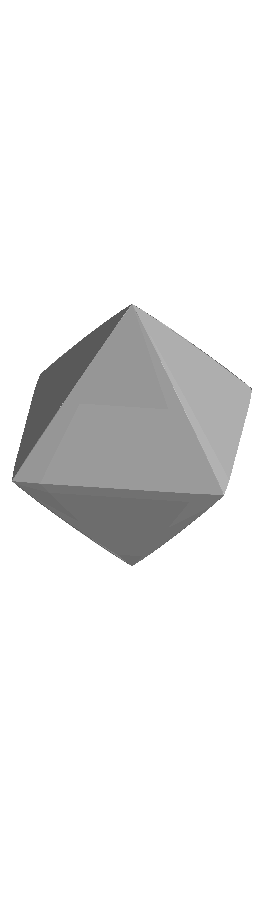} \qquad
\includegraphics[angle=0,totalheight=\lheight]{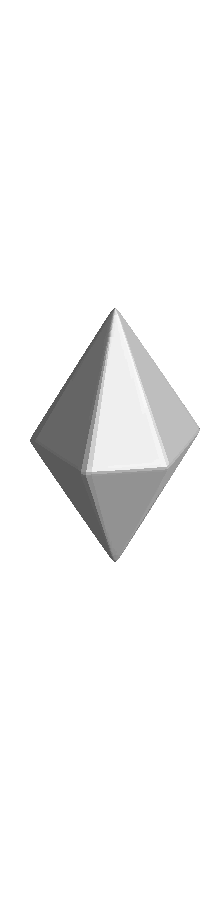} 
\includegraphics[angle=0,totalheight=\lheight]{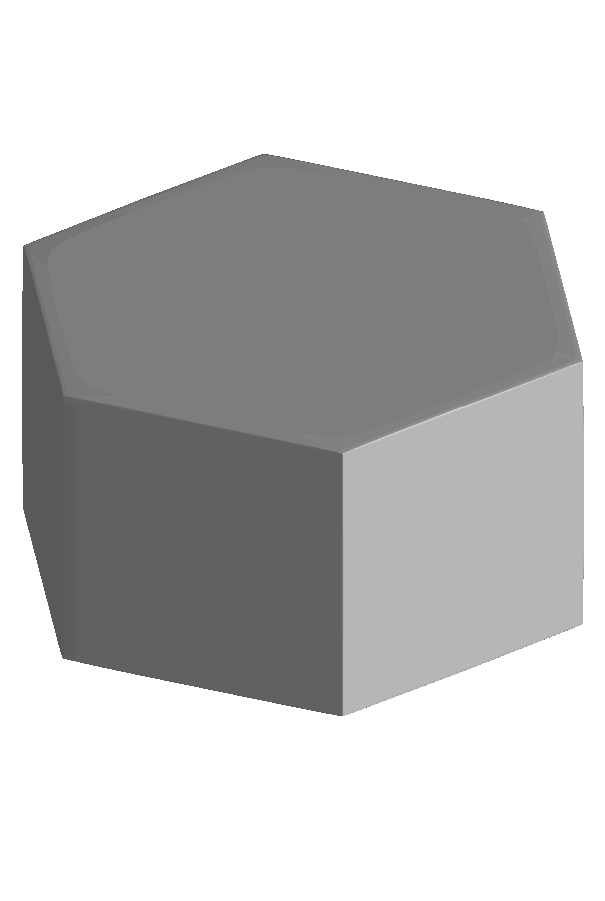}
\caption{Frank diagram and Wulff shape in $\bR^3$ for a
regularized $l^1$-norm, 
$\gamma(\vec p) = \sum_{\ell=1}^3\,[ \epsilon^2\,|\vec p|^2+
p_\ell^2(1-\epsilon^2)]^\frac12$, $\epsilon = 0.01$, left,
a cubic anisotropy, 
$\gamma(\vec p) = [\sum_{\ell=1}^3\,[ \epsilon^2\,|\vec p|^2+
p_\ell^2(1-\epsilon^2)]^\frac r2]^\frac1r$, $\epsilon = 0.01$, $r =
30$, middle, and a hexagonal anisotropy, 
$\gamma(\vec p) = \sum_{\ell=1}^4\,[ \vec p\,.\,
\protect\mat R_\ell^{\protect\transT}\,
\diag(1, \epsilon^2, \epsilon^2)\,\protect\mat R_\ell^{}\,\vec p ]^\frac12$, 
$\epsilon = 0.01$, right. Here $\protect\mat R_\ell^{}$, 
$\ell = 1,\ldots,4$ are suitable rotation matrices.
}
\label{fig:BGNwulff}
\end{figure}%
We remark that in the planar case, $d=2$, from a modelling point of view, there
is no benefit in choosing $r>1$. Moreover, the choice $r=1$ has the advantage
that it leads to linear schemes, i.e.\ a linear system of equations needs to be
solved at each time level.

Most of the surface calculus discussed in Section~\ref{sec:geometry} 
can be repeated in the
context of the energies discussed above. We obtain for example, see
\citet[Theorem~2.1]{ani3d}, that for a $C^2$-hypersurface $\Gamma$ it holds
that
\begin{equation}
\label{varkappaBGN}
\varkappa_\gamma \,\vec\nu=
\sum_{\ell=1}^L \gamma_\ell(\vec\nu)\,\mat \tG_\ell\,
\nabs^{\tG_\ell}\,.\left[
\left[\frac{\gamma_\ell(\vec\nu)}{\gamma(\vec\nu)}\right]^{r-1}
\nabs^{\tG_\ell}\,\vec{\rm id}\right]
\qquad\text{on }\Gamma\,,
\end{equation}
where $\mat \tG_\ell = \left[\det \mat G_\ell\right]^{\frac1{d-1}} 
\mat G_\ell^{-1}$, $\ell=1,\ldots,L$, and
$\nabs^{\tG_\ell}$, $\left(\nabs^{\tG_\ell}\,.\right)$ are defined as follows. 
Let $\Gamma \subset \bR^d$ be a $C^1$-hypersurface and let $\vec p \in \Gamma$.
Let $\{\vec\tau_1,\ldots,\vec\tau_{d-1}\}$
be an orthonormal basis of $\tanspace_{\vec p}\,\Gamma$ with respect to the
inner product $(\cdot,\cdot)_{\tG_\ell}$ on $\bR^d$ induced by $\mat\tG_\ell$, 
recall (\ref{eq:anisopro}).
Let $f:\Gamma\to\bR$, $\vec f :\Gamma\to\bR^d$ be $C^1$-functions. 
The anisotropic surface gradient of $f$, 
the anisotropic surface divergence of $\vec f$
and the anisotropic surface Jacobian of $\vec f$
at the point $\vec p \in \Gamma$ are defined as 
\begin{subequations} \label{eq:calgamma}
\begin{align}
(\nabs^{\tG_\ell}\, f) (\vec p) &= \sum_{i=1}^{d-1}
(\partial_{\vec\tau_i}\,f)(\vec p)\,\vec\tau_i\,,
\label{eq:nabsgamma} \\
(\nabs^{\tG_\ell}\,.\, \vec f) (\vec p) &= 
\sum_{i=1}^{d-1} \mat\tG_\ell\, \vec\tau_i \,.\,
(\partial_{\vec\tau_i}\,\vec f)(\vec p)\,,
\label{eq:divsgamma} \\
(\nabs^{\tG_\ell}\, \vec f) (\vec p) &= \sum_{i=1}^{d-1}
(\partial_{\vec\tau_i}\,\vec f)(\vec p)\otimes \mat\tG_\ell\,\vec\tau_i\,.
\label{eq:nabsgamma_vec}
\end{align}
\end{subequations}
The definitions (\ref{eq:calgamma}) are anisotropic
versions of Definition \ref{def:2.5}\ref{item:def2.5ii}--\ref{item:def2.5v}
in the case $n=d-1$. If, in addition, $\vec g \in C^1(\Gamma)$, we also define 
\begin{equation*}
\left(\nabs^{\tG_\ell}\,\vec f, \nabs^{\tG_\ell}\,\vec g\right)_{\tG_\ell} 
(\vec p) =
\sum_{i=1}^{d-1} \left((\partial_{\vec\tau_i}\,\vec f)(\vec p),
(\partial_{\vec\tau_i}\,\vec g)(\vec p) \right)_{\tG_\ell}\,.
\end{equation*}

Using a generalised divergence theorem, see \citet[Lemma~2.8]{ani3d},
we can obtain the following weak formulations of the anisotropic flows 
\eqref{eq:aniflows} in the case that the anisotropic energy density $\gamma$
is of the form \eqref{eq:g}. 
Given a closed hypersurface $\Gamma(0)$, find an evolving hypersurface 
$(\Gamma(t))_{t\in[0,T]}$
with a global parameterization and induced velocity field
$\vec{\mathcal{V}}$, and $\varkappa_\gamma \in L^2(\GT)$ as follows. 
For almost all $t \in (0,T)$, find
$\vec{\mathcal V}(\cdot,t)\in [L^2(\Gamma(t))]^d$
and $\varkappa_\gamma(\cdot,t)\in L^2(\Gamma(t))$,
respectively $\Wt$, such that 
\begin{subequations} \label{eq:waBGN}
\begin{align}
\left\langle \vec{\mathcal{V}},\chi\,\vec\nu  \right\rangle_{\Gamma(t)} =
\begin{cases} \displaystyle
\left\langle \beta(\vec\nu)\,\varkappa_\gamma, \chi \right\rangle_{\Gamma(t)}
\qquad &\forall\ \chi \in L^2(\Gamma(t))\\
\left\langle \beta(\vec\nu)\,\nabs\,\varkappa_\gamma,
\nabs\,\chi \right\rangle_{\Gamma(t)} 
\qquad &\forall\ \chi \in \Wt
\end{cases} ,\label{eq:wa1} \\
\left\langle \varkappa_\gamma \,\vec\nu,\vec\eta \right\rangle_{\Gamma(t)}
+ \left\langle 
\nabs^{\tG}\,\vec{\rm id} , \nabs^{\tG}\,\vec\eta 
\right\rangle_{\Gamma(t),\gamma}
=0 \quad\forall\ \vec\eta \in \Vt\,,
\label{eq:wa2BGN}
\end{align}
\end{subequations}
where 
\begin{equation} \label{eq:ipGG}
\left\langle \nabs^{\tG}\,\vec\chi, \nabs^{\tG}\,\vec\eta
\right\rangle_{\Gamma(t),\gamma}
=
\sum_{\ell=1}^L \int_{\Gamma(t)} \left[
\frac{\gamma_\ell(\vec\nu)}{\gamma(\vec\nu)}\right]^{r-1}
(\nabs^{\tG_\ell}\,\vec\chi, \nabs^{\tG_\ell}\,\vec\eta)_{\tG_\ell}\,
\gamma_\ell(\vec\nu)\dH{d-1}
\end{equation} 
for all $\vec\chi,\vec\eta \in \Vt$.
In the above weak formulations only derivatives up to first
order appear, and hence once again the equations can be discretized with
the help of continuous piecewise linear finite elements.

\subsection{Finite element approximation}
The weak formulations (\ref{eq:waBGN}) 
can be used to formulate finite element approximations
of \eqref{eq:aniflows}.
Discretizing the velocity and the geometry,
similarly as in (\ref{eq:MCIP}) for isotropic mean curvature flow, 
we recall the following finite element approximations from
\cite{triplejANI,ani3d}. 

Let the closed polyhedral hypersurface $\Gamma^0$ be an approximation of
$\Gamma(0)$. Then, for $m=0,\ldots,M-1$, find 
$\Gamma^{m+1} = \vec X^{m+1}(\Gamma^m)$, for $\vec X^{m+1} \in \Vhm$, 
with unit normal $\vec\nu^{m+1}$, and $\kappa^{m+1}_\gamma \in \Whm$ such that 
\begin{subequations} \label{eq:aniBGN}
\begin{align}
& \left\langle\frac{\vec X^{m+1}-\vec\id}{\ttau_m}, \chi\,\vec\nu^m 
\right\rangle^h_{\Gamma^m}
= \begin{cases} 
\left\langle \beta(\vec\nu^m)\,\kappa^{m+1}_\gamma, \chi 
\right\rangle_{\Gamma^m}^h \\
\left\langle \beta(\vec\nu^m)\,\nabs\,\kappa^{m+1}_\gamma, 
\nabs\,\chi \right\rangle_{\Gamma^m}
\end{cases}
\ \forall\ \chi \in \Whm\,, 
\label{eq:discaniso1} \\
& \left\langle\kappa^{m+1}_\gamma\,\vec\nu^m, \vec\eta
\right\rangle^h_{\Gamma^m}
+ \left\langle \nabs^{\tG_{\ell}}\,\vec X^{m+1},
\nabs^{\tG_{\ell}}\,\vec\eta \right\rangle_{\Gamma^m,\gamma}
=0 \quad\forall\ \vec\eta \in \Vhm \,.
\label{eq:discaniso2}
\end{align}
\end{subequations}
In the above, we have used the notation
\begin{align*} 
& \left\langle \nabs^{\tG_{\ell}}\,\vec X^{m+1},
\nabs^{\tG_{\ell}}\,\vec\eta \right\rangle_{\Gamma^m,\gamma} 
\nonumber \\ & \
=\sum_{\ell=1}^L \int_{\Gamma^m} \left[
\frac{\gamma_\ell(\vec\nu^{m+1} \circ \vec X^{m+1})}
{\gamma(\vec\nu^{m+1}\circ \vec X^{m+1})} \right]^{r-1}
(\nabs^{\tG_{\ell}}\,\vec X^{m+1},\nabs^{\tG_{\ell}}\,\vec\eta)_{\tG_\ell}
\,\gamma_{\ell}(\vec\nu^m) \dH{d-1} \,,
\end{align*}
i.e.\ a discrete analogue of (\ref{eq:ipGG}).

\begin{rem}[Curves in the plane] \label{rem:ipGGh2d}
On using the notation from {\rm \S\ref{subsubsec:polycurves}}, and similarly to
\eqref{eq:MCIP2d}, the systems \eqref{eq:aniBGN} simplify considerably in the
case $d=2$, i.e.\ for the evolution of curves in the plane, recall also
{\rm \citet[(3.5)]{ani3d}}. 
In particular, we can reformulate \eqref{eq:aniBGN} as follows.
Let $\vec X^0 \in \VhI$ be such that $\Gamma^0 = \vec X^0(\bI)$ is a polygonal
approximation of $\Gamma(0)$. Then, for $m=0,\ldots,M-1$, find 
$(\vec X^{m+1},\kappa^{m+1}_\gamma) \in\VhI\times\WhI$ such that
\begin{subequations} 
\begin{align}\label{eq:aniBGN2da}
& - \left\langle \frac{\vec X^{m+1}-\vec X^m}{\ttau_m} ,
  \chi\,[\vec X^{m}_\rho]^\perp \right\rangle^h_{\bI} 
= \begin{cases}
\left\langle \beta\left(-\frac{[\vec X^{m}_\rho]^\perp}{|\vec
X^{m}_\rho|}\right)
\,\kappa^{m+1}_\gamma,
\chi\,|\vec X^{m}_\rho| \right\rangle^h_{\bI} \\
\left\langle \beta\left(-\frac{[\vec X^{m}_\rho]^\perp}{|\vec
X^{m}_\rho|}\right)
(\kappa^{m+1}_\gamma)_\rho, \chi_\rho\,|\vec X^{m}_\rho|^{-1} 
\right\rangle_{\bI}
\end{cases} \,,\\
&  - \left\langle \kappa^{m+1}_\gamma\,[\vec X^m_\rho]^\perp,\vec\eta
  \right\rangle^h_{\bI} +
\sum_{\ell=1}^L \left\langle 
\left[ \frac{\gamma_\ell([\vec X^{m+1}_\rho]^\perp)}
{\gamma([\vec X^{m+1}_\rho]^\perp)} \right]^{r-1}
\frac{\mat G_\ell\,[\vec X^{m+1}_\rho]^\perp}
{\gamma_{\ell}([\vec X^m_\rho]^\perp)},\vec\eta_\rho^\perp
\right\rangle_\bI = 0 
\label{eq:aniBGN2db}
\end{align}
\end{subequations}
for all $\chi \in \WhI$ and $\vec\eta \in \VhI$.
\end{rem}

For $r=1$ in \eqref{eq:g} the systems (\ref{eq:aniBGN}) do not depend on
$\vec\nu^{m+1}$, and so they reduce to linear
systems for the unknowns $\vec X^{m+1}$ and $\kappa^{m+1}_\gamma$. 
For these discrete systems we have the
following existence and uniqueness result, which can be shown by
generalizing the proofs of the corresponding isotropic cases, 
see Theorems~\ref{thm:uniqueMC} and \ref{thm:uniqueSD}.

\begin{thm}\label{thm:uniqueMCani}
Let $\Gamma^m$ satisfy {\rm Assumption~\ref{ass:A}\ref{item:assA1}}, and let
$\gamma$ be of the form \eqref{eq:g} with $r=1$.
Then there exists a unique solution
$(\vec X^{m+1},\kappa^{m+1}_\gamma) \in \Vhm\times \Whm$
to the systems \eqref{eq:aniBGN}.
\end{thm}
\begin{proof}
The two results can be shown as in the proofs of
Theorem~\ref{thm:uniqueMC} and Theorem~\ref{thm:uniqueSD}, see
also \citet[Theorem~3.1]{ani3d}.
\end{proof}

\begin{rem} 
For the nonlinear discretizations with $r>1$, neither
existence nor uniqueness results are known. However, in practice 
there are no difficulties in finding solutions to the nonlinear systems
{\rm(\ref{eq:aniBGN})}, and the employed iterative solvers always
converged.
\end{rem}

The main property of the schemes \eqref{eq:aniBGN} are that they can be 
shown to be stable. This enables one to compute solutions for strongly
anisotropic, nearly crystalline anisotropies. Something which is
difficult with other schemes in the literature. The main insight is
that a local inequality for the anisotropic energy is true, which is
similar to the isotropic version stated in Lemma~\ref{lem:stab3d} for $d=3$. 
In fact, we have the following result.

\begin{lem}\label{lem:anistab3d}
Let $d=3$ and let $\Gamma^h =\bigcup^J_{j=1}\overline{\sigma}_j$ be a 
polyhedral hypersurface, with unit normal $\vec\nu^h$, in $\bR^d$.
Then we have for $j=1,\ldots,J$ and $\ell=1,\ldots,L$ that
\begin{equation*} 
\tfrac12\,\int_{\sigma_j} \gamma_\ell(\vec\nu^h)\,
|\nabs^{\tG_\ell}\,\vec X|^2_{\tG_\ell} \dH{2} \geq  
\int_{\vec X(\sigma_j)} \gamma_\ell(\vec\nu^h_{\vec X}) \dH{2}
\quad\forall\ \vec X\in \Vhh\,,
\end{equation*}
with equality for $\vec X= \vec\id_{\mid_{\Gamma^h}} \in \Vhh$.
Here $\vec\nu^h_{\vec X}$ is the unit normal on the polyhedral hypersurface
$\vec X(\Gamma^h)$ 
and $|\cdot|^2_{\tG_\ell}=(\cdot,\cdot)_{\tG_\ell}$.
\end{lem}
\begin{proof} 
See \citet[Lemma~3.1]{ani3d}.
\end{proof}

\begin{lem} \label{lem:anistab2d3d}
Let $\gamma$ be of the form \eqref{eq:g} and let $d=2$ or $d=3$.
Let $\Gamma^h=\bigcup^J_{j=1}\overline{\sigma_j}$ be a 
polyhedral hypersurface, with unit normal $\vec\nu^h$, in $\bR^d$.
Let $\gamma$ be of the form \eqref{eq:g} 
and let $\vec X \in \Vhh$. 
Then it holds that
\begin{equation*} 
\left\langle\nabs^\tG\,\vec X, \nabs^\tG\,(\vec X-\vec\id) 
\right\rangle_{\Gamma^h,\gamma}
\geq |\vec X(\Gamma^h)|_\gamma - |\Gamma^h|_\gamma\,,
\end{equation*}
where $|\Gamma^h|_\gamma = \left\langle 1, \gamma(\vec\nu^h) \right\rangle_{\Gamma^h}$.
\end{lem}
\begin{proof}
The proof for $d=3$ hinges on Lemma~\ref{lem:anistab3d} and
can be found in \citet[Theorem~3.2]{ani3d}.
On recalling Remark~\ref{rem:ipGGh2d},
the result for $d=2$, and $r=1$, can be shown by using ideas similar to 
(\ref{eq:stab2d}) in the proof of Lemma~\ref{lem:stab2d3d}, see 
\citet[Theorem~2.5]{triplejANI}.
It can easily be extended to the case $r \geq 1$,
on using the techniques in {\rm \citet[Theorem~3.2]{ani3d}}.
\end{proof}

\begin{thm} \label{thm:stab3dani}
Let $\gamma$ be of the form \eqref{eq:g} and let $d=2$ or $d=3$.
Let $(\vec X^{m+1},\kappa_\gamma^{m+1}) \in \Vhm\times\Whm$ be a solution of 
{\rm (\ref{eq:aniBGN})}. Then it holds that
\begin{equation}
|\Gamma^{m+1}|_\gamma + \ttau_m
\begin{cases}
\left\langle \beta(\vec\nu^m)\,\kappa_\gamma^{m+1},\kappa_\gamma^{m+1}
\right\rangle_{\Gamma^m}^h   \\
\left\langle\beta(\vec\nu^m)\,\nabs\,\kappa_\gamma^{m+1}, 
\nabs\,\kappa_\gamma^{m+1} \right\rangle_{\Gamma^m} 
\end{cases}
\leq |\Gamma^m|_\gamma \,, \label{eq:stab00}
\end{equation}
where we recall {\rm (\ref{eq:11})}.
\end{thm}
\begin{proof}
Similarly to the proof of Theorem~\ref{thm:stabMC}, we can
choose $\chi = \kappa^{m+1}_\gamma\in \Whm$ in (\ref{eq:discaniso1}) and
$\vec\eta = \frac1{\ttau_m}\,(\vec X^{m+1}-\vec\id_{\mid_{\Gamma^m}})\in\Vhm$
in (\ref{eq:discaniso2}), and then combine with Lemma~\ref{lem:anistab2d3d},
in order to obtain the desired results.
\end{proof}

The inequalities \eqref{eq:stab00} are the discrete analogues of
\eqref{eq:anisoMCin} and \eqref{eq:anisoSDin}, respectively,
and anisotropic analogues of Theorems \ref{thm:stabMC} and \ref{thm:stabSD}, 
respectively. 

\subsection{Solution method and discrete systems} \label{subsec:anisolve}
In comparison to the isotropic case the assembly of the matrices is
only slightly more complicated. In addition to the matrices in
\S\ref{subsec:3.3}, we define $M_{\Gamma^m,\beta},$ $A_{\Gamma^m,\beta} 
\in \bR^{K\times K}$ with entries
\[
\left[M_{\Gamma^m,\beta}\right]_{kl} 
=\left\langle \beta(\vec\nu^m)\,\phi^{\Gamma^m}_k,
\phi_l^{\Gamma^m} \right\rangle_{\Gamma^m}^h, \
\left[A_{\Gamma^m,\beta}\right]_{kl} =
\left\langle\beta(\vec\nu^m)\,\nabs\,\phi^{\Gamma^m}_k,\nabs\,\phi_l^{\Gamma^m}
\right\rangle_{\Gamma^m}.
\]
Moreover, given a $\vec\vvv \in \Vchm$, we introduce 
$\mat {A}_{\Gamma^m,\gamma}(\vec\vvv) \in (\bR^{d\times d})^{K\times K}$ 
with entries 
\[
\left[\mat {A}_{\Gamma^m,\gamma}(\vec\vvv) \right]_{kl}
 =  \sum_{\ell=1}^L \left\langle \left[
\frac{\gamma_\ell(\vec\vvv)}{\gamma(\vec\vvv)} \right]^{r-1}
\gamma_{\ell}(\vec\nu^m)\,
\nabs^{\tG_{\ell}}\,\phi^{\Gamma^m}_k,\nabs^{\tG_{\ell}}
\,\phi_l^{\Gamma^m} \right\rangle_{\Gamma^m} \mat\tG_{\ell}\,, 
\]
where we have noted (\ref{eq:ipGG}) and (\ref{eq:anisopro}). 
Assembling these matrices is straightforward, since
assembling e.g.\ $\left(\left\langle\nabs^{\tG_\ell}\,\phi^{\Gamma^m}_k,
\nabs^{\tG_\ell}\,\phi_l^{\Gamma^m}\right\rangle_{\Gamma^m}\right)_{k,l=1}^K$ 
is very similar to assembling $A_{\Gamma^m}$ in (\ref{eq:mat0}).
Using the notation from \eqref{eq:MCvec}, we can then formulate 
\eqref{eq:aniBGN} as:
Find $(\delta\vec X^{m+1},\kappa^{m+1}_\gamma)$ $\in (\bR^d)^K \times \bR^K$ 
such that
\begin{equation}
\begin{pmatrix} \ttau_m
\begin{cases} M_{\Gamma^m,\beta} \\ A_{\Gamma^m,\beta} \end{cases} 
& -\vec N^\transT_{\Gamma^m} \\ 
\vec N_{\Gamma^m} & \mat A_{\Gamma^m,\gamma}(\vec\vvv^{m+1})
\end{pmatrix}
\begin{pmatrix} \kappa^{m+1}_\gamma \\ \delta\vec X^{m+1}
\end{pmatrix}
= \begin{pmatrix} 0 \\ -\mat A_{\Gamma^m,\gamma}(\vec\vvv^{m+1})
\,\vec X^m \end{pmatrix},
\label{eq:anivec}
\end{equation}
where we have used the notation $\vec\vvv^{m+1} = \vec\nu^{m+1} \circ \vec
X^{m+1}$, and where we recall that
$\delta\vec X^{m+1}$ is the vector of coefficients with respect to the 
standard basis for $\vec X^{m+1}- \vec\id_{\mid_{\Gamma^m}} \in \Vhm$.

If $\gamma$ is of the form \eqref{eq:g} with $r=1$, then the system 
\eqref{eq:anivec} does not depend on $\vec\vvv^{m+1}$ and so is linear. 
Hence it can be solved analogously to \eqref{eq:MCvec}, i.e.\ either with a 
sparse direct solver or with the help of a Schur complement approach and
a preconditioned conjugate gradient solver, see \citet[\S4]{ani3d} for
details. 

If $r > 1$, on the other hand, then the nonlinear system \eqref{eq:anivec} 
can be solved with a lagged fixed-point type iteration, where each iterate is 
obtained by solving a linear system of the form \eqref{eq:anivec}, with 
$\vec\vvv^{m+1}$ replaced by a given quantity $\vec\vvv^{m+1,i}$.
We note that in practice such an iteration always converges, and
refer to \citet[\S4]{ani3d} for further details.

\subsection{Volume conservation for semidiscrete schemes}
For the semidiscrete continuous-in-time variants of \eqref{eq:aniBGN} 
it is straightforward to prove the obvious anisotropic analogues of
Theorem~\ref{thm:semidis}\ref{item:semidisstab} 
and Theorem~\ref{thm:SDSD}\ref{item:SDSDstab}, as well as 
Theorem~\ref{thm:SDSD}\ref{item:SDSDvol} in the case of anisotropic surface
diffusion. 

\begin{rem}[Tangential motion] \label{rem:aniTM}
In the anisotropic situation 
the tangential motion induced by the semidiscrete analogue of
\eqref{eq:discaniso2} will no longer lead to surfaces satisfying
{\rm Definition~\ref{def:conformal}}. In particular, 
{\rm Theorem~\ref{thm:semidis}\ref{item:semidisTM}} 
will in general not hold.
In the planar case, and if $\gamma$ is of the form \eqref{eq:g} with
$L=1$ and $r=1$, then equidistribution with respect to $\gamma$ can be shown,
see {\rm \citet[Remark~2.7]{triplejANI}}. 
However, the fully discrete schemes \eqref{eq:aniBGN}, for any $\gamma$ of the
form \eqref{eq:g} and for $d=2$ and $d=3$, exhibit good quality
meshes in practice, without coalescence that could lead to a breakdown of the 
schemes.
See, for example, the numerical simulations in
{\rm \citet[\S5]{ani3d}}.
\end{rem}

\subsection{Alternative numerical approaches}
\label{subsec:alternativeANI}
For the case $d=2$ a numerical scheme for anisotropic mean curvature flow
has been proposed in \cite{Dziuk99}, and an error estimate 
for a semidiscrete continuous-in-time variant is shown, 
on assuming that the approximated solution is sufficiently smooth.

Moreover, and still for $d=2$, fully implicit fully discrete schemes,
similar to \eqref{eq:21617}, for
anisotropic evolution equations are considered in \cite{fdfi}.

The extension of the methods discussed in this section to anisotropic 
evolution laws of surface clusters and surfaces with boundary has been
discussed in \cite{triplejANI,clust3d,ejam3d}.

For parametric methods for anisotropic curvature driven flows in higher
codimension we refer to \cite{Pozzi07,Pozzi08,curves3d}.

We refer also to \cite{DeckelnickD99}, where, on assuming a sufficiently
smooth solution, an error analysis is presented
for a semidiscrete finite element approximation 
by continuous piecewise linears of a graph formulation
of anisotropic mean curvature flow.
We refer to \cite{DeckelnickDE05b} for the corresponding error analysis
for a fully discrete finite element approximation 
of a graph formulation of anisotropic surface diffusion.

We also refer to \cite{DeckelnickDE05} for the approximation of 
anisotropic mean curvature flow in the context of level set and 
phase field methods. 
For the former we mention \cite{ClarenzHRVW05,BurgerHSV07},
while for the latter we refer to 
\cite{CaginalpL87,GarckeSN99,Benes03,GraserKS13}
and the present authors' work in \cite{eck,vch} on stable approximations of 
anisotropic mean curvature flow and anisotropic surface diffusion.

Finally we mention that for strongly
anisotropic surface energies, some authors propose a curvature energy
regularization, which leads to higher order flows
with similarities to Willmore flow, 
see \cite{Burger05,HausserV05,TorabiLVW09}.

\section{Coupling bulk equations to geometric equations on the
surface and applications to crystal growth}
We now consider how curvature driven interface evolutions can involve
quantities defined in the bulk regions surrounding the interface.

\subsection{The Mullins--Sekerka problem}
Let us start with one of the simplest problems that couple geometric equations
on the interface to equations in the bulk. Let
$\Omega\subset\bR^d$, $d\geq2$, 
be a domain with a Lipschitz boundary $\partial\Omega$, 
which is separated by an
interface into two different phases. The interface is at each time $t$
assumed to be a closed $C^2$--hypersurface $\Gamma(t)\subset\Omega$. 
We adopt the notation in \S\ref{subsec:evolve}, recall Figure~\ref{fig:sketch},
and assume that two phases occupy regions $\Omega_-(t)$ and
$\Omega_+(t) = \Omega\setminus \overline{\Omega_-(t)}$
with $\Gamma(t)=\partial\Omega_-(t)$,
and with $\vec\nu$ denoting the outer unit normal to $\Omega_-(t)$
on $\Gamma(t)$. 
Then the Mullins--Sekerka problem is given as follows. 
Given $\Gamma(0)$, find $u:\Omega\times [0,T]\to\bR$ and the evolving interface
$(\Gamma(t))_{t\in[0,T]}$ such that for all $t\in(0,T]$ the following
conditions hold
\begin{subequations} \label{eq:MS}
\begin{alignat}{4}
-\Delta\,u &=0 \quad && \text{in } \Omega_-(t)\,,\qquad &
-\Delta\, u &=0 \quad &&\text{in } \Omega_+(t)\,,\label{eq:MS1}\\
u &= \varkappa \quad &&\text{on } \Gamma(t)\,,
& \left[\partial_{\vec\nu}\,u\right]^+_-
 &= -\mathcal{V}\quad && \text{on } \Gamma(t)\,,\label{eq:MS2}\\
\partial_{\vec\nu_\Omega}\,u &= 0 \quad && \text{on } 
\partial\Omega\,,\label{eq:MS3}
\end{alignat}
\end{subequations}
where we have recalled the notation \eqref{eq:dnufjump} and $\vec\nu_{\Omega}$
is the outer unit normal to $\Omega$ on $\partial\Omega$.

\subsubsection{Weak formulation of the Mullins--Sekerka problem}
A finite element approximation of the Mullins--Sekerka problem needs
a suitable weak formulation, which is given as follows.
Given a closed hypersurface $\Gamma(0) \subset \Omega$, 
we seek an evolving hypersurface $(\Gamma(t))_{t\in[0,T]}$
that separates $\Omega$ into $\Omega_-(t)$ and $\Omega_+(t)$,
with a global parameterization and induced velocity field
$\vec{\mathcal{V}}$, and $\varkappa \in L^2(\GT)$ as well as 
$u : \Omega \times [0,T] \to \bR$ 
as follows. 
For almost all $t \in (0,T)$, find
$(u(\cdot,t),\vec {\mathcal V}(\cdot,t),\varkappa(\cdot,t))\in H^1(\Omega) 
\times [L^2(\Gamma(t))]^d \times L^2(\Gamma(t))$ such that 
\begin{subequations} \label{eq:WMS}
\begin{align} \label{eq:WMS1}
&\left(\nabla\, u,\nabla\,\phi\right) =
\left\langle \vec{\mathcal{V}},\phi\,\vec\nu \right\rangle_{\Gamma(t)} 
\quad\forall\ \phi\in H^1(\Omega)\,,\\
&\left\langle u,\chi \right\rangle_{\Gamma(t)} 
= \left\langle \varkappa, \chi \right\rangle_{\Gamma(t)}
\quad\forall\ \chi\in L^2(\Gamma(t))\label{eq:WMS2}\,,\\
&\left\langle \varkappa\,\vec\nu,\vec\eta \right\rangle_{\Gamma(t)} 
+ \left\langle \nabs\,\vec\id , \nabs\,\vec\eta 
\right\rangle_{\Gamma(t)}=0 \qquad
\forall\ \vec\eta\in \Vt\,. \label{eq:WMS3}
\end{align}
\end{subequations}
Here $(\cdot,\cdot) = \left\langle\cdot,\cdot\right\rangle_\Omega$ 
is the $L^2$--inner product on $\Omega$. It is easy to show that a sufficiently
smooth solution to \eqref{eq:WMS} solves \eqref{eq:MS}. To this end, 
we observe that it follows from \eqref{eq:WMS1} and Remark~\ref{rem:dnufjump},
on using a density argument, that
\begin{align}
\left\langle \mathcal{V},\phi\right\rangle_{\Gamma(t)} & 
= \left\langle \vec{\mathcal{V}},\phi\,\vec\nu \right\rangle_{\Gamma(t)} 
= \left(\nabla\, u,\nabla\,\phi\right) 
 \nonumber \\ &
= - \int_{\Omega_-(t) \cup \Omega_+(t)} \phi\,\Delta u\dL{d}
+ \left\langle \partial_{\vec\nu_\Omega}\,u,\phi 
\right\rangle_{\partial\Omega} 
- \left\langle 
   \left[\partial_{\vec\nu}\,u\right]^+_-,\phi \right\rangle_{\Gamma(t)} 
 \nonumber \\ & \hspace{6cm}
\quad\forall\ \phi \in H^1(\Omega)\,. \label{eq:divjump}
\end{align}
Now the fundamental lemma of the calculus of variations, together
with \eqref{eq:divjump}, yields that \eqref{eq:MS1}, \eqref{eq:MS3}
and the second condition in \eqref{eq:MS2} hold.
Similarly, it follows from \eqref{eq:WMS2} and \eqref{eq:WMS3},
on recalling Remark~\ref{rem:ibp}\ref{item:weakvarkappa},
that the first condition in \eqref{eq:MS2} is also satisfied.

\begin{rem} \label{rem:WMS}
For a sufficiently smooth weak solution of \eqref{eq:WMS}, we can formally
prove the following results.
\begin{enumerate}
\item \label{item:WMSi}
The Mullins--Sekerka problem decreases the surface area of the interface. 
This follows from the transport theorem, 
{\rm Theorem~\ref{thm:trans}}, \eqref{eq:WMS2} with $\chi = \mathcal{V}$ and
\eqref{eq:WMS1} with $\phi=u$, as 
\[
\ddt\,|\Gamma(t)|
= - \left\langle \varkappa,\mathcal{V} \right\rangle_{\Gamma(t)} 
= - \left\langle u,\mathcal{V} \right\rangle_{\Gamma(t)}
= - \left\langle \vec{\mathcal{V}},u\,\vec\nu \right\rangle_{\Gamma(t)}
= - \left|\nabla\,u\right|_\Omega^2 \leq 0\,.
\]
\item \label{item:WMSii}
The Mullins--Sekerka problem preserves the volume $\mathcal{L}^d(\Omega_-(t))$,
and \arxivyesno{}{\linebreak}%
hence $\mathcal{L}^d(\Omega_+(t))$,
as the transport theorem, {\rm Theorem~\ref{thm:transvol}}, yields that
\[
\ddt\,\mathcal{L}^d(\Omega_-(t))
= \left\langle\mathcal{V},1 \right\rangle_{\Gamma(t)}
= \left\langle \vec{\mathcal{V}},\vec\nu \right\rangle_{\Gamma(t)}
= \left(\nabla\, u, \nabla\, 1\right) = 0\,,
\]
where we have chosen $\phi= 1$ in \eqref{eq:WMS1}.
\end{enumerate}
\end{rem}

\begin{rem} \label{rem:FWMS}
Given the hypersurface $\Gamma(t)\subset\Omega$, and hence its normal
$\vec\nu(t)$ and curvature
$\varkappa(t)$, we can solve \eqref{eq:MS1}, \eqref{eq:MS3}
and the first equation in \eqref{eq:MS2} to obtain a bulk function
$u(t)$ that depends  on $\varkappa(t)$. Defining now
\begin{equation*}
\mathfrak{F}(\varkappa) = -\left[\partial_{\vec\nu}\,u\right]^+_- ,
\end{equation*}
we obtain that
\[
\left\langle \mathfrak{F}(\varkappa),\varkappa \right\rangle_{\Gamma(t)}
= - \left\langle \left[\partial_{\vec\nu}\,u\right]^+_-,\varkappa 
  \right\rangle_{\Gamma(t)} 
=- \left\langle \left[\partial_{\vec\nu}\,u\right]^+_-, u 
  \right\rangle_{\Gamma(t)} 
= \left|\nabla\,u\right|_\Omega^2 \geq 0\,,
\]
where we have observed the first equation in \eqref{eq:MS2}, 
and the choice $\phi=u$ in \eqref{eq:divjump}, on noting \eqref{eq:MS1}
and \eqref{eq:MS3}. 
Therefore, the Mullins--Sekerka problem \eqref{eq:MS} can be written 
in the general form \eqref{eq:Fgeomev} satisfying \eqref{eq:Fineq}.
\end{rem}

\subsubsection{An unfitted finite element approximation of the 
Mullins--Sekerka problem}
In addition to the discretization of the evolving hypersurface, we need
a discretization of the domain $\Omega$. For simplicity we assume that
$\Omega$ is a polyhedral domain. Although a generalization to a curved
domain $\Omega$ is possible using suitable boundary approximations, see
\cite{Ciarlet78}. For all $m\geq0$, let $\mathcal{T}^m$ be a
regular partitioning of $\Omega$ into disjoint open simplices, so that
$\overline{\Omega}=\cup_{\sigmaO\in\mathcal{T}^m}\overline{\sigmaO}$, 
see \cite{Ciarlet78} for a definition of regular partitioning
and further details about finite elements.
Note that we allow for time-dependent bulk triangulations.
Associated with $\mathcal{T}^m$ is the finite element space 
\begin{equation} \label{eq:Sm}
S^m = \left\{\chi \in C(\overline{\Omega}) : \chi_{\mid_{\sigmaO}}
\text{ is affine } \forall\ \sigmaO \in \mathcal{T}^m\right\} 
\subset H^1(\Omega)\,. 
\end{equation}
Let $K_{S^m}$ be the number of nodes of $\mathcal{T}^m$ and  
let $\{\vec p^m_{k}\}_{k=1}^{K_{S^m}}$ be the coordinates of these nodes.
Let $\{\phi_{k}^{S^m}\}_{k=1}^{K_{S^m}}$ be the standard basis functions 
for $S^m$.
We introduce $I^m:C(\overline{\Omega})\to S^m$, the interpolation
operator, such that $(I^m \eta)(\vec p_k^m)= \eta(\vec p_k^m)$ 
for $k=1,\ldots, K_{S^m}$. 
A discrete semi-inner product on $C(\overline{\Omega})$ is then defined by 
\begin{equation*} 
\left(\eta_1,\eta_2\right)^h_m = \left(I^m[\eta_1\,\eta_2],1\right) 
\quad\forall\ \eta_1,\,\eta_2 \in C(\overline{\Omega})\,,
\end{equation*} 
with the induced seminorm given by 
$|\eta|_{\Omega,m}  = [\,(\eta,\eta)^h_m\,]^{\frac{1}{2}}$
for $\eta \in C(\overline{\Omega})$.

We can now introduce the finite element approximation of the
Mullins--Sekerka problem from \cite{dendritic},  
based on the weak formulation \eqref{eq:WMS}, as follows. 
Let the closed polyhedral hypersurface $\Gamma^0$ be an approximation of
$\Gamma(0)$ and recall the time interval partitioning \eqref{eq:ttaum}. 
Then, for $m=0,\ldots,M-1$, find 
$(U^{m+1},\vec X^{m+1},\kappa^{m+1}) \in S^m \times \Vhm \times \Whm$ such that
\begin{subequations}
\label{eq:DMS}
\begin{align}\label{eq:DMS1}
&\left(\nabla\, U^{m+1}, \nabla\,\varphi\right) 
= \left\langle \pi_{\Gamma^m}\left[
\frac{\vec X^{m+1}-\vec\id}{\ttau_m} \,.\,\vec\omega^m \right], \varphi
\right\rangle_{\Gamma^m} \quad\forall\ \varphi \in S^m\,,\\
&\left\langle\kappa^{m+1},\chi\right\rangle^h_{\Gamma^m} =\left\langle
U^{m+1},\chi\right\rangle_{\Gamma^m}
\quad\forall\ \chi \in \Whm \,,\label{eq:DMS2}\\
&\left\langle \kappa^{m+1}\,\vec\nu^m,\vec\eta\right\rangle^h_{\Gamma^m} +
  \left\langle\nabs\,\vec X^{m+1},\nabs\,\vec\eta\right\rangle_{\Gamma^m}
= 0
\quad\forall\ \vec\eta \in \Vhm \label{eq:DMS3}
\end{align}
\end{subequations}
and set $\Gamma^{m+1} = \vec X^{m+1}(\Gamma^m)$. 
In the above we have recalled (\ref{eq:omegah}). 

\begin{rem}[Implementation] \label{rem:unfitted}
The approximation \eqref{eq:DMS} is called unfitted because the surface mesh is
totally independent of the bulk mesh, and is not fitted to the bulk
mesh in the sense that the surface mesh does not consist of faces of the
bulk mesh, see {\rm \cite{BarrettE82}}.
As a consequence, special quadrature rules need to be employed in order to
calculate the right hand sides in \eqref{eq:DMS1} and \eqref{eq:DMS2}. 
Here the most challenging aspect is to compute intersections 
$\sigma^m \cap \sigmaO^m$ between an arbitrary element 
$\sigma^m \subset \Gamma^m$ and an element $\sigmaO^m\in\mathcal{T}^m$
of the bulk mesh.
An algorithm that describes how these intersections can be calculated is given
in {\rm \citet[p.\ 6284]{dendritic}}, see also Figures~4 and 5 in 
{\rm \cite{dendritic}}. 
\end{rem}

Before stating our next result, we need an assumption on the compatibility
between bulk and surface mesh.

\begin{assumption} \label{ass:B}
Let $\Gamma^m$ and $\mathcal{T}^m$ be such that 
\[
\dim\left\{ \int_{\Gamma^m} \varphi\,\vec\omega^m\dH{d-1} : \varphi \in S^m
\right\} = d\,.\]
\end{assumption}

\begin{thm} \label{thm:HDMS}
Let $\Gamma^m$ and $\mathcal{T}^m$ satisfy 
{\rm Assumption~\ref{ass:B}}. 
Then there exists a unique solution $(U^{m+1}, \vec X^{m+1},\kappa^{m+1})\in
S^m\times \Vhm\times \Whm$ to \eqref{eq:DMS}.
In addition, if $d=2$ or $d=3$,
then a solution to \eqref{eq:DMS} satisfies
the stability estimate
\begin{equation*}
|\Gamma^{m+1}| + \ttau_m \left|\nabla\, U^{m+1}\right|_\Omega^2
\leq |\Gamma^m|\,.
\end{equation*}
\end{thm}
\begin{proof} 
Existence follows from uniqueness and hence we consider
the homogeneous linear system. Find $(U,\vec X,\kappa) \in
S^m\times\underline{V} (\Gamma^m)\times \Whm$ such that
\begin{subequations} \label{eq:HDMS}
\begin{align}\label{eq:HDMS1}
\ttau_m\left(\nabla\,U,\nabla\,\varphi\right) 
- \left\langle\pi_{\Gamma^m}[\vec X\,.\,\vec\omega^m],\varphi
\right\rangle_{\Gamma^m} & = 0\quad\forall\ \varphi\in S^m\,,\\
\left\langle \kappa,\chi\right\rangle^h_{\Gamma^m} - \left\langle
U,\chi\right\rangle_{\Gamma^m} &=0\quad\forall\ \chi\in
\Whm\,,\label{eq:HDMS2}\\
\left\langle\kappa\,\vec\omega^m,\vec\eta\right\rangle^h_{\Gamma^m} +
\left\langle\nabs\,\vec X, \nabs\,\vec\eta\right\rangle_{\Gamma^m} &=0
\quad\forall\ \vec\eta\in\Vhm\,,\label{eq:HDMS3}
\end{align}
\end{subequations}
where we have noted (\ref{eq:omegahnuh}) in (\ref{eq:HDMS3}).
Choosing $\varphi=U$ in \eqref{eq:HDMS1}, 
$\chi=\pi_{\Gamma^m}[\vec X\,.\,\vec\omega^m]$ in \eqref{eq:HDMS2} 
and $\vec\eta = \vec X$ in \eqref{eq:HDMS3} gives
\begin{equation*}
\ttau_m\left|\nabla\,U\right|_\Omega^2 
+ \left|\nabs\,\vec X\right|_{\Gamma^m}^2 = 0\,.
\end{equation*}
We hence obtain that $U$ is constant in $\Omega$ and that 
$\vec X$ is constant on $\Gamma^m$.
Now \eqref{eq:HDMS1} and Assumption~\ref{ass:B} yield that $\vec X = \vec 0$. 
Moreover, it follows from \eqref{eq:HDMS2} that $\kappa = U$ is constant on
$\Gamma^m$. Hence we can proceed as in the proof of Theorem~\ref{thm:uniqueSD}
to show that $\kappa = U = 0$, and so we obtain uniqueness. 

It remains to prove the stability bound. Here we choose
$\varphi= U^{m+1}$ in \eqref{eq:DMS1}, 
$\chi = \pi_{\Gamma^m}[(\vec X^{m+1}-\vec\id)\,.\,\vec\omega^m]$ 
in \eqref{eq:DMS2} and 
$\vec\eta = \vec X^{m+1}-\vec\id_{\mid_{\Gamma^m}}$ in \eqref{eq:DMS3} 
and obtain
\begin{equation*}
\ttau_m\left|\nabla\,U^{m+1}\right|_\Omega^2 + 
\left\langle\nabs\,\vec X^{m+1},\nabs\,(\vec X^{m+1}-\vec\id) 
\right\rangle_{\Gamma^m} =0\,.
\end{equation*}
Now Lemma~\ref{lem:stab2d3d} yields the claim.
\end{proof}

\begin{rem}[Semidiscrete scheme] \label{rem:MSsd}
We note that the above stability bound is a discrete analogue of
{\rm Remark~\ref{rem:WMS}\ref{item:WMSi}}. However, the fully discrete scheme
\eqref{eq:DMS} will in general not satisfy an exact discrete analogue of
the volume conservation property
{\rm Remark~\ref{rem:WMS}\ref{item:WMSii}}. That is because
choosing $\varphi=1$ in \eqref{eq:DMS1} only leads to
\begin{align} \label{eq:Xiddiff}
0 & = \left\langle \pi_{\Gamma^m}\left[
\left(\vec X^{m+1}-\vec\id\right).\,\vec\omega^m \right], 1
\right\rangle_{\Gamma^m}
= \left\langle \vec X^{m+1}-\vec\id,\vec\omega^m \right\rangle_{\Gamma^m}^h
\nonumber \\ &
= \left\langle \vec X^{m+1}-\vec\id, \vec\nu^m\right\rangle_{\Gamma^m}^h
= \left\langle \vec X^{m+1}-\vec\id, \vec\nu^m\right\rangle_{\Gamma^m},
\end{align}
where we have noted \eqref{eq:omegahnuh} and \eqref{eq:intnuh}. 
In general, the terms in \eqref{eq:Xiddiff} do not equal the 
discrete volume change
${\mathcal L}^{d}(\Omega_-^{m+1}) - {\mathcal L}^{d}(\Omega_-^m)$,
where $\Omega_-^m$ denotes the interior of $\Gamma^m$.
However, on recalling {\rm Remark~\ref{rem:SDSD}\ref{item:remSDSDii}}
it is possible to prove an exact discrete analogue of 
{\rm Remark~\ref{rem:WMS}\ref{item:WMSii}}
for a semidiscrete version of the scheme \eqref{eq:DMS}.
Moreover, this semidiscrete scheme will satisfy the mesh property
{\rm Theorem~\ref{thm:SDSD}\ref{item:SDSDTM}}.
\end{rem}

\begin{rem}[Discrete linear systems] \label{rem:MSsolve}
On recalling the notation from {\rm \S\ref{subsec:3.3}}, 
we can formulate the
linear systems of equations to  be solved at each time level for
\eqref{eq:DMS} as follows.
Find $(U^{m+1},\kappa^{m+1},\delta\vec X^{m+1})
\in \bR^{K_{S^m}}\times \bR^K\times (\bR^d)^K$ such that
\begin{equation}
\begin{pmatrix}
 A_\Omega & 0 & -\frac1{\ttau_m}\,\NbulkTMS \\
-\Mbulk & M_{\Gamma^m} & 0 \\
0 & \vec N_{\Gamma^m} & \mat A_{\Gamma^m}
\end{pmatrix}
\begin{pmatrix} U^{m+1} \\ \kappa^{m+1} \\ \delta\vec X^{m+1} \end{pmatrix}
=
\begin{pmatrix} 0 \\ 0 \\ -\mat A_{\Gamma^m}\,\vec X^m \end{pmatrix} \,,
\label{eq:lin}
\end{equation}
where we use a similar abuse of notation as in \eqref{eq:MCvec}. The
definitions of the matrices in \eqref{eq:lin} are either given in
\eqref{eq:mat0}, or they follow directly from 
\eqref{eq:DMS}, see also {\rm \citet[\S4.1]{dendritic}} for details.
In practice, the linear system  \eqref{eq:lin} can either be 
solved with a sparse direct solution method like UMFPACK, see 
{\rm \cite{Davis04}}, 
or by first using a Schur complement approach and then use a (precondioned) 
conjugate gradient solver. Once again, we refer to 
{\rm \citet[\S4.1]{dendritic}} for more details.
\end{rem}
           
\subsection{The Stefan problem with a (kinetic) Gibbs--Thomson law}
In general crystal growth phenomena involve more complex models, in
comparison to the above Mullins--Sekerka model. The overall model is
the Stefan problem with a Gibbs--Thomson law and kinetic undercooling
with anisotropic surface energy \eqref{eq:11} taken into account. It
is given as follows with the same notation and conventions as for (\ref{eq:MS}).

Given $\Gamma(0)$ and, if $\vartheta >0$, $u_0 : \Omega \to \bR$,
find $u : \Omega \times [0,T] \to \bR$ 
and the evolving interface $(\Gamma(t))_{t\in[0,T]}$ such that
$\vartheta\,u(\cdot,0) = \vartheta\,u_0$ and 
for all $t\in (0,T]$ the following conditions hold
\begin{subequations} \label{eq:MS1abc}
\begin{alignat}{4}
& \vartheta\,\partial_t\,u - \conduct_-\,\Delta u = f
\quad && \text{in } \Omega_-(t)\,, \quad && 
\vartheta\,\partial_t\,u - \conduct_+\,\Delta u = f \quad 
&&\text{in } \Omega_+(t)\,, \label{eq:MS1a} \\
& \frac{\MSrho\,\mathcal{V}}{\beta(\vec\nu)} = 
\alpha\,\varkappa_\gamma - a\,u \quad && \text{on } \Gamma(t)\,, &&
\left[ \conduct\,\partial_{\vec\nu}\, u \right]_-^+ 
=-\lambda\,{\cal V} \quad && \text{on } \Gamma(t)\,, \label{eq:MS1b} \\ 
& \partial_{\vec\nu_\Omega} u = 0 \quad &&\mbox{on } \partial_{\rm N}\Omega\,, 
&& u = u_{\rm D} \quad && \text{on } \partial_{\rm D} \Omega \,, 
\label{eq:MS1d}
\end{alignat}
\end{subequations}
where $[\conduct\,\partial_{\vec\nu}\,u]_-^+$ 
is defined similarly to (\ref{eq:dnufjump}), and $\partial\Omega =
\overline{\partial_{\rm N}\Omega}\cup\overline{\partial_{\rm D}\Omega}$ with
$\partial_{\rm N}\Omega\cap\partial_{\rm D}\Omega = \emptyset$.
In the above system $u$ denotes the deviation from the melting
temperature at a planar interface, $f$
describes heat sources, 
$\vartheta \in \bRgeq$ is the volumetric heat capacity, and
$\conduct(\cdot,t) = \conduct_+\,\charfcn{\Omega_+(t)} 
+ \conduct_-\,\charfcn{\Omega_-(t)}$, with $\conduct_{\pm} \in \bRplus$,
is the phase-dependent thermal conductivity, recall \eqref{eq:fpm}.
Moreover, 
$\lambda\in \bRplus$ is the latent heat
per unit volume, $\alpha\in\bRplus$ is an interfacial energy density per 
surface area, $\MSrho\in\bRgeq$ is a kinetic coefficient and 
$a \in \bRplus$ is a coefficient having the dimension entropy/volume. 
In addition, as in \S\ref{subsec:ani1},
$\beta: \bS^{d-1} \to\bRplus$ is a dimensionless mobility
function, which allows one to describe the dependence of the mobility on
the local orientation of the interface. 
Clearly, on choosing $\vartheta=\MSrho=0$, $\conduct=a=\alpha=\lambda=1$
and $\partial\Omega= \partial_{\rm N} \Omega$ then (\ref{eq:MS1abc})
in the isotropic case collapses to (\ref{eq:MS}).

In order to state the weak form of (\ref{eq:MS1abc}), we introduce the 
function spaces 
\begin{equation} \label{eq:S0D}
S_0 = \{ \phi \in H^1(\Omega) : \phi = 0 \ \mbox{ on } \partial_{\rm D}\Omega \}
\; \text{and}\;
S_{\rm D} = \{ \phi \in H^1(\Omega) : \phi = u_{\rm D} \ \mbox{ on } \partial_{\rm D}\Omega\}\,, 
\end{equation}
where we assume for simplicity from now on that $u_{\rm D} \in \bR$,
with $u_{\rm D} = 0$ in the case 
$\partial_{\rm N}\Omega = \partial\Omega$.
Now a weak formulation of (\ref{eq:MS1abc}) for an anisotropic surface energy
density of the form \eqref{eq:g} can be obtained by 
generalizing the weak formulation of the
Mullins--Sekerka problem, (\ref{eq:WMS}), as follows.

Given a closed hypersurface $\Gamma(0) \subset \Omega$ and, 
if $\vartheta >0$, $u_0 \in L^2(\Omega)$,
we seek an evolving hypersurface $(\Gamma(t))_{t\in[0,T]}$
that separates $\Omega$ into $\Omega_-(t)$ and $\Omega_+(t)$,
with a global parameterization and induced velocity field
$\vec{\mathcal{V}}$, and $\varkappa \in L^2(\GT)$ as well as 
$u\in H^1(0,T;L^2(\Omega)) \cap L^2(0,T;H^1(\Omega)))$ 
with $\vartheta\,u(\cdot,0)=\vartheta\,u_0$ as follows.
For almost all $t \in (0,T)$, find
\arxivyesno{
$(u(\cdot,t),$ $\vec {\mathcal V}(\cdot,t),\varkappa(\cdot,t))\in S_{\rm D} 
\times [L^2(\Gamma(t))]^d \times L^2(\Gamma(t))$}
{$(u(\cdot,t),\vec {\mathcal V}(\cdot,t),\varkappa(\cdot,t))\in S_{\rm D} 
\times [L^2(\Gamma(t))]^d \times L^2(\Gamma(t))$} 
such that 
\begin{subequations}
\label{eq:2}
\begin{align}
&\vartheta\left(\partial_t\,u, \phi\right) 
+ \left(\conduct\,\nabla\,u, \nabla\,\phi\right) - \left(f, \phi\right) 
= \lambda\left\langle \vec{\mathcal{V}},\phi\,\vec\nu\right\rangle_{\Gamma(t)} 
\quad\forall\ \phi\in S_0\,,\label{eq:2aa} \\ &
\MSrho\left\langle \frac{\vec{\mathcal{V}}}{\beta(\vec\nu)},\chi \,\vec\nu
\right\rangle_{\Gamma(t)} 
= \left\langle \alpha\,\varkappa_\gamma - a\,u, \chi \right\rangle_{\Gamma(t)} 
\quad\forall\ \chi \in L^2(\Gamma(t))
\,,\label{eq:2bb} \\
& \left\langle \varkappa_\gamma\,\vec\nu,\vec\eta \right\rangle_{\Gamma(t)} 
+ \left\langle \nabs^\tG\,\vec\id,\nabs^\tG\,\vec\eta 
\right\rangle_{\Gamma(t),\gamma} = 0 
\quad\forall\ \vec\eta \in \Vt\,, \label{eq:2cc}
\end{align}
\end{subequations}
where we have adopted the notation (\ref{eq:ipGG}).

We can establish the following formal a priori bound for a sufficiently 
\arxivyesno{}{\linebreak}%
smooth solution of \eqref{eq:2}. 
Choosing $\phi=u-u_{\rm D}$ in (\ref{eq:2aa}), 
$\chi=\frac\lambda{a}\,\vec{\mathcal{V}}\,.\,\vec\nu$ in (\ref{eq:2bb}) and 
$\vec\eta=\frac{\alpha\,\lambda}a\,\vec {\mathcal V}$ in (\ref{eq:2cc})  
we obtain, on noting the transport theorem, Theorem \ref{thm:transvol},
and Lemma \ref{lem:ani} that
\begin{align}
&\ddt\left(\frac\vartheta2\left|u-u_{\rm D}\right|_\Omega^2 + 
\frac{\alpha\,\lambda}a\,|\Gamma(t)|_\gamma
+\lambda\,u_{\rm D}\,\mathcal{L}^d(\Omega_-(t))\right) \nonumber \\ & 
\qquad + \left(\conduct\,\nabla\,u, \nabla\,u\right) 
+\frac{\lambda\,\rho}{a} \left\langle \beta(\vec\nu), {\mathcal V}^2
\right\rangle_{\Gamma(t)}
= \left(f, u - u_{\rm D}\right) . \label{eq:testD}
\end{align}

We now introduce the finite element approximations of $S_0$ and $S_{\rm D}$.
On recalling (\ref{eq:Sm}) and \eqref{eq:S0D}, we set 
\begin{equation} \label{eq:Sm0D}
S^m_0=S^m\cap S_0 \quad\text{and}\quad
\SmD = S^m \cap S_{\rm D}\,.
\end{equation}
Then a finite element approximation of (\ref{eq:2}) 
from \cite{dendritic}, combining the ideas 
introduced for mean curvature flow in Section~\ref{sec:mc}, for anisotropic
flows in Section~\ref{sec:ani} and for the Mullins--Sekerka
problem above, is given as follows.
Let the closed polyhedral hypersurface $\Gamma^0$ be an approximation of 
$\Gamma(0)$, 
and, if $\vartheta>0$, let $U^0\in S^0_{\rm D}$ be an approximation of $u_0$, e.g.\
$U^0 = I^0\,u_0$ if $u_0 \in C(\overline{\Omega})$.
Then, for $m=0,\ldots, M-1$, find $(U^{m+1},\vec X^{m+1},\kappa^{m+1}_\gamma)
\in \SmD \times \Vhm \times \Whm$ such that
\begin{subequations} \label{eq:MSHG}
\begin{align}
& \vartheta\left(\frac{U^{m+1}-U^m}{\ttau_m}, \varphi \right)^h_m + 
\left(\conduct\,\nabla\,U^{m+1}, \nabla\,\varphi\right) 
- \lambda \left\langle \pi_{\Gamma^m}\left[
\frac{\vec X^{m+1}-\vec\id}{\ttau_m} \,.\,\vec\omega^m \right], \varphi
\right\rangle_{\Gamma^m} \nonumber \\ & \hspace{4cm}
 = \left(f^{m+1}, \varphi\right)^h_m \quad\forall\ \varphi \in S^m_0
\,, \label{eq:MSHGa}\\
& \MSrho\left\langle 
[\beta(\vec\nu^m)]^{-1}\,\frac{\vec X^{m+1}-\vec\id}{\ttau_m},
\chi\,\vec\omega^m \right\rangle_{\Gamma^m}^h - \alpha \left\langle \kappa^{m+1}_\gamma, 
\chi\right\rangle_{\Gamma^m}^h
+ a \left\langle U^{m+1}, \chi \right\rangle_{\Gamma^m} =0 
\nonumber \\
& \hspace {6cm} \quad\forall\ \chi \in \Whm
\,, \label{eq:MSHGb} \\
& \left\langle \kappa^{m+1}_\gamma\,\vec\nu^m, \vec\eta \right\rangle_{\Gamma^m}^h + 
\left\langle \nabs^\tG\,\vec X^{m+1}, \nabs^\tG\,\vec\eta \right\rangle_{\Gamma^m,\gamma}
= 0 \quad\forall\ \vec\eta \in \Vhm
 \label{eq:MSHGc} 
\end{align}
\end{subequations}
and set $\Gamma^{m+1} = \vec X^{m+1}(\Gamma^m)$.
In the above, we have introduced 
\arxivyesno{$f^{m+1} = f(\cdot,t_{m+1})$,}
{$f^{m+1} = $ \linebreak $f(\cdot,t_{m+1})$,}
where we assume for convenience that $f(\cdot,t) \in C(\overline{\Omega})$
for all $t \in [0,T]$.
Moreover, we have recalled (\ref{eq:omegah}) and observe that using
$\vec\omega^m$ in \arxivyesno{}{\linebreak}%
\eqref{eq:MSHGb}, and not $\vec\nu^m$, is necessary in order
be able to prove existence, uniqueness and stability results for
\eqref{eq:MSHG}, see the proof of Theorem~\ref{thm:stab} below.

\begin{rem}[Implementation] \label{rem:MSHG}
Compared to the computational challenges discussed in 
{\rm Remark~\ref{rem:unfitted}} for \eqref{eq:DMS}, the only new
difficulty in \eqref{eq:MSHG} is the term
$(\conduct\,\nabla\,U^{m+1},$ $\nabla\,\varphi)$. 
In order to compute this in the case $\conduct_+ \not= \conduct_-$, it is
necessary to calculate $\mathcal{L}^d(\sigmaO \cap \Omega^m_-)$ for
every $\sigmaO \in \mathcal{T}^m$, where $\Omega^m_-$ denotes the interior of
$\Gamma^m$. This can be done as described in {\rm \citet[Remark~4.2]{crystal}},
but turns out to be computationally expensive due to the unfitted nature of 
the finite element approximation.
Alternatively, suitable numerical approximations of the integral
$\left(\conduct\,\nabla\,U^{m+1}, \nabla\,\varphi\right)$ can be introduced,
similarly to what we present in {\rm \S\ref{subsec:one-sided}}
and {\rm Section~\ref{sec:tpf}}, below.
The case of phase-dependent forcings $f_\pm$ in \eqref{eq:MS1a} can be dealt
with similarly, on introducing a suitable analogue of
$(f^{m+1}, \varphi)_m^h$ in \eqref{eq:MSHGa}. 
\end{rem}

On defining 
\[
\mathcal{E}^m(U,\vec X) = 
\frac\vartheta2\,|U - u_{\rm D}|_{\Omega,m}^2 +
\frac{\alpha\,\lambda}a\,|\vec X(\Gamma^m)|_\gamma
\]
for $U \in C(\overline\Omega)$ and $\vec X \in \Vhm$,
we have the following results.

\begin{thm} \label{thm:stab}
Let $\Gamma^m$ and $\mathcal{T}^m$ satisfy 
{\rm Assumption~\ref{ass:A}\ref{item:assA2}} 
and {\rm Assumption~\ref{ass:B}}, and let $U^m \in C(\overline\Omega)$.
Let $\gamma$ be of the form \eqref{eq:g} with $r=1$.
Then there exists a unique solution 
$(U^{m+1}, \vec X^{m+1}, \kappa^{m+1}_\gamma)$ $ 
\in \SmD\times \Vhm \times \Whm$ to {\rm (\ref{eq:MSHG})}. 
Moreover, if $d=2$ or $d=3$ and if 
$\gamma$ is of the form \eqref{eq:g} with $r\in[1,\infty)$, 
then a solution to \eqref{eq:MSHG} satisfies
\begin{align}
& \mathcal{E}^m(U^{m+1}, \vec X^{m+1})
+\lambda\,u_{\rm D}\left\langle \vec X^{m+1}-\vec\id, 
\vec\nu^m\right\rangle_{\Gamma^m}
+ \frac\vartheta2\left|U^{m+1} - U^m\right|_{\Omega,m}^2
\nonumber \\ & \hspace{0.2cm}
+ \ttau_m\left(\conduct\,\nabla\,U^{m+1}, \nabla\,U^{m+1}\right)
+\ttau_m\,\frac{\lambda\,\MSrho}{a}\left(\left|[\beta(\vec\nu^m)]^{-\frac12}\,
\frac{\vec X^{m+1}-\vec\id}{\ttau_m}\,.\,\vec\omega^m\right|_{\Gamma^m}^h
\right)^2
\nonumber \\ & \hspace{2cm}
\leq \mathcal{E}^m(U^m,\vec\id)
+ \ttau_m\left(f^{m+1}, U^{m+1} - u_{\rm D}\right)^h_m . \label{eq:stab}
\end{align}
\end{thm}
\begin{proof}
The existence and uniqueness proof for $r=1$ proceeds analogously to the proof 
of Theorem~\ref{thm:HDMS}. In particular, if
$(U,\vec X,\kappa_\gamma) \in S^m_0\times\underline{V} (\Gamma^m)\times \Whm$ 
denotes a solution to the homogeneous analogue of \eqref{eq:MSHG},
similarly to \eqref{eq:HDMS}, then choosing 
$\varphi=U$, $\chi = \frac\lambda{a}\,
\pi_{\Gamma^m}\,[\vec X\,.\,\vec\omega^m]$
and $\vec\eta=\frac{\alpha\,\lambda}a\,\vec X$ yields
\begin{align*} 
& \vartheta\left|U\right|_{\Omega,m}^2 + 
\ttau_m\left(\conduct\,\nabla\, U, \nabla\, U\right)
+\frac{\lambda\,\MSrho}{\ttau_m\,a} \left(\left|[\beta(\vec\nu^m)]^{-\frac12}\,
\vec X\,.\,\vec\omega^m\right|_{\Gamma^m}^h\right)^2 \nonumber \\ & \qquad
+ \frac{\alpha\,\lambda}a\left\langle\nabs^\tG\,\vec X, \nabs^\tG\,\vec X 
\right\rangle_{\Gamma^m,\gamma} =0 \,,
\end{align*}
which implies that $U$ is constant in $\Omega$, with $U=0$ if 
$\vartheta > 0$ or $S^m_0 \not= S^m$, 
and that $\vec X$ is constant on $\Gamma^m$,
since all the involved physical parameters are nonnegative, with
$\conduct$, $\alpha$, $a$ and $\lambda$ being positive.
As in the proof of Theorem~\ref{thm:HDMS} we can then infer 
from Assumption~\ref{ass:B} that $\vec X = \vec 0$.
This implies that $\kappa_\gamma = \frac a\alpha\,U$ is constant on $\Gamma^m$
and so we can proceed as in the proof of Theorem~\ref{thm:HDMS}
to show that $\kappa_\gamma = U = 0$. Overall this proves
existence of a unique solution $(U^{m+1}, \vec X^{m+1}, \kappa^{m+1}_\gamma) 
\in \SmD\times \Vhm \times \Whm$ to \eqref{eq:MSHG}. 

It remains to establish (\ref{eq:stab}). 
Choosing $\varphi=U^{m+1}-u_{\rm D}$ in (\ref{eq:MSHGa}), 
$\chi = \frac\lambda{a}\,
\pi_{\Gamma^m}[(\vec X^{m+1}-\vec\id_{\mid_{\Gamma^m}})\,.\,\vec\omega^m]$ 
in (\ref{eq:MSHGb}) and
$\vec\eta=\frac{\alpha\,\lambda}a\,({\vec X^{m+1}-\vec\id_{\mid_{\Gamma^m}}})$ 
in (\ref{eq:MSHGc}) yields that
\begin{align*}
& \vartheta\left(U^{m+1}-U^m, U^{m+1} - u_{\rm D}\right)^h_m + 
\ttau_m\left(\conduct\,\nabla\, U^{m+1}, \nabla\, U^{m+1}\right)
\nonumber \\ & \qquad
+ \frac{\alpha\,\lambda}a
\left\langle \nabs^\tG\,\vec X^{m+1},\nabs^\tG\,(\vec X^{m+1} - \vec\id) 
\right\rangle_{\Gamma^m,\gamma}
\nonumber \\ & \qquad
+\ttau_m\,\frac{\lambda\,\rho}{a} \left(\left|[\beta(\vec\nu^m)]^{-\frac12}\,
\frac{\vec X^{m+1}-\vec\id}{\ttau_m}
\,.\,\vec\omega^m\right|_{\Gamma^m}^h\right)^2
\nonumber \\ & \qquad
= -\lambda\,u_{\rm D} \left\langle \vec X^{m+1}-\vec\id,
\vec\omega^m\right\rangle_{\Gamma^m}^h
+ \ttau_m\left(f^{m+1}, U^{m+1} - u_{\rm D}\right)^h_m
\end{align*}
and hence (\ref{eq:stab}) follows immediately from
Lemma~\ref{lem:anistab2d3d} and \eqref{eq:intnuh}.
\end{proof}

\begin{rem}[Semidiscrete scheme] \label{rem:stab1}
We note that \eqref{eq:stab} closely mimics the corres\-ponding continuous 
energy law \eqref{eq:testD}. 
The reason why it is not an exact discrete analogue of \eqref{eq:testD} 
has been discussed in {\rm Remark~\ref{rem:MSsd}} already.
However, on recalling {\rm Remark~\ref{rem:SDSD}\ref{item:remSDSDii}}
it is possible to prove an exact discrete analogue of {\rm (\ref{eq:testD})}
for a semidiscrete version of the scheme \eqref{eq:MSHG},
see also {\rm \citet[Remark 3.5]{dendritic}}.
In addition, this semidiscrete scheme will feature the tangential motion
discussed in {\rm Remark~\ref{rem:aniTM}}.
\end{rem}

\begin{rem}[Discrete systems] \label{rem:MSHGsolve}
It is a simple matter to combine the techniques in 
{\rm Remark~\ref{rem:MSsolve}} and {\rm \S\ref{subsec:anisolve}} in order to
derive the discrete systems that need to be solved at each time level
for \eqref{eq:MSHG}.  
If $\gamma$ is of the form \eqref{eq:g} with $r=1$, then the systems
are linear and can hence it can be solved analogously to 
{\rm Remark~\ref{rem:MSsolve}}, i.e.\ either with a 
sparse direct solver or with the help of a Schur complement approach and
a preconditioned conjugate gradient solver.
If $r > 1$, on the other hand, then the systems are nonlinear and can be
solved with a lagged fixed-point type iteration. 
We refer to {\rm \citet[\S4]{dendritic}} for further details.
\end{rem}

\subsection{One-sided free boundary problems} \label{subsec:one-sided}
In this section we consider the situation where diffusion is
restricted to one phase. We hence study one-sided versions of the 
Mullins--Sekerka or Stefan problems. This is, for example, relevant for snow
crystal growth, where diffusion can be restricted to the gas phase. In this
case we only find the unknown $u$ in one-phase, which occupies at time
$t$ a domain $\Omega_+(t)$. 
Once again, for simplicity of the
presentation, we always assume that the region
$\Omega_-(t) = \Omega\setminus \overline{\Omega_+(t)}$ is compactly contained
in $\Omega$.
The problem now reads as follows. 
Given $\Gamma(0)$ and, if $\vartheta >0$, $u_0 : \Omega_+(0) \to \bR$,
find the evolving interface $(\Gamma(t))_{t\in[0,T]}$ and 
$u(\cdot,t) : \Omega_+(t) \to \bR$, $t \in [0,T]$, such that
$\vartheta\,u(\cdot,0) = \vartheta\,u_0(\cdot)$ and 
for all $t\in (0,T]$ the following conditions hold
\begin{subequations} \label{eq:os1}
\begin{alignat}{2}
&\vartheta\,\partial_t\,u - \conduct_+\,\Delta u = f \qquad  \mbox{in } 
\Omega_+(t)\,, \qquad u = u_{\rm D} \qquad &&\mbox{on } \partial\Omega \,,
\label{eq:os1a} \\
&\frac{\MSrho\,\mathcal{V}}{\beta(\vec\nu)} = 
\alpha\,\varkappa_\gamma - a\,u \quad\text{and } \quad
\conduct_+\,\partial_{\vec\nu}\, u =-\lambda\,{\cal V} \quad 
&&\mbox{on } \Gamma(t),  \label{eq:os1b}
\end{alignat}
\end{subequations}
where the given data satisfies the same conditions as in (\ref{eq:MS1abc}). 
For snow crystal growth $-u$ is a suitably scaled concentration with 
$-u_{\rm D}$ being the scaled supersaturation, see 
\cite{Libbrecht05,jcg,crystal}. 

As we wish to model snow crystal growth, we now state a relevant
hexagonal surface energy. Let $d=3$ and
$l_\epsilon(\vec p) = \left[ \epsilon^2\,|\vec p|^2 +
  p_1^2\,(1-\epsilon^2) \right]^{\frac12}$ for $\epsilon>0$.
We then introduce the rotation matrices
\[
\mat{R}_{1}(\theta)=\left(\!\!\!\scriptsize
\begin{array}{rrr} \cos\theta & \sin\theta&0 \\
-\sin\theta & \cos\theta & 0 \\ 0 & 0 & 1 \end{array}\!\! \right) \quad 
\mbox{and} \quad  
\mat{R}_{2}(\theta)=\left(\!\!\!\scriptsize
\begin{array}{rrr} \cos\theta & 0 & \sin\theta \\
0 & 1 & 0 \\ -\sin\theta & 0 & \cos\theta \end{array}\!\! \right).
\]
Then setting
\begin{equation*} 
\gamma(\vec p) = 
l_\epsilon\left(\mat{R}_2(\tfrac{\pi}2)\,\vec p\right) +
\tfrac{1}{\sqrt{3}}\sum_{\ell = 1}^3
l_\epsilon\left(\mat{R}_1(\theta_0+\tfrac{\ell\,\pi}3)\,\vec p\right)\,,
\end{equation*}
where $\theta_0 \in [0,\frac\pi3)$ rotates the
anisotropy in the $x_1-x_2$ plane,
defines a density of the form \eqref{eq:g}, with $r=1$ and $L=4$, 
that approximates
a crystalline surface energy density with a regular hexagonal prism as its
Wulff shape, where each face of the Wulff shape has the same distance to the 
origin.

We recall from \cite{crystal} an unfitted finite element approximation of 
the above one-sided problem, following ideas of \cite{BarrettE82}. 
Given an approximation of the interface $\Gamma^m$, we 
let $\Omega^m_+$ denote the exterior of $\Gamma^m$ and let
$\Omega^m_-$ denote the interior of $\Gamma^m$, so that
$\Gamma^m = \partial\Omega^m_- = \overline{\Omega^m_-} \cap 
\overline{\Omega^m_+}$.
We now introduce the appropriate discrete trial and test function
spaces. To this end, let $\Omega^{m,h}_+$ be an approximation to
$\Omega^m_+$ and set $\Omega^{m,h}_- = \Omega \setminus
\overline \Omega^{m,h}_+$. We stress that $\Omega^{m,h}_+$ need not necessarily
be a union of elements from $\mathcal{T}^m$.
Then we define, on recalling (\ref{eq:Sm}) and \eqref{eq:Sm0D},
the finite element spaces
\begin{align}
\Sml & = \left\{ \chi \in S^m : \chi(\vec p^m_k) = 0 \text{ if } 
\supp \phi^{S^m}_k \subset \overline{\Omega^{m,h}_-},\ k=1,\ldots, 
K_{S^m} \right\}, 
\nonumber \\ 
\quad\Smhoml & = \Smhom \cap \Sml\,,
\qquad\SmDl = \SmD \cap \Sml\,.
\label{eq:Sml}
\end{align}

Our finite element approximation of (\ref{eq:os1}) is then given as follows.
Let the closed polyhedral hypersurface $\Gamma^0$ be an approximation of
$\Gamma(0)$, and, if $\vartheta>0$, let $U^0\in S^0_{\rm D}$ be an 
approximation of $u_0$, e.g.\ $U^0 = I^0\,u_0$ if 
$u_0 \in C(\overline{\Omega})$ is an 
extension of the given $u_0 \in C(\overline{\Omega_+(0)})$.
Then, for $m=0,\ldots, M-1$, find $(U^{m+1},\vec X^{m+1},\kappa^{m+1}_\gamma) 
\in \SmDl \times \Vhm \times \Whm$ such that
\begin{subequations} \label{eq:uHG}
\begin{align}
&
\vartheta\left(\frac{U^{m+1}-U^m}{\ttau_m}, \varphi \right)^h_{m,+} +
 \conduct_+\left(\nabla\,U^{m+1}, \nabla\,\varphi\right)_{m,+} 
\nonumber \\ & \
- \lambda \left\langle \pi_{\Gamma^m}\left[
\frac{\vec X^{m+1}-\vec\id}{\ttau_m} \,.\,\vec\omega^m \right], \varphi
\right\rangle_{\Gamma^m}  
= \left(f^{m+1}, \varphi\right)^h_{m,+} 
\quad\forall\ \varphi \in \Smhoml\,, \label{eq:uHGa}\\
& \MSrho\left\langle 
[\beta(\vec\nu^m)]^{-1}\,\frac{\vec X^{m+1}-\vec\id}{\ttau_m},
\chi\,\vec\omega^m \right\rangle_{\Gamma^m}^h 
- \alpha \left\langle \kappa^{m+1}_\gamma, 
\chi\right\rangle_{\Gamma^m}^h
+ a \left\langle U^{m+1}, \chi \right\rangle_{\Gamma^m} =0
\nonumber \\ & \hspace{6cm} \quad\forall\ \chi \in \Whm
\,, \label{eq:uHGb} \\
& \left\langle \kappa^{m+1}_\gamma\,\vec\nu^m, \vec\eta 
\right\rangle_{\Gamma^m}^h 
+ \left\langle \nabs^\tG\,\vec X^{m+1},\nabs^\tG\,\vec\eta 
\right\rangle_{\Gamma^m,\gamma} = 0 \quad\forall\ \vec\eta \in \Vhm
\label{eq:uHGc} 
\end{align}
\end{subequations}
and set $\Gamma^{m+1} = \vec X^{m+1}(\Gamma^m)$.
Here we define for all $\chi, \varphi \in S^m$
\begin{align} \label{eq:intml}
\left(\nabla\,\chi , \nabla\,\varphi\right)_{m,+} & =
\int_{\Omega^{m,h}_+} \nabla\,\chi \,.\, \nabla\,\varphi \dL{d}
\nonumber \\ & 
= \sum_{o \in \mathcal{T}^m} \frac{\mathcal{L}^d(o \cap \Omega^{m,h}_+)}
{\mathcal{L}^d(o)}\, \int_{o} \nabla\,\chi \,.\, \nabla\,\varphi \dL{d}
\end{align}
and, in a similar fashion, 
\begin{equation*} 
\left(\chi, \varphi\right)_{m,+}^h 
= \sum_{o \in \mathcal{T}^m}\frac{\mathcal{L}^d(o \cap \Omega^{m,h}_+)}
{\mathcal{L}^d(o)} \, \int_{o} I^m [\chi \,\varphi ] \dL{d}\,.
\end{equation*}
It follows immediately from (\ref{eq:Sml}) and (\ref{eq:intml}) that
\begin{equation*} 
\left(\nabla\,\varphi, \nabla\,\varphi\right)_{m,+} > 0 
\quad\forall\ \varphi \in \Smhoml \setminus \{0\}\,.
\end{equation*}

\begin{rem}
\rule{0pt}{0pt}
\begin{enumerate}
\item 
We note that for $\vartheta>0$ the approximation {\rm (\ref{eq:uHG})} is
only meaningful when the discrete solid region does not shrink, see
{\rm \citet[Remark~3.1]{crystal}}.
In practice this technical constraint is not very restrictive, since in most
physically relevant applications the solid region grows.
\item 
Existence and uniqueness for $r=1$, as well
as stability for $r \in [1,\infty)$, can be shown for 
\eqref{eq:uHG} under appropriate assumptions, 
see {\rm \citet[Theorems~3.1 and 3.2]{crystal}} 
for the proofs for the case $r=1$.
We note that the stability proof in {\rm \citet[Theorem~3.2]{crystal}},
for $d=2$ and $d=3$, immediately carries over to the case $r>1$, on recalling
{\rm Lemma~\ref{lem:anistab2d3d}}.
\item 
The main new computational challenge in implementing the scheme 
\arxivyesno{}{\linebreak}%
\eqref{eq:uHG}, compared to \eqref{eq:MSHG} with $\conduct_+ = \conduct_-$, 
is the determination of $\Smhoml$ and the calculation of the terms involving
$(\cdot,\cdot)_{m,+}$ and $(\cdot,\cdot)_{m,+}^h$ at each time step.
Clearly, the involved difficulty crucially depends on the choice of
$\Omega^{m,h}_+$ as an approximation to $\Omega^m_+$. A thorough discussion of
possible choices can be found in {\rm \citet[\S4.1]{crystal}}.
\end{enumerate}
\end{rem}

We present two numerical snow crystal growth simulations based on the scheme
\eqref{eq:uHG} with a hexagonal surface energy density $\gamma$,
and with a time-dependent mobility $\beta$, in
Figure~\ref{fig:crystal}.
\begin{figure}
\center
\arxivyesno{
\newcommand\lwidth{0.45\textwidth} 
\newcommand\lheight{2cm} 
\newcommand\llheight{1.2cm} 
}{
\newcommand\lwidth{0.3\textwidth} 
\newcommand\lheight{1.5cm} 
\newcommand\llheight{1.0cm} 
}
\includegraphics[angle=-90,width=\lwidth]{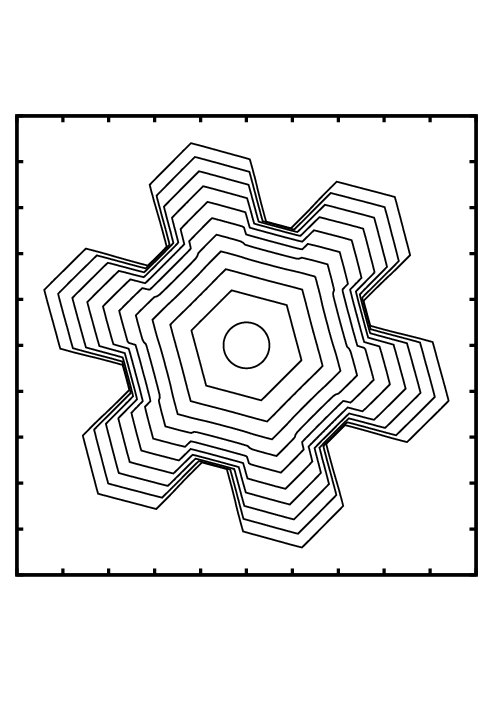}
\includegraphics[angle=-90,width=\lwidth]{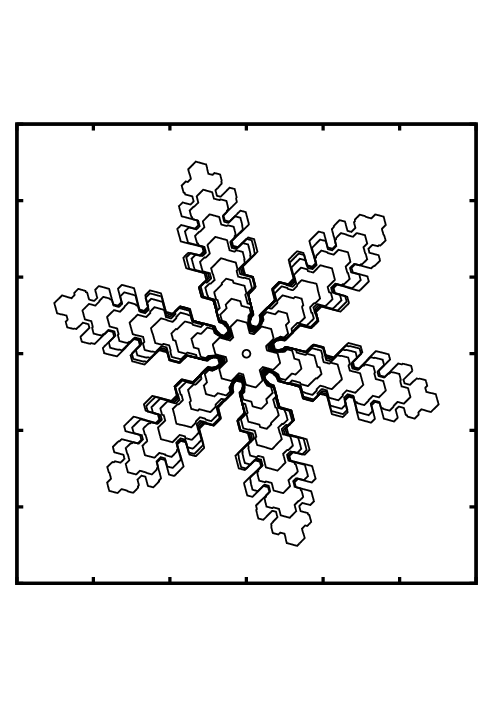} \\
\includegraphics[angle=-0,totalheight=\lheight]{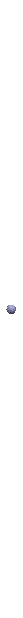}
\includegraphics[angle=-0,totalheight=\lheight]{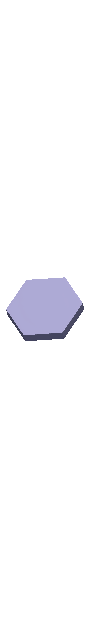}
\includegraphics[angle=-0,totalheight=\lheight]{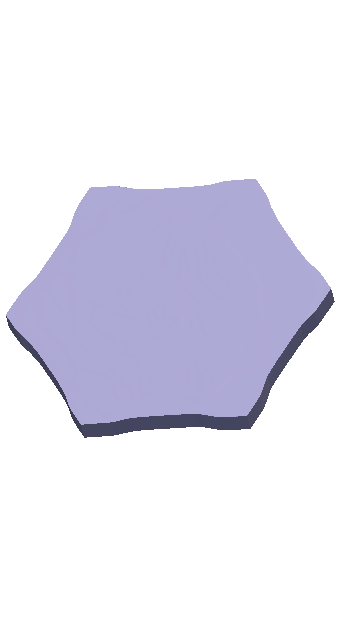}
\includegraphics[angle=-0,totalheight=\lheight]{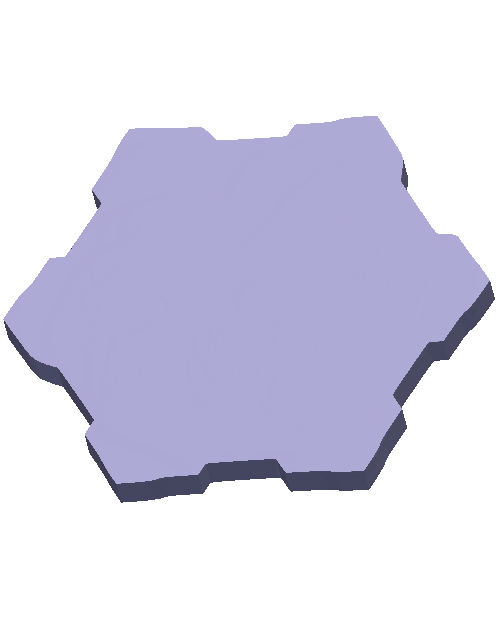}
\includegraphics[angle=-0,totalheight=\lheight]{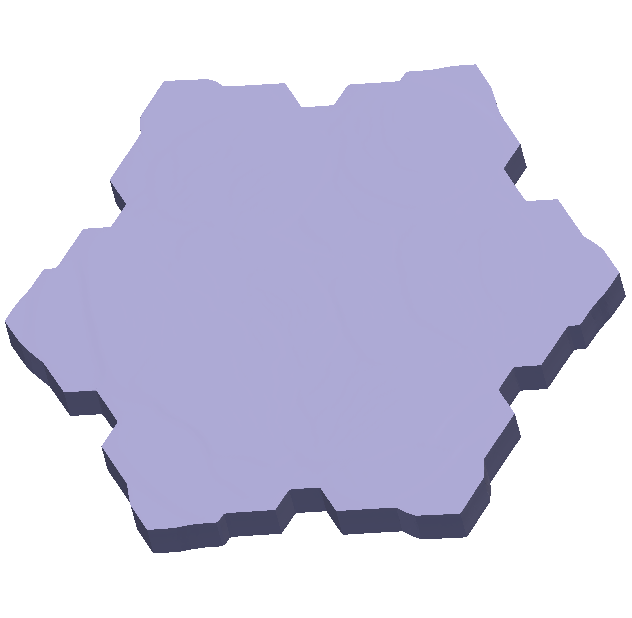}
\includegraphics[angle=-0,totalheight=\lheight]{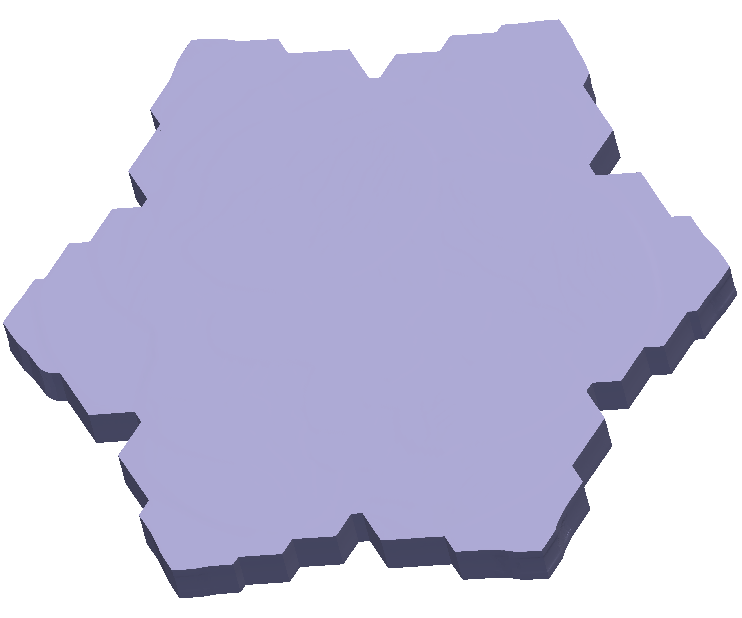}\\
\includegraphics[angle=-0,totalheight=\llheight]{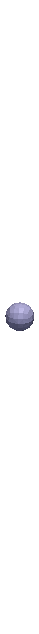} \
\includegraphics[angle=-0,totalheight=\llheight]{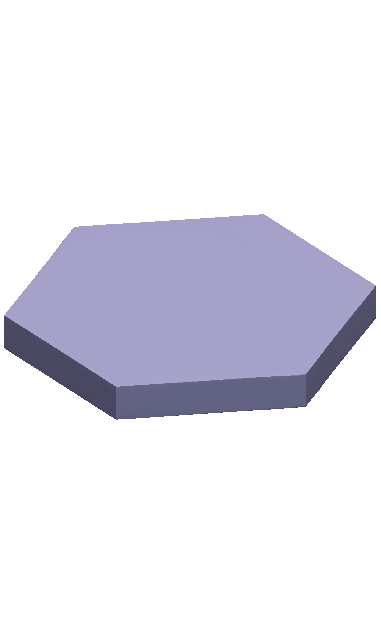} \
\includegraphics[angle=-0,totalheight=\llheight]{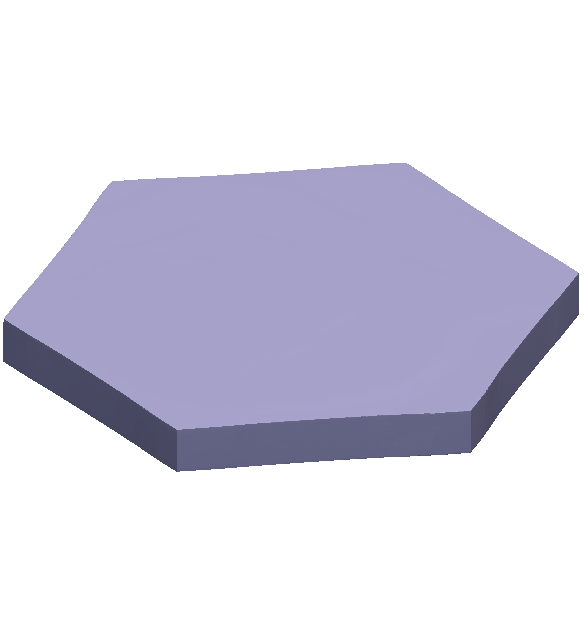} \
\includegraphics[angle=-0,totalheight=\llheight]{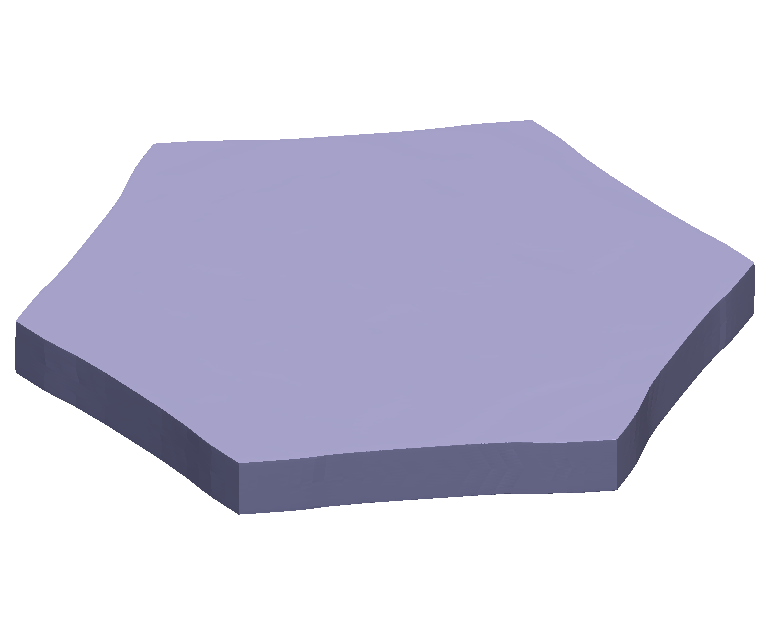} \
\includegraphics[angle=-0,totalheight=\llheight]{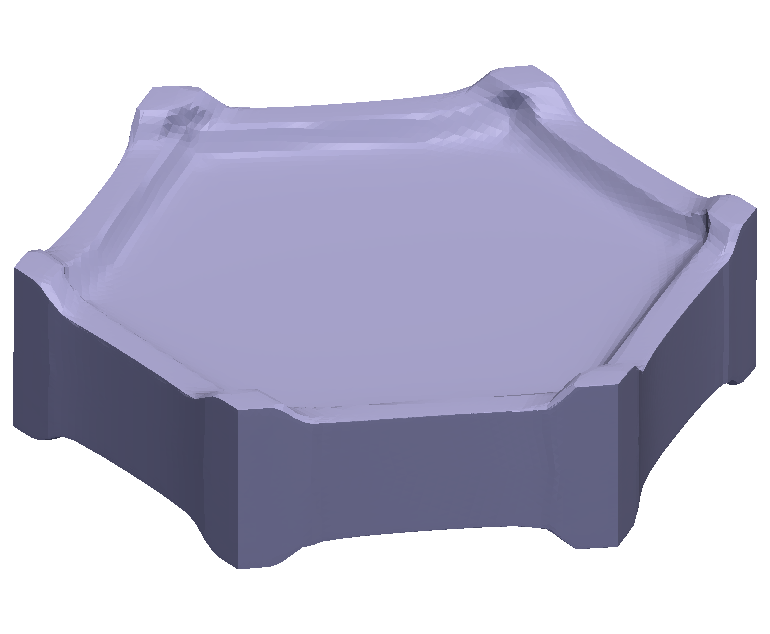} \
\includegraphics[angle=-0,totalheight=\llheight]{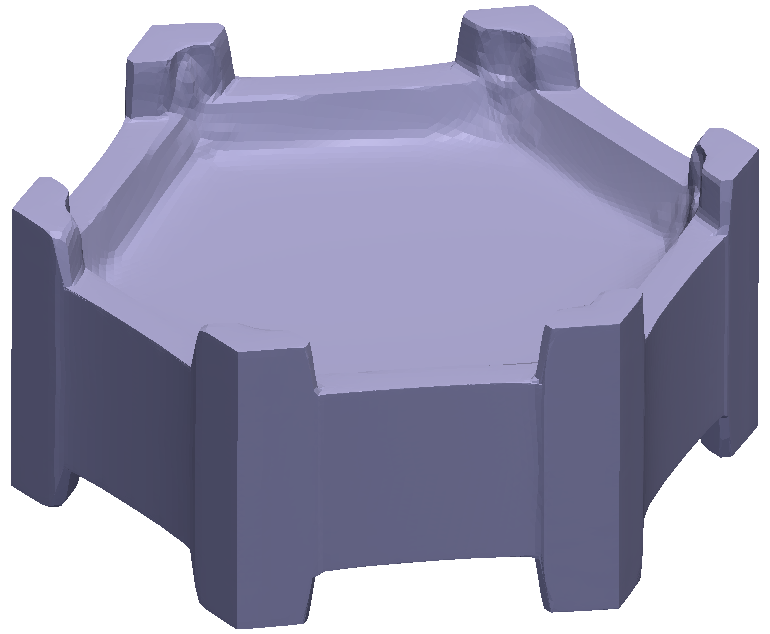}
\caption{Visualizations of the numerical experiments from
\citet[Fig.\ 12]{crystal}, above, \citet[Fig.\ 22]{crystal}, middle,
and \citet[Fig.\ 27 (left)]{crystal}, below.
The 2D solutions are shown at times $t=0,\,0.5,\ldots,5$ (left) and 
$t=0,\,5,\ldots,40$ (right).
The snapshots for the 3D simulations are taken at times 
$t=0,\ 5,\ 50,\ 100,\ 150,\ 200$ (middle) and
$t=0,\ 10,\ldots, 50$ (below).
}
\label{fig:crystal}
\end{figure}%

\subsection{Alternative numerical approaches}
\label{subsec:alternativeMS}

Other authors who numerically studied crystal growth with the help of a front
tracking type method are 
\cite{RoosenT91,Yokoyama93,Almgren93,Schmidt93,Schmidt96,Schmidt98,JuricT96,%
BanschS00}. %
A level set method for crystal growth was studied in
\cite{SethianS92}.
In the context of phase field approximations, we refer to
\cite{Kobayashi93,WheelerMS93,KarmaR98,DebierreKCG03,Nestler05}, as well
as to the works by the present authors \cite{eck,vch,pfsi}. 
Solidification phenomena can also be described with the
help of cellular automata and we refer to
\cite{Reiter05,Libbrecht08,GravnerG09} for details.

\section{Two-phase flow} \label{sec:tpf}
We now study free boundary problems appearing in fluid flow. 
Here we consider two
immiscible fluids that are separated by an interface,
which moves in time with a velocity given by that of the fluid. We
start with the simplest situation in which the evolution in the two
phases is given by Stokes flow. For the discretization of the evolving
hypersurface we consider the parametric finite element approximation
introduced above. The good mesh properties are particularly helpful in
two-phase flow as the mesh, which is moved by the fluid velocity, 
can change quite drastically. Many other approaches, using a surface mesh, have
problems in such situations due the mesh deteriorating. This is not
the case, in most situations, for the approach stated below. 
Moreover, on using a simple XFEM pressure
space enrichment, we obtain exact volume conservation for the two phase regions.
Furthermore, our fully discrete finite element approximation can be shown to be
unconditionally stable.
A common feature in two-phase flow is that nonphysical velocities
appear in numerical approximations. If surface tension effects are
taken into account, a jump discontinuity in the pressure results, and
this poses serious challenges for the numerical method. This can lead
to nonphysical velocities. These so-called spurious velocities are
typically avoided by the XFEM approach discussed in this section.

\subsection{Two-phase Stokes flow} \label{subsec:stokes}
In this section we consider the flow of viscous, incompressible,
immiscible two fluid systems in a low Reynolds regime, i.e.\ we
can neglect inertia terms. The governing equations for the velocity $\vec u$ 
and the pressure $p$ are given by the momentum equation and conservation of
mass, i.e.\
\begin{equation} \label{eq:Stokes}
-\mu_\pm \,\Delta\,\vec u +\nabla\, p =\vec f\,,\qquad 
\nabla\,.\,\vec u =0 \qquad\text{in } \Omega_\pm(t)\,,
\end{equation}
where $\vec f$
is a possible forcing, and 
$\Omega_+(t)$ and $\Omega_-(t)$ are the time dependent regions
occupied by the two fluid phases as in Figure~\ref{fig:sketch} 
in \S\ref{subsec:evolve},
where $\Omega \subset \bR^d$ is again a fixed domain, with $d\geq2$.
As usual, we let $\vec\nu$ denote the outer unit normal to $\Omega_-(t)$
on $\Gamma(t) = \partial\Omega_-(t)$.
We consider the case in which the viscosity of the two fluids can be
different and introduce 
$\mu(\cdot,t) = \mu_+\,\charfcn{\Omega_+(t)} + \mu_-\,\charfcn{\Omega_-(t)}$,
with $\mu_\pm \in \bRplus$ denoting the fluid viscosities, 
recall \eqref{eq:fpm}. 
We now define the conditions that have to hold on the
interface. Therefore, we introduce the stress tensor 
\begin{equation} \label{eq:sigma}
\mat\sigma = \mu \,(\nabla\,\vec u + (\nabla\,\vec u)^\transT) - p\,\mat\Id
= 2\,\mu\, \mat D (\vec u)-p\,\mat\Id\,,
\end{equation}
where 
$\mat D(\vec u)=\tfrac12 \,(\nabla\vec u+(\nabla\vec u)^\transT)$ is the
rate of deformation tensor.
Using the fact that the velocity is divergence free, we can rewrite 
\eqref{eq:Stokes} as
\[
- \nabla\,.\,\mat \sigma = \vec f\,,\qquad 
\nabla\,.\,\vec u =0 \qquad\text{in } \Omega_\pm(t)\,.
\]
On the moving interface, we require
\begin{equation*}
\left[\vec u\right]_-^+ = \vec 0\,,\quad 
\big[\mat\sigma\,\vec\nu\big]_-^+= -\gamma_0\,\varkappa\,\vec\nu\,,\quad 
\mathcal{V} = \vec u\,.\,\vec\nu \qquad\text{on } \Gamma(t)\,, 
\end{equation*}
where $\gamma_0 \in \bRplus$ represents surface tension.
To close the system we prescribe initial data $\Gamma(0)$ and the boundary 
condition $\vec u = \vec 0$ on $\partial\Omega$.
Overall, we can rewrite the total system as follows. Given $\Gamma(0)$, find 
$\vec u : \Omega \times [0,T] \to \bR^d$,
$p : \Omega \times [0,T] \to \bR$
and the evolving interface $(\Gamma(t))_{t \in [0,T]}$
such that for all $t \in (0,T]$ the following conditions hold 
\begin{subequations} \label{eq:tf2}
\begin{align}
&- 2\,\mu\,\nabla\,.\,\mat D(\vec u) + \nabla\,p = \vec f\,, \ \
\nabla\,.\,\vec u = 0 \quad\text{in } \Omega_\pm(t)\,, 
\quad\vec u = \vec 0 \quad\mbox{on } \partial\Omega\,,
\label{eq:tf2a} \\
&\left[\vec u\right]_-^+ = \vec 0\,, \ \
\left[\left(2\,\mu\,\mat D(\vec u) - p\,\mat\Id \right) \vec\nu \right]_-^+ 
= -\gamma_0\,\varkappa\,\vec\nu\,,\ \
\mathcal{V} = \vec u\,.\,\vec\nu 
\quad\text{on } \Gamma(t)\,.
\label{eq:tf2b} 
\end{align}
\end{subequations}
A finite element approximation of \eqref{eq:tf2} needs
a suitable weak formulation. Here a choice has to be made regarding
the tangential velocity $\vec{\mathcal{V}}_\subT$, recall 
Definition~\ref{def:globalx}\ref{item:VT}, which will have important
repercussions on the discrete level. Given our results derived in 
Sections~\ref{sec:mc} and \ref{sec:surfdiff}, it is natural to not fix
$\vec{\mathcal{V}}_\subT$ explicitly and to use the formulations
\begin{equation} \label{eq:BGNweakStokes}
\vec{\mathcal{V}}\,.\,\vec\nu =\vec u\,.\,\vec\nu\quad\text{and}\quad
\varkappa\,\vec\nu = \Delta_s\,\vec\id
\qquad\text{on } \Gamma(t)\,,
\end{equation}
compare with \eqref{eq:BGNweakmcf}. An alternative approach would let
$\vec{\mathcal{V}}_\subT$ be given by the fluid flow, and so use the
formulations
\begin{equation*} 
\vec{\mathcal{V}} =\vec u\quad\text{and}\quad
\vec\varkappa = \Delta_s\,\vec\id
\qquad\text{on } \Gamma(t)\,,
\end{equation*}
compare with \eqref{eq:Dziukweakmcf}. We will consider this alternative
approach in \S\ref{subsubsec:DziukStokes}. However, our main focus is on the
formulation \eqref{eq:BGNweakStokes}, which will lead to good mesh properties
for the discrete schemes.

We begin with the following simple generalization of Remark~\ref{rem:dnufjump}
\begin{align}
& - \int_{\Omega_-(t)\cup\Omega_+(t)} \left(\nabla\,.\,\mat\sigma\right).
\,\vec\xi \dL{d} \nonumber \\ &
= 2\left(\mu\,\mat D(\vec u), \mat D(\vec\xi)\right)
- \left(p, \nabla\,.\,\vec\xi\right) + \left\langle  
\left[\mat\sigma\,\vec\nu\right]_-^+,\vec\xi \right\rangle_{\Gamma(t)}
\quad\forall\ \vec\xi \in [H^1_0(\Omega)]^d\,,
\label{eq:dnufjumpgen}
\end{align}
where we have noted for symmetric matrices $\mat A \in \bR^{d \times d}$
that $\mat A : \mat B = \mat A : \tfrac12\,(\mat B + \mat B^\transT)$ 
for all $\mat B \in \bR^{d \times d}$.  
Now (\ref{eq:dnufjumpgen}) leads to the following weak formulation of 
(\ref{eq:tf2}), using the formulations \eqref{eq:BGNweakStokes}.
Given a closed hypersurface $\Gamma(0) \subset \Omega$,
we seek an evolving hypersurface $(\Gamma(t))_{t\in[0,T]}$
that separates $\Omega$ into $\Omega_-(t)$ and $\Omega_+(t)$,
with a global parameterization and induced velocity field
$\vec{\mathcal{V}}$, and $\varkappa \in L^2(\GT)$ as well as 
$\vec u : \Omega \times [0,T] \to \bR^d$ and
$p : \Omega \times [0,T] \to \bR$ as follows. 
For almost all $t \in (0,T)$, find
$(\vec u(\cdot,t),p(\cdot,t),\vec {\mathcal V}(\cdot,t),\varkappa(\cdot,t))
\in [H^1_0(\Omega)]^d \times L^2(\Omega) \times [L^2(\Gamma(t))]^d
\times L^2(\Gamma(t))$ such that
\begin{subequations}
\label{eq:weak}
\begin{align}
& 2\left(\mu\,\mat D(\vec u), \mat D(\vec\xi)\right) 
- \left(p, \nabla\,.\,\vec\xi\right)
= \left(\vec f, \vec\xi\right) 
+ \gamma_0\left\langle \varkappa\,\vec\nu, \vec\xi \right\rangle_{\Gamma(t)}
\quad\forall\ \vec\xi \in [H^1_0(\Omega)]^d\,,\label{eq:weaka} \\
&\left(\nabla\,.\,\vec u, \varphi\right) = 0 \quad\forall\ \varphi \in 
L^2(\Omega) \,, \label{eq:weakb} \\
&  \left\langle \vec{\mathcal{V}}\,.\,\vec\nu,\chi \right\rangle_{\Gamma(t)}
  = \left\langle \vec u\,.\,\vec\nu,\chi \right\rangle_{\Gamma(t)} 
\quad\forall\ \chi \in L^2(\Gamma(t))\,, \label{eq:weakc} \\
& \left\langle \varkappa\,\vec\nu, \vec\eta \right\rangle_{\Gamma(t)}
+ \left\langle \nabs\,\vec{\id}, \nabs\,\vec\eta \right\rangle_{\Gamma(t)}
 = 0  \quad\forall\ \vec\eta \in \Vt\,. \label{eq:weakd}
\end{align}
\end{subequations}

\begin{rem} \label{rem:weak}
For a sufficiently smooth weak solution of \eqref{eq:weak}, we can formally
prove the following results.
\begin{enumerate}
\item 
The two-phase Stokes flow, in the absence of external forcings, 
decreases the surface area of the interface. 
This follows from the transport theorem, {\rm Theorem~\ref{thm:trans}},
\eqref{eq:weakc} with $\chi= \varkappa$, 
\eqref{eq:weaka} with $\xi=\vec u$ and
\eqref{eq:weakb} with $\varphi=p$, as one then obtains the energy identity
\begin{align}
\gamma_0\,\ddt\,|\Gamma(t)|
& =- \gamma_0\left\langle \varkappa\,\vec\nu, \vec{\mathcal{V}}
  \right\rangle_{\Gamma(t)}
=- \gamma_0\left\langle \varkappa\,\vec\nu, \vec u
  \right\rangle_{\Gamma(t)}
  \nonumber \\ &
=-2 \left( \mu \,\mat D(\vec u),\mat D(\vec u) \right) 
+ \left( \vec f,  \vec u \right) .\label{eq:testD1}
\end{align}
\item \label{item:weakii}
The two-phase Stokes flow preserves $\mathcal{L}^d(\Omega_\pm(t))$, as
the transport theorem, {\rm Theorem~\ref{thm:transvol}},
\eqref{eq:weakc} with $\chi=1$, the divergence theorem
and \eqref{eq:weakb} with $\varphi=\charfcn{\Omega_-(t)}$  
yield that
\begin{equation}
\ddt\,\mathcal{L}^d(\Omega_-(t)) = 
\left\langle \vec{\mathcal{V}} , \vec\nu \right\rangle_{\Gamma(t)}
=\left\langle\vec u,\vec\nu \right\rangle_{\Gamma(t)}
=\left( \nabla\,.\,\vec u, \charfcn{\Omega_-(t)} \right)
=0\,. \label{eq:conserved}
\end{equation}
\end{enumerate}
\end{rem}

\begin{rem} \label{rem:Fweak}
In the case $\vec f= \vec 0$ the evolution \eqref{eq:weak} decreases
the surface area of the interface $\Gamma(t)$, recall \eqref{eq:testD1}, 
and can be written in the general form \eqref{eq:Fgeomev} satisfying 
\eqref{eq:Fineq}.
This can be seen as follows. For a given interface $\Gamma(t)$, with normal 
$\vec\nu(t)$ and curvature $\varkappa(t)$, we solve
\eqref{eq:weaka}, \eqref{eq:weakb} for $(\vec u(\cdot,t),p(\cdot,t))
\in [H^1_0(\Omega)]^d \times L^2(\Omega)$ and set
\[ 
\mathfrak{F} (\varkappa) = \vec u \,.\, \vec\nu\,.
\]
We then compute, similarly to \eqref{eq:testD1}, 
\begin{equation}\label{eq:Fineq2phaseflow} 
\left\langle \mathfrak{F} (\varkappa), \varkappa \right\rangle_{\Gamma(t)} 
= \left\langle \vec u\,.\, \vec\nu, \varkappa \right\rangle_{\Gamma(t)}
= \frac 2\gamma_0 \left(\mu\,\mat D(\vec u), \mat D(\vec u)\right)
\geq 0\,.
\end{equation}
\end{rem}

\subsubsection{Finite element approximation} \label{subsubsec:stokesFEA}

In order to approximate the velocity and pressure on $\mathcal{T}^m$,
recall the assumptions and notation above (\ref{eq:Sm}), we introduce
finite element spaces $\uspace^m \subset [H^1_0(\Omega)]^d$ and
$\pspace^m\subset L^2(\Omega)$.
We require also the space
$\widehat\pspace^m= \pspace^m \cap \widehat\pspace$, 
where $\widehat\pspace = \{ \varphi \in L^2(\Omega) : (\varphi,1)=0\}$. The
velocity/pressure finite element spaces $(\uspace^m,\pspace^m)$
satisfy the LBB inf-sup condition if there exists a
$C \in \bRplus$ independent of $h^m = \max_{o \in \mathcal{T}^m}
\{ \diam(\sigmaO)\}$, such that
\begin{equation} \label{eq:LBB}
\inf_{\varphi \in \widehat\pspace^m} \sup_{\vec\xi \in \uspace^m}
\frac{\left( \varphi, \nabla\,.\,\vec\xi\right)}
{|\varphi|_\Omega\,\|\vec\xi\|_{1,\Omega}} \geq C > 0\,,
\end{equation}
where $\|\cdot\|_{1,\Omega}^2 = |\cdot|_\Omega^2 + |\nabla\,\cdot|_\Omega^2$ 
defines the $H^1$--norm on $\Omega$, see e.g.\ \citet[p.~114]{GiraultR86}. 

Next we introduce a generalization of (\ref{eq:Sm}) to
\begin{equation} \label{eq:calSmk}
\mathcal{S}^m_k = \left\{\chi \in C(\overline{\Omega}) : \chi_{\mid_{\sigmaO}}
\in \mathcal{P}_k(\sigmaO) \quad\forall\ 
\sigmaO \in \mathcal{T}^m\right\} \subset H^1(\Omega)\,,
\end{equation}
where, for $k \in \bN$, 
$\mathcal{P}_k(\sigmaO)$ denotes the space of polynomials of degree
$k$ on $\sigmaO$. In addition, we denote by $\mathcal{S}^m_0$
the space of piecewise constant functions on $\mathcal{T}^m$.
Then, for example,
we may choose the lowest order Taylor--Hood element P2-P1, the P2-P0 element, or
the P2-(P1+P0) element on setting
$\uspace^m = [\mathcal{S}^m_2]^d\cap\uspace$, and 
$\pspace^m = \mathcal{S}^m_1, \mathcal{S}^m_0$
or $\mathcal{S}^m_1 + \mathcal{S}^m_0$, respectively.  
It is well-known that the P2-P1 element 
satisfies the LBB condition (\ref{eq:LBB}) for $d=2$
and $d=3$, where the latter requires 
the weak constraint that all simplices
have a vertex in $\Omega$, see \cite{Boffi97}.
While the other two choices, P2-P0 and P2-(P1+P0),
satisfy it for $d=2$.

Moreover, as in earlier sections, we approximate the interface at time $t_m$
by a polyhedral surface $\Gamma^m$ with the same notation.
Similarly to \S\ref{subsec:one-sided},
let $\Omega^m_+$ denote the exterior of $\Gamma^m$ and 
$\Omega^m_-$ the interior of $\Gamma^m$, so that
$\Gamma^m = \partial\Omega^m_- = \overline{\Omega^m_-} \cap 
\overline{\Omega^m_+}$. For simplicity we assume that the unit normal
$\vec\nu^m$ on $\Gamma^m$ points into $\Omega^m_+$.
Due to the phase-dependent viscosity, $\mu_\pm$,
we now subdivide the elements of the bulk mesh $\mathcal{T}^m$ into exterior, 
interior and interfacial elements as follows. Let
\begin{equation} \label{eq:partT}
\mathcal{T}^m_\pm = \left\{ o \in \mathcal{T}^m : o \subset
\Omega^m_\pm \right\} , 
\quad\mathcal{T}^m_{\Gamma^m} = \left\{ o \in \mathcal{T}^m : o \cap
\Gamma^m \not = \emptyset \right\} .
\end{equation}
The disjoint partition $\mathcal{T}^m = \mathcal{T}^m_- \cup \mathcal{T}^m_+ 
\cup \mathcal{T}^m_{\Gamma^m}$ 
can easily be found e.g.\ with the Algorithm~4.1 in \cite{crystal}. 

Similarly to \S\ref{subsec:one-sided},
we use an unfitted finite element approximation of \arxivyesno{}{\linebreak}%
(\ref{eq:weak}).
We define the discrete viscosity
$\mu^m \in {\mathcal S}^m_0$, for $m\geq 0$, as
\begin{equation} \label{eq:rhoma}
\mu^m_{\mid_{o}} = \begin{cases}
\mu_- & o \in \mathcal{T}^m_-\,, \\
\mu_+ & o \in \mathcal{T}^m_+\,, \\
\tfrac12\,(\mu_- + \mu_+) & o \in \mathcal{T}^m_{\Gamma^m}\,.
\end{cases}
\end{equation}
Then the unfitted finite element approximation of (\ref{eq:weak}) 
from \cite{spurious} is given as follows.
Let the closed polyhedral hypersurface $\Gamma^0$ be an approximation of
$\Gamma(0)$ and recall the time interval partitioning \eqref{eq:ttaum}. 
Then, for $m=0,\ldots,M-1$, find 
$(\vec U^{m+1}, P^{m+1},\vec X^{m+1},$ $\kappa^{m+1})
\in \uspace^m \times \widehat\pspace^m \times \Vhm \times \Whm$ such that
\begin{subequations} \label{eq:HG}
\begin{align}
& 2\left(\mu^m\,\mat D(\vec U^{m+1}), \mat D(\vec\xi) \right)
- \left(P^{m+1}, \nabla\,.\,\vec\xi\right)
= \left(\vec f^{m+1}, \vec\xi\right)
+ \gamma_0\left\langle \kappa^{m+1}\,\vec\nu^m,
  \vec\xi\right\rangle_{\Gamma^m} \nonumber \\ & \hspace{7cm} 
\quad\forall\ \vec\xi \in \uspace^m \,, \label{eq:HGa}\\
& \left(\nabla\,.\,\vec U^{m+1}, \varphi\right)  = 0 
\quad\forall\ \varphi \in \widehat\pspace^m\,, \label{eq:HGb} \\
&  \left\langle \frac{\vec X^{m+1} - \vec\id}{\ttau_m}\,.\,\vec\nu^m ,
\chi \right\rangle_{\Gamma^m}^h
= \left\langle \vec U^{m+1}\,.\,\vec\nu^m, 
\chi \right\rangle_{\Gamma^m} \quad\forall\ \chi \in \Whm\,, \label{eq:HGc} \\
& \left\langle \kappa^{m+1}\,\vec\nu^m, \vec\eta \right\rangle_{\Gamma^m}^h
+ \left\langle \nabs\,\vec X^{m+1}, \nabs\,\vec\eta \right\rangle_{\Gamma^m}
 = 0 \quad\forall\ \vec\eta \in \Vhm \label{eq:HGd}
\end{align}
\end{subequations}
and set $\Gamma^{m+1} = \vec X^{m+1}(\Gamma^m)$.
Here, with the help of the interpolation operator 
$\vec I^m_2 : [C(\overline{\Omega})]^d
\to [\mathcal{S}^m_2]^d$, the natural  generalization of 
$I^m : C(\overline{\Omega}) \to S^m$,
we define $\vec f^{m+1} = \vec I^m_2\,\vec f(\cdot,t_{m+1})$.

We note that the  scheme (\ref{eq:HG}) is linear in the unknowns
$(\vec U^{m+1}, P^{m+1},$ $\vec X^{m+1}, \kappa^{m+1})$.
For the mathematical analysis of \eqref{eq:HG} it is convenient to introduce
a reduced version. Given $\uspace^m$ and $\pspace^m$, we define
\begin{equation} \label{eq:Um0}
\uspace^m_0 =\{ \vec U \in \uspace^m : (\nabla\,.\,\vec U, \varphi) = 0 
\quad\forall\ \varphi \in \widehat\pspace^m \} \,.
\end{equation}
Then the discrete pressure can be eliminated from \eqref{eq:HG} to yield the
following reduced variant. For $m=0,\ldots,M-1$, find 
$(\vec U^{m+1}, \vec X^{m+1},$ $\kappa^{m+1})
\in \uspace^m_0 \times \Vhm \times \Whm$ such that
\begin{subequations} \label{eq:redHG}
\begin{align}
& 2\left(\mu^m\,\mat D(\vec U^{m+1}), \mat D(\vec\xi) \right)
= \left(\vec f^{m+1}, \vec\xi\right)
+ \gamma_0\left\langle \kappa^{m+1}\,\vec\nu^m,
  \vec\xi\right\rangle_{\Gamma^m}
\quad\forall\ \vec\xi \in \uspace^m_0 \,, \label{eq:redHGa}\\
&  \left\langle \frac{\vec X^{m+1} - \vec\id}{\ttau_m}\,.\,\vec\nu^m ,
\chi \right\rangle_{\Gamma^m}^h
= \left\langle \vec U^{m+1}\,.\,\vec\nu^m, 
\chi \right\rangle_{\Gamma^m} \quad\forall\ \chi \in \Whm\,, 
\label{eq:redHGb} \\
& \left\langle \kappa^{m+1}\,\vec\nu^m, \vec\eta \right\rangle_{\Gamma^m}^h
+ \left\langle \nabs\,\vec X^{m+1}, \nabs\,\vec\eta \right\rangle_{\Gamma^m}
 = 0 \quad\forall\ \vec\eta \in \Vhm \label{eq:redHGc}
\end{align}
\end{subequations}
and set $\Gamma^{m+1} = \vec X^{m+1}(\Gamma^m)$.

We now show existence, uniqueness and stability results for the two
schemes \eqref{eq:redHG} and \eqref{eq:HG}.

\begin{thm} \label{thm:2phaseEU}
\rule{0pt}{0pt}
\begin{enumerate}
\item \label{item:redEU}
Let $\Gamma^m$ satisfy {\rm Assumption~\ref{ass:A}}. Then
there exists a unique solution
\arxivyesno{$(\vec U^{m+1}, \vec X^{m+1},$ $ \kappa^{m+1}) 
\in \uspace^m_0 \times \Vhm \times \Whm$}
{$(\vec U^{m+1}, \vec X^{m+1}, \kappa^{m+1}) 
\in \uspace^m_0 \times \Vhm \times \Whm$}
to \eqref{eq:redHG}.
\item \label{item:EredEU} 
If $(\vec U^{m+1}, P^{m+1}, \vec X^{m+1}, \kappa^{m+1}) 
\in \uspace^m\times \widehat\pspace^m \times \Vhm \times \Whm$ solves
\arxivyesno{}{\linebreak}%
\eqref{eq:HG}, then $(\vec U^{m+1}, \vec X^{m+1}, \kappa^{m+1})$
is a solution to \eqref{eq:redHG}.
\item \label{item:2phaseEU}
Let $(\uspace^m,\widehat \pspace^m)$ satisfy the LBB condition
{\rm (\ref{eq:LBB})} and let $\Gamma^m$ satisfy {\rm Assumption~\ref{ass:A}}.
Then there exists a unique solution
$(\vec U^{m+1}, P^{m+1}, \vec X^{m+1},$ $ \kappa^{m+1}) 
\in \uspace^m\times \widehat\pspace^m \times \Vhm \times \Whm$ to
\eqref{eq:HG}.
\end{enumerate}
\end{thm}
\begin{proof}
\ref{item:redEU}
Existence follows from uniqueness and hence we consider the homogeneous
linear system. 
Find $(\vec U, \vec X, \kappa) \in \uspace^m_0\times \Vhm \times \Whm$ 
such that
\begin{subequations} 
\begin{align}
&  2\left(\mu^m\,\mat D(\vec U), \mat D(\vec\xi) \right)
= \gamma_0\left\langle \kappa\,\vec\nu^m, \vec\xi\right\rangle_{\Gamma^m}
 \quad\forall\ \vec\xi \in \uspace^m_0 \,, \label{eq:proofa}\\
&  \left\langle \vec X\,.\, \vec\nu^m,
\chi \right\rangle_{\Gamma^m}^h
= \ttau_m \left\langle \vec U\,.\,\vec\nu^m , \chi\right\rangle_{\Gamma^m} 
 \quad\forall\ \chi \in \Whm\,, \label{eq:proofc} \\
& \left\langle \kappa\,\vec\nu^m, \vec\eta \right\rangle_{\Gamma^m}^h
+ \left\langle \nabs\,\vec X, \nabs\,\vec\eta \right\rangle_{\Gamma^m}
 = 0  \quad\forall\ \vec\eta \in \Vhm\,. \label{eq:proofd}
\end{align}
\end{subequations}
We now need to show that the zero solution is the only possible solution.
Choosing $\vec\xi=\ttau_m\,\vec U$ in (\ref{eq:proofa}),
$\chi = \gamma_0\,\kappa$ in (\ref{eq:proofc})
and $\vec\eta=\gamma_0\,\vec X$ in (\ref{eq:proofd}), we obtain
\[
2\,\ttau_m\left(\mu^m\,\mat D(\vec U), \mat D(\vec U) \right)
+ \gamma_0\left| \nabs\,\vec X\right|_{\Gamma^m}^2 
=0\,. 
\]
Now Korn's inequality, see e.g.\ \citet[\S62.15]{Zeidler88}, 
yields $\vec U = \vec 0$. 
Hence it follows from \eqref{eq:proofc}, \eqref{eq:proofd} and the proof of
Lemma~\ref{lem:exXk} that $\vec X = \vec 0$ and $\kappa= 0$. 
\\
\ref{item:EredEU}
The claim follows trivially from \eqref{eq:HGb} and the definition of
$\uspace^m_0$.
\\
\ref{item:2phaseEU}
Upon considering the homogeneous system, for $(\vec U, P, \vec X, \kappa) 
\in \uspace^m\times \widehat\pspace^m\times\Vhm \times \Whm$ it follows 
immediately from \ref{item:EredEU} and \ref{item:redEU} that 
$\vec U = \vec 0$, $\vec X = \vec 0$ and $\kappa = 0$.
Then the LBB inf-sup condition (\ref{eq:LBB}) yields that $P = 0$.
\end{proof}

\begin{thm} \label{thm:2phasestab}
Let $d=2$ or $d=3$.
Let $(\vec U^{m+1}, \vec X^{m+1}, \kappa^{m+1}) 
\in \uspace^m_0 \times \Vhm \times \Whm$ be a solution to \eqref{eq:redHG},
or let $(\vec U^{m+1}, P^{m+1}, \vec X^{m+1}, \kappa^{m+1}) 
\in \uspace^m\times \widehat\pspace^m \times \Vhm \times \Whm$ be a solution to
\eqref{eq:HG}. Then it holds that
\begin{equation} \label{eq:2phasestab}
 \gamma_0\left|\Gamma^{m+1}\right| 
+ 2\,\ttau_m \left(\mu^m\,\mat D(\vec U^{m+1}), \mat D(\vec U^{m+1})\right)
\leq \gamma_0\left|\Gamma^m\right|
+ \ttau_m \left( \vec f^{m+1}, \vec U^{m+1} \right).
\end{equation}
\end{thm}
\begin{proof}
It follows from Theorem~\ref{thm:2phaseEU}\ref{item:EredEU} that we only need
to consider \eqref{eq:redHG}.
Choosing $\vec\xi = \vec U^{m+1} \in \uspace^m_0$ in (\ref{eq:redHGa}), 
$\chi = \gamma_0\,\kappa^{m+1}$ in (\ref{eq:redHGb}) and
$\vec\eta=\gamma_0\,(\vec X^{m+1}-\vec\id_{\mid_{\Gamma^m}})$ 
in (\ref{eq:redHGc}) yields that
\begin{align*}
& 2\,\ttau_m\left(\mu^m\,\mat D(\vec U^{m+1}), \mat D(\vec U^{m+1}) \right)
+ \gamma_0\left\langle \nabs\,\vec X^{m+1}, \nabs\,(\vec X^{m+1} - \vec\id) 
\right\rangle_{\Gamma^m}
\nonumber \\ & \qquad 
= \ttau_m\left(\vec f^{m+1}, \vec U^{m+1} \right).
\end{align*}
Hence (\ref{eq:2phasestab}) follows immediately, on recalling
Lemma~\ref{lem:stab2d3d}.
\end{proof}

\begin{rem} \label{rem:stabstab}
\rule{0pt}{0pt}
\begin{enumerate}
\item 
The stability bound \eqref{eq:2phasestab} is a natural fully discrete 
analogue of \eqref{eq:testD1}.
\item
In the case $\vec f = \vec 0$, it is possible to derive the 
stability bound \eqref{eq:2phasestab} 
with the help of the general strategy from {\rm \S\ref{subsec:otherflows}},
on recalling \eqref{eq:Fineq2phaseflow}. 
To this end, for a given $\Gamma^m$ and $\kappa \in \Whm$, we determine
$\vec U \in \uspace^m_0$ as the unique solution to
\begin{equation} \label{eq:HGa2}
 2\left(\mu^m\,\mat D(\vec U), \mat D(\vec\xi) \right)
= \gamma_0\left\langle \kappa\,\vec\nu^m,
   \vec\xi\right\rangle_{\Gamma^m}
\quad\forall\ \vec\xi \in \uspace^m_0 \,.
\end{equation}
We then define $\mathfrak{F}^m (\kappa) \in \Whm$ such that 
$\left\langle \mathfrak{F}^m (\kappa),\chi \right\rangle_{\Gamma^m}^h
=\left\langle \vec U, \chi\,\vec\nu^m \right\rangle_{\Gamma^m}^h$
for all $\chi \in \Whm$. 
Choosing $\vec\xi =\vec U$ in \eqref{eq:HGa2} we obtain that
\[
\left\langle \mathfrak{F}^m (\kappa), \kappa \right\rangle_{\Gamma^m}^h
= \frac 2\gamma_0 \left(\mu^m\,\mat D(\vec U), \mat D(\vec U) \right)
\geq 0 \quad\forall\ \kappa \in \Whm\,.
\]
Now {\rm Theorem~\ref{thm:other}\ref{item:otherstab}} implies the
desired stability result.
\item
The scheme \eqref{eq:HG} leads to well-behaved meshes for 
$\Gamma^m$, $m=1,\ldots,M$, which follows as in 
{\rm \S\ref{subsec:equi}} by considering a semidiscrete version, see 
{\rm \citet[Remark~3]{spurious}}.
In particular, we obtain equidistribution in two space dimensions and 
conformal polyhedral hypersurfaces in three space dimensions.
\end{enumerate}
\end{rem}

\begin{rem}[Discrete linear systems] \label{rem:Stokessolve}
We recall the notations and definitions from {\rm Remark~\ref{rem:MSsolve}}.
Moreover, as is standard practice for the solution of linear systems 
arising from discretizations of (Navier--)Stokes equations, 
we avoid the complications of the constrained pressure space 
$\widehat\pspace^m$ by considering an overdetermined linear system 
with $\pspace^m$ instead. 
Introducing the obvious abuse of notation, the linear system \eqref{eq:HG},
with $\widehat\pspace^m$ replaced by $\pspace^m$, can be formulated as: 
Find $(\vec U^{m+1},P^{m+1},\kappa^{m+1},\delta\vec X^{m+1})\in
(\bR^d)^{K_{\uspace^m}}\times \bR^{K_{\pspace^m}} \times
\bR^K \times (\bR^d)^K$ such that
\begin{equation}
\begin{pmatrix}
 \mat B_\Omega & \vec C_\Omega & -\gamma_0\,\Nbulk & 0 \\
 \vec C^\transT_\Omega & 0 & 0 & 0 \\
 \NbulkT & 0 & 0 & -\frac1{\ttau_m}\,\vec N_{\Gamma^m}^\transT \\
0 & 0 & \vec N_{\Gamma^m} & \mat A_{\Gamma^m}
\end{pmatrix} 
\begin{pmatrix} \vec U^{m+1} \\ P^{m+1} \\ \kappa^{m+1} \\ 
\delta\vec X^{m+1} \end{pmatrix}
=
\begin{pmatrix} \mat M_\Omega\,\vec f^{m+1} \\ 0 \\
0 \\ -\mat A_{\Gamma^m}\,\vec X^m \end{pmatrix} \,,
\label{eq:Stokeslin}
\end{equation}
where $K_{\uspace^m}$ and $K_{\pspace^m}$ denote the degrees of freedom for
the finite element spaces $\uspace^m$ and $\pspace^m$, respectively.
The definitions of the matrices in \eqref{eq:Stokeslin} are either given in
\eqref{eq:mat0}, or they follow directly from 
\eqref{eq:HG}, see also {\rm \citet[\S5]{fluidfbp}} for details.

The overdetermined linear system \eqref{eq:Stokeslin} can either be solved
directly, with the help of a sparse $QR$ factorization method such as 
SPQR, see {\rm \cite{Davis11}}. Or it can be solved with the help of a Schur 
complement approach that eliminates  $(\kappa^{m+1}, \delta\vec X^{m+1})$ 
from \eqref{eq:Stokeslin}, and
then uses an iterative solver for the remaining system in $(\vec U^{m+1},
P^{m+1})$. This approach has the advantage that for the reduced system
well-known solution methods for finite element discretizations for the 
standard (Navier--)Stokes equations may be employed.
In particular, we let 
\[
\Xi_{\Gamma^m}= \begin{pmatrix}
 0 & - \frac1{\ttau_m}\,\vec N_{\Gamma^m}^\transT \\
\vec N_{\Gamma^m} & \mat A_{\Gamma^m}
\end{pmatrix} 
\]
and recall from {\rm Lemma~\ref{lem:exXk}} that if {\rm Assumption~\ref{ass:A}}
holds, then the matrix $\Xi_{\Gamma^m}$ is nonsingular. 
On defining 
$\mat T_\Omega = (\Nbulk \ 0)\,\Xi_{\Gamma^m}^{-1}\,
\begin{pmatrix} \NbulkT \\ 0 \end{pmatrix}$, we can 
reduce \eqref{eq:Stokeslin} to
\begin{equation} \label{eq:SchurkX}
\begin{pmatrix}
\vec B_\Omega + \gamma_0\,\mat T_\Omega
& \vec C_\Omega \\
\vec C_\Omega^\transT & 0 
\end{pmatrix}
\begin{pmatrix}
\vec U^{m+1} \\ P^{m+1} 
\end{pmatrix} 
= \begin{pmatrix}
\mat M_\Omega\,\vec f^{m+1} -\gamma_0\,(\Nbulk \ 0)\, \Xi_{\Gamma^m}^{-1}\,
\begin{pmatrix} 0 \\ \mat {A}_{\Gamma^m}\,\vec X^{m} \end{pmatrix} \\
0
\end{pmatrix}
\end{equation}
and
$\begin{pmatrix} \kappa^{m+1} \\ \delta\vec X^{m+1} \end{pmatrix} 
= \Xi_{\Gamma^m}^{-1}\,
\begin{pmatrix} -\NbulkT\,\vec U^{m+1} \\ -\mat A_{\Gamma^m}\,\vec X^m
\end{pmatrix}.$
The linear system \eqref{eq:SchurkX} can be sol\-ved, 
for example, with preconditioned GMRES iterative 
solvers for standard \arxivyesno{}{\linebreak}%
\mbox{(Navier--)}Stokes discretizations,
see e.g.\ {\rm \cite{ElmanSW05}} for some examples. 
For particular preconditioners for \eqref{eq:SchurkX} 
and further details on possible solution procedures, 
we refer to {\rm \citet[\S5]{fluidfbp}}.
\end{rem}

\subsubsection{Semidiscrete finite element approximation}
\label{subsubsec:sdStokes}

In this section we introduce a continuous-in-time semidiscrete variant of 
\arxivyesno{}{\linebreak}%
\eqref{eq:HG}. Similarly to \eqref{eq:Sm} and \eqref{eq:calSmk}, 
for a fixed regular partitioning $\mathcal{T}^h$ of $\Omega$, with
$\overline{\Omega}=\cup_{\sigmaO\in\mathcal{T}^h}\overline{\sigmaO}$,
we introduce the finite element spaces 
\begin{equation*} 
\mathcal{S}^h_k = \left\{
\chi \in C(\overline{\Omega}) : \chi_{\mid_{\sigmaO}} \in
\mathcal{P}_k(\sigmaO) \quad\forall\ \sigmaO \in {\mathcal{T}}^h\right\} 
\subset H^1(\Omega)\,, \qquad k \in \bN\,.
\end{equation*}
As before, we let $\mathcal{S}^h_0$ denote the space of  
piecewise constant functions on ${\mathcal{T}}^h$.
Let $\uspace^h\subset[H^1_0(\Omega)]^d$
and $\pspace^h(t)\subset L^2(\Omega)$ be the finite element spaces for the
semidiscrete velocity and pressure approximations, and set
$\widehat\pspace^h(t) = \pspace^h(t) \cap \widehat\pspace$.
Note that while $\uspace^h$ is fixed, 
for later developments we allow a time-dependent discrete pressure space
$\pspace^h(t)$, see \S\ref{subsubsec:XFEM} below.
In addition, we use the notation of \S\ref{subsec:ESFEM} for
evolving polyhedral surfaces, the corresponding finite element spaces,
and discrete time derivatives. In addition, we define
\begin{equation} \label{eq:pspacehT}
\widehat\pspace^h_T = \left\{
\varphi\in L^2(0,T;L^2(\Omega)): \varphi(t)\in \widehat\pspace^h(t)
\ \forall\ t\in(0,T]\right\}.
\end{equation}
Given $\Gamma^h(t)$, we denote by
$\Omega^h_+(t)$ the exterior of $\Gamma^h(t)$ and by
$\Omega^h_-(t)$ the interior of $\Gamma^h(t)$, so that
$\Gamma^h(t) = \partial\Omega^h_-(t) = \overline{\Omega^h_-(t)} \cap 
\overline{\Omega^h_+(t)}$. 
The elements of the bulk mesh $\mathcal{T}^h$ are partitioned
into interior, exterior and interfacial elements precisely as in
\eqref{eq:partT}, and the discrete viscosity $\mu^h(t) \in \mathcal{S}^h_0$ is
defined as the natural semidiscrete analogue of
\eqref{eq:rhoma}.

Then we can formulate the semidiscrete analogue of \eqref{eq:HG} as follows.
Given the closed polyhedral hypersurface $\Gamma^h(0)$,
find an evolving polyhedral hypersurface $\GhT$ 
with induced velocity $\vec{\mathcal{V}}^h \in \VhGhT$,
$\kappa^h \in \WhGhT$, 
$\vec U^h \in L^2(0,T; \uspace^h)$ and $P^h \in \widehat \pspace^h_T$
as follows. For all $t \in (0,T]$, find \arxivyesno{}{\linebreak}%
$(\vec U^h(\cdot,t), P^h(\cdot,t),\vec{\mathcal{V}}^h(\cdot,t),
\kappa^h(\cdot,t))\in \uspace^h\times\widehat\pspace^h(t)\times
\Vht\times\Wht$ such that
\begin{subequations} \label{eq:sdHG}
\begin{align}
& 2\left(\mu^h\,\mat D(\vec U^h), \mat D(\vec\xi) \right)
- \left(P^h, \nabla\,.\,\vec\xi\right)
= \left(\vec f^h, \vec\xi\right)
+ \gamma_0\left\langle \kappa^h\,\vec\nu^h,
  \vec\xi\right\rangle_{\Gamma^h(t)} 
\quad\forall\ \vec\xi \in \uspace^h \,, \label{eq:sdHGa}\\
& \left(\nabla\,.\,\vec U^h, \varphi\right)  = 0 
\quad\forall\ \varphi \in \widehat\pspace^h(t)\,, \label{eq:sdHGb} \\
&  \left\langle \vec{\mathcal{V}}^h\,.\,\vec\nu^h ,
\chi \right\rangle_{\Gamma^h(t)}^h
= \left\langle \vec U^h\,.\,\vec\nu^h, 
\chi \right\rangle_{\Gamma^h(t)} \quad\forall\ \chi \in \Wht\,, \label{eq:sdHGc} \\
& \left\langle \kappa^h\,\vec\nu^h, \vec\eta \right\rangle_{\Gamma^h(t)}^h
+ \left\langle \nabs\,\vec\id, \nabs\,\vec\eta \right\rangle_{\Gamma^h(t)}
 = 0 \quad\forall\ \vec\eta \in \Vht\,, \label{eq:sdHGd}
\end{align}
\end{subequations}
where $\vec f^h$ is the natural semidiscrete analogue of the fully discrete 
forcings $\vec f^{m+1}$, $m= 0,\ldots,M-1$.

\begin{thm} \label{thm:sdStokes}
Let $(\GhT,\kappa^h,\vec U^h,P^h)$ be a solution of \eqref{eq:sdHG}.
\begin{enumerate}
\item \label{item:sdStokesstab}
It holds that
\[
\gamma_0\,\ddt\left|\Gamma^h(t)\right| +
2\left( \mu^h\,\mat D(\vec U^h),\mat D(\vec U^h)\right)
= \left( \vec f^h, \vec U^h\right) .
\]
\item \label{item:sdStokesvol}
If $\charfcn{\Omega_-^h(t)} \in \pspace^h(t)$, then
it holds that
\[
\ddt\,\mathcal{L}^d(\Omega^h_-(t)) = 0\,.
\] 
\item \label{item:sdStokesTM}
For any $t \in (0,T]$, it holds that
$\Gamma^h(t)$ is a conformal polyhedral surface.
In particular, for $d=2$, any two neighbouring elements of the curve 
$\Gamma^h(t)$ either have equal length, or they are parallel.
\end{enumerate}
\end{thm}
\begin{proof}
\ref{item:sdStokesstab}
Similarly to the proof of Theorem~\ref{thm:semidis}\ref{item:semidisstab},
choosing $\vec\xi = \vec U^h(\cdot,t) \in \uspace^h$, $\varphi = P^h(\cdot,t)
\in \widehat\pspace^h(t)$, $\chi = \kappa^h(\cdot,t)\in \Wht$ 
$\vec\eta = \vec{\mathcal{V}}^h(\cdot, t) \in \Vht$ in \eqref{eq:sdHG} gives
\begin{align*}
\gamma_0\,\ddt\left|\Gamma^h(t)\right| &
= \gamma_0 \left\langle \nabs\,\vec\id, \nabs\,\vec{\mathcal{V}}^h 
 \right\rangle_{\Gamma^h(t)} 
= -\gamma_0 \left\langle \vec{\mathcal{V}}^h,
\kappa^h\,\vec\nu^h\right\rangle_{\Gamma^h(t)}^h \nonumber \\ &
= -\gamma_0
\left\langle \vec U^h, \kappa^h\,\vec\nu^h\right\rangle_{\Gamma^h(t)} 
= \left( \vec f^h, \vec U^h\right) 
- 2\left( \mu^h\,\mat D(\vec U^h),\mat D(\vec U^h)\right) ,
\end{align*}
which is the claim. \\
\ref{item:sdStokesvol}
Similarly to the proof of Theorem~\ref{thm:SDSD}\ref{item:SDSDvol},
choosing $\chi = 1$ in \eqref{eq:sdHGc} 
and $\varphi = \charfcn{\Omega_-^h(t)} -
\frac{\mathcal{L}^d(\Omega_-^h(t))}{\mathcal{L}^d(\Omega)} \in 
\widehat\pspace^h(t)$ in \eqref{eq:sdHGb} yields,
on using the divergence theorem, that
\begin{align*}
\ddt\,\mathcal{L}^d(\Omega^h_-(t)) &
= \left\langle \vec{\mathcal{V}}^h,\vec\nu^h\right\rangle_{\Gamma^h(t)}^h 
= \left\langle \vec U^h,\vec\nu^h\right\rangle_{\Gamma^h(t)}
= \left( \nabla\,.\,\vec U^h, \charfcn{\Omega^h_-(t)} \right) \nonumber\\ &
= \left( \nabla\,.\,\vec U^h, \charfcn{\Omega^h_-(t)} 
- \frac{\mathcal{L}^d(\Omega_-^h(t))}{\mathcal{L}^d(\Omega)} \right)
= 0\,,
\end{align*}
in a discrete analogue to \eqref{eq:conserved}.\\
\ref{item:sdStokesTM}
This follows directly from Definition~\ref{def:conformal} and
Theorem~\ref{thm:equid}. 
\end{proof}

\subsubsection{\XFEMGamma\ for conservation of the phase volumes} 
\label{subsubsec:XFEM}
Conservation of the total mass, equivalent to the conservation
of $\mathcal{L}^d(\Omega_-(t))$, (\ref{eq:conserved}), 
is clearly a desirable property on the discrete level.
We have seen in Theorem~\ref{thm:sdStokes}\ref{item:sdStokesvol}
that the semidiscrete scheme \eqref{eq:sdHG} conserves
$\mathcal{L}^d(\Omega^h_-(t))$ only if the time-dependent discrete pressure
spaces $\pspace^h(t)$ contain the characteristic function of the discrete 
inner phase $\charfcn{\Omega_-^h(t)}$ for all $t \in [0,T]$.
Hence, for the fully discrete approximation \eqref{eq:HG} 
it was suggested in \citet[\S3.4]{spurious} to extend
the pressure space $\pspace^m$ by one single basis function, namely
$\charfcn{\Omega_-^m}$. There, we referred to this as the
\XFEMGamma\ approach, because the extra contributions to \eqref{eq:HGa}
and \eqref{eq:HGb} coming from $\charfcn{\Omega_-^m} -
\frac{\mathcal{L}^d(\Omega_-^m)}{\mathcal{L}^d(\Omega)} \in \widehat\pspace^m$
can be written in terms of integrals over $\Gamma^m$, on noting
from the divergence theorem that
\begin{equation} \label{eq:XFEMGamma}
\left(\nabla\,.\,\vec\xi,\charfcn{\Omega_-^m}-
 \frac{\mathcal{L}^d(\Omega_-^m)}{\mathcal{L}^d(\Omega)} \right) 
= \left(\nabla\,.\,\vec\xi,\charfcn{\Omega_-^m}\right) 
= \left\langle \vec\nu^m, \vec\xi \right\rangle_{\Gamma^m} 
\quad\forall\ \vec\xi \in \uspace^m\,.
\end{equation}
For the fully discrete approximation \eqref{eq:HG},
even with the \XFEMGamma\ pressure space extension,
it is not possible to show that the total mass is conserved.
However, we observe that combining \eqref{eq:XFEMGamma}, 
with $\vec\xi=\vec U^{m+1}$, \eqref{eq:HGb} and \eqref{eq:HGc} leads to
\begin{equation*} 
\left\langle \vec X^{m+1} - \vec\id, \vec\nu^m \right\rangle_{\Gamma^m} = 0\,,
\end{equation*}
which means that in practice this fully discrete approximation
conserves the volume of the two phases well,
see also Remark~\ref{rem:SDSD}\ref{item:remSDSDii}.

Moreover, it turns out that the \XFEMGamma\ approach avoids spurious 
velocities. To make this precise, we state the following theorem.
\begin{thm} \label{thm:xfem}
\rule{0pt}{0pt}
\begin{enumerate}
\item \label{item:xfemi}
Let $d=2$ or $d=3$.
Let $(\vec U^{m+1},\vec X^{m+1}, \kappa^{m+1}) \in
\uspace^m\times \Vhm\times\Whm$ 
be a solution to \eqref{eq:redHG} with $\vec f^{m+1} = \vec 0$.
If $\vec X^{m+1} = \vec\id_{\mid_{\Gamma^m}}$, then $\vec U^{m+1} = \vec 0$. 
\item \label{item:xfemii}
Let $\charfcn{\Omega_-^m} \in \pspace^m$,
let $\Gamma^m$ satisfy {\rm Assumption~\ref{ass:A}}, 
and let $\Gamma^m$ be a polyhedral surface with constant discrete mean
curvature, i.e.\ there exists a constant
$\overline{\kappa}\in \bR$ such that
\begin{equation*} 
\overline{\kappa} \left\langle\vec\nu^m,\vec\eta\right\rangle_{\Gamma^m}
+\left\langle
\nabs\, \vec\id,\nabs\,\vec\eta\right\rangle_{\Gamma^m}=0
\quad\forall\ \vec\eta\in\Vhm\,.
\end{equation*}
Then $\Gamma^m$ is a conformal polyhedral surface,
recall {\rm Definition~\ref{def:conformal}}, and
\arxivyesno{$(\vec U^{m+1}, \vec X^{m+1},$ $ \kappa^{m+1}) = 
(\vec 0, \vec X^m, \overline{\kappa})\in\uspace^m_0\times\Vhm\times \Whm$}
{$(\vec U^{m+1}, \vec X^{m+1}, \kappa^{m+1}) = 
(\vec 0, \vec X^m, \overline{\kappa})\in\uspace^m_0\times\Vhm\times \Whm$}
is the unique solution to \eqref{eq:redHG} with $\vec f^{m+1} = \vec 0$. 
\item \label{item:xfemiii}
Let the assumptions in \ref{item:xfemii} hold and let 
$(\uspace^m,\widehat \pspace^m)$ satisfy the LBB condition \eqref{eq:LBB}.
Then $\Gamma^m$ is a conformal polyhedral surface, and
\arxivyesno{$(\vec U^{m+1},$ $ P^{m+1}, \vec X^{m+1}, \kappa^{m+1}) = $
\linebreak $(\vec 0,-\gamma_0\,\overline\kappa\,\bigl[
\charfcn{\Omega_-^m} - \frac{\mathcal{L}^d(\Omega_-^m)}{\mathcal{L}^d(\Omega)}
\bigr],\vec\id_{\mid_{\Gamma^m}}, 
\overline{\kappa})\in\uspace^m\times\widehat\pspace^m\times
\Vhm\times \Whm$}
{$(\vec U^{m+1},$ $ P^{m+1}, \vec X^{m+1}, \kappa^{m+1}) = 
(\vec 0,-\gamma_0\,\overline\kappa\,\bigl[
\charfcn{\Omega_-^m} - \frac{\mathcal{L}^d(\Omega_-^m)}{\mathcal{L}^d(\Omega)}
\bigr],\vec\id_{\mid_{\Gamma^m}}, 
\overline{\kappa})\in\uspace^m\times\widehat\pspace^m\times
\Vhm\times \Whm$}
is the unique solution to \eqref{eq:HG} with $\vec f^{m+1} = \vec 0$.
\end{enumerate}
\end{thm}
\begin{proof}
\ref{item:xfemi}
It follows from Theorem~\ref{thm:2phasestab} that
the solution fulfills \eqref{eq:2phasestab} with $\Gamma^{m+1}$ 
replaced by $\Gamma^m$ and $\vec f^{m+1} = \vec 0$. 
Hence we obtain 
\arxivyesno{$(\mu^m\,\mat D (\vec U^{m+1}), \mat D(\vec U^{m+1})) = 0$,} 
{$(\mu^m\,\mat D (\vec U^{m+1}), $ \linebreak $\mat D(\vec U^{m+1}))= 0$,}
and so Korn's inequality implies $\vec U^{m+1} = \vec 0$. 
\\
\ref{item:xfemii}
It immediately follows from \eqref{eq:intnuh} that $\Gamma^m$
is a conformal polyhedral surface.
Theorem~\ref{thm:2phaseEU}\ref{item:redEU} implies that in order to establish 
the remaining result, 
we only need to show that $(\vec U^{m+1}, \vec X^{m+1}, \kappa^{m+1}) = 
(\vec 0, \vec\id_{\mid_{\Gamma^m}}, \overline{\kappa})$ is a solution to 
\eqref{eq:redHG} with $\vec f^{m+1} = \vec 0$. 
But this follows immediately from
$\overline{\kappa}\,\langle\vec\nu^m,\vec\eta\rangle_{\Gamma^m}
= \langle\overline{\kappa}\,\vec\nu^m,\vec\eta\rangle_{\Gamma^m}^h$
for all $\vec\eta \in \Vhm$,
and
\begin{equation*}
\left\langle \vec\nu^m, \vec\xi \right\rangle_{\Gamma^m} 
=\left(\nabla\,.\,\vec\xi,\charfcn{\Omega_-^m}-
 \frac{\mathcal{L}^d(\Omega_-^m)}{\mathcal{L}^d(\Omega)}
\right)=0
\quad\forall\ \vec\xi \in \uspace^m_0\,,
\end{equation*}
where we have 
recalled \eqref{eq:XFEMGamma} and \eqref{eq:Um0}.
\\
\ref{item:xfemiii}
On recalling Theorem~\ref{thm:2phaseEU}\ref{item:2phaseEU}, the proof is
analogous to the proof of \ref{item:xfemii}. It holds that
\begin{align*}
&
\gamma_0\left\langle \kappa^{m+1}\,\vec\nu^m, \vec\xi\right\rangle_{\Gamma^m} 
+ \left(P^{m+1}, \nabla\,.\,\vec\xi\right) \nonumber \\ & \qquad
= \gamma_0\,\overline\kappa 
\left\langle \vec\nu^m, \vec\xi\right\rangle_{\Gamma^m} 
-\gamma_0\,\overline\kappa\left(
\charfcn{\Omega_-^m} - \frac{\mathcal{L}^d(\Omega_-^m)}{\mathcal{L}^d(\Omega)},
\nabla\,.\,\vec\xi\right)
= 0 \quad\forall\ \vec\xi \in \uspace^m\,,
\end{align*}
and this proves the claim.
\end{proof}

\begin{rem} \label{rem:xfem}
It follows from {\rm Theorem~\ref{thm:xfem}} that, 
independently of the choice of $\mu_\pm$, no spurious velocities appear for 
discrete stationary solutions, $\Gamma^{m+1} = \Gamma^m$.
Moreover, for the \XFEMGamma\ approach it holds that polyhedral surfaces with
constant discrete mean curvature are discrete stationary solutions.
In particular, spherical bubbles can be approximated by such polyhedral 
surfaces, and so our method admits a stationary solution with zero velocity
in these situations.
This is not the case for many other discretizations and is one of the
reasons for spurious velocities in simple situations like a spherical bubble.
\end{rem}

\subsubsection{Approximations based on the fluidic tangential velocity} 
\label{subsubsec:DziukStokes}

Let us briefly discuss an alternative approximation of two-phase Stokes flow,
that is based on a weak formulation of \eqref{eq:tf2a} and
\[
\left[\vec u\right]_-^+ = \vec 0\,, \ \
\left[\left(2\,\mu\,\mat D(\vec u) - p\,\mat\Id \right) \vec\nu \right]_-^+ 
= -\gamma_0\,\vec\varkappa\,,\ \
\vec\varkappa = \Delta_s\,\vec\id\,,\ \
\vec{\mathcal{V}} = \vec u \quad\text{on } \Gamma(t)
\]
as opposed to \eqref{eq:tf2} with \eqref{eq:BGNweakStokes}.
The semidiscrete finite element approximation, in line with \eqref{eq:sdHG},
then features the equations
\begin{subequations} \label{eq:sdGD}
\begin{align}
& 2\left(\mu^h\,\mat D(\vec U^h), \mat D(\vec\xi) \right)
- \left(P^h, \nabla\,.\,\vec\xi\right)
= \left(\vec f^h, \vec\xi\right)
+ \gamma_0\left\langle \vec\kappa^h, \vec\xi\right\rangle_{\Gamma^h(t)}^h 
\quad\forall\ \vec\xi \in \uspace^h \,, \label{eq:sdGDa}\\
& \left(\nabla\,.\,\vec U^h, \varphi\right)  = 0 
\quad\forall\ \varphi \in \widehat\pspace^h(t)\,, \label{eq:sdGDb} \\
&  \left\langle \vec{\mathcal{V}}^h, \vec\chi \right\rangle_{\Gamma^h(t)}^h
= \left\langle \vec U^h, \vec\chi \right\rangle_{\Gamma^h(t)}^h 
\quad\forall\ \vec\chi \in \Vht\,, \label{eq:sdGDc} \\
& \left\langle \vec\kappa^h, \vec\eta \right\rangle_{\Gamma^h(t)}^h
+ \left\langle \nabs\,\vec\id, \nabs\,\vec\eta \right\rangle_{\Gamma^h(t)}
 = 0 \quad\forall\ \vec\eta \in \Vht\,, \label{eq:sdGDd}
\end{align}
\end{subequations}
for $(\vec U^h(\cdot,t), P^h(\cdot,t),\vec{\mathcal{V}}^h(\cdot,t),
\vec\kappa^h(\cdot,t))\in \uspace^h\times\widehat\pspace^h(t)\times
\Vht\times\Vht$. We note that a variant of \eqref{eq:sdGD} without numerical
integration can also be considered, see \citet[\S3.6]{spurious} for details on
the fully discrete case.
It is a simple matter to prove
that solutions to \eqref{eq:sdGD} satisfy the stability result
Theorem~\ref{thm:sdStokes}\ref{item:sdStokesstab}. However, since
$\vec\nu^h(\cdot,t)$ is not a valid test function in \eqref{eq:sdGDc}, it is not
possible to prove the volume conservation result in 
Theorem~\ref{thm:sdStokes}\ref{item:sdStokesvol} for a solution of
\eqref{eq:sdGD}, even if $\charfcn{\Omega_-^h(t)} \in \pspace^h(t)$.
However, on choosing $\vec\chi = \vec\omega^h(\cdot,t) \in \Vht$
in \eqref{eq:sdGDc}, it follows
from Theorem~\ref{thm:disctransvol} and \eqref{eq:intnuh} that
\begin{equation} \label{eq:virtualGD}
\ddt\,\mathcal{L}^d(\Omega^h_-(t)) 
= \left\langle \vec{\mathcal{V}}^h,\vec\nu^h\right\rangle_{\Gamma^h(t)}^h 
= \left\langle \vec{\mathcal{V}}^h,\vec\omega^h\right\rangle_{\Gamma^h(t)}^h
= \left\langle \vec U^h, \vec\omega^h \right\rangle_{\Gamma^h(t)}^h .
\end{equation}
Hence, by enforcing the needed condition
\[
\left\langle \vec U^h, \vec\omega^h \right\rangle_{\Gamma^h(t)}^h = 0
\]
directly, together with a suitable Lagrange multiplier $\Psing^h(t) \in \bR$,
we can introduce the following semidiscrete approximation of two-phase Stokes
flow that satisfies Theorem~\ref{thm:sdStokes}\ref{item:sdStokesstab},
\ref{item:sdStokesvol}. Note that in order for \eqref{eq:virtualGD} to hold, it
is crucial to employ numerical integration on $\Gamma^h(t)$ throughout.

Given the closed polyhedral hypersurface $\Gamma^h(0)$,
find an evolving polyhedral hypersurface $\GhT$ 
with induced velocity $\vec{\mathcal{V}}^h \in \VhGhT$,
$\vec\kappa^h \in \VhGhT$, 
$\vec U^h \in L^2(0,T; \uspace^h)$, $P^h \in \widehat \pspace^h_T$
and $\Psing^h \in L^2(0,T; \bR)$ as follows. 
For all $t \in (0,T]$, find 
\arxivyesno{%
$(\vec U^h(\cdot,t), P^h(\cdot,t), \Psing^h(t), \vec{\mathcal{V}}^h(\cdot,t),$
$\vec\kappa^h(\cdot,t))\in \uspace^h\times\widehat\pspace^h(t)\times\bR\times
\Vht\times\Vht$}
{%
$(\vec U^h(\cdot,t), P^h(\cdot,t), \Psing^h(t), \vec{\mathcal{V}}^h(\cdot,t),
\vec\kappa^h(\cdot,t))\in \uspace^h\times\widehat\pspace^h(t)\times\bR\times
\Vht\times\Vht$}
such that
\begin{subequations} \label{eq:sdGD2}
\begin{align}
& 2\left(\mu^h\,\mat D(\vec U^h), \mat D(\vec\xi) \right)
- \left(P^h, \nabla\,.\,\vec\xi\right)
- \Psing^h \left\langle\vec\omega^h, \vec\xi \right\rangle_{\Gamma^h(t)}^h
\nonumber \\ & \hspace{3.5cm}
= \left(\vec f^h, \vec\xi\right)
+ \gamma_0\left\langle \vec\kappa^h, \vec\xi\right\rangle_{\Gamma^h(t)}^h
\quad\forall\ \vec\xi \in \uspace^h \,, \label{eq:sdGD2a}\\
& \left(\nabla\,.\,\vec U^h, \varphi\right)  = 0 
\quad\forall\ \varphi \in \widehat\pspace^h(t)
\qquad\text{and}\qquad
\left\langle\vec U^h, \vec\omega^h \right\rangle_{\Gamma^h(t)}^h = 0
\,, \label{eq:sdGD2b} \\
&  \left\langle \vec{\mathcal{V}}^h, \vec\chi \right\rangle_{\Gamma^h(t)}^h
= \left\langle \vec U^h, \vec\chi \right\rangle_{\Gamma^h(t)}^h 
\quad\forall\ \vec\chi \in \Vht\,, \label{eq:sdGD2c} \\
& \left\langle \vec\kappa^h, \vec\eta \right\rangle_{\Gamma^h(t)}^h
+ \left\langle \nabs\,\vec\id, \nabs\,\vec\eta \right\rangle_{\Gamma^h(t)}
 = 0 \quad\forall\ \vec\eta \in \Vht\,. \label{eq:sdGD2d}
\end{align}
\end{subequations}
We note that in terms of pressure space enrichment, the above procedure
may be viewed as a virtual element method, see e.g.\ \cite{BeiraoBCMMR13}.

\subsection{Two-phase Navier--Stokes flow} \label{subsec:NS}
In \cite{fluidfbp} the present authors extended the approximation 
\eqref{eq:HG} to two-phase Navier--Stokes flow, which is given by the model 
\eqref{eq:tf2} with the first equation in \eqref{eq:tf2a} replaced by
\[
\rho\,(\partial_t\,\vec u + (\vec u \,.\,\nabla)\,\vec u)
- 2\,\mu\,\nabla\,.\,\mat D(\vec u) + \nabla\,p = \vec f
\quad\text{in } \Omega_\pm(t)\,, 
\]
where 
$\rho(\cdot,t) = \rho_+\,\charfcn{\Omega_+(t)}+\rho_-\,\charfcn{\Omega_-(t)}$,
with $\rho_\pm \in \bRgeq$ denoting the two fluid densities, and with the
additional initial condition 
$\rho(\cdot,0)\,\vec u(\cdot,0) = \rho(\cdot,0)\,\vec u_0$ in $\Omega$.
The treatment of
the interface evolution, and its coupling to the quantities in the bulk,
remains unchanged, and for the approximation of the fluid flow in the bulk
standard techniques for the finite element approximation of one-phase
Navier--Stokes flow can be employed, see e.g.\ \cite{Temam01}. 
In the following, we recall the fully discrete approximation from 
\cite{fluidfbp}, which is based on the weak formulation of 
two-phase Navier--Stokes flow defined by \eqref{eq:weak} with the additional
terms
\begin{equation} \label{eq:addNS}
\tfrac12\, \ddt \left(\rho\,\vec u, \vec\xi\right) + 
\tfrac12\left(\rho\,\partial_t\,\vec u, \vec\xi\right)
+ \tfrac12\left(\rho, [(\vec u\,.\,\nabla)\,\vec u]\,.\,\vec\xi
- [(\vec u\,.\,\nabla)\,\vec\xi]\,.\,\vec u\right)
\end{equation}
on the left hand side of \eqref{eq:weaka}, recall \citet[(3.9)]{fluidfbp}. 

Let $\rho^m \in \mathcal{S}^m_0$ be defined analogously to \eqref{eq:rhoma},
for $m \geq 0$, and set $\rho^{-1} = \rho^0$. 
In addition we define the standard projection
operator $I^m_0:L^1(\Omega)\to \mathcal{S}^m_0$, such that
$(I^m_0 \eta)_{\mid_{o}} = \frac1{\mathcal{L}^d(o)}\,\int_{o}
\eta \dL{d}$ for all $o \in \mathcal{T}^m$.
In this section, we consider the partitioning 
$t_m = m\,\ttau$, $m=0,\ldots,M$, of $[0,T]$ into uniform time steps 
$\ttau = \frac TM$. Uniform time steps are required in order to be able
to introduce a consistent fully discrete approximation
of the time derivative terms terms in \eqref{eq:addNS}.
Let the closed polyhedral hypersurface $\Gamma^0$ be an approximation of
$\Gamma(0)$, and let $\vec U^0 \in \uspace^0$ be an approximation to 
$\vec u_0$. 
Then, for $m=0,\ldots,M-1$, find 
$(\vec U^{m+1}, P^{m+1},\vec X^{m+1},$ $\kappa^{m+1})
\in \uspace^m \times \widehat\pspace^m \times \Vhm \times \Whm$ such that
\begin{subequations} \label{eq:NSHG}
\begin{align}
& \tfrac12 \left( \frac{\rho^m\,\vec U^{m+1} - (I^m_0\,\rho^{m-1})
\,\vec I^m_2\,\vec U^m}{\ttau}
+(I^m_0\,\rho^{m-1}) \,\frac{\vec U^{m+1}- \vec I^m_2\,\vec U^m}{\ttau}, 
\vec\xi \right)
 \nonumber \\ & \
+ \tfrac12\left(\rho^m, 
 [(\vec I^m_2\,\vec U^m\,.\,\nabla)\,\vec U^{m+1}]\,.\,\vec\xi
- [(\vec I^m_2\,\vec U^m\,.\,\nabla)\,\vec\xi]\,.\,\vec U^{m+1} \right)
\nonumber \\ & \
+ 2\left(\mu^m\,\mat D(\vec U^{m+1}), \mat D(\vec\xi) \right)
- \left(P^{m+1}, \nabla\,.\,\vec\xi\right)
= \left(\vec f^{m+1}, \vec\xi\right)
+ \gamma_0\left\langle \kappa^{m+1}\,\vec\nu^m,
   \vec\xi\right\rangle_{\Gamma^m} \nonumber \\
& \hspace{7cm}\quad\forall\ \vec\xi \in \uspace^m \,, \label{eq:NSHGa}\\
& \left(\nabla\,.\,\vec U^{m+1}, \varphi\right)  = 0 
\quad\forall\ \varphi \in \widehat\pspace^m\,,
\label{eq:NSHGb} \\
&  \left\langle \frac{\vec X^{m+1} - \vec\id}{\ttau}\,.\,\vec\nu^m ,
\chi \right\rangle_{\Gamma^m}^h
= \left\langle \vec U^{m+1}\,.\,\vec\nu^m, 
\chi \right\rangle_{\Gamma^m} 
 \quad\forall\ \chi \in \Whm\,,
\label{eq:NSHGc} \\
& \left\langle \kappa^{m+1}\,\vec\nu^m, \vec\eta \right\rangle_{\Gamma^m}^h
+ \left\langle \nabs\,\vec X^{m+1}, \nabs\,\vec\eta \right\rangle_{\Gamma^m}
 = 0  \quad\forall\ \vec\eta \in \Vhm\,,
\label{eq:NSHGd}
\end{align}
\end{subequations}
and set $\Gamma^{m+1} = \vec X^{m+1}(\Gamma^m)$.
Clearly, in the case $\rho_- = \rho_+ = 0$, the approximation \eqref{eq:NSHG} 
collapses to the scheme \eqref{eq:HG}, with uniform time steps,
for two-phase Stokes flow.

\begin{thm} \label{thm:NSEU}
\rule{0pt}{0pt}
\begin{enumerate}
\item \label{item:NSEU}
Let $(\uspace^m,\widehat \pspace^m)$ satisfy the LBB condition
{\rm (\ref{eq:LBB})}, let $\Gamma^m$ satisfy {\rm Assumption~\ref{ass:A}}
and let $\vec U^m \in [C(\overline\Omega)]^d$.
Then there exists a unique solution
$(\vec U^{m+1}, P^{m+1}, \vec X^{m+1},$ $ \kappa^{m+1}) 
\in \uspace^m\times \widehat\pspace^m \times \Vhm \times \Whm$ to 
\eqref{eq:NSHG}.
\item \label{item:NSstab}
Let $d=2$ or $d=3$. Then a solution to \eqref{eq:NSHG} satisfies
\begin{align}
& \tfrac12\left( \rho^m, |\vec U^{m+1}|^2 \right) + \gamma_0\,|\Gamma^{m+1}|
+ \tfrac12\left(I^m_0\rho^{m-1},|\vec U^{m+1} - \vec I^m_2\,\vec U^m|^2 
\right) \nonumber \\ & \qquad\quad
+ 2\,\ttau\left(\mu^m\,\mat D(\vec U^{m+1}), \mat D(\vec U^{m+1}) \right)
\nonumber \\ & \quad
\leq \tfrac12\left( I^m_0\,\rho^{m-1}, |\vec I^m_2\,\vec U^m|^2 \right) 
+ \gamma_0\,|\Gamma^m| + \ttau\left( \vec f^{m+1}, \vec U^{m+1} \right).
\label{eq:NSstab}
\end{align}
\end{enumerate}
\end{thm}
\begin{proof}
\ref{item:NSEU}
The result can be shown as in the proof of Theorem~\ref{thm:2phaseEU}. \\
\ref{item:NSstab}
We choose $\vec\xi = \vec U^{m+1}$ in (\ref{eq:NSHGa}), 
$\varphi = P^{m+1}$ in (\ref{eq:NSHGb}), 
$\chi = \gamma_0\,\kappa^{m+1}$ in (\ref{eq:NSHGc}) and
$\vec\eta=\gamma_0\,(\vec X^{m+1}-\vec\id_{\mid_{\Gamma^m}})$ in 
(\ref{eq:NSHGd}) to obtain
\begin{align*}
& \tfrac12\left(\rho^m\,\vec U^{m+1}, \vec U^{m+1}\right)
+ \tfrac12\left((I^m_0\,\rho^{m-1})\,(\vec U^{m+1} - \vec I^m_2\,\vec U^m), 
\vec U^{m+1} - \vec I^m_2\,\vec U^m \right) \nonumber \\ & \qquad
+ 2\,\ttau\left(\mu^m\,\mat D(\vec U^{m+1}), \mat D(\vec U^{m+1}) \right)
+ \gamma_0\left\langle \nabs\,\vec X^{m+1}, \nabs\,(\vec X^{m+1} - \vec\id) 
\right\rangle_{\Gamma^m} \nonumber \\ & \hspace{1cm}
= \tfrac12\left((I^m_0\,\rho^{m-1})\,\vec I^m_2\,\vec U^{m}, 
 I^m_2\,\vec U^{m}\right)
+ \ttau\left( \vec f^{m+1}, \vec U^{m+1} \right) .
\end{align*}
Hence (\ref{eq:NSstab}) follows immediately, on recalling
Lemma~\ref{lem:stab2d3d}.
\end{proof}

\begin{rem} \label{rem:NSHG}
\rule{0pt}{0pt}
\begin{enumerate}
\item
If $d=2$ or $d=3$ then, on assuming that
\[
\left(I^m_0\rho^{m-1},|\vec I^m_2\,\vec U^m|^2\right) \leq
\left(\rho^{m-1},|\vec U^m|^2\right)
\quad\text{for $m=1,\ldots, M-1$,}
\]
we can prove an unconditional stability bound for the scheme \eqref{eq:NSHG},
see {\rm \citet[Theorem~4.2]{fluidfbp}}. 
The condition is always satisfied if no bulk 
mesh coarsening in time is performed.
\item
If
\[
\charfcn{\Omega_-^m} \in \pspace^m \quad\text{for $m=0,\ldots, M-1$,}
\]
then a semidiscrete continuous-in-time version of \eqref{eq:NSHG} conserves the
volume of the two phase exactly, which follows from the direct discrete
analogue of {\rm Remark~\ref{rem:weak}\ref{item:weakii}}, as
discussed previously in {\rm \S\ref{subsubsec:XFEM}}.
\item
The scheme \eqref{eq:NSHG} leads to well-behaved meshes, which follows as 
\arxivyesno{}{\linebreak}%
usual by considering a semidiscrete version, see {\rm \S\ref{subsec:equi}}.
In particular, we obtain equidistribution in two space dimensions and 
conformal polyhedral hypersurfaces in three space dimensions.
\item \label{item:NSHGbc}
It is a simple matter to extend the scheme \eqref{eq:NSHG}, and hence
\eqref{eq:HG}, to more general boundary conditions than $\vec u = \vec 0$ 
on $\partial\Omega$ for the fluid flow. Apart from this no-slip condition, 
also free-slip and stress-free boundary conditions,
as well as their inhomogeneous analogues, may be considered. 
See {\rm \cite{spurious,fluidfbp,nsns}} for details.
\item 
The discrete linear systems arising from \eqref{eq:NSHG} can be solved as
described in {\rm Remark~\ref{rem:Stokessolve}}.
\end{enumerate}
\end{rem}

In Figure~\ref{fig:fluidfbp} we show some numerical results for a 
generalization
of the scheme \eqref{eq:NSHG} to include, for example, free-slip boundary 
conditions on parts of the boundary $\partial\Omega$ and gravitational forces
$\vec f = \rho\,\vec f_1$.
In Figure~\ref{fig:fluidfbp} we show the interface of a rising bubble
together with a visualization of the fluid flow for three different simulations
from \cite{fluidfbp}. The two 2D simulations have density values 
$10\,\rho_- = \rho_+ = 10^3$ and 
$10^3\,\rho_- = \rho_+ = 10^3$, respectively, while the 3D simulation
has $10\,\rho_- = \rho_+ = 10^3$.
As the density of the inner fluid is chosen smaller than the density of the 
outer fluid in each case, the bubble rises in the presence of gravity. 
\begin{figure}
\center
\arxivyesno{
\newcommand\lwidth{0.25\textwidth} 
}{
\newcommand\lwidth{0.24\textwidth} 
}
\includegraphics[angle=0,width=\lwidth]{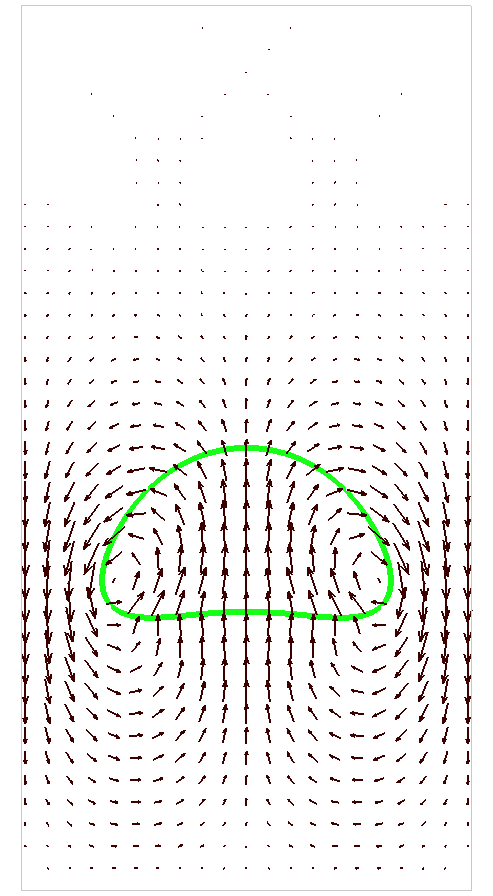} \qquad
\includegraphics[angle=0,width=\lwidth]{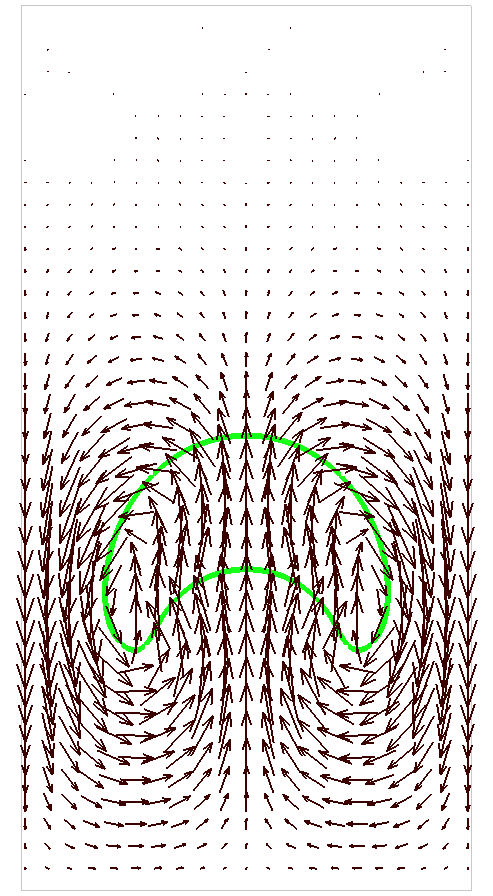} \qquad
\includegraphics[angle=0,width=\lwidth]{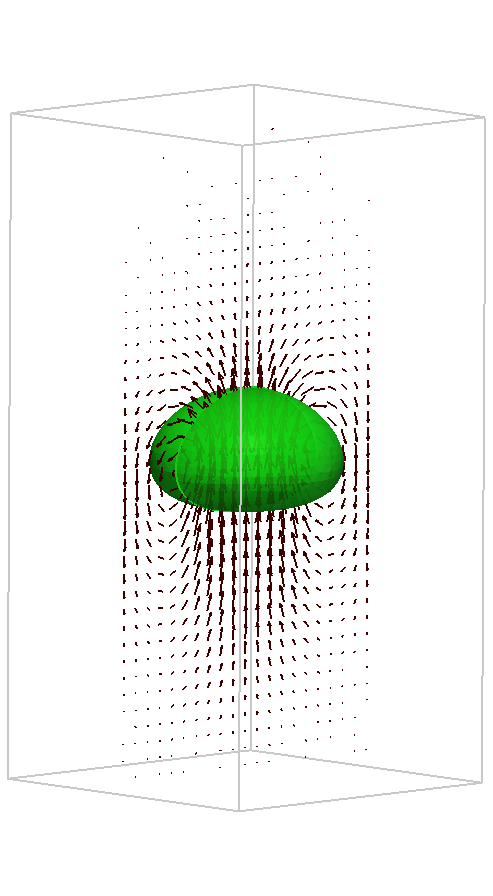}
\caption{Visualization of the numerical results for
the rising droplet experiments shown in 
Figures 3, 7 and 11 in \cite{fluidfbp}. Each plot shows the interface 
$\Gamma^m$ and the velocity $\vec U^m$ at time $t=1.5$. 
For the 3D experiment, the fluid velocity is only visualized within a 2D cut 
through $\Omega$.}
\label{fig:fluidfbp}
\end{figure}%

\subsection{Alternative numerical approaches}
\label{subsec:alternativeNS}

Numerical methods based on interface tracking methods using an indicator
function to describe the
interface are also popular methods to numerically solve two-phase flow
problems. The volume of fluid (VOF)
method uses a characteristic function of one of the phases to evolve the
interface and has been used by
\cite{HirtN81} and \cite{RenardyR02}. Another interface tracking method is
the level set method, which uses a level set function to track the interface. 
We refer to \cite{SussmanSO94} and to \cite{GrossR11} and the references 
therein for details.
Phase field methods, which are also called diffuse interface methods in this
context, have been studied numerically by
\cite{KimKL04,KaySW08,GrunK14,GarckeHK16}.
Other parametric methods, which use a polyhedral mesh to directly represent the
interface, are discussed in
\cite{UnverdiT92,Bansch01,Tryggvason_etal01,GanesanMT07,%
AgneseN16,AgneseN19}.

It is possible to generalize the approximation \eqref{eq:NSHG} to the case when 
surfactants are present. Then the surface tension $\gamma_0$ depends on the
local concentration of surface active agents on the moving interface. The cases
of insoluble and soluble surfactants have been considered by the authors in
\cite{tpfs} and \cite{solsurf}, respectively.
Other approaches to two phase flow with
surfactants are discussed in \cite{JamesL04,GrossR11,GanesanT12,AlandHKN17}
and the references therein.

\section{Willmore flow} \label{sec:willmore_flow}

\subsection{Derivation of the flow}

Willmore flow is the $L^2$--gradient flow of the Willmore energy 
\begin{equation*}
  E(\Gamma) = \tfrac12 \int_\Gamma \varkappa^2 \dH{d-1}\,,
\end{equation*}
for a sufficiently smooth hypersurface $\Gamma$ in $\bR^d$, $d\geq2$.
We remark that in the case $d=2$ this evolution law is often called 
elastic flow. In order to derive Willmore flow, we need the first variation of 
$E(\Gamma)$, which is given in the following lemma.
Here we make use of the notations and conventions introduced in
\S\ref{subsec:evolve}. 

\begin{lem} \label{lem:Willmorevar}
Let 
$\GT$ be a closed  $C^4$--evolving orientable hypersurface.
Then it holds that
\begin{equation*}
\ddt\,E(\Gamma(t)) = \left\langle 
\Delta_s\,\varkappa + \varkappa\,|\nabs\,\vec\nu|^2-\tfrac12\,
\varkappa^3, \mathcal{V} \right\rangle_{\Gamma(t)} .
\end{equation*}
\end{lem}
\begin{proof} 
Using Theorem~\ref{thm:trans}, Lemma~\ref{lem:derkappa}\ref{item:lem10.3ii}
and Remark~\ref{rem:ibp}\ref{item:ibp}, we compute
\begin{align}
\ddt\,E(\Gamma(t)) &= 
\left\langle \varkappa, \matpartn\,\varkappa
 -\tfrac12\,\varkappa^2\,\mathcal{V} \right\rangle_{\Gamma(t)} 
= \left\langle \varkappa , \Delta_s\,\mathcal{V} 
 + \mathcal{V}\,|\nabs\,\vec\nu|^2-\tfrac12\,\varkappa^2\,\mathcal{V} 
\right\rangle_{\Gamma(t)} \nonumber \\ &
= \left\langle \Delta_s\,\varkappa 
+ \varkappa\,|\nabs\,\vec\nu|^2-\tfrac12\,\varkappa^3
, \mathcal{V} \right\rangle_{\Gamma(t)} ,
\label{eq:WEvar}
\end{align}
where we have used the fact that $\Gamma(t)$ has no boundary.
\end{proof}

Hence we obtain that an evolving hypersurface $(\Gamma(t))_{t\in[0,T]}$, with
\begin{equation}
  \mathcal{V} = -\Delta_s\,\varkappa - \varkappa\,|\nabs\,\vec\nu|^2 +
  \tfrac12\,\varkappa^3
\qquad\text{on } \Gamma(t)\,,
\label{eq:WF}
\end{equation}
most efficiently decreases the Willmore energy, and this evolution law is
called Willmore flow. Therefore, (\ref{eq:WF}) is the $L^2$--gradient 
flow of the Willmore energy.  

\subsection{A finite element approximation of Willmore flow}
\label{subsec:1.9.2}
We begin with finite element approximations of Willmore flow from
\cite{triplej,willmore} or in the spirit of those papers. 
They are based on the following formulation of Willmore flow
\begin{equation} \label{eq:SWF}
\vec{\mathcal{V}}\,.\,\vec\nu = 
-\Delta_s\,\varkappa - \varkappa\,|\nabs\,\vec\nu|^2 + \tfrac12\, \varkappa^3, 
\qquad
\varkappa\,\vec\nu = \Delta_s\,\vec\id
\qquad\text{on } \Gamma(t)\,.
\end{equation}
Comparing \eqref{eq:strongSD} with \eqref{eq:SWF}, we note that
once a suitable approximation of $|\nabs\,\vec\nu|^2$ is given, then 
it is a simple matter to extend the techniques in Section~\ref{sec:surfdiff} 
in order to derive finite element approximations for Willmore flow.
For $d=2$, $|\nabs\,\vec\nu|^2$ collapses to $\varkappa^2$, recall Lemma 
\ref{lem:nabsnu}\ref{item:nabsnu2}, and so we can consider the scheme in 
\citet[\S2.3]{triplej}. In general, we
rely on one of the approximations in \eqref{eq:matWh}
to obtain a discrete approximation of the Weingarten map $\nabs\,\vec\nu$.
Hence, we introduce the following finite element
approximations for this formulation of Willmore flow. 
Let the closed polyhedral hypersurface $\Gamma^0$ be an approximation to
$\Gamma(0)$, and let $\kappa^0_{\Gamma^0} \in V(\Gamma^0)$ 
be an approximation to its mean curvature.
We also recall the time interval partitioning \eqref{eq:ttaum}. 
Then, for $m = 0,\ldots,M-1$, first find
$\mat W^{m+1}\in\matVhm$, or $\mat W^{m+1}\in\matVchm$, as an approximation 
of the Weingarten map on $\Gamma^m$,
and then find $(\vec X^{m+1},\kappa^{m+1}) \in \Vhm\times \Whm$ such that 
\begin{subequations}
\label{eq:Will}
\begin{align}
& \left\langle \frac{\vec X^{m+1}-\vec\id}{\ttau_m}, \chi\,\vec\nu^m\right\rangle_{\Gamma^m}^h
- \left\langle \nabs\, \kappa^{m+1}, \nabs\, \chi \right\rangle_{\Gamma^m} 
- \tfrac12 \left\langle(\kappa^m_{\Gamma^m})^2 \,\kappa^{m+1}, 
\chi \right\rangle_{\Gamma^m}^h
\nonumber \\ & \hspace*{3cm}
  = - \left\langle \kappa^m_{\Gamma^m} \,|\mat W^{m+1}|^2,\chi\right\rangle_{\Gamma^m}^h
\quad\forall\ \chi \in \Whm\,, \label{eq:Willa}  \\
& \left\langle \kappa^{m+1}\,\vec\nu^m,\vec\eta\right\rangle^h_{\Gamma^m}
+ \left\langle \nabs\,\vec X^{m+1}, \nabs\,\vec\eta 
\right\rangle_{\Gamma^m} = 0
\quad\forall\ \vec\eta \in \Vhm \label{eq:Willb}
\end{align}
\end{subequations}
and set $\Gamma^{m+1} = \vec X^{m+1}(\Gamma^m)$ and
$\kappa^{m+1}_{\Gamma^{m+1}} = \kappa^{m+1} \circ (\vec X^{m+1})^{-1} \in
\Whmp$.
For the definition of $\mat W^{m+1}$ we may, for example, choose one of the
formulations \eqref{eq:matWh}, based on $\Gamma^m$ and possibly
$\kappa^m_{\Gamma^m}$ or 
$\vec\kappa^m = \vec\pi_{\Gamma^m}[\kappa^m_{\Gamma^m}\,\vec\omega^m]$.

\begin{thm}\label{thm:uniqueWF}
Let $\Gamma^m$ satisfy {\rm Assumption \ref{ass:A}\ref{item:assA1}}, 
let $\kappa^m_{\Gamma^m} \in \Whm$ and 
$\mat W^{m+1}\in\matVhm$, or $\mat W^{m+1}\in\matVchm$, be given.
Then there exists a unique solution
$(\vec X^{m+1},\kappa^{m+1}) \in \Vhm\times \Whm$ to \eqref{eq:Will}. 
\end{thm}
\begin{proof}
The desired result follows similarly to the proof of Theorem~\ref{thm:uniqueSD}.
\end{proof}

\begin{rem} \label{rem:BGN08WF}
\rule{0pt}{0pt}
\begin{enumerate}
\item
Similarly to {\rm \S\ref{subsec:vol}}, a semidiscrete variant of the scheme 
\eqref{eq:Will} can also be considered, and it will satisfy 
{\rm Theorem~\ref{thm:SDSD}\ref{item:SDSDTM}}.
\item Similarly to {\rm \S\ref{subsec:TM}}, one can consider 
a variant of the scheme \eqref{eq:Will} with reduced or induced 
tangential motion.
\item The discrete linear systems arising at each time level of
\eqref{eq:Will} are very similar to the ones induced by \eqref{eq:sdfea}.
They can be solved, for example, with the help of a sparse factorization
package such as UMFPACK, see {\rm \cite{Davis04}}.
\end{enumerate}
\end{rem}

\subsection{A stable approximation of Willmore flow}\label{subsec:stableappro}
Unfortunately, it does not seem possible to prove a stability result for 
the fully discrete approximation (\ref{eq:Will}) of (\ref{eq:SWF}). 
However, the important paper \cite{Dziuk08} introduced a stable semidiscrete
finite element approximation of Willmore flow. The discretization is
based on an alternative formulation of the first variation of the
Willmore energy, which leads to a weak formulation of Willmore
flow. 
In order to derive the weak formulation, we prove the following result, which is
inspired by \citet[Lemma 3]{Dziuk08}. 

\begin{lem} \label{lem:Willmorevar2}
Let $\GT$ be a closed $C^3$--evolving orientable hypersurface.
Let $\vec{\mathcal{V}}$ be the velocity field induced by a  
global parameterization of $\GT$.
Then it holds that
\[
\ddt\, E(\Gamma(t)) =
- \left\langle \tfrac12\,|\vec\varkappa|^2 + \nabs\,.\,\vec\varkappa,
\nabs\,.\,\vec{\mathcal{V}}\right\rangle_{\Gamma(t)} 
+ \left\langle \nabs\,\vec\varkappa, 2\,\mat D_s(\vec{\mathcal{V}}
)-\nabs\,\vec{\mathcal{V}}\right\rangle_{\Gamma(t)} ,
\]
where $\mat D_s(\vec{\mathcal{V}})$ is the rate of deformation tensor,
recall {\rm Definition~\ref{def:globalx}\ref{item:DsV}}. 
\end{lem}
\begin{proof} 
It follows from Definition~\ref{def:2.8}\ref{item:def2.8iii} and
Theorem~\ref{thm:trans} that
\begin{equation} \label{eq:willdiff1}
\ddt\,E(\Gamma(t)) = \tfrac12\,\ddt \left\langle
|\vec\varkappa|^2,1 \right\rangle_{\Gamma(t)}
= \left\langle \vec\varkappa,
\matpartx\,\vec\varkappa\right\rangle_{\Gamma(t)}
+\tfrac12\left\langle|\vec\varkappa|^2,
\nabs\,.\,\vec{\mathcal{V}}\right\rangle_{\Gamma(t)} .
\end{equation}
We now slightly modify the arguments in the proof of \citet[Lemma~3]{Dziuk08}, 
in order to compute
the term $\left\langle\vec\varkappa, \matpartx\,\vec\varkappa
\right\rangle_{\Gamma(t)}$.

Let $\vec x : \Upsilon \times [0,T] \to \bR^d$ be a global parameterization of
$\GT$ as described in Definition~\ref{def:globalx}\ref{item:globalx}. 
For a given $\vec\eta_0\in [C^1(\Upsilon)]^d$, we define
$\vec\eta : \GT \to \bR^d$ via
\begin{equation}
\vec\eta  (\vec x(\vec z, t), t) 
= \vec\eta_0(\vec z) \quad\mbox{ for all } \quad 
\forall\ (\vec z,t) \in \Upsilon \times [0,T]\,,
\label{eq:etaflow}
\end{equation}
i.e.\ the values of $\vec\eta$ are transported with the map
$\vec x$. Hence it follows from Definition~\ref{def:2.18}\ref{item:def2.18i} 
that
\begin{equation} \label{eq:dteta0}
\matpartx\,\vec\eta = \vec 0 \qquad\text{on } \GT\,.
\end{equation}
On recalling Remark~\ref{rem:ibp}\ref{item:weakvarkappa},
Definition \ref{def:2.8}\ref{item:def2.8iii} and a density argument, we have that
\begin{equation}\label{eq:epscurv}
\left\langle \vec\varkappa, \vec\eta\right\rangle_{\Gamma(t)} +
\left\langle\nabs\,\vec\id, \nabs\,\vec\eta\right\rangle_{\Gamma(t)}
= 0  \quad\forall\ \vec\eta \in \Vt\,,
\end{equation}
since we assume $\Gamma(t)$ to be closed.
Differentiating \eqref{eq:epscurv} with respect to $t$, 
we obtain from Theorem~\ref{thm:trans}, on noting (\ref{eq:dteta0}), that
\begin{equation}
  \left\langle \matpartx\,\vec\varkappa , \vec\eta\right\rangle_{\Gamma(t)}
  + \left\langle\vec\varkappa \,.\,\vec\eta ,
  \nabs\,.\,\vec{\mathcal{V}}\right\rangle_{\Gamma(t)} + 
  \ddt \left\langle\nabs\,\vec\id,
  \nabs\,\vec\eta\right\rangle_{\Gamma(t)}
  = 0\,.
\label{eq:epscurva}  
\end{equation}
Using Lemma~\ref{lem:nabsid}\ref{item:nabsidnabsf}, 
Theorem~\ref{thm:trans},
Lemma~\ref{lem:commdiff}\ref{item:commdiffiii} and (\ref{eq:dteta0}), 
we now compute
\begin{align}\label{eq:dtdivdisc}
& \ddt \left\langle\nabs\,\vec\id, \nabs\,\vec\eta\right\rangle_{\Gamma(t)}
= \ddt \left\langle\nabs\,.\,\vec\eta,1 \right\rangle_{\Gamma(t)} 
\nonumber\\ & \qquad
= \left\langle \nabs\,.\,\vec\eta, 
\nabs\,.\,\vec{\mathcal{V}}\right\rangle_{\Gamma(t)}
+ \left\langle \nabs\,\vec\eta,
 \nabs\,\vec{\mathcal{V}}- 2\,\mat D_s(\vec{\mathcal{V}}) 
\right\rangle_{\Gamma(t)}.
\end{align}
On combining (\ref{eq:epscurva}) and (\ref{eq:dtdivdisc}), and then 
choosing $\vec\eta = \vec\varkappa$, we obtain
\[
\left\langle
\matpartx\,\vec\varkappa,\vec\varkappa
\right\rangle_{\Gamma(t)} =
-\left\langle |\vec\varkappa|^2 + \nabs\,.\,\vec\varkappa , 
\nabs\,.\,\vec{\mathcal{V}} \right\rangle_{\Gamma(t)} 
- \left\langle \nabs\,\vec\varkappa, \nabs\,\vec{\mathcal{V}} - 
2\,\mat D_s(\vec{\mathcal{V}}) \right\rangle_{\Gamma(t)}.
\]
Together with \eqref{eq:willdiff1} this yields the desired result.
\end{proof}

A weak formulation of Willmore flow based on Lemma~\ref{lem:Willmorevar2} 
is then the following.
Given a closed hypersurface $\Gamma(0)$, we seek an evolving hypersurface 
$(\Gamma(t))_{t\in[0,T]}$, 
with a global parameterization and induced velocity field
$\vec{\mathcal{V}}$, and $\vec\varkappa \in [L^2(\GT)]^d$ as follows. 
For almost all $t \in (0,T)$, find 
$(\vec {\mathcal V}(\cdot,t),\vec\varkappa(\cdot,t))
\in [L^2(\Gamma(t))]^d \times \Vt$ such that
\begin{subequations}
\label{eq:Dziukwill}
\begin{align}
&\left\langle \vec{\mathcal{V}}, \vec\chi \right\rangle_{\Gamma(t)} =
\left\langle \nabs\,\vec\varkappa , \nabs\,\vec\chi 
- 2 \,\mat D_s(\vec{\chi}) \right\rangle_{\Gamma(t)}
+\left\langle \nabs\,.\, \vec\varkappa, \nabs\,.\, \vec\chi 
\right\rangle_{\Gamma(t)} 
\nonumber \\
& \hspace{3cm}
+ \tfrac12\left\langle |\vec\varkappa|^2, 
\nabs\,.\,\vec\chi \right\rangle_{\Gamma(t)}
\quad\forall\ \vec\chi \in \Vt \,, \label{eq:Dziukwilla} \\
&\left\langle\vec\varkappa,\vec\eta\right\rangle_{\Gamma(t)} +
    \left\langle\nabs\,\vec\id, \nabs\,\vec\eta\right\rangle_{\Gamma(t)} =
    0\quad\forall\ \vec\eta\in \Vt\,,
\label{eq:Dziukwillb}    
\end{align}
\end{subequations}
where we have noted 
a density argument.
 
\begin{rem}[Comparison to \cite{Dziuk08}]
We note that our notation is such that 
$\nabs\,\vec\chi = (\nabla_{\!\Gamma}\,\vec\chi)^\transT$, with 
$\nabla_{\!\Gamma}\,\vec\chi$
defined as in {\rm \citet[(2.4)]{Dziuk08}}. 
In addition, our $\mat D_s(\vec\chi) = 
\tfrac12\,\mat P_\Gamma\,\mat D_\Gamma(\vec\chi)\,\mat P_\Gamma$,
where we recall {\rm Definition \ref{def:2.5}\ref{item:def2.5vi}}
and $\mat D_\Gamma(\vec\chi) = \nabla_{\!\Gamma}\,\vec\chi +
(\nabla_{\!\Gamma}\,\vec\chi)^\transT$ as defined in 
{\rm \citet[(3.14)]{Dziuk08}}.  
Hence, it is easily deduced from {\rm Remark \ref{rem:2.6}\ref{item:Pnabs}}
and {\rm Lemma \ref{lem:nabsid}\ref{item:nabsid}} that
$2\,\nabs\,\vec f : \mat D_s(\vec\chi) = \nabla_{\!\Gamma}\, \vec f : 
\mat D_\Gamma(\vec\chi) \, \nabla_{\!\Gamma}\,\vec\id$ for all
$\vec f,\,\vec\chi \in [H^1(\Gamma(t)]^d$,
which implies that \eqref{eq:Dziukwill} agrees with
{\rm \citet[Problem~2]{Dziuk08}}.
\end{rem}

We now relate the weak formulation (\ref{eq:Dziukwill}) to the
strong formulation (\ref{eq:SWF}).

\begin{lem}\label{lem:WFweakstrong}
A sufficiently smooth solution  of {\rm (\ref{eq:Dziukwill})} is a solution of the strong 
  formulation {\rm (\ref{eq:SWF})} of 
  Willmore flow.
\end{lem}
\begin{proof} First of all, (\ref{eq:Dziukwillb}) and 
Remark \ref{rem:ibp}\ref{item:weakvarkappa} imply that 
$\vec\varkappa = \varkappa\,\vec\nu$ 
and the second equation in (\ref{eq:SWF}) holds.

Next, we have, on noting Definition~\ref{def:2.5}\ref{item:def2.5vi},
Remark~\ref{rem:2.6}\ref{item:Pnabs}, 
Lemma \ref{lem:productrules}\ref{item:priii} 
and Lemma \ref{lem:nabsnu}\ref{item:nabsnuT}, that 
\[
\nabs\,\vec\varkappa :  
\left(\nabs\,\vec\chi-2\,\mat D_s(\vec\chi)\right) 
= (\vec\nu \otimes \nabs\,\varkappa) : \nabs\,\vec\chi
-\varkappa\,\nabs\,\vec\nu : \nabs\,\vec\chi\,.
\]
Hence, it follows from Remark~\ref{rem:ibp}\ref{item:ibp},\ref{item:ibpvec}
and Lemma~\ref{lem:productrules}\ref{item:prii},\ref{item:priii} that 
the first term on the right hand side of (\ref{eq:Dziukwilla}) can be 
rewritten as
\begin{equation} \label{eq:GDWF1}
\left \langle \nabs\,\vec\varkappa ,
\nabs\,\vec\chi - 2\,\mat D_s(\vec\chi) \right \rangle_{\Gamma(t)}
=\left\langle \varkappa\,\Delta_s\,\vec\nu - (\Delta_s\,\varkappa)\,\vec\nu,
\vec\chi\right\rangle_{\Gamma(t)}.
\end{equation}
In addition, Definition~\ref{def:2.5}\ref{item:def2.5ii} and 
Lemma~\ref{lem:varkappa}\ref{item:varkappa} yield that 
\arxivyesno{$\nabs\,.\,\vec\varkappa =
\nabs\,.\,(\varkappa\,\vec\nu) = \varkappa\,\nabs\,.\,\vec\nu =-\varkappa^2$.}  
{$\nabs\,.\,\vec\varkappa =$\linebreak 
$\nabs\,.\,(\varkappa\,\vec\nu) = \varkappa\,\nabs\,.\,\vec\nu =-\varkappa^2$.}
Hence, we obtain from the second and third terms on the right hand side of 
(\ref{eq:Dziukwilla}),
on noting Lemma~\ref{lem:productrules}\ref{item:pri} and Theorem~\ref{thm:div},
that
\begin{align}
\left\langle\nabs\,.\,\vec\varkappa +\tfrac12\,|\vec\varkappa|^2 ,
\nabs\,.\,\vec\chi\right\rangle_{\Gamma(t)} &
= - \tfrac12 \left\langle \varkappa^2,
\nabs\,.\,\vec\chi\right\rangle_{\Gamma(t)} \nonumber \\ &
=\left\langle\varkappa\,\nabs\,\varkappa,\vec\chi\right\rangle_{\Gamma(t)}
+ \tfrac12\left\langle\varkappa^3, \vec\chi\,.\,\vec\nu
\right\rangle_{\Gamma(t)}.
\label{eq:GDWF2}
\end{align}
Combining (\ref{eq:Dziukwilla}), (\ref{eq:GDWF1}) and (\ref{eq:GDWF2}) yields 
that
\begin{equation} \label{eq:GDWF3}
\vec{\mathcal{V}} = \left(-\Delta_s\,\varkappa + \tfrac12\,\varkappa^3 \right)\vec\nu 
+ \varkappa \, ( \Delta_s\,\vec\nu + \nabs\,\varkappa) = 
\left( - \Delta_s\,\varkappa  - \varkappa\,|\nabs\,\vec\nu|^2 
+ \tfrac12 \,\varkappa^3 \right)\vec\nu\,,
\end{equation}
where we have recalled Lemma \ref{lem:Deltasnu}. 
Therefore, we obtain the desired first result in (\ref{eq:SWF}).
\end{proof}

We now introduce a natural semidiscrete variant of (\ref{eq:Dziukwill}).
Given the closed polyhedral hypersurface $\Gamma^h(0)$,
find an evolving polyhedral hypersurface $\GhT$ 
with induced velocity $\vec{\mathcal{V}}^h \in \VhGhT$,
and $\vec\kappa^h \in \VhGhT$, i.e.\ 
$(\vec{\mathcal{V}}^h(\cdot,t),\vec\kappa^h(\cdot,t))$ $\in\Vht\times\Vht$ 
for all $t \in [0,T]$, such that, for all $t \in (0,T]$,
\begin{subequations}
\label{eq:Dzwill1}
\begin{align}
&\left\langle \vec{\mathcal{V}}^h,\vec\chi\right\rangle^h_{\Gamma^h(t)} =
\left\langle\nabs\,\vec\kappa^h,\nabs\,\vec\chi
-2\,\mat D_s(\vec\chi)
\right\rangle_{\Gamma^h(t)} +
\left\langle\nabs\,.\,\vec\kappa^h,\nabs\,.\,\vec\chi
\right\rangle_{\Gamma^h(t)}
\nonumber \\ & \hspace{3cm}
+ \tfrac12\left\langle|\vec\kappa^h|^2, \nabs\,.\,\vec\chi
\right\rangle_{\Gamma^h(t)}^h 
\quad\forall\ \vec\chi\in\Vht\,,
\label{eq:Dzwill1a}\\
& \left\langle\vec\kappa^h,\vec\eta\right\rangle^h_{\Gamma^h(t)} +
\left\langle\nabs\,\vec\id, \nabs\,\vec\eta\right\rangle_{\Gamma^h(t)} = 0
\quad\forall\ \vec\eta\in \Vht\,, \label{eq:Dzwill2b}
\end{align}
\end{subequations}
where we recall the definition of 
$\mat D_s$ on $\Gamma^h(t)$ from Lemma~\ref{lem:commdiffh}. 

We now prove the following stability theorem.
\begin{thm} \label{thm:Dzwill} 
Let $(\GhT,\vec\kappa^h)$ be a solution of \eqref{eq:Dzwill1},
and let $\vec\kappa^h \in \VhTGhT$.
Then it holds that  
\[
\ddt\, \tfrac12 \left(\left| \vec\kappa^h \right|^h_{\Gamma^h(t)}\right)^2
=- \left( \left| \vec{\mathcal{V}}^h \right|^h_{\Gamma^h(t)}\right)^2 \leq 0\,.
\]
\end{thm}
\begin{proof} 
We argue as in the proof of Lemma~\ref{lem:Willmorevar2}. 
Similarly to (\ref{eq:dteta0}), we extend test functions
$\vec\eta \in \Vht$ in (\ref{eq:Dzwill2b}) 
to $\vec\eta \in \VhTGhT$
such that $\matpartxh\,\vec\eta = \vec 0$ on $\GhT$. 
Then taking the time
derivative of (\ref{eq:Dzwill2b}), on noting Theorem \ref{thm:disctrans}
and Lemma \ref{lem:commdiffh}\ref{item:commdiffhiii}, yields
\begin{align}
&    \left\langle\matpartxh\,
\vec\kappa^h,\vec\eta\right\rangle^h_{\Gamma^h(t)} +
\left\langle \vec\kappa^h\,.\,\vec\eta, 
\nabs\,.\,\vec{\mathcal{V}}^h\right\rangle^h_{\Gamma^h(t)} +
\left\langle
\nabs\,.\,\vec{\mathcal{V}}^h,\nabs\,.\,\vec\eta\right\rangle_{\Gamma^h(t)}
\nonumber \\ &\qquad 
+ \left\langle\nabs\,\vec{\mathcal{V}}^h,\nabs\,\vec\eta
\right\rangle_{\Gamma^h(t)}
- 2\left\langle\mat D_s(\vec{\mathcal{V}}^h)
,\nabs\,\vec\eta\right\rangle_{\Gamma^h(t)} = 0\,.
\label{eq:Dzwill2c}
\end{align}
We now choose $\vec\eta = \vec\kappa^h$ in (\ref{eq:Dzwill2c}) and
$\vec\chi = \vec{\mathcal V}^h$ in (\ref{eq:Dzwill1a}), 
to obtain that
\[
\left\langle \matpartxh\,\vec\kappa^h,\vec\kappa^h\right\rangle^h_{\Gamma^h(t)}
+ \tfrac12 \left\langle|\vec\kappa^h|^2,
\nabs\,.\,\vec{\mathcal{V}}^h\right\rangle^h_{\Gamma^h(t)}
+\left\langle\vec{\mathcal{V}}^h,\vec{\mathcal{V}}^h
\right\rangle^h_{\Gamma^h(t)} = 0\,.
\]
The desired result then follows on noting
Theorem~\ref{thm:disctrans}\ref{item:disctransh} and (\ref{eq:ip0norm}).
\end{proof}

\begin{rem} \label{rem:GDWill}
\rule{0pt}{0pt}
\begin{enumerate}
\item
For the case $d=2$ an error analysis of \eqref{eq:Dzwill1} can be found in
{\rm \cite{DeckelnickD09}}.
\item \label{item:GDWillii} 
We note the version of \eqref{eq:Dzwill1} proposed in
{\rm \citet[Problem~3]{Dziuk08}} is without mass lumping.
Nevertheless, either version of \eqref{eq:Dzwill1}, in contrast
to \eqref{eq:Will}, does not have good
mesh properties. This is due to the fact that the mesh movements are
almost exclusively in the normal direction, which in general leads to
bad meshes. This is because \eqref{eq:Dzwill1a} approximates   
\eqref{eq:Dziukwilla}, which is a weak formulation of 
\[
\vec{\mathcal{V}} = \left[ 
-\Delta_s\,\varkappa - \varkappa\,|\nabs\,\vec\nu|^2 +
  \tfrac12\,\varkappa^3
\right] \vec\nu  \qquad\mbox{on } \Gamma(t)\,,
\]
where we have recalled \eqref{eq:GDWF3}.
\end{enumerate}
\end{rem}

We now wish to derive a weak formulation, which leads to
semidiscretizations that are both stable and have good mesh properties. 
A main ingredient is to ensure that the equation
\begin{equation}\label{eq:weakcurv12}
  \left\langle\varkappa\,\vec\nu,\vec\eta\right\rangle_{\Gamma(t)} +
  \left\langle\nabs\,\vec\id, \nabs\,\vec\eta\right\rangle_{\Gamma(t)} = 0
\quad\forall\ \vec\eta\in \Vt
\end{equation}
holds. This then leads to good meshes on the discrete level, 
recall \S\ref{subsec:equi}.
We hence want to compute the time derivative of 
$\tfrac12\,\langle \varkappa,\varkappa\rangle_{\Gamma(t)}$ by
taking the constraint (\ref{eq:weakcurv12}) into
account. To this end, we use the calculus of PDE constrained optimization,
see e.g.\ \cite{HinzePUU09,Troltzsch10}, and define the Lagrangian
\begin{equation}
L(\Gamma(t), \varkappa^\star, \vec y) =
\tfrac12\left\langle \varkappa^\star, \varkappa^\star\right\rangle_{\Gamma(t)} 
- \left\langle \varkappa^\star\,\vec\nu, \vec y \right\rangle_{\Gamma(t)} 
- \left\langle \nabs\,\vec\id,
\nabs\,\vec y \right\rangle_{\Gamma(t)} \,, \label{eq:Lag2}
\end{equation}
where $\vec y \in [H^1(\Gamma(t)]^d$ is the Lagrange
multiplier associated with the constraint \eqref{eq:weakcurv12}
with $\varkappa$ replaced by $\varkappa^\star$. 
We note at this stage that $\varkappa^\star(t)$ is an independent variable,
and not the mean curvature, $\varkappa(t)$, of $\Gamma(t)$.
We now need to take variations of $L(\Gamma(t),\varkappa^\star, \vec y)$
with respect to $\Gamma(t)$, $\varkappa^\star$ and $\vec y$. 

To this end, for any $\vec\chi \in [H^1(\Gamma(t)]^d$ 
and for any $\epsilon \in \bR$ let
\begin{equation} \label{eq:Gammadelta}
\Gamma_\epsilon(t) = \vec\Phi(\Gamma(t),\epsilon)\,, \qquad  \text{where}
\quad 
\vec\Phi(\cdot,\epsilon) = \vec\id_{\mid_{\Gamma(t)}} + \epsilon\,\vec\chi\,.
\end{equation}
For $\epsilon_0>0$, we now consider a smooth function $f$ defined on 
$\bigcup_{\epsilon\in[-\epsilon_0,\epsilon_0]} 
(\Gamma_\epsilon(t)\times\{\epsilon\})$. Then, similarly to
Definition~\ref{def:2.18}\ref{item:def2.18i}, we define
\begin{equation} \label{eq:ederiv0}
(\matpartxpar{\epsilon} \,f) (\vec\Phi(\vec z,\epsilon), \epsilon) = 
\ddeps\,f(\vec\Phi(\vec z,\epsilon),\epsilon)
\quad\forall\ (\vec z,\epsilon) \in \Gamma(t) \times [-\epsilon_0,\epsilon_0]
\,,
\end{equation}
and note that
\begin{equation} \label{eq:ederiv}
\matpartxpar{\epsilon} \,f 
=  \left(\ddeps\,f(\vec\Phi(\cdot,\epsilon),\epsilon)
\right)_{\mid_{\epsilon=0}} 
= \lim_{\epsilon \to 0} \frac{f(\vec\Phi(\cdot,\epsilon),\epsilon) - f}
{\epsilon} \quad\text{on } \Gamma(t)\,.
\end{equation}
In what follows, and similarly to \S\ref{subsec:evolve}, 
we often identify $\Gamma_\epsilon(t) \times \{\epsilon\}$ with 
$\Gamma_\epsilon(t)$,
and hence $(\Gamma_\epsilon(t) \times \{\epsilon\})_{\mid_{\epsilon=0}}$ 
with $\Gamma(t)$.
With the help of a direct analogue of Theorem~\ref{thm:trans}, we then obtain
that the first variation of $\left\langle f, 1 \right\rangle_{\Gamma(t)}$
is given by
\begin{equation} \label{eq:transeq}
\left[\deldel{\Gamma} \left\langle f, 1 \right\rangle_{\Gamma(t)}
\right](\vec\chi) 
= \left( \ddeps\left\langle f, 1 \right\rangle_{\Gamma_\epsilon(t)}
\right)_{\mid_{\epsilon=0}}
=\left\langle \matpartxpar{\epsilon}\,f+f\,\nabs\,.\,\vec\chi, 1 
\right\rangle_{\Gamma(t)}.
\end{equation}
In addition, if $\vec\nu_\epsilon$ is the unit normal on $\Gamma_\epsilon(t)$,
corresponding to $\vec\nu$ on $\Gamma(t)$, then a direct analogue of
Lemma~\ref{lem:dtnu}\ref{item:dtnu} yields that
\begin{equation} \label{eq:denu}
\left[\deldel{\Gamma}\,\vec\nu \right](\vec\chi) 
= \left(\matpartxpar{\epsilon}\,\vec\nu_\epsilon\right)_{\mid_{\epsilon=0}} 
= -(\nabs\,\vec\chi)^\transT\,\vec\nu \qquad\mbox{on } \Gamma(t) \,.
\end{equation}
Similarly to (\ref{eq:etaflow}), for any $\eta \in L^\infty(\Gamma(t))$, 
one can define $\eta_{(\epsilon)}$ on 
$\bigcup_{\epsilon\in[-\epsilon_0,\epsilon_0]}$ $ 
(\Gamma_\epsilon(t)\times\{\epsilon\})$ via
\begin{equation}\label{eq:1919}
\eta_{(\epsilon)}(\vec\Phi(\vec z,\epsilon),\epsilon)= \eta(\vec z) 
\quad\forall\ (\vec z,\epsilon) \in \Gamma(t) \times [-\epsilon_0,\epsilon_0]
\,,
\end{equation}
and analogously for $\vec\eta \in [L^\infty(\Gamma(t))]^d$.
Similarly to (\ref{eq:dteta0}), it follows from this extension and
\eqref{eq:ederiv} that 
\begin{equation*} 
\left[\deldel{\Gamma}\,\vec\eta \right](\vec\chi) 
= \left(\matpartxpar{\epsilon}\,\vec\eta_{(\epsilon)}
\right)_{\mid_{\epsilon=0}} = \vec 0 
\quad\text{on } \Gamma(t) \,.
\end{equation*}
In particular, we have the following analogue of (\ref{eq:dtdivdisc}), 
\begin{align}\label{eq:dedivdisc}
& \left[\deldel{\Gamma} \left\langle\nabs\,\vec\id,
  \nabs\,\vec\eta\right\rangle_{\Gamma(t)} \right](\vec\chi) 
= \left( \ddeps \left\langle\nabs\,\vec\id,
  \nabs\,\vec\eta_{(\epsilon)}\right\rangle_{\Gamma_\epsilon(t)}
\right)_{\mid_{\epsilon=0}} \nonumber \\ & \qquad\qquad
= \left\langle \nabs\,.\,\vec\eta, \nabs\,.\,\vec\chi\right\rangle_{\Gamma(t)}
+ \left\langle \nabs\,\vec\eta, \nabs\,\vec\chi- 2\,\mat D_s(\vec\chi) 
\right\rangle_{\Gamma(t)}.
\end{align}

We now consider the variations of $L(\Gamma(t),\varkappa^\star, \vec y)$
with respect to $\Gamma(t)$, $\varkappa^\star$ and $\vec y$. 
In particular, for all $\vec\chi \in \Vt$, 
$\xi \in L^2(\Gamma(t))$ and $\vec\eta \in \Vt$, we let
\begin{subequations} \label{eq:deltaL}
\begin{align}
\left[\deldel{\Gamma}\,L\right](\vec\chi) 
& = \left( \ddeps\, 
L(\Gamma_\epsilon(t), \varkappa^\star_{(\epsilon)}, \vec y_{(\epsilon)})
\right)_{\mid_{\epsilon=0}} , \label{eq:deltaLa} \\
\left[\deldel{\varkappa^\star}\,L\right](\xi) & =
\left( \ddeps\, L(\Gamma(t), \varkappa^\star + \epsilon\,\xi, \vec y) 
\right)_{\mid_{\epsilon=0}} , \label{eq:deltaLb} \\
\left[\deldel{\vec y}\,L\right](\vec\eta) & =
\left( \ddeps\, L(\Gamma(t), \varkappa^\star , \vec y + \epsilon\,\vec\eta) 
\right)_{\mid_{\epsilon=0}} , \label{eq:deltaLc}
\end{align}
\end{subequations}
where $\varkappa^\star_{(\epsilon)} \in L^2(\Gamma_\epsilon(t))$ and 
$\vec y_{(\epsilon)} \in [H^1(\Gamma_\epsilon(t))]^d$
are defined by transporting the values of 
$\varkappa^\star \in L^2(\Gamma(t))$ and 
$\vec y \in \Vt$ as defined in (\ref{eq:1919}), recall
\eqref{eq:Gammadelta} for the $\vec\chi$ at hand.
Setting the variation
$\left[\deldel{\vec y}\,L\right](\vec\eta)=0$, yields
(\ref{eq:weakcurv12}) with $\varkappa$ replaced by $\varkappa^\star$. Hence, 
comparing this with the original (\ref{eq:weakcurv12}), we obtain that 
$\varkappa^\star=\varkappa$.
Setting $\left[\deldel{\varkappa^\star}\,L\right](\xi)=0$, 
yields, on noting $\varkappa^\star=\varkappa$, that  
\begin{equation} \label{eq:kappaynu} 
\varkappa = \vec y\,.\,\vec\nu 
\qquad\mbox{on } \Gamma(t) \,.
\end{equation}
Finally, setting the variation $\left[\deldel{\Gamma}\,L\right](\vec\chi) =
-\left\langle \vec {\mathcal V}\,.\,\vec\nu, \vec\chi\,.\,\vec\nu 
\right\rangle_{\Gamma(t)}$
and noting (\ref{eq:transeq}), (\ref{eq:denu}), (\ref{eq:dedivdisc})
and that $\varkappa^\star= \varkappa$, we obtain
\begin{align}
\left\langle \vec{\mathcal{V}}\,.\,\vec\nu, \vec\chi\,.\,\vec\nu 
\right\rangle_{\Gamma(t)} &= 
\left\langle \nabs\,\vec y, \nabs\,\vec\chi 
-2\,\mat{D}_s(\vec\chi) \right\rangle_{\Gamma(t)} 
+ \left\langle \nabs\,.\,\vec y, \nabs\,.\,\vec\chi \right\rangle_{\Gamma(t)}
\nonumber \\ &\qquad 
- \left\langle \varkappa \left(\tfrac12\,\varkappa - 
\vec y\,.\,\vec\nu\right), \nabs\,.\,\vec\chi \right\rangle_{\Gamma(t)} 
-  \left\langle \varkappa \,\vec y,  (\nabs\,\vec\chi)^\transT \,\vec\nu 
\right\rangle_{\Gamma(t)}
\nonumber \\ & \hspace{1.3in}
\quad\forall\ \vec\chi \in \Vt \,.
\label{eq:weak2a}
\end{align}

Therefore we have the following weak formulation. 
Given a closed hypersurface $\Gamma(0)$, we seek an evolving hypersurface 
$(\Gamma(t))_{t\in[0,T]}$, with a global parameterization and induced velocity 
field $\vec{\mathcal{V}}$, $\varkappa \in L^2(\GT)$
and $\vec y \in [L^2(\GT)]^d$ as follows. 
For almost all $t \in (0,T)$, find
$(\vec {\mathcal V}(\cdot,t),\varkappa(\cdot,t), \vec y(\cdot,t))
\in [L^2(\Gamma(t))]^d
\times L^2(\Gamma(t)) \times \Vt$ such that
(\ref{eq:weak2a}), (\ref{eq:kappaynu}) and (\ref{eq:weakcurv12}) hold.

\begin{rem} \label{rem:WWF}
\rule{0pt}{0pt}
\begin{enumerate}
\item \label{item:WWFi}
Using the techniques in {\rm \citet[Appendix A]{pwfopen}} one can show,
similarly to {\rm Lemma~\ref{lem:WFweakstrong}}, that a sufficiently
smooth solution of this weak formulation is a solution of the strong formulation
\eqref{eq:SWF}.   
\item \label{item:WWFii} 
Clearly, using \eqref{eq:kappaynu} one can eliminate $\varkappa$
from \eqref{eq:weak2a} and \eqref{eq:weakcurv12} in this weak formulation. 
\end{enumerate}
\end{rem}

We now consider a discrete analogue of (\ref{eq:weak2a}), (\ref{eq:kappaynu}) 
and (\ref{eq:weakcurv12}), by first introducing the discrete
analogue of (\ref{eq:Lag2})
\begin{equation}
L^h(\Gamma^h(t), \kappa^h, \vec Y^h) =
\tfrac12 \left\langle \kappa^h, \kappa^h \right\rangle^h_{\Gamma^h(t)} 
- \left\langle \kappa^h\,\vec\nu^h, \vec Y^h \right\rangle^h_{\Gamma^h(t)} 
- \left\langle \nabs\,\vec\id,
\nabs\,\vec Y^h \right\rangle_{\Gamma^h(t)} \,, \label{eq:Lag2h}
\end{equation}
where $\vec Y^h(\cdot,t) \in \Vht$ is the Lagrange
multiplier associated with the constraint $\kappa^h(\cdot,t) \in \Wht$ 
satisfying
\begin{equation}\label{eq:weakcurv12h}
  \left\langle\kappa^h\,\vec\nu^h,\vec\eta\right\rangle^h_{\Gamma^h(t)} +
  \left\langle\nabs\,\vec\id, \nabs\,\vec\eta\right\rangle_{\Gamma^h(t)} = 0
\quad\forall\ \vec\eta\in \Vht\,.
\end{equation}

Setting $\left[\deldel{\vec Y^h}\,L^h\right](\vec\eta)=0$ yields
(\ref{eq:weakcurv12h}).
Setting $\left[\deldel{\kappa^h}\,L^h\right](\xi)=0$ yields that
\begin{equation} \label{eq:kappaynuh} 
\left\langle \kappa^h - \vec Y^h\,.\,\vec\nu^h, \xi 
\right\rangle^h_{\Gamma^h(t)} = 0 \quad\forall\ \xi \in \Wht\,.
\end{equation}
Finally, we need to take the variation of $L^h(\Gamma^h(t),\kappa^h, \vec Y^h)$
with respect to $\Gamma^h(t)$. 
To this end, we have the following discrete analogue of \eqref{eq:Gammadelta}.
For any $\vec\chi \in \Vht$ and for any $\epsilon \in \bR$, let
\begin{equation*} 
\Gamma^h_\epsilon(t) = \vec\Phi^h(\Gamma^h(t),\epsilon)\,, \qquad  \text{where}
\quad 
\vec\Phi^h(\cdot,\epsilon) = \vec\id_{\mid_{\Gamma^h(t)}} + 
\epsilon\,\vec\chi\,.
\end{equation*}
We also define $\matpartxhpar{\epsilon}$ to be the discrete analogue of
\eqref{eq:ederiv0}. 
Similarly to Theorem~\ref{thm:disctrans}\ref{item:disctransh}, we then have 
that
\begin{align} \label{eq:transeqh}
& \left[\deldel{\Gamma^h} \left\langle \kappa^h, 
\vec Y^h\,.\,\vec\nu^h \right\rangle_{\Gamma^h(t)}^h
\right](\vec\chi) =
\left( \ddeps \left\langle \kappa^h_{(\epsilon)}, 
\vec Y^h_{(\epsilon)}\,.\,\vec\nu^h_\epsilon 
\right\rangle_{\Gamma^h_\epsilon(t)}^h
\right)_{\mid_{\epsilon=0}} \nonumber \\ & \qquad
 = \left(\left\langle \kappa^h_{(\epsilon)}, \vec Y^h_{(\epsilon)}\,.\, 
\matpartxhpar{\epsilon}\,\vec\nu^h_\epsilon
\right\rangle_{\Gamma^h_{\epsilon}(t)}^h
\right)_{\mid_{\epsilon=0}}
+ \left\langle \kappa^h \,\vec Y^h\,.\,\vec\nu^h, \nabs\,.\,\vec\chi 
\right\rangle_{\Gamma^h(t)}^{h} \nonumber \\ & \qquad
 =\left\langle \kappa^h , \vec Y^h\,.
\left[\deldel{\Gamma^h}\,\vec\nu^h\right](\vec\chi)
\right\rangle_{\Gamma^h(t)}^h
+ \left\langle \kappa^h \,\vec Y^h\,.\,\vec\nu^h, \nabs\,.\,\vec\chi 
\right\rangle_{\Gamma^h(t)}^{h} ,
\end{align}
where $\kappa^h_{(\epsilon)}$ and $\vec Y^h_{(\epsilon)}$ are defined
similarly to (\ref{eq:1919}), with the help of the map $\vec\Phi^h$, and
$\vec\nu^h_\epsilon$ is the unit normal on $\Gamma^h_\epsilon(t)$. 
Hence, setting the variation 
\arxivyesno{$\left[\deldel{\Gamma^h}\, L^h\right](\vec\chi) = $ \linebreak $
-\left\langle \vec {\mathcal V}^h\,.\,\vec\omega^h, \vec\chi\,.\,\vec\omega^h 
\right\rangle^h_{\Gamma(t)}$,}
{$\left[\deldel{\Gamma^h}\, L^h\right]\!\!\,(\vec\chi)$ $ =
-\left\langle \vec {\mathcal V}^h\,.\,\vec\omega^h, \vec\chi\,.\,\vec\omega^h 
\right\rangle^h_{\Gamma(t)}$,}
on noting (\ref{eq:transeqh}),
the discrete analogue of (\ref{eq:denu}) for $\vec\nu^h$ on $\Gamma^h(t)$,
compare Lemma~\ref{lem:dtnuh}, 
and (\ref{eq:dedivdisc}) for $\Gamma^h(t)$, we obtain
\begin{align}
&\left\langle \vec{\mathcal{V}}^h\,.\,\vec\omega^h, \vec\chi\,.\,\vec\omega^h 
\right\rangle^h_{\Gamma^h(t)} = 
\left\langle \nabs\,\vec Y^h, \nabs\,\vec\chi 
-2\,\mat{D}_s(\vec\chi) \right\rangle_{\Gamma^h(t)}
+ \left\langle \nabs\,.\,\vec Y^h, \nabs\,.\,\vec\chi
  \right\rangle_{\Gamma^h(t)}\nonumber \\
& \qquad\qquad
- \left\langle \kappa^h \,(\tfrac12\,
\kappa^h - \vec Y^h\,.\,\vec\nu^h), \nabs\,.\,\vec\chi
\right\rangle^h_{\Gamma^h(t)} 
-  \left\langle \kappa^h \,\vec Y^h,  (\nabs\,\vec\chi)^\transT \,\vec\nu^h 
\right\rangle^h_{\Gamma^h(t)}
\nonumber \\ & \hspace{2.5in}\quad\forall\ \vec\chi \in \Vht\,.
\label{eq:weak2ah}
\end{align}

Therefore we have the following semidiscrete finite element approximation of
(\ref{eq:weak2a}), (\ref{eq:kappaynu}) and (\ref{eq:weakcurv12}). 
Given the closed polyhedral hypersurface $\Gamma^h(0)$,
find an evolving polyhedral hypersurface $\GhT$ 
with induced velocity $\vec{\mathcal{V}}^h \in \VhGhT$,
and $\kappa^h \in \WhGhT$, $\vec Y^h \in \VhGhT$,
i.e.\ $(\vec{\mathcal{V}}^h(\cdot, t),\kappa^h(\cdot,t),\vec Y^h(\cdot,t)) 
\in \Vht \times \Wht \times \Vht$ for all $t \in [0,T]$,
such that, for all $t \in (0,T]$,
(\ref{eq:weak2ah}), (\ref{eq:kappaynuh}) and (\ref{eq:weakcurv12h}) hold.

\begin{rem} \label{rem:WWFh}
\rule{0pt}{0pt}
\begin{enumerate}
\item
Similarly to {\rm Remark~\ref{rem:WWF}\ref{item:WWFii}}, 
using \eqref{eq:kappaynuh}, one can eliminate $\kappa^h$ from 
\eqref{eq:weak2ah} and \eqref{eq:weakcurv12h} in this semidiscrete 
finite element approximation.
\item
The scheme \eqref{eq:weak2ah}, \eqref{eq:kappaynuh} and \eqref{eq:weakcurv12h}
satisfies {\rm Theorem~\ref{thm:SDSD}\ref{item:SDSDTM}}.
\end{enumerate}
\end{rem}

We have the following stability result.

\begin{thm} \label{thm:WWFh}
Let $(\GhT, \kappa^h, \vec Y^h)$ be a solution of
\eqref{eq:weak2ah}, \eqref{eq:kappaynuh} and \arxivyesno{}{\linebreak}%
\eqref{eq:weakcurv12h},
and let $\kappa^h \in \WhTGhT$. Then it holds that  
\[
\ddt\, \tfrac12 \left(\left| \kappa^h
\right|^h_{\Gamma^h(t)}\right)^2 
=- \left( \left| \vec{\mathcal{V}}^h\,.\,\vec\omega^h \right|^h_{\Gamma^h(t)}\right)^2 
\le 0\,.
\]
\end{thm}
\begin{proof} 
Similarly to (\ref{eq:Dzwill2c}), we extend test functions
$\vec\eta \in \Vht$ in (\ref{eq:weakcurv12h}) 
to $\vec\eta \in \VhTGhT$
such that $\matpartxh\,\vec\eta = \vec 0$ on $\GhT$. 
Then taking the time
derivative of (\ref{eq:weakcurv12h}), on noting Theorem~\ref{thm:disctrans}
and Lemma \ref{lem:commdiffh}\ref{item:commdiffhiii},
yields
\begin{align}
& \left\langle\matpartxh\,
\kappa^h,\vec\nu^h \,.\,\vec\eta\right\rangle^h_{\Gamma^h(t)} +
\left\langle
\kappa^h\, \matpartxh\,\vec\nu^h, \vec\eta\right\rangle^h_{\Gamma^h(t)} +    
\left\langle\kappa^h\,\vec\nu^h.\,\vec\eta, 
\nabs\,.\,\vec{\mathcal{V}}^h\right\rangle^h_{\Gamma^h(t)} 
\nonumber \\ &\quad 
+ \left\langle
\nabs\,.\,\vec{\mathcal{V}}^h,\nabs\,.\,\vec\eta\right\rangle_{\Gamma^h(t)}
+\left\langle\nabs\,\vec{\mathcal{V}}^h,\nabs\,\vec\eta\right\rangle_{\Gamma^h(t)}
- 2\left\langle\mat D_s(\vec{\mathcal{V}}^h)
,\nabs\,\vec\eta\right\rangle_{\Gamma^h(t)} = 0\,.
\label{eq:BGNwill2c}
\end{align}
We now take $\vec\eta = \vec Y^h$ in (\ref{eq:BGNwill2c}) and
$\vec\chi = \vec{\mathcal V}^h$ in
(\ref{eq:weak2ah}), to obtain, 
on noting (\ref{eq:kappaynuh}), (\ref{eq:omegahnuh}) 
and Lemma~\ref{lem:dtnuh}, that
\[
\left\langle \matpartxh\,\kappa^h,\kappa^h\right\rangle^h_{\Gamma^h(t)}
+ \tfrac12 \left\langle|\kappa^h|^2, 
\nabs\,.\,\vec{\mathcal{V}}^h\right\rangle^h_{\Gamma^h(t)}
+\left\langle
\vec{\mathcal{V}}^h\,.\,\vec\omega^h,\vec{\mathcal{V}}^h\,.\,
\vec\omega^h \right\rangle^h_{\Gamma^h(t)} 
= 0\,.
\]
The desired result then follows on noting
Theorem~\ref{thm:disctrans}\ref{item:disctransh} and (\ref{eq:ip0norm}).
\end{proof}

We now state a fully discrete version of \eqref{eq:weak2ah},
\eqref{eq:kappaynuh} and \eqref{eq:weakcurv12h}, on recalling (\ref{eq:omegahnuh}).

Let the closed polyhedral hypersurface $\Gamma^0$ be an approximation to
$\Gamma(0)$, and let $\kappa^0_{\Gamma^0} \in V(\Gamma^0)$ and 
$\vec Y^0_{\Gamma^0} \in \underline{V}(\Gamma^0)$ be approximations to its 
mean curvature and mean curvature vector, respectively.
We also recall the time interval partitioning \eqref{eq:ttaum}. 
Then, for $m = 0,\ldots,M-1$, find 
$(\vec X^{m+1},\kappa^{m+1},\vec Y^{m+1}) \in \Vhm\times \Whm \times \Vhm$ 
such that 
\begin{subequations} \label{eq:FDBGNWill}
\begin{align}
& \left\langle \frac{\vec X^{m+1}-\vec\id}{\ttau_m}\,.\,\vec\omega^m, 
\vec\chi\,.\,\vec\omega^m\right\rangle_{\Gamma^m}^h
- \left\langle \nabs\,\vec Y^{m+1}, \nabs\,\vec\chi \right\rangle_{\Gamma^m} 
\nonumber \\ &\
= - \left\langle \kappa^m_{\Gamma^m} \left(\tfrac12\,
\kappa^m_{\Gamma^m}- \vec Y^m_{\Gamma^m}\,.\,\vec\nu^m\right), 
\nabs\,.\,\vec\chi \right\rangle^h_{\Gamma^m} 
-  \left\langle \kappa^m_{\Gamma^m} \,\vec Y^m_{\Gamma^m},  
(\nabs\,\vec\chi)^\transT \,\vec\nu^m 
\right\rangle^h_{\Gamma^m}
\nonumber \\ & \quad 
+ \left\langle \nabs\,.\,\vec Y^m_{\Gamma^m}, \nabs\,.\,\vec\chi
  \right\rangle_{\Gamma^m}  
-2 \left\langle \nabs\,\vec Y^m_{\Gamma^m}, 
\mat{D}_s(\vec\chi) \right\rangle_{\Gamma^m} 
\quad\forall\ \vec\chi \in \Vhm
\,, \label{eq:FDBGNWilla}  \\
& \kappa^{m+1} = \pi_{\Gamma^m}\left[ \vec Y^{m+1}\,.\,\vec\omega^m \right]
,
\label{eq:FDBGNWillb}  \\
& \left\langle \kappa^{m+1}\,\vec\nu^m,\vec\eta\right\rangle^h_{\Gamma^m}
+ \left\langle \nabs\,\vec X^{m+1}, \nabs\,\vec\eta 
\right\rangle_{\Gamma^m} = 0
\quad\forall\ \vec\eta \in \Vhm \label{eq:FDBGNWillc}
\end{align}
\end{subequations}
and set $\Gamma^{m+1} = \vec X^{m+1}(\Gamma^m)$, 
$\kappa^{m+1}_{\Gamma^{m+1}}= \kappa^{m+1}\circ (\vec X^{m+1})^{-1}\in\Whmp$
and \arxivyesno{}{\linebreak}%
$\vec Y^{m+1}_{\Gamma^{m+1}}= \vec Y^{m+1} \circ (\vec X^{m+1})^{-1}
\in\Vhmp$. 
We have the following result.

\begin{thm}\label{thm:FDBGNWill}
Let $\Gamma^m$ satisfy {\rm Assumption \ref{ass:A}\ref{item:assA1}}
and let $\kappa^m_{\Gamma^m} \in \Whm$ and $\vec Y^m_{\Gamma^m} \in \Vhm$.
Then there exists a unique solution
$(\vec X^{m+1},\kappa^{m+1}, \vec Y^{m+1}) \in \Vhm\times \Whm \times \Vhm$
to {\rm (\ref{eq:FDBGNWill})}. 
\end{thm}
\begin{proof}
The proof is very similar to the proof of Theorem \ref{thm:uniqueSD}.
We consider a solution $(\vec X,\kappa,\vec Y)\in\Vhm\times \Whm \times \Vhm$ 
of the homogeneous system
\begin{subequations} \label{eq:homoWill}
\begin{align}
\left\langle \vec X\,.\,\omega^m, \vec\chi\,.\,\vec\omega^m \right
\rangle^h_{\Gamma^m}
- \ttau_m\left\langle \nabs\,\vec Y, \nabs\,\vec\chi \right\rangle_{\Gamma^m} 
& = 0 \quad\forall\ \vec\chi \in \Vhm\,, \label{eq:homoWilla} \\
\kappa - \pi_{\Gamma^m}\left[ \vec Y\,.\,\vec\omega^m \right] & = 0\,,
\label{eq:homoWillb}  \\
\left\langle \kappa\,\vec\nu^m, \vec\eta\right\rangle^h_{\Gamma^m}
+ \left\langle \nabs\,\vec X, \nabs\,\vec\eta \right\rangle_{\Gamma^m} & = 0
\quad\forall\ \vec\eta \in \Vhm\,. \label{eq:homoWillc}
\end{align}
\end{subequations}
Choosing $\vec\chi = \vec Y\in \Vhm$ in (\ref{eq:homoWilla}) and
$\vec\eta = \vec X\in\Vhm$ in (\ref{eq:homoWillc}) yields,
on noting \eqref{eq:omegahnuh} and \eqref{eq:homoWillb}, that
\begin{equation*} 
\left|\nabs\,\vec X\right|_{\Gamma^m}^2 
+\ttau_m\left|\nabs\,\vec Y\right|_{\Gamma^m}^2 =0\,.
\end{equation*}
We deduce from this that 
$\vec X = \vec X^c$ and $\vec Y = \vec Y^c$ on $\Gamma^m$, for 
$\vec X^c,\,\vec Y^c\in \bR^d$. Hence, on substituting \eqref{eq:homoWillb}
into \eqref{eq:homoWillc}, it follows from \eqref{eq:homoWill} and 
\eqref{eq:omegahnuh} that
\[
\left(\left|\vec X^c\,.\,\vec\omega^m\right|_{\Gamma^m}^h\right)^2
+ \left(\left|\vec Y^c\,.\,\vec\omega^m\right|_{\Gamma^m}^h\right)^2 = 0\,,
\]
i.e.\ $\vec X^c \,.\,\vec\omega^m = \vec Y^c \,.\,\vec\omega^m = 0$. 
Therefore Assumption~\ref{ass:A}\ref{item:assA1} yields that 
$\vec X^c=\vec Y^c=\vec 0$, and hence $\kappa=0$.
\end{proof} 

\begin{rem}[Stability] \label{rem:WFstab}
Unfortunately, it does not seem possible to prove a stability bound for
the fully discrete scheme \eqref{eq:FDBGNWill} or its generalizations.
A similar comment applies to a fully discrete approximation of the
Dziuk scheme \eqref{eq:Dzwill1}.
\end{rem}

\begin{rem}[Discrete linear systems] \label{rem:WFsolve}
On recalling the notation from {\rm \S\ref{subsec:3.3}}, 
we can formulate the
linear systems of equations to  be solved at each time level for
\eqref{eq:FDBGNWill} as follows.
Find $(\vec Y^{m+1},\kappa^{m+1},\delta\vec X^{m+1})
\in (\bR^d)^K\times\bR^K\times (\bR^d)^K$ such that
\begin{equation}
\begin{pmatrix}
 \mat A_{\Gamma^m} & 0 & -\frac1{\ttau_m}\,\mat{\mathcal M}_{\Gamma^m} \\
-\vec N_{\Gamma^m}^\transT & M_{\Gamma^m} & 0 \\
0 & \vec N_{\Gamma^m} & \mat A_{\Gamma^m}
\end{pmatrix}
\begin{pmatrix} \vec Y^{m+1} \\ \kappa^{m+1} \\  \delta\vec X^{m+1} 
\end{pmatrix}
=
\begin{pmatrix} \vec g_{\Gamma^m} 
\\ 0 \\ -\mat A_{\Gamma^m}\,\vec X^m \end{pmatrix} \,,
\label{eq:FDBGNlin}
\end{equation}
where we use a similar abuse of notation as in \eqref{eq:MCvec}. The
definitions of the matrices and vectors in \eqref{eq:FDBGNlin} 
are either given in \eqref{eq:mat0}, or they follow directly from 
\eqref{eq:FDBGNWill}.
In practice, the linear system  \eqref{eq:FDBGNlin} can be efficiently
solved with a sparse direct solution method like UMFPACK, see 
{\rm \cite{Davis04}}.
\end{rem}

\subsection{Willmore flow with spontaneous curvature and area difference 
elasticity effects} \label{subsec:WFspont}
Curvature energies also play an important role for vesicles and biomembranes.
As in many membrane elastic energies, the curvature of the membrane enters
the elastic energy density.  
However, for biomembranes, additional effects play a role, which we would like
to discuss now. In the original curvature energies for biomembranes a possible
asymmetry of the membrane in the normal direction was taken into account,
which can result, for example, from a different chemical environment on the two
sides of the membrane. This
top-down asymmetry makes it necessary to generalize the Willmore energy, and one
example of such a model is the spontaneous curvature model introduced by
 \citet{Canham70} and
\citet{Helfrich73}. 
The simplest such energy, in a nondimensional form, is
\begin{equation*}
E_\spont(\Gamma) = \tfrac12 \int_{\Gamma}
(\varkappa - \spont)^2 \dH{d-1}
\,, 
\end{equation*}
where $\spont\in\bR$ is the given so-called spontaneous curvature.
Biomembranes consist of two layers of lipids, and the
 number of lipid molecules is conserved. In addition, 
there are osmotic pressure effects,
arising from the chemistry around the lipid. These two effects
lead to constraints on the possible membrane configurations. Early
models for bilayer membranes modelled this by taking hard constraints on the 
total
surface area and the enclosed volume of the membrane into account. The
fact that it is difficult to exchange molecules between the two layers
imply that the total number of lipids in each layer is conserved,
and hence a surface
area difference between the two layers will appear. The area difference is
to first order given by the total integrated mean curvature. This follows
from the first variation formula for surface area, recall (\ref{eq:firstvar}). 
Different models incorporate this area difference either by a hard constraint
or penalize deviations from an optimal
area difference. In the latter case, we obtain the energy
\begin{equation}\label{eq:ADESCa}
E_{\spont,\beta}(\Gamma) = E_{\spont}(\Gamma) +
 \tfrac12\,\beta \left( \left\langle \varkappa,1 \right\rangle_{\Gamma}
-M_0 \right)^2
\end{equation}
with given constants $\beta\in\bRgeq$, $M_0\in \bR$. 
A model based on the energy (\ref{eq:ADESCa}) is  called an area difference 
elasticity (ADE) model, see \cite{Seifert97}. 
We now extend Lemma~\ref{lem:Willmorevar} to $E_{\spont,\beta}(\Gamma(t))$,
where for notational convenience we define
\begin{equation} \label{eq:At}
A(t) = \left\langle \varkappa,1 \right\rangle_{\Gamma(t)}  -M_0\,.
\end{equation}

\begin{lem} \label{lem:Willmorevargen}
Let $\GT$ be a closed $C^4$--evolving orientable hypersurface.  
Then it holds that
\begin{align*}
& \ddt\,E_{\spont,\beta}(\Gamma(t)) \nonumber \\ & \quad
= \left\langle\Delta_s\,\varkappa + 
(\varkappa-\spont)\,|\nabs\,\vec\nu|^2-\tfrac12\,
(\varkappa-\spont)^2\,\varkappa 
+ \beta\,A(t) \left(|\nabs\,\vec\nu|^2-\varkappa^2\right), \mathcal{V} 
\right\rangle_{\Gamma(t)} .
\end{align*}
\end{lem}
\begin{proof} 
Generalising (\ref{eq:WEvar}) to $E_\spont(\Gamma(t)$ yields that
\begin{align}
\ddt\,E_{\spont}(\Gamma(t)) & 
= \left\langle \varkappa - \spont,\matpartn\,\varkappa
-\tfrac12\,(\varkappa-\spont)\,\varkappa\,\mathcal{V} \right\rangle_{\Gamma(t)} 
\nonumber \\ &
= \left\langle \varkappa -\spont , \Delta_s\,\mathcal{V} 
 + \mathcal{V}\,|\nabs\,\vec\nu|^2
-\tfrac12\,(\varkappa-\spont)\,\varkappa\,\mathcal{V} \right\rangle_{\Gamma(t)}
\nonumber \\ &
= \left\langle \Delta_s\,\varkappa + (\varkappa-\spont) \,|\nabs\,\vec\nu|^2
-\tfrac12\,(\varkappa-\spont)^2\,\varkappa , \mathcal{V}
\right\rangle_{\Gamma(t)}.
\label{eq:WEvargen}
\end{align}
Next, we compute, on recalling
Theorem \ref{thm:trans} and Lemma \ref{lem:derkappa}\ref{item:lem10.3ii}, that
\begin{align*}
\ddt \left[ \tfrac12
\left(\left\langle \varkappa, 1\right\rangle_{\Gamma(t)} -M_0\right)^2 \right]
&= 
A(t)\,\ddt \left\langle \varkappa, 1\right\rangle_{\Gamma(t)}
= A(t) \left\langle \matpartn\,\varkappa-\varkappa^2\,\mathcal{V}
, 1 \right\rangle_{\Gamma(t)} \\
& = A(t) \left\langle 
|\nabs\,\vec\nu|^2-\varkappa^2, \mathcal{V} \right\rangle_{\Gamma(t)} .
\end{align*}
Combining the above with (\ref{eq:WEvargen}) yields the claim,
on noting (\ref{eq:ADESCa}). 
\end{proof}

Hence the $L^2$--gradient flow
of $E_{\spont,\beta}(\Gamma(t))$, (\ref{eq:ADESCa}), is given as
\begin{equation} \label{eq:WFgen}
{\mathcal V} = -\Delta_s\,\varkappa
-(\varkappa-\spont) \,|\nabs\,\vec\nu|^2
+\tfrac12\,(\varkappa-\spont)^2\,\varkappa
- \beta\,A(t)\left(|\nabs\,\vec\nu|^2-\varkappa^2\right)
\quad\text{on }\Gamma(t)\,.
\end{equation}

Taking our discussion above into account, the volume preserving 
and surface area preserving flows
are also of interest. In the case $\beta=0$, the 
volume and surface area preserving flow is called Helfrich
flow. 
We consider the two side constraints
\begin{equation} \label{eq:HFcons}
\left\langle \varkappa, {\mathcal V} \right\rangle_{\Gamma(t)} =0
\qquad\mbox{and} \qquad 
\left\langle 1, {\mathcal V} \right\rangle_{\Gamma(t)} =0 
\end{equation}
for surface area and volume preservation, where
we have recalled Theorem \ref{thm:trans} and Theorem~\ref{thm:transvol}, 
respectively.
In particular, we introduce Lagrange multipliers 
$(\lambda_A(t),\lambda_V(t))^\transT \in \bR^2$ and then,
on writing (\ref{eq:WFgen}) as ${\mathcal V} = f$, we adapt 
(\ref{eq:WFgen}) to
\begin{equation} \label{WF:gencons}
{\mathcal V} = f + \lambda_A\,\varkappa + \lambda_V\qquad\mbox{on }\Gamma(t)\,.
\end{equation}
We see that the constraints (\ref{eq:HFcons}) are satisfied if 
$(\lambda_A(t),\lambda_V(t))^\transT \in \bR^2$ solve the symmetric system
\begin{equation} \label{eq:matcons}
\begin{pmatrix} 
\left\langle \varkappa, \varkappa \right\rangle_{\Gamma(t)} 
&  \left\langle \varkappa, 1 \right\rangle_{\Gamma(t)} \\ 
\left\langle \varkappa, 1 \right\rangle_{\Gamma(t)} 
& \left\langle 1, 1 \right\rangle_{\Gamma(t)}
\end{pmatrix}\begin{pmatrix} \lambda_A \\ \lambda_V \end{pmatrix}
= - \begin{pmatrix} \left\langle f,\varkappa \right\rangle_{\Gamma(t)} 
\\ \left\langle f, 1 \right\rangle_{\Gamma(t)} \end{pmatrix}.
\end{equation}
We observe that the matrix in (\ref{eq:matcons}) is symmetric positive 
semidefinite,
and it is singular if and only if $\varkappa$ is constant, i.e.\ $\Gamma(t)$
is sphere. In the case of just one constraint, the Lagrange
multiplier corresponding to the other constraint
is set to zero and the corresponding equation is removed from
(\ref{eq:matcons}). Equivalently to (\ref{eq:matcons}), in the two constraint case and
on assuming that $\varkappa$ is not constant,
$\lambda_A = -\left\langle f,\varkappa - \Gmint{\Gamma(t)}\varkappa\right\rangle_{\Gamma(t)}/
|\varkappa-\Gmint{\Gamma(t)}\varkappa|_{\Gamma(t)}^2$
and $\lambda_V$ is obtained from the second equation in (\ref{eq:matcons}).    
Here, we have recalled the definition of $\Gmint{\Gamma(t)} \varkappa$ from (\ref{eq:mint}).

The corresponding changes to the finite element approximation (\ref{eq:Will})
is to replace (\ref{eq:Willa}) by 
\begin{align}
& \left\langle \frac{\vec X^{m+1}-\vec\id}{\ttau_m}, \chi\,\vec\nu^m\right\rangle_{\Gamma^m}^h
- \left\langle \nabs\, \kappa^{m+1}, \nabs\, \chi \right\rangle_{\Gamma^m} 
- \tfrac12 \left\langle(\kappa^m_{\Gamma^m}-\spont)^2 \,\kappa^{m+1}, 
\chi \right\rangle_{\Gamma^m}^h
\nonumber \\ & \qquad
-[\lambda_A^m]_+ \,\left\langle 
\kappa^{m+1}, \chi \right\rangle_{\Gamma^m}^h
\nonumber \\ & \
= \left\langle g^m, \chi \right\rangle_{\Gamma^m}^h +
[\lambda_A^m]_- \,\left\langle 
\kappa^m_{\Gamma^m}, \chi \right\rangle_{\Gamma^m}^h
+  \lambda_V^m \,\left\langle 1, \chi \right\rangle_{\Gamma^m}^h 
\quad\forall\ \chi \in \Whm\,, \label{eq:Willcons}
\end{align}
where $g^m \in \Whm$ is such that 
\begin{align*}  
\left\langle g^m, \chi \right\rangle_{\Gamma^m}^h  
&= -\left\langle \left(\kappa^m_{\Gamma^m} -\spont\right) |\mat W^{m+1}|^2,
\chi\right\rangle_{\Gamma^m}^h
- \beta \,A^m \left\langle |\mat W^{m+1}|^2-
\left(\kappa^m_{\Gamma^m}\right)^2 ,\chi\right\rangle_{\Gamma^m}^h
\nonumber \\
& \hspace{5cm}
\quad\forall\ \chi \in \Whm\,, 
\\
A^m &= \left\langle \kappa^m_{\Gamma^m} , 1
\right\rangle_{\Gamma^m}^h - M_0\,. 
\end{align*} 
In addition, $[\lambda^m_A]_{\pm}= \pm \max\{\pm \lambda^m_A,0\}$ and
the Lagrange multipliers $(\lambda_A^m,\lambda_V^m)^\transT \in \bR^2$
satisfy the symmetric system
\begin{align}
&\begin{pmatrix} 
\left\langle \kappa^m_{\Gamma^m}, \kappa^m_{\Gamma^m}
\right\rangle_{\Gamma^m}^h 
&  \left\langle \kappa^m_{\Gamma^m}, 1 \right\rangle_{\Gamma^m} \\ 
\left\langle \kappa^m_{\Gamma^m}, 1 \right\rangle_{\Gamma^m} 
& \left\langle 1, 1 \right\rangle_{\Gamma^m}
\end{pmatrix}\begin{pmatrix} \lambda_A^m \\ \lambda_V^m \end{pmatrix}
\nonumber \\ 
& \hspace{1cm}
= - \begin{pmatrix} \left\langle g^m +
\tfrac12 \left(\kappa^m_{\Gamma^m}-\spont\right)^2 \kappa^m_{\Gamma^m},\kappa^m_{\Gamma^m} \right\rangle_{\Gamma^m}^h
+ \left|\nabs\,\kappa^m_{\Gamma^m}\right|_{\Gamma^m}^2
\\ \left\langle g^m +
\tfrac12 \left(\kappa^m_{\Gamma^m}-\spont\right)^2 \kappa^m_{\Gamma^m},  1 
\right\rangle_{\Gamma^m}^h \end{pmatrix}.
\label{eq:matconsdis}
\end{align}
Once again, 
we observe that the matrix in (\ref{eq:matconsdis}) is symmetric positive 
semi\-definite,
and it is singular if and only if $\kappa^m_{\Gamma^m}$ is constant. 
Similarly, in the case of just one constraint, the Lagrange 
multiplier corresponding to the other constraint
is set to zero and the corresponding equation removed from
(\ref{eq:matconsdis}). 
Moreover, this modified scheme satisfies 
the suitably modified versions of Theorem~\ref{thm:uniqueWF}
and Remark~\ref{rem:BGN08WF}; see \cite{willmore} for details.     

We now discuss generalizations of the stable approximations of Willmore
flow in \S\ref{subsec:stableappro}
to the case of Willmore flow with spontaneous curvature and area
difference effects.
In \cite{pwfade}, the present authors extended the approach of \cite{Dziuk08}, 
see (\ref{eq:Dzwill1}), to the case of nonzero $\spont$ and $\beta$.

First of all, we consider the variation of $E_{\spont,\beta}(\Gamma(t))$
subject to the side constraint (\ref{eq:epscurv}).
Similarly to (\ref{eq:Lag2}), we introduce the Lagrangian
\begin{align*}
\widehat L_{\spont,\beta}(\Gamma(t), \vec\varkappa^\star, \vec y) &=
\tfrac12 \left\langle \left|\vec\varkappa^\star - \spont\,\vec\nu\right|^2,1 
\right\rangle_{\Gamma(t)} 
+ \tfrac12\,\beta\left( \left\langle \vec\varkappa^\star , \vec\nu 
\right\rangle_{\Gamma(t)} - M_0 \right)^2 \nonumber \\ &
\qquad - \left\langle \vec\varkappa^\star, \vec y \right\rangle_{\Gamma(t)} - 
\left\langle \nabs\,\vec\id, \nabs\,\vec y \right\rangle_{\Gamma(t)} , 
\end{align*}
where $\vec y(\cdot,t)\in \Vt$ is the Lagrange multiplier for (\ref{eq:epscurv})
with $\vec\varkappa(\cdot,t) \in [L^2(\Gamma(t))]^d$ replaced by
$\vec\varkappa^\star(\cdot,t) \in [L^2(\Gamma(t))]^d$.
As in (\ref{eq:deltaL}), we consider the variations,
for all $\vec\chi \in \Vt$, $\vec\xi \in [L^2(\Gamma(t))]^d$ 
and $\vec\eta \in \Vt$, 
\begin{align*}
\left[\deldel{\Gamma}\,\widehat L_{\spont,\beta}\right](\vec\chi) 
& = 
\left( \ddeps\, \widehat L_{\spont,\beta}(\Gamma_\epsilon(t), 
\varkappa^\star_{(\epsilon)}, \vec y_{(\epsilon)})
\right)_{\mid_{\epsilon=0}} , 
\\
\left[\deldel{\vec\varkappa^\star}\,\widehat L_{\spont,\beta}\right](\vec\xi) 
& = \left( \ddeps\,
\widehat 
L_{\spont,\beta}(\Gamma(t), \vec\varkappa^\star + \epsilon\,\vec\xi, \vec y) 
\right)_{\mid_{\epsilon=0}} , 
\\
\left[\deldel{\vec y}\,\widehat L_{\spont,\beta}\right](\vec\eta) & =
\left( \ddeps\,
\widehat
L_{\spont,\beta}(\Gamma(t), \vec\varkappa^\star , \vec y + \epsilon\,\vec\eta) 
\right)_{\mid_{\epsilon=0}} , 
\end{align*}
where $\vec\varkappa_{(\epsilon)}^\star(\cdot,t),\,\vec y_{(\epsilon)}(\cdot,t)
\in [H^1(\Gamma_\epsilon(t))]^d$ are 
defined by transporting the values of $\vec\varkappa^\star$ and $\vec y$ 
as in (\ref{eq:1919}). 
Setting $\left[\deldel{\vec y}\,\widehat L_{\spont,\beta}\right](\vec\eta)=0$ 
and comparing with (\ref{eq:epscurv}) leads to 
$\vec\varkappa^\star= \vec\varkappa$.
Moreover, on setting 
$\left[\deldel{\vec\varkappa^\star}\,\widehat L_{\spont,\beta}\right](\vec\xi)
=0$ 
and $\langle \vec{\mathcal V}, \vec\chi \rangle_{\Gamma(t)}
= - \left[\deldel{\Gamma}\,\widehat L_{\spont,\beta}\right](\vec\chi)$, we 
obtain,
on noting (\ref{eq:transeq}), (\ref{eq:denu}) and (\ref{eq:dedivdisc}), the
following weak formulation of $\vec{\mathcal{V}} = f\,\vec\nu$
and Lemma~\ref{lem:varkappa}\ref{item:vecvarkappa}, 
where $f$ is the right hand side of (\ref{eq:WFgen}); 
recall Remark~\ref{rem:GDWill}\ref{item:GDWillii}  
in the absence of spontaneous curvature and ADE effects.
Given a closed hypersurface $\Gamma(0)$, we seek an evolving hypersurface 
$(\Gamma(t))_{t\in[0,T]}$, 
with a global parameterization and induced velocity field
$\vec{\mathcal{V}}$, $\vec\varkappa \in [L^2(\GT)]^d$ 
and $\vec y \in [L^2(\GT)]^d$ as follows. 
For almost all $t \in (0,T)$, find
$(\vec {\mathcal V}(\cdot,t),\vec\varkappa(\cdot,t), \vec y(\cdot,t))
\in [L^2(\Gamma(t))]^d \times [L^2(\Gamma(t))]^d 
\times \Vt$ such that
\begin{subequations} \label{eq:Dzasym}
\begin{align}
& \left\langle \vec{\mathcal{V}}, \vec\chi \right\rangle_{\Gamma(t)} =
\left\langle \nabs\,\vec y, \nabs\,\vec\chi 
-2\,\mat D_s(\vec\chi)\right\rangle_{\Gamma(t)}+
\left\langle \nabs\,.\,\vec y, \nabs\,.\,\vec\chi \right\rangle_{\Gamma(t)}
\nonumber \\ & \hspace{1in} 
-\left\langle  \tfrac12\,|\vec\varkappa - \spont\,\vec\nu|^2
- \vec\varkappa\,.\left( \vec y- \beta\, A(t)\,\vec\nu\right), 
\nabs\,.\,\vec\chi \right\rangle_{\Gamma(t)}
\nonumber \\ & \hspace{1in}
+ \left(\beta\,A(t) - \spont\right) \left\langle \vec\varkappa, 
(\nabs\,\vec\chi)^\transT \vec\nu \right\rangle_{\Gamma(t)}
\quad\forall\ \vec\chi\in \Vt\,, \label{eq:Dzasyma} \\
& \left\langle \vec\varkappa + 
\left(\beta\,A(t)-\spont\right)\vec\nu - \vec y, \vec\xi 
\right\rangle_{\Gamma(t)}
= 0 \qquad
\forall\ \vec\xi \in [L^2(\Gamma(t))]^d \,, \label{eq:Dzasymb}  \\
& \left\langle \vec\varkappa , \vec\eta \right\rangle_{\Gamma(t)}+ \left\langle
\nabs\, \vec\id,\nabs\,\vec\eta\right\rangle_{\Gamma(t)} = 0 
\quad\forall\ \vec\eta \in \Vt\label{eq:Dzasymc} 
\end{align} 
with
\begin{equation} \label{eq:Dzasymd}
A(t)= \left\langle \vec\varkappa , \vec\nu
\right\rangle_{\Gamma(t)} - M_0 \,.
\end{equation}
\end{subequations}
We note that if $\spont = \beta = 0$ then $\vec y = \vec\varkappa$, and so
the system (\ref{eq:Dzasym}) reduces to (\ref{eq:Dziukwill}).

We now introduce a semidiscrete variant of (\ref{eq:Dzasym}) with
the help of the first variation of the discrete energy
\begin{equation}
E^h_{\spont,\beta}(\Gamma^h(t)) 
= \tfrac12 \left\langle \left|\vec\kappa^h - \spont\,\vec\nu^h\right|^2, 1
\right\rangle_{\Gamma^h(t)}^h
+ \tfrac12\,\beta \left( \left\langle \vec\kappa^h,\vec\nu^h
\right\rangle_{\Gamma^h(t)} - M_0 \right)^2 \label{eq:discent}
\end{equation}
subject to the side constraint (\ref{eq:Dzwill2b}).
Hence, we define the Lagrangian
\[
\widehat L^h_{\spont,\beta}(\Gamma^h(t), \vec\kappa^h, \vec Y^h) =
E^h_{\spont,\beta}(\Gamma^h(t)) 
- \left\langle \vec\kappa^h, \vec Y^h \right\rangle_{\Gamma^h(t)}^h 
- \left\langle \nabs\,\vec\id, \nabs\,\vec Y^h \right\rangle_{\Gamma^h(t)} ,
\]
where 
$\vec Y^h(\cdot,t)\in \Vht$ is the Lagrange multiplier for (\ref{eq:Dzwill2b}).
Similarly to (\ref{eq:Lag2h}), we set 
$\left[\deldel{\Gamma^h}\,\widehat L^h_{\spont,\beta}\right](\vec\chi) = 
- \langle \vec{\mathcal{V}^h}, \vec\chi \rangle_{\Gamma^h(t)}$
for $\vec\chi\in\Vht$, \arxivyesno{}{\linebreak}%
$\left[\deldel{\vec\kappa^h}\,\widehat L^h_{\spont,\beta}\right](\vec\xi) = 0$ 
for $\vec\xi\in\Vht$
and $\left[\deldel{\vec Y^h}\,\widehat L^h_{\spont,\beta}\right](\vec\eta) = 0$ 
for $\vec\eta\in\Vht$. 
Therefore, on noting the analogue of (\ref{eq:transeqh}),
the discrete analogue of (\ref{eq:denu}) for $\vec\nu^h$ on $\Gamma^h(t)$
and (\ref{eq:dedivdisc}) for $\Gamma^h(t)$,  
we have the following semidiscrete finite element approximation of
(\ref{eq:Dzasym}). 
Given the closed polyhedral hypersurface $\Gamma^h(0)$,
find an evolving polyhedral hypersurface $\GhT$ 
with induced velocity $\vec{\mathcal{V}}^h \in \VhGhT$,
$\vec\kappa^h \in \VhGhT$ and $\vec Y^h \in \VhGhT$ as follows.
For all $t \in (0,T]$, find
$(\vec{\mathcal{V}}^h(\cdot, t),\vec\kappa^h(\cdot,t),\vec Y^h(\cdot,t)) 
\in \Vht \times \Vht \times \Vht$ such that
\begin{subequations} \label{eq:Dzasymdisc}
\begin{align}
& \left\langle \vec{\mathcal{V}}^h, \vec\chi \right\rangle_{\Gamma^h(t)}^h
= \left\langle \nabs\,\vec Y^h, \nabs\,\vec\chi 
-2\,\mat D_s(\vec\chi) \right\rangle_{\Gamma^h(t)} +
\left\langle \nabs\,.\,\vec Y^h, \nabs\,.\,\vec\chi
 \right\rangle_{\Gamma^h(t)}
\nonumber \\ & \hspace{0.9in}
- \left\langle \tfrac12\left|\vec\kappa^h - \spont\,\vec\nu^h\right|^2
- \vec\kappa^h\,.\left(\vec Y^h - \beta\,A^h(t)\,\vec\nu^h\right), 
\nabs\,.\,\vec\chi
\right\rangle_{\Gamma^h(t)}^h
\nonumber \\ & \hspace{0.9in} 
+\left(\beta\,A^h(t)-\spont\right) 
\left\langle \vec\kappa^h, (\nabs\,\vec\chi)^\transT \vec\nu^h
\right\rangle_{\Gamma^h(t)}^h
\quad\forall\ \vec\chi\in \Vht\,, \label{eq:Dzasymdisca} \\
& \left\langle \vec\kappa^h + 
\left(\beta\,A^h(t)- \spont\right)\vec\nu^h - \vec Y^h, \vec\xi 
\right\rangle_{\Gamma^h(t)}^h
= 0 \quad\forall\ \vec\xi \in \Vht \,, \label{eq:Dzasymdiscb}  \\
& \left\langle \vec\kappa^h, \vec\eta \right\rangle_{\Gamma^h(t)}^h
+ \left\langle \nabs\, \vec\id,\nabs\,\vec\eta\right\rangle_{\Gamma^h(t)} = 0 
\quad\forall\ \vec\eta \in \Vht\,,\label{eq:Dzasymdiscc} 
\end{align} 
where
\begin{equation} \label{eq:Dzasymdiscd}
A^h(t)= \left\langle \vec\kappa^h , \vec\nu^h 
\right\rangle_{\Gamma^h(t)}^h - M_0\,.
\end{equation}
\end{subequations}
We note that (\ref{eq:Dzasymdiscb}) and (\ref{eq:omegahnuh}) imply that 
$\vec\kappa^h + \left(\beta\,A^h(t) - \spont\right)\vec\omega^h= \vec Y^h$ 
on $\Gamma^h(t)$.
Of course, $\vec\kappa^h$ can be eliminated from (\ref{eq:Dzasymdisc}).

We now prove the analogue of Theorem \ref{thm:Dzwill} for the semidiscrete 
scheme (\ref{eq:Dzasymdisc}).

\begin{thm} \label{thm:Dzasymdisc}
Let $(\GhT, \vec\kappa^h, \vec Y^h)$ be a solution of \eqref{eq:Dzasymdisc},
and let \arxivyesno{}{\linebreak}%
$\vec\kappa^h\in\VhTGhT$. Then it follows that  
\begin{equation}
\ddt \, E_{\spont,\beta}^h(\Gamma^h(t)) 
= - \left( \left|\vec{\mathcal{V}}^h\right|^h_{\Gamma^h(t)} \right)^2
\le 0\,.
\label{eq:sd3stab}
\end{equation}
\end{thm}
\begin{proof} 
As in the proof of Theorem~\ref{thm:Dzwill}, on
taking the derivative with respect to $t$ of
(\ref{eq:Dzasymdiscc}) 
we obtain (\ref{eq:Dzwill2c}).
Choosing $\vec\eta = \vec Y^h$ in (\ref{eq:Dzwill2c}) leads to
\begin{align} \label{eq:timderkappa}
& \left\langle \matpartxh\,\vec\kappa^h, \vec Y^h\right\rangle_{\Gamma^h(t)}^h
+ \left\langle \vec\kappa^h\,.\,\vec Y^h, 
\nabs\,.\,\vec{\mathcal{V}}^h \right\rangle_{\Gamma^h(t)}^h
+ \left\langle \nabs\,.\, \vec{\mathcal{V}}^h,
  \nabs\,.\,\vec Y^h\right\rangle_{\Gamma^h(t)} \nonumber \\ & \qquad
+ \left\langle \nabs\,\vec{\mathcal{V}}^h, \nabs\,\vec Y^h \right\rangle_{\Gamma^h(t)} 
- 2 \left\langle \mat D_s(\vec{\mathcal{V}}^h)\,, 
\nabs\,\vec Y^h \right\rangle_{\Gamma^h(t)} 
= 0\,.
\end{align}
Combining (\ref{eq:Dzasymdisca}) with $\vec\chi = \vec{\mathcal{V}}^h$
and (\ref{eq:timderkappa}), 
we obtain, on noting Lemma \ref{lem:dtnuh} and (\ref{eq:Dzasymdiscb}), that 
\begin{align} \label{eq:s32h}
& \left\langle \vec{\mathcal{V}}^h, \vec{\mathcal{V}}^h 
\right\rangle_{\Gamma^h(t)}^h
+ \left\langle \tfrac12 \left|\vec\kappa^h - \spont\,\vec\nu^h\right|^2
+ \beta\,A^h(t)\,\vec\kappa^h\,.\,\vec\nu^h, 
\nabs\,.\,\vec{\mathcal{V}}^h\right\rangle_{\Gamma^h(t)}^h
\nonumber \\ & \qquad\qquad 
+ \left(\beta\,A^h(t) - \spont\right) 
\left\langle \vec\kappa^h,\matpartxh\,\vec\nu^h
\right\rangle_{\Gamma^h(t)}^h \nonumber \\
& \ =
- \left\langle \matpartxh\,\vec\kappa^h,\vec Y^h \right\rangle_{\Gamma^h(t)}^h 
= 
- \left\langle \matpartxh\,\vec\kappa^h, 
\vec\kappa^h + \left(\beta\,A^h(t) - \spont\right)\vec\nu^h
\right\rangle_{\Gamma^h(t)}^h 
\,.
\end{align}
The desired result (\ref{eq:sd3stab}) then follows from 
(\ref{eq:s32h}), (\ref{eq:discent}), (\ref{eq:ip0norm}), 
Theorem~\ref{thm:disctrans} 
and (\ref{eq:Dzasymdiscd}), where we have observed that
\begin{align*} 
&\tfrac12\,\ddt \left(\left\langle\vec\kappa^h,\vec\nu^h
  \right\rangle_{\Gamma^h(t)}^h - M_0\right)^2  \nonumber \\ 
  &\quad = A^h(t) \left[  
\left\langle \matpartxh\,\vec\kappa^h,\vec\nu^h \right\rangle_{\Gamma^h(t)}^h
+ \left\langle \vec\kappa^h, \matpartxh\,\vec\nu^h \right\rangle_{\Gamma^h(t)}^h
+ \left\langle \vec\kappa^h\,.\,\vec\nu^h, 
  \nabs\,.\,\vec{\mathcal{V}}^h\right\rangle_{\Gamma^h(t)}^h
 \right].
\end{align*}
\end{proof}

We now state a fully discrete variant of the semidiscrete
scheme (\ref{eq:Dzasymdisc}). 
Let the closed polyhedral hypersurface $\Gamma^0$ be an approximation to 
$\Gamma(0)$, and let $\vec\kappa^0_{\Gamma^0}, \vec Y^0_{\Gamma^0} \in \Vhz$ 
and $A^0\in\bR$ be given.
Then, for $m=0,\ldots, M-1$, find
$(\vec X^{m+1},\vec\kappa^{m+1},\vec Y^{m+1}) \in\Vhm\times\Vhm\times\Vhm$
such that
\begin{subequations} \label{eq:GDade}
\begin{align}
& \left\langle\frac{\vec X^{m+1}-\vec\id}{\ttau_m},
    \vec\chi\right\rangle_{\Gamma^m}^h
= \left\langle \nabs\,\vec Y^{m+1} , \nabs\,\vec\chi
\right\rangle_{\Gamma^m}
+ \left\langle \nabs\,.\,\vec Y^{m}_{\Gamma^m} ,
\nabs\,.\,\vec\chi \right\rangle_{\Gamma^m}
\nonumber \\ & \hspace{1cm}
 - 2 \left\langle \nabs\,\vec Y^m_{\Gamma^m} ,
 \mat D_s(\vec\chi) \right\rangle_{\Gamma^m}
+ \left(\beta\,A^m - \spont\right) \left\langle \,\vec\kappa^m_{\Gamma^m} ,
  (\nabs\,\vec\chi)^\transT \,\vec\nu^m \right\rangle_{\Gamma^m}^h
\nonumber \\ & \hspace{1cm}
- \left\langle \tfrac12\left|\vec\kappa^m_{\Gamma^m} -
\spont\,\vec\nu^m\right|^2
 -\vec\kappa^m_{\Gamma^m}\,.\left(\vec Y^m_{\Gamma^m} -
\beta\,A^m\,\vec\nu^m\right), \nabs\,.\,\vec\chi \right\rangle_{\Gamma^m}^h
\nonumber \\ & \hspace{6.5cm}
\quad\forall\ \vec\chi \in \Vhm\,, \label{eq:GDadea}
\\
&\vec\kappa^{m+1} = \vec Y^{m+1} - \left(\beta\, A^m - \spont\right)
\vec\omega^m \,, \label{eq:GDadeb} \\
& \left\langle \vec\kappa^{m+1}, \vec\eta \right\rangle_{\Gamma^m}^h
+ \left\langle \nabs\,\vec X^{m+1}, \nabs\,\vec\eta
\right\rangle_{\Gamma^m} = 0
 \quad\forall\ \vec\eta \in \Vhm \label{eq:GDadec} 
\end{align}
\end{subequations}
and set $\Gamma^{m+1} = \vec X^{m+1}(\Gamma^m)$, 
$A^{m+1} = \left\langle \vec\kappa^{m+1}, \vec\nu^{m}
\right\rangle_{\Gamma^{m}} - M_0$,
$\vec\kappa^{m+1}_{\Gamma^{m+1}}= \vec\kappa^{m+1}\circ (\vec X^{m+1})^{-1}
\in\Vhmp$
and 
$\vec Y^{m+1}_{\Gamma^{m+1}}= \vec Y^{m+1} \circ (\vec X^{m+1})^{-1}
\in\Vhmp$. Of course, $\vec\kappa^{m+1}$ 
can be eliminated from (\ref{eq:GDade}).

\begin{thm}\label{thm:GDade}
Let $\vec\kappa^m_{\Gamma^m},\,\vec Y^m_{\Gamma^m} \in \Vhm$ and $A^m \in \bR$.
Then there exists a unique solution
$(\vec X^{m+1},\vec\kappa^{m+1}, \vec Y^{m+1}) \in \Vhm\times \Vhm \times \Vhm$
to {\rm (\ref{eq:GDade})}. 
\end{thm}
\begin{proof}
It is a simple matter to adapt the proof of Theorem \ref{thm:FDBGNWill}.
\end{proof}

\begin{rem}[Discrete linear systems] \label{rem:GDsolve}
Similarly to {\rm Remark~\ref{rem:WFsolve}}, the 
linear systems of equations to be solved at each time level for
\eqref{eq:GDade} can be formulated as follows.
Find $(\vec Y^{m+1},\vec\kappa^{m+1},\delta\vec X^{m+1})
\in (\bR^d)^K\times(\bR^d)^K\times (\bR^d)^K$ such that
\begin{equation*}
\begin{pmatrix}
 \mat A_{\Gamma^m} & 0 & -\frac1{\ttau_m}\,\mat M_{\Gamma^m} \\
-\mat M_{\Gamma^m} & \mat M_{\Gamma^m} & 0 \\
0 & \mat M_{\Gamma^m} & \mat A_{\Gamma^m}
\end{pmatrix}
\begin{pmatrix} \vec Y^{m+1} \\ \vec\kappa^{m+1} \\  \delta\vec X^{m+1} 
\end{pmatrix}
=
\begin{pmatrix} \vec {\mathfrak g}_{\Gamma^m} \\ 
- \left(\beta\, A^m - \spont\right)\,\mat M_{\Gamma^m}\,\vec\omega^m
 \\ -\mat A_{\Gamma^m}\,\vec X^m \end{pmatrix} \,,
\end{equation*}
and this system can be efficiently
solved with a sparse direct solution method like UMFPACK, see 
{\rm \cite{Davis04}}.
\end{rem}

Similarly to Remark \ref{rem:GDWill}\ref{item:GDWillii}, 
the semidiscrete scheme (\ref{eq:Dzasymdisc}) and its fully discrete
version (\ref{eq:GDade}) do not have good mesh properties.
In order to obtain a scheme with good mesh properties, we extend
the semidiscrete approximation
(\ref{eq:weak2ah}), (\ref{eq:kappaynuh}) and (\ref{eq:weakcurv12h}) to take into
account spontaneous curvature and ADE effects.
First, we extend the weak formulation
(\ref{eq:weak2a}), (\ref{eq:kappaynu}) and (\ref{eq:weakcurv12})
to include spontaneous curvature and ADE effects. 
In order to do so, we extend the Lagrangian (\ref{eq:Lag2}) to
\begin{align}
L_{\spont,\beta}(\Gamma(t), \varkappa^\star, \vec y) &=
\tfrac12 \left\langle \varkappa^\star - \spont, 
\varkappa^\star - \spont\right\rangle_{\Gamma(t)} 
+ \tfrac12\,\beta \left( \left\langle \varkappa^\star , 1 
\right\rangle_{\Gamma(t)} - M_0 \right)^2 \nonumber \\ & \hspace{2cm}
- \left\langle \varkappa^\star \,\vec\nu, \vec y \right\rangle_{\Gamma(t)} - 
\left\langle \nabs\,\vec\id, \nabs\,\vec y \right\rangle_{\Gamma(t)} . 
\label{eq:Lag33}
\end{align}
Once again, setting the variation 
$\left[\deldel{\vec y}\,L_{\spont,\beta}\right](\vec\eta)=0$, yields
(\ref{eq:weakcurv12}) with $\varkappa$ replaced by $\varkappa^\star$,
and so $\varkappa^\star=\varkappa$.
Setting $\left[\deldel{\varkappa^\star}\,L_{\spont,\beta}\right](\xi)=0$, 
yields, on noting $\varkappa^\star=\varkappa$ and \eqref{eq:At}, that 
\begin{equation} \label{eq:def3A}
\varkappa - \spont = \vec y\,.\,\vec\nu - \beta\,A(t)
\qquad\mbox{on } \Gamma(t) \,.
\end{equation}
Finally, on setting the variation $\left[\deldel{\Gamma}\,L_{\spont,\beta}\right](\vec\chi) =
-\left\langle \vec {\mathcal V}\,.\,\vec\nu, \vec\chi\,.\,\vec\nu 
\right\rangle_{\Gamma(t)}$
and noting (\ref{eq:transeq}), (\ref{eq:denu}), (\ref{eq:dedivdisc})
and that $\varkappa^\star= \varkappa$, we obtain
\begin{align}
\left\langle \vec{\mathcal{V}}\,.\,\vec\nu, \vec\chi\,.\,\vec\nu 
\right\rangle_{\Gamma(t)} &= 
\left\langle \nabs\,\vec y, \nabs\,\vec\chi 
-2\,\mat{D}_s(\vec\chi) \right\rangle_{\Gamma(t)} 
+ \left\langle \nabs\,.\,\vec y, \nabs\,.\,\vec\chi
  \right\rangle_{\Gamma(t)}
\nonumber \\  
& \qquad  
- \left\langle \tfrac12\,( 
\varkappa -\spont)^2 - \varkappa\left(\vec y\,.\,\vec\nu-\beta\,A(t)\right), 
\nabs\,.\,\vec\chi \right\rangle_{\Gamma(t)} \nonumber \\
& \qquad -  \left\langle \varkappa \,\vec y, 
(\nabs\,\vec\chi)^\transT \,\vec\nu \right\rangle_{\Gamma(t)}
\quad\forall\ \vec\chi \in \Vt \,.
\label{eq:LM3a}
\end{align}

Therefore we have the following weak formulation. 
Given a closed hypersurface $\Gamma(0)$, 
we seek an evolving hypersurface $(\Gamma(t))_{t\in[0,T]}$, 
with a global parameterization and induced velocity field
$\vec{\mathcal{V}}$, $\varkappa \in L^2(\GT)$
and $\vec y \in [L^2(\GT)]^d$ as follows. 
For almost all $t \in (0,T)$, find
$(\vec {\mathcal V}(\cdot,t),\varkappa(\cdot,t), \vec y(\cdot,t))
\in [L^2(\Gamma(t))]^d\times L^2(\Gamma(t)) \times \Vt$ 
such that (\ref{eq:LM3a}), (\ref{eq:def3A}) and (\ref{eq:weakcurv12}) hold.

Once again, a remark equivalent to Remark \ref{rem:WWF} for this
weak formulation holds. In addition, 
one can extend   
the semidiscrete finite element approximation 
(\ref{eq:weak2ah}), (\ref{eq:kappaynuh}) and
(\ref{eq:weakcurv12h}) to approximate the weak formulation
(\ref{eq:LM3a}), (\ref{eq:def3A}) and (\ref{eq:weakcurv12}).
Moreover, one can extend Remark \ref{rem:WWFh} and Theorem \ref{thm:WWFh}
to this approximation. Furthermore, one can also incorporate    
volume and surface area constraints and still prove stability
of the approximation.

In \cite{pwfade}, we considered an extension of the semidiscrete
approximation (\ref{eq:weak2ah}), (\ref{eq:kappaynuh}) and 
(\ref{eq:weakcurv12h}),
and its extension to incorporate spontaneous curvature and ADE effects, 
which possibly reduces the tangential motion. This is based on the Lagrangian
\begin{equation}
\widehat
L^h_{\spont,\beta,\theta}(\Gamma^h(t), \vec\kappa^h, \vec Y^h) =
E^h_{\spont,\beta}(\Gamma^h(t))
- \left\langle \mat Q^h_\theta \,\vec\kappa^h, \vec Y^h \right\rangle^h_{\Gamma^h(t)} 
- \left\langle \nabs\,\vec\id,
\nabs\,\vec Y^h \right\rangle_{\Gamma^h(t)} , \label{eq:Lag2th}
\end{equation}
where $\vec Y^h(\cdot,t) \in \Vht$ is the Lagrange
multiplier associated with the constraint $\vec\kappa^h(\cdot,t) \in \Vht$ 
satisfying
\begin{equation}\label{eq:weakcurv12th}
\left\langle\mat Q^h_\theta\,\vec\kappa^h,\vec\eta\right\rangle^h_{\Gamma^h(t)}
+ \left\langle\nabs\,\vec\id, \nabs\,\vec\eta\right\rangle_{\Gamma^h(t)} = 0
\quad\forall\ \vec\eta\in \Vht\,.
\end{equation}
Here, $\mat Q^h_\theta$, for a given $\theta \in [0,1]$, 
is the semidiscrete version of (\ref{eq:Qm}), where $\vec\omega^m$
is replaced by $\vec\omega^h$, 
and where we assume that Assumption~\ref{ass:A}\ref{item:assA2} holds. 

Setting $\left[\deldel{\vec Y^h}\,\widehat 
L^h_{\spont,\beta,\theta}\right](\vec\eta)=0$,
for all $\vec\eta \in \Vht$, yields (\ref{eq:weakcurv12th}).
Setting 
\arxivyesno{$\left[\deldel{\kappa^h}\,\widehat 
L^h_{\spont,\beta,\theta}\right](\vec\xi)$ $=0$,}
{$\left[\deldel{\kappa^h}\,\widehat 
L^h_{\spont,\beta,\theta}\right](\vec\xi)=0$,}
for all $\vec\xi \in \Vht$, yields that  
$\vec\kappa^h + (\beta\,A^h(t)-\spont)\,\vec\omega^h
= \vec\pi_{\Gamma^h} \left[\mat Q^h_\theta\,\vec Y^h\right]$,
where $A^h(t)$ is given by (\ref{eq:Dzasymdiscd}).
Finally, we set 
$\left[\deldel{\Gamma^h}\,\widehat
L^h_{\spont,\beta,\theta}\right](\vec\chi) =
-\langle \mat Q^h_\theta\,\vec {\mathcal V}^h, \vec\chi \rangle^h_{\Gamma(t)}$,
for all $\vec\chi \in \Vht$.
If $\theta=1$, the resulting scheme collapses to the semidiscrete Dziuk scheme 
(\ref{eq:Dzasymdisc}). 
If $\theta =0$, this scheme collapses to a variant of 
(\ref{eq:weak2ah}), (\ref{eq:kappaynuh}) and (\ref{eq:weakcurv12h}),
which takes spontaneous curvature and ADE effects into account,  
and this scheme still satisfies 
{\rm Theorem~\ref{thm:SDSD}\ref{item:SDSDTM}}.
For any $\theta \in (0,1)$, the scheme interpolates between these two extremes.
Moreover, this scheme, for any given $\theta \in [0,1]$,
satisfies a stability bound, a generalization of Theorem \ref{thm:Dzasymdisc};
see \citet[Theorem 3.3]{pwfade} for details.        
In addition, this stability bound holds for a variant 
involving volume and surface area constraints.

\begin{rem}[Surfaces with boundary] \label{rem:WFboundaries}
For an evolving surface with a 
boun\-dary, the result of 
{\rm Lemma~\ref{lem:Willmorevargen}} can be generalized as follows, where we
recall the shorthand notation $\mathcal{V} = f$ for \eqref{eq:WFgen}.
On noting {\rm Theorem~\ref{thm:trans}},
{\rm Lemma~\ref{lem:derkappa}\ref{item:lem10.3ii}} and   
{\rm Theorem~\ref{thm:div}}, it holds that 
\begin{align}
& \ddt\,E_{\spont,\beta}(\Gamma(t)) 
= - \left\langle f, \mathcal{V} \right\rangle_{\Gamma(t)}
+ \int_{\partial\Gamma(t)} \left( \tfrac12\,(\varkappa-\spont)^2 +
\beta\,A(t)\,\varkappa\right) \vec{\mathcal{V}}\,.\,\conormal\dH{d-2}
\nonumber \\ & \qquad
- \int_{\partial\Gamma(t)} \mathcal{V}\,\conormal\,.\,\nabs\,\varkappa\dH{d-2} 
+ \int_{\partial\Gamma(t)} \left(\varkappa-\spont+
\beta\,A(t)\right) \conormal\,.\,\nabs\,\mathcal{V}\dH{d-2}\,,
\label{eq:WFbc}
\end{align}
where $\vec{\mathcal{V}}$ is the velocity field induced by a global
parameterization of the evolving hypersurface, and 
$\conormal(t)$ denotes the outer unit conormal on $\partial\Gamma(t)$.
Hence, in order to ensure that \eqref{eq:WFgen} is still the $L^2$--gradient
flow of $E_{\spont,\beta}(\Gamma(t))$, conditions need to be prescribed at the
boundary so that the boundary terms in \eqref{eq:WFbc} vanish.
In particular, the following boundary conditions may be considered.
In the simplest situation the boundary is kept fixed, i.e.\ 
$\partial\Gamma(t) = \partial\Gamma(0)$ for all $t \in (0,T]$, 
which is the same as \eqref{eq:fixedbc}. 
This leads to the first two boundary terms vanishing in \eqref{eq:WFbc}.
The third boundary term vanishes if either 
$\varkappa-\spont + \beta\,A(t)=0$ or 
$\conormal\,.\,\nabs\,\mathcal{V}=0$ on $\partial\Gamma(t)$.
Setting $\partial\Gamma(t) = \partial\Gamma(0)$ and 
$\varkappa-\spont + \beta\,A(t)=0$ on $\partial\Gamma(t)$
for all $t \in (0,T]$ are called Navier boundary conditions.
Whereas, for a given $\vec\zeta\in C(\partial\Gamma(0),\bS^{d-1})$, 
setting $\partial\Gamma(t) = \partial\Gamma(0)$ and 
$\conormal(t) = \vec\zeta$ on $\partial\Gamma(t)$
for all $t \in (0,T]$ are called clamped boundary conditions.
It is a simple matter to show that clamped boundary conditions lead to
$\conormal\,.\,\nabs\,\mathcal{V}=0$ on $\partial\Gamma(t)$, 
and hence to all three boundary terms vanishing in \eqref{eq:WFbc}.
 
A third type of possible boundary conditions are called free boundary
conditions. Here the boundary can move freely and the three boundary terms
in \eqref{eq:WFbc} vanish on imposing
\[
\tfrac12\left(\varkappa-\spont\right)^2 +
\beta\,A(t)\,\varkappa=0\,, \quad
\conormal\,.\,\nabs\,\varkappa =0\,, \quad 
\varkappa-\spont + \beta\,A(t)=0 \quad\mbox{on }\partial\Gamma(t)\,.
\]
We refer to {\rm \cite{pwf}} for more details in the case $d=2$, 
and to {\rm \cite{pwfopen}} for more details in the case $d=3$.
The discussion in the latter article easily generalizes to $d>3$.
Moreover, in {\rm \cite{pwfopen}} the present authors 
also included a boundary energy, 
$\varsigma\,\mathcal{H}^{d-2}(\partial\Gamma(t))$ for $\varsigma\in\bRgeq$,  
which is called line energy in the case $d=3$,
and Gaussian curvature effects.
Here the inclusion of the Gaussian curvature effects 
is achieved via the Gauss--Bonnet theorem, 
recall {\rm Theorem~\ref{thm:GB}}.

In the case of Navier boundary conditions, one can easily adapt the 
finite element approximation \eqref{eq:Willcons}, \eqref{eq:Willb}
by replacing $\Vhm$ in \eqref{eq:Willb} by $\VhmD$, recall \eqref{eq:VhmD},
by replacing $\Whm$ in \eqref{eq:Willcons} by $\WhmD$, where
$\VhmD = [\WhmD]^d$, and
by seeking 
\arxivyesno{$(\vec X^{m+1},$ $\kappa^{m+1}) \in \Vhm \times \Whm$}
{$(\vec X^{m+1},\kappa^{m+1}) \in \Vhm \times \Whm$}
such that $\vec X^{m+1} - \vec\id_{\mid_{\Gamma^m}} \in \VhmD$ 
and $\kappa^{m+1} - (\spont -\beta\,A^m) \in \WhmD$.
For clamped and free boundary conditions this finite element approximation 
cannot be adapted.
However, the approaches based on 
\eqref{eq:Lag33} and \eqref{eq:Lag2th} can be adapted to all three
types of boundary conditions. 
For the latter approach, we once again refer to {\rm \cite{pwf,pwfopen}} 
for more details, including stability results for semidiscrete discretizations.
Furthermore, in {\rm \cite{pwfc0c1}} we extended this approach to approximate
gradient flows for two-phase biomembranes, where instead of boundary 
conditions certain matching conditions need to hold across interfaces between
different phases on the hypersurface. 
Finally, a related problem is the evolution of curve networks under elastic
flow, where two or more curves meet at junction points. This problem has been
considered by the authors in {\rm \cite{pwftj}}.
\end{rem}

\subsection{Alternative numerical approaches}
\label{subsec:alternativeWF}
We note that \cite{Rusu05} first introduced a mixed variational form
for Willmore flow, based on position and mean curvature vectors, 
which allowed for the approximation by continuous piecewise
linear finite elements. This approach was extended to surfaces with 
boundaries and applied to surface restoration problems in \cite{ClarenzDDRR04}.  
We refer also to \cite{DeckelnickD06,DeckelnickS10,DeckelnickKS15}, 
where, on assuming a sufficiently
smooth solution, an error analysis is presented
for semidiscrete finite element approximations 
for a graph formulation of Willmore flow.  

Other parametric numerical methods are discussed in 
\cite{DziukKS02,MayerS02,BobenkoS05,BonitoNP10,ElliottS10,ElliottS13,%
BartelsDNR12,BartezzaghiDQ19}.
In addition, a numerical method based on a level set approach has been
studied in \cite{DroskeR04}, while phase field type approaches are discussed 
in \cite{DuLRW05,EsedogluRR14,BretinMO15}.

\section{Biomembranes} \label{sec:bio}

In this section we will discuss how the first variation
of curvature energies discussed in Section~\ref{sec:willmore_flow} 
appear in evolution  laws involving also bulk
quantities. This is relevant for the evolution of vesicles and biomembranes,
where the curvature energy interacts with a fluid surrounding the membrane.
Biomembranes are lipid bilayers and they form bag-like structures containing
fluid and they are surrounded by a possibly different fluid. Membranes appear
in a multitude of biological systems and a proper understanding of the form
and the evolution of biomembranes is of major interest in the life sciences.
As an example, we mention that the biconcave shapes of red blood cells
appear as stationary states of the curvature energies discussed in 
Section~\ref{sec:willmore_flow}.

\subsection{A model for the dynamics of fluidic biomembranes}
\label{subsec:1.10.1}

We will introduce a model, which is a variant of a model introduced by
\cite{ArroyoS09}, that couples the first variation of a curvature energy
to the (Navier--)Stokes equations in the bulk and on the surface.
To be more precise a bulk (Navier--)Stokes system is coupled to
a tangential (Navier--)Stokes system
on the membrane. The tangential (Navier--)Stokes system also includes
a surface incompressibility condition, which will lead to good mesh 
properties on the discrete level even
in situations where the fluid velocity leads to a high deformation of the
surface.  

A generalized elastic energy for a biomembrane is given by
\begin{equation}\label{eq:CEHEn}
\int_\Gamma \tfrac12\, \alpha(\varkappa-\spont)^2 + \alpha^G\, \Gauss\dH{d-1}
+  \tfrac12\,\alpha\,\beta 
\left( \left\langle \varkappa,1 \right\rangle_{\Gamma} -M_0 \right)^2
\end{equation}
for a compact hypersurface $\Gamma\subset\bR^d$ without boundary,
$d\geq2$. 
Here $\alpha\in\bRplus$ and $\alpha^G\in\bR$ are the bending and Gaussian 
bending rigidities. As before, $\spont$ is the spontaneous
curvature, which arises for example from local inhomogeneities within
the membrane and the parameters $\beta\in\bRgeq$, $M_0\in \bR$ relate
the area difference elasticity (ADE) model, recall \S\ref{subsec:WFspont}.
Moreover, $\Gauss$ is the Gaussian curvature. It is discussed
in \cite{Nitsche93} that the most general form of a curvature energy
of the form $\int_\Gamma q(\varkappa_1,\ldots,\varkappa_{d-1}) \dH{d-1}$,
with $q$ being at most quadratic in the principal curvatures 
and invariant under permutations of its arguments, 
has the form $q(\varkappa_1,\ldots,\varkappa_{d-1}) = \frac12\,
\alpha\,\varkappa^2+\alpha^G\,\Gauss + \alpha_1\,\varkappa + \alpha_2$,
which leads to the first term in \eqref{eq:CEHEn} by choosing
$\alpha_1=-\alpha\,\spont$ and $\alpha_2=\frac12\,\alpha\,\spont^2$. 
If $\alpha^G$ is constant we obtain
from the Gauss--Bonnet theorem, recall Theorem~\ref{thm:GB}, that, 
as the surface has no boundary, the contribution 
$\int_\Gamma \alpha^G\, \Gauss\dH{d-1}$ is
constant in a fixed topological class, and hence will always disappear
in a first variation. For now, we will hence set $\alpha^G=0$ and only
if one considers inhomogeneous membranes, or open membranes with boundary, 
one has to consider this term, see 
\cite{ElliottS10,pwfopen,nsns2phase,pwfc0c1}.
Hence, overall, we consider from now on
$\alpha\,E_{\spont,\beta}(\Gamma(t))$ as the elastic energy for an 
evolving biomembrane, where $E_{\spont,\beta}(\Gamma(t))$ is as defined in
\eqref{eq:ADESCa}.

We adopt the notation in \S\ref{subsec:evolve}, and assume that 
two fluids occupy regions $\Omega_-(t)$ and
$\Omega_+(t) = \Omega\setminus \overline{\Omega_-(t)}$
with $\Gamma(t)=\partial\Omega_-(t)$,
and with $\vec\nu$ denoting the outer unit normal to $\Omega_-(t)$
on $\Gamma(t)$, recall Figure~\ref{fig:sketch} in \S\ref{subsec:evolve}.
Here $\Omega \subset \bR^d$ is a fixed domain, with $d\geq2$.
In these two fluid regions we require, as in Section~\ref{sec:tpf},
the incompressible Stokes equations with a no-slip boundary condition on 
$\partial\Omega$, i.e.\
\begin{equation}\label{eq:NSNS1}
 - \nabla\,.\,\mat\sigma = \vec 0\,,\quad\nabla\,.\,\vec u = 0
  \qquad\text{in } \Omega_\pm(t),\qquad
  \vec u = \vec 0 \quad\mbox{on }\partial\Omega
\end{equation}
with the stress tensor $\mat\sigma$ as defined in \eqref{eq:sigma}. 
Once again, we refer to Remark~\ref{rem:NSHG}\ref{item:NSHGbc} for more general
boundary conditions.

We remark that for 
biomembranes scales are typically such that the Stokes approximation of
the Navier--Stokes equations is a very good approximation, see
\cite{nsnspre}. Therefore, we will consider the Stokes system in the
bulk in what follows and refer to \cite{nsns,nsnsade} and
Section~\ref{sec:tpf} for a discussion on how to generalize the
following considerations to Navier--Stokes flow. 

We now follow and generalize an approach of \cite{ArroyoS09}, who
used the theory of interfacial fluid dynamics, which goes back to
\cite{Scriven60}, to formulate evolution laws that take the fluidic
behaviour of vesicles and biomembranes into account. In fact, these
lipid bilayer membranes can be described as a two-dimensional surface,
where the lipid molecules have a fluid-like behaviour in the tangential
direction. However, elastic forces stemming from the elastic bending
energy act in normal direction.
For more details on the modelling and the analysis of fluidic interfaces
we refer to \cite{SlatterySO07,BotheP10}. For our purposes it suffices to 
mention that in the absence of mass transfer to/from the interface from/to the
bulk, and on assuming a no-slip condition between the outer and inner fluid 
at the interface, it is natural to assume that the bulk velocity is 
continuous
across the interface and that the interface is moved with the bulk velocity.
We refer to \citet[p.\ 1830]{bssf} for more details.

Overall the following conditions need to hold on the free surface $\Gamma(t)$:
\begin{subequations} \label{eq:bio}
\begin{alignat}{2}
\left[\vec u\right]_-^+ & = \vec 0 \qquad &&\mbox{on } \Gamma(t)\,, 
\label{eq:bio1} \\ 
\rho_\Gamma \, \matpartu\,\vec u - \nabs\,.\,\mat\sigma_\Gamma & =
\left[\mat\sigma\,\vec\nu\right]_-^+ + \alpha\, \vec f_\Gamma
\qquad &&\mbox{on } \Gamma(t)\,, \label{eq:bio2} \\ 
\nabs\,.\,\vec u & = 0
\qquad &&\mbox{on } \Gamma(t)\,, \label{eq:bio3} \\ 
\mathcal{V} &= \vec u\,.\,\vec\nu \qquad &&\mbox{on } \Gamma(t)\,, 
\label{eq:bio4} 
\end{alignat}
\end{subequations}
where $\rho_\Gamma \in \bRgeq$
denotes the surface material density,
$\alpha \in \bRplus$ is the bending rigidity and 
$\vec f_\Gamma = f_\Gamma\,\vec\nu$, with
$f_\Gamma$ denoting minus the first variation of the bending
energy $E_{\spont,\beta}$, see \eqref{eq:ADESCa}, i.e.\
\begin{equation}
f_\Gamma = -\Delta_s\,\varkappa - (\varkappa-\spont)|\nabs\,\vec\nu|^2
+ \tfrac12 (\varkappa-\spont)^2\,\varkappa 
- \beta\,A(t)\left(|\nabs\,\vec\nu|^2-\varkappa^2\right)\,,
\label{eq:fWF}
\end{equation}
where $\varkappa$ satisfies Lemma~\ref{lem:varkappa}\ref{item:vecvarkappa}
and $A(t) = \langle\varkappa,1\rangle_{\Gamma(t)} - M_0$, recall 
Lemma \ref{lem:Willmorevargen}. 
Moreover,
\begin{equation}\label{eq:materialder}
\matpartu\,f = \partial_t\,f + \vec u\,.\,\nabla\,f
\end{equation}
is the material time derivative with respect to
the fluid velocity $\vec u$. This definition is related to
Remark~\ref{rem:timeder}\ref{item:rem2.19ii}, where the parameterizations 
$\vec x(\cdot,t)$ here are defined such that
\begin{equation} \label{eq:vecuvecx}
\vec u(\vec x(\vec q,t),t) = (\partial_t\,\vec x)(\vec q,t) \qquad
\forall\, (\vec q,t)\in \Upsilon \times [0,T]\,,
\end{equation}
as opposed to Definition \ref{def:globalx}\ref{item:vecV}, 
so that $\vec x(\vec q,t)$ really is the trajectory of a material
point. Equation~\eqref{eq:bio3} is conservation of mass of the interface, 
similarly to standard incompressible (Navier--)Stokes in the bulk, recall 
\eqref{eq:NSNS1}. In particular, it implies that the membrane is
locally incompressible.
This can be seen by considering 
a surface patch $\sigma(t) \subset \Gamma(t)$ that is transported with the 
fluid velocity $\vec u$, i.e.\ $\sigma(t) = \vec x(\sigma_\Upsilon,t)$
with $\sigma_\Upsilon \subset \Upsilon$ and 
$\vec x$ satisfying \eqref{eq:vecuvecx}. Then it follows
with the help of the transport theorem, Theorem~\ref{thm:trans}, that
\[
\ddt\,\mathcal{H}^{d-1}(\sigma(t)) = 
\int_{\sigma(t)} \nabs\,.\,\vec{\mathcal{V}} \dH{d-1} =
\int_{\sigma(t)} \nabs\,.\,\vec u \dH{d-1} = 0\,,
\]
see also \cite{ArroyoS09}.
The balance of momentum on the surface implies \eqref{eq:bio2}, see
\cite{ArroyoS09,KohneL18,Lengeler18}. Here the surface stress tensor
is defined by
\begin{equation}
\mat\sigma_\Gamma=2\,\mu_\Gamma\,\mat D_s(\vec u)-p_\Gamma\,
\mat P_\Gamma \qquad\text{on } \Gamma(t)\,,
\label{eq:surfstress}
\end{equation}
where $p_\Gamma$ is the surface pressure, $\mu_\Gamma \in \bRplus$ 
is the surface shear viscosity, $\mat P_\Gamma$ is the projection onto the 
tangent space and $\mat D_s(\vec u)$ is the surface rate of deformation tensor, 
recall {\rm Definition~\ref{def:2.5}\ref{item:def2.5vi}}.
The tensor $\mat D_s(\vec u)$ is the relevant tensor for
measuring the rate of change of lengths and angles through the metric
tensor, cf.\ Lemma~\ref{lem:dtgij}.

Finally, the term $\left[\mat\sigma\,\vec\nu\right]^+_-$, 
recall \eqref{eq:fjump} and its generalization to vector-valued quantities
in Remark~\ref{rem:dnufjump},
is the force exerted by the bulk on the surface. The total bending energy
considered for now is given by $\alpha\, E_{\spont,\beta}(\Gamma)$
with $E_{\spont,\beta}$, the dimensionless energy defined in (\ref{eq:ADESCa}), 
and $\alpha$, having the dimension of energy, is the bending rigidity. As in
\S\ref{subsec:stokes}, \eqref{eq:bio4} states that the
interface evolves with the normal component of the fluid velocity.

The system \eqref{eq:NSNS1}, \eqref{eq:bio} is closed with an initial 
condition for $\Gamma(0)$ and
\begin{equation*}
\rho_\Gamma\,\vec u(\cdot,0) = \rho_\Gamma\,\vec u_0 \qquad\text{on } 
\Gamma(0)\,,
\end{equation*} 
where for simplicity we assume that $\vec u_0 : \Omega \to \bR^d$ is 
a given initial velocity.
 
\begin{lem} \label{lem:NSNSweak}
Let $(\GT,\vec u, p)$ be a sufficiently smooth solution of 
\eqref{eq:NSNS1}, \eqref{eq:bio} and \eqref{eq:fWF}.
\begin{enumerate}
\item \label{item:nsnsi}
It holds that
\[\ddt\, |\Gamma(t)| = 0\,.\]
\item \label{item:nsnsii}
It holds that
\[\ddt\, \mathcal{L}^d(\Omega_-(t)) = 0\,.\]
\item \label{item:nsnsiii}
It holds that
\begin{align*}
&\ddt \left [ \alpha\,E_{\spont,\beta}(\Gamma(t)) +
\tfrac12\,\rho_\Gamma\left|\vec u\right|^2_{\Gamma(t)}\right ]
\\ & \hspace{2cm}
= -2\,\mu_\Gamma 
\left\langle\mat D_s(\vec u),\mat D_s(\vec u)\right\rangle_{\Gamma(t)}
-2\left\langle \mu \,\mat D(\vec u), \mat D(\vec u)\right\rangle\leq 0\,.
\end{align*}
\end{enumerate}
\end{lem}
\begin{proof} 
Let $\vec x$ be an arbitrary global parameterization of $\GT$, with induced
velocity field $\vec{\mathcal{V}}$.

\ref{item:nsnsi}
Combining the transport theorem, Theorem~\ref{thm:trans}, 
and the divergence theorem, Theorem~\ref{thm:div},
with the tangential fields $\vec f = \vec{\mathcal{V}}_\subT$
and $\vec f = \mat P_\Gamma\,\vec u$, gives
\begin{align}
\ddt\left\langle 1,1\right\rangle_{\Gamma(t)} & = 
\left\langle 1, \nabs\,.\,\vec{\mathcal{V}} \right\rangle_{\Gamma(t)} =
\left\langle 1, \nabs\,.\left(\vec{\mathcal{V}} - \vec{\mathcal{V}}_\subT
\right) \right\rangle_{\Gamma(t)} =
\left\langle 1, \nabs\,.\left(\mathcal{V}\,\vec\nu\right) 
\right\rangle_{\Gamma(t)} \nonumber \\ &
= \left\langle 1, \nabs\,.\left((\vec u\,.\,\vec\nu)\,\vec\nu\right) 
\right\rangle_{\Gamma(t)} 
= \left\langle 1, \nabs\,.\left(\vec u - \mat P_\Gamma\,\vec u\right) 
\right\rangle_{\Gamma(t)} \nonumber \\ &
= \left\langle 1, \nabs\,.\,\vec u \right\rangle_{\Gamma(t)}  = 0\,,
\label{eq:tancon}
\end{align}
which is the claim.

\ref{item:nsnsii}
As in Remark~\ref{rem:weak}\ref{item:weakii}, the result follows from
the transport theorem, Theorem~\ref{thm:transvol}, and the facts that
$\vec{\mathcal{V}}\,.\,\vec\nu = \vec u\,.\, \vec\nu$ on
$\Gamma(t)$ and $\nabla\,.\,\vec u = 0$ in $\Omega_-(t)$.

\ref{item:nsnsiii}
From Lemma~\ref{lem:Willmorevargen}, (\ref{eq:fWF}) and \eqref{eq:bio4}
we obtain
\begin{align}
& \ddt\,E_{\spont,\beta}(\Gamma(t)) \nonumber \\ & \quad
=
\left\langle\Delta_s\,\varkappa + 
(\varkappa-\spont)\,|\nabs\,\vec\nu|^2-\tfrac12\,
(\varkappa-\spont)^2\,\varkappa 
+ \beta\,A(t) \left(|\nabs\,\vec\nu|^2-\varkappa^2\right), \mathcal{V} 
\right\rangle_{\Gamma(t)} \nonumber\\ & \quad
= - \left\langle
f_\Gamma,\mathcal{V}\right\rangle_{\Gamma(t)}
=- \left\langle f_\Gamma\,\vec\nu, \vec
u\right\rangle_{\Gamma(t)}
= -\left\langle \vec f_\Gamma, \vec
u\right\rangle_{\Gamma(t)}. 
\label{eq:stabtan1}
\end{align}
It remains to analyse the right hand side in \eqref{eq:stabtan1}. To this end,
we first of all note that it follows from
Lemma~\ref{lem:nabsid}\ref{item:nabsid},\ref{item:nabsidnabsf},
Definition~\ref{def:2.5}\ref{item:def2.5tensor}, 
Lemma \ref{lem:productrules}\ref{item:priii},
the divergence theorem on hypersurfaces, i.e.\ Theorem~\ref{thm:div}
together with a density argument, and Remark~\ref{rem:2.6}\ref{item:PGamma}
that
\begin{align}
\left\langle p_\Gamma, \nabs\,.\,\vec\xi \right\rangle_{\Gamma(t)} &
= \left\langle p_\Gamma\,\mat P_\Gamma, \nabs\,\vec\xi\right\rangle_{\Gamma(t)}
\nonumber \\ & 
= \left\langle \nabs\,.\left(
p_\Gamma\,\mat P_\Gamma\,\vec\xi\right), 1 \right\rangle_{\Gamma(t)}
- \left\langle \nabs\,.\left(
p_\Gamma\,\mat P_\Gamma\right), \vec\xi \right\rangle_{\Gamma(t)}
\nonumber \\ & 
= - \left\langle \nabs\,.\left(
p_\Gamma\,\mat P_\Gamma\right), \vec\xi \right\rangle_{\Gamma(t)}
\quad\forall\ \vec\xi \in \Vt\,.
\label{eq:pgammaweak}
\end{align}
Hence it follows from \eqref{eq:pgammaweak}, 
Remark~\ref{rem:ibp}\ref{item:ibpvec}, together with a density argument,
\eqref{eq:surfstress}, \eqref{eq:bio2}, \eqref{eq:dnufjumpgen} and 
\eqref{eq:NSNS1} that
\begin{align}
& \rho_\Gamma
\left\langle \matpartu\, \vec u, \vec\xi 
\right\rangle_{\Gamma(t)}
+ 2\,\mu_\Gamma \left\langle \mat D_s (\vec u) , \mat D_s ( \vec\xi )
\right\rangle_{\Gamma(t)}
- \left\langle p_\Gamma, \nabs\,.\,\vec\xi \right\rangle_{\Gamma(t)}
\nonumber \\ & \quad
= \left\langle \rho_\Gamma\,\matpartu\,\vec u - \nabs\,.\,\mat\sigma_\Gamma,
\vec\xi \right \rangle_{\Gamma(t)} 
= \left \langle \left[\mat\sigma\,\vec\nu\right]_-^+, 
\vec\xi \right \rangle_{\Gamma(t)}
+ \alpha \left\langle \vec f_\Gamma, \vec\xi \right\rangle_{\Gamma(t)}
\nonumber \\ & \quad
= -2 \left( \mu\, \mat D(\vec u), \mat D(\vec\xi) \right) + \left(p,\nabla\,.\,\vec\xi \right) 
+ \alpha \left\langle \vec f_\Gamma, \vec\xi \right\rangle_{\Gamma(t)}
\quad\forall\ \vec\xi \in {\mathbb V}_{\Gamma(t)}\,,
\label{eq:weakGam}
\end{align}
where
\begin{equation} \label{eq:VGamma}
{\mathbb V}_{\Gamma(t)} = \left\{\vec\xi \in [H^1_0(\Omega)]^d :
\mat P_\Gamma\,\vec\xi_{\mid_{\Gamma(t)}}\in \Vt\right\} .
\end{equation}
Here we recall from Lemma~\ref{lem:divf}\ref{item:divfi},\ref{item:divfiii},
together with a density argument, and a trace theorem for $\Omega_-(t)$ that
all the terms in \eqref{eq:weakGam} 
are well-defined for $\vec\xi \in {\mathbb V}_{\Gamma(t)}$.
Combining (\ref{eq:stabtan1}) and (\ref{eq:weakGam}) with $\vec\xi= \vec u$,
on noting (\ref{eq:NSNS1}) and (\ref{eq:bio3}), yields that
\begin{align}
& \alpha\,\ddt\,E_{\spont,\beta}(\Gamma(t)) 
+\rho_\Gamma
\left\langle \matpartu\, \vec u, \vec u 
\right\rangle_{\Gamma(t)}
\nonumber \\ & \qquad
= - 2\,\mu_\Gamma \left\langle \mat D_s (\vec u) , \mat D_s ( \vec u )
\right\rangle_{\Gamma(t)} 
-2 \left( \mu\, \mat D(\vec u), \mat D(\vec u) \right) . 
\label{eq:stabtan2}
\end{align}
Next, similarly to (\ref{eq:tancon}), it follows 
from (\ref{eq:materialder}), Remark \ref{rem:timeder}\ref{item:rem2.19ii}, 
\eqref{eq:bio4}, Remark \ref{rem:altforms}\ref{item:eq:grad}, 
Theorem \ref{thm:div}, 
Lemma~\ref{lem:productrules}\ref{item:pri}, (\ref{eq:bio3}) and 
Theorem \ref{thm:trans} that
\begin{align*}
\left\langle \matpartu\, \vec u, \vec u 
\right\rangle_{\Gamma(t)} & = \left\langle \matpartx\, \vec u, \vec u 
\right\rangle_{\Gamma(t)} + 
\left\langle \left[\left(\vec u - \vec{\mathcal{V}}\right).\,\nabla\right] 
\vec u,\vec u \right\rangle_{\Gamma(t)} \nonumber \\
&=  \tfrac12 \left\langle \matpartx\left|\vec u\right|^2, 1 
\right\rangle_{\Gamma(t)} + 
\left\langle \left[\mat P_\Gamma\left(\vec u - \vec{\mathcal{V}}\right).\,\nabs
\right] \vec u,\vec u \right\rangle_{\Gamma(t)} \nonumber \\ 
&=  \tfrac12 \left\langle \matpartx\left|\vec u\right|^2, 1 
\right\rangle_{\Gamma(t)}
- \tfrac12 \left\langle \nabs\,.\left[\mat P_\Gamma 
\left(\vec u - \vec{\mathcal{V}}\right)\right],\left|\vec u\right|^2 
\right\rangle_{\Gamma(t)} \nonumber \\ 
&=  \tfrac12 \left\langle \matpartx\left|\vec u\right|^2, 1 
\right\rangle_{\Gamma(t)} 
- \tfrac12 \left\langle \nabs\,.\left(\vec u - \vec{\mathcal{V}}\right),
\left|\vec u\right|^2 \right\rangle_{\Gamma(t)} 
= \tfrac12\,\ddt \left\langle \left|\vec u\right|^2, 1 
\right\rangle_{\Gamma(t)} \! .  
\end{align*}
Combining this with (\ref{eq:stabtan2}) yields the desired result.
\end{proof}

We note, in particular, that the conservation properties 
Lemma~\ref{lem:NSNSweak}\ref{item:nsnsi}, \ref{item:nsnsii} 
follow from the continuity equations $\nabs\,.\,\vec u=0$ on $\Gamma(t)$ 
and $\nabla\,.\,\vec u = 0$ in $\Omega_\pm(t)$, respectively.
Hence, in contrast to the situation in e.g.\ \S\ref{subsec:WFspont},
no additional 
side constraints for surface area and enclosed volume are required.

\subsection{A weak formulation for the dynamics of biomembranes}
\label{subsec:1.10.2}
The most natural weak formulation for the model introduced in
\S\ref{subsec:1.10.1} uses the tangential velocity of the
fluid for the evolution of $\Gamma(t)$, and hence \eqref{eq:bio4} is
replaced by $\vec{\mathcal{V}} = \vec u$ on $\Gamma(t)$.
This mimics the procedure in \S\ref{subsubsec:DziukStokes} for two-phase Stokes
flow. However, contrary to the situation there, the presence of
the continuity equation $\nabs\,.\,\vec u=0$ on $\Gamma(t)$ will lead to good
mesh properties for discretizations based on this natural weak formulation.

Now 
(\ref{eq:weakGam})
leads to the following weak formulation of the system
\arxivyesno{}{\linebreak}%
\eqref{eq:NSNS1}, \eqref{eq:bio} and (\ref{eq:fWF}). 
Given a closed hypersurface $\Gamma(0)$
and, if $\rho_\Gamma > 0$, $\vec u_0 \in [H^1_0(\Omega)]^d$, 
we seek an evolving hypersurface $(\Gamma(t))_{t\in[0,T]}$, 
with a global parameterization and induced velocity field
$\vec{\mathcal{V}}$, $\vec\varkappa \in [L^2(\GT)]^d$,
$p_\Gamma \in L^2(\GT)$, $\vec f_\Gamma \in [L^2(\GT)]^d$, as well as
$\vec u : \Omega \times [0,T] \to \bR^d$, 
with $\vec u_{\mid_{\GT}} \in [H^1(\GT)]^d$
and $\rho_\Gamma\,\vec u(\cdot,0)_{\mid_{\Gamma(0)}} = 
\rho_\Gamma\,(\vec u_0)_{\mid_{\Gamma(0)}}$, and
$p:\Omega\times[0,T] \to \bR$ as follows.
For almost all $t \in (0,T)$, find 
$(\vec {\mathcal V}(\cdot,t),\vec\varkappa(\cdot,t),p_\Gamma(\cdot,t),
\vec f_\Gamma(\cdot,t),\vec u(\cdot,t),p(\cdot,t))
\in [L^2(\Gamma(t))]^d \times \Vt \times L^2(\Gamma(t))\times
[L^2(\Gamma(t))]^d \times 
{\mathbb V}_{\Gamma(t)}\times L^2(\Omega)$ such that
\begin{subequations} \label{eq:weakGD}
\begin{align}
&2\left(\mu\,\mat D(\vec u), \mat D(\vec\xi)\right)
- \left(p, \nabla\,.\,\vec\xi\right)
+ \rho_\Gamma \left\langle \matpartu\,\vec u , \vec\xi
 \right\rangle_{\Gamma(t)}
+ 2\,\mu_\Gamma \left\langle \mat D_s (\vec u) , \mat D_s ( \vec\xi )
\right\rangle_{\Gamma(t)}
\nonumber \hspace*{1cm}\\ & \qquad
- \left\langle p_\Gamma, \nabs\,.\,\vec\xi \right\rangle_{\Gamma(t)}
 = \alpha \left\langle \vec f_\Gamma, \vec\xi \right\rangle_{\Gamma(t)}
\quad\forall\ \vec\xi \in 
{\mathbb V}_{\Gamma(t)}
\,, \label{eq:weakGDa} \\
&(\nabla\,.\,\vec u, \varphi) = 0  \quad\forall\ \varphi \in L^2(\Omega)
\,, \label{eq:weakGDb} \\
& \left\langle \nabs\,.\,\vec u, \eta \right\rangle_{\Gamma(t)}
 = 0  \quad\forall\ \eta \in L^2(\Gamma(t))\,, \label{eq:weakGDc} \\
&  \left\langle \vec{\mathcal{V}}, \vec\chi \right\rangle_{\Gamma(t)} 
= \left\langle \vec u, \vec\chi \right\rangle_{\Gamma(t)} 
 \quad\forall\ \vec\chi \in [L^2(\Gamma(t))]^d\,,
\label{eq:weakGDd} \\
  & \left\langle \vec\varkappa, \vec\eta \right\rangle_{\Gamma(t)}
+ \left\langle \nabs\,\vec\id, \nabs\,\vec\eta
\right\rangle_{\Gamma(t)}
 = 0  \quad\forall\ \vec\eta \in \Vt\,, \label{eq:weakGDe} 
\end{align}
together with an equation for $\vec f_\Gamma$. In the
simplest case, when $\spont = \beta = 0$, we recall 
Lemma~\ref{lem:Willmorevar2} and (\ref{eq:Dziukwilla}), and set
\begin{align}
\left\langle \vec f_\Gamma, \vec\chi \right\rangle_{\Gamma(t)} &  =
\left\langle \nabs\,\vec\varkappa , \nabs\,\vec\chi - 2\,\mat D_s(\vec{\chi}) 
\right\rangle_{\Gamma(t)}
+\left\langle \nabs\,.\, \vec\varkappa, \nabs\,.\, \vec\chi
\right\rangle_{\Gamma(t)}
\nonumber \\ & \qquad
+ \tfrac12\left\langle |\vec\varkappa|^2, 
\nabs\,.\,\vec\chi \right\rangle_{\Gamma(t)}
\quad\forall\ \vec\chi \in \Vt \,. \label{eq:weakGDf}
\end{align}
\end{subequations}
In the general case, for nonzero $\spont$ or $\beta$,
we recall (\ref{eq:Dzasym}). 
Hence we require, instead of (\ref{eq:weakGDf}), that in
addition a $\vec y\in[L^2(\GT)]^d$ exists such that for almost all $t \in
(0,T)$ it holds that $\vec y(\cdot,t) \in \Vt$ and
\begin{subequations}
\begin{align}
& \left\langle \vec f_\Gamma, \vec\chi \right\rangle_{\Gamma(t)} =
\left\langle \nabs\,\vec y, \nabs\,\vec\chi -2\, \mat D_s(\vec\chi)
\right\rangle_{\Gamma(t)}+
\left\langle \nabs\,.\,\vec y, \nabs\,.\,\vec\chi \right\rangle_{\Gamma(t)}
\nonumber \\ & \quad\qquad
-\left\langle  \tfrac12\,|\vec\varkappa - \spont\,\vec\nu|^2
- \vec\varkappa\,.\left( \vec y- \beta\,A(t)\,\vec\nu\right), 
\nabs\,.\,\vec\chi \right\rangle_{\Gamma(t)}
\nonumber \\ & \quad\qquad
+ \left(\beta\,A(t) - \spont\right) \left\langle \vec\varkappa, 
(\nabs\,\vec\chi)^\transT\,\vec\nu \right\rangle_{\Gamma(t)}
 \quad\forall\ \vec\chi\in \Vt\,,  \label{eq:weakGDg} \\
& \left\langle \vec\varkappa + \left(\beta\,A(t) - \spont\right)
\vec\nu - \vec y, \vec\xi \right\rangle_{\Gamma(t)} = 0 
\quad\forall\ \vec\xi \in \Vt\,,  \label{eq:weakGDh}
\end{align}
\end{subequations}
with $A(t)$ defined by (\ref{eq:Dzasymd}).   

We now argue that if no connected component of
$\Gamma(t)$ is spherical, then the surface pressure is unique and the
bulk pressure is unique up to an additive constant. 
Recall that Figure~\ref{fig:sketch} shows the special case of $\Gamma(t)$
having just a single connected component.
We assume that two solutions to \eqref{eq:weakGD} are found that only differ in
the pressure. Say one solution features the pressure pair $(p_1,p_{\Gamma,1})$
and the other $(p_2,p_{\Gamma,2})$. Let $\bar p = p_1-p_2$ and 
$\bar p_\Gamma =p_{\Gamma,1} -p_{\Gamma,2}$. 
Then we obtain 
\[
\left(\bar p, \nabla\,.\,\vec\xi\right) + 
\left\langle \bar p_\Gamma, \nabs\,.\,\vec\xi \right\rangle_{\Gamma(t)}
 = 0
\quad\forall\ \vec\xi \in [H^1_0(\Omega)]^d\,,
\]
which first of all implies that $\bar p$ is equal to some constants 
$\bar p_\pm$ in each connected component of $\Omega_\pm(t)$.
Applying the divergence theorem in the bulk regions $\Omega_{\pm}(t)$,
Lemma~\ref{lem:productrules}\ref{item:pri}, 
and the divergence theorem on $\Gamma(t)$, Theorem~\ref{thm:div}, 
yields that
\[
\nabs\,\bar p_\Gamma + \varkappa\,\bar p_\Gamma\,\vec\nu
 = - \left[\bar p\,\vec\nu\right]^+_-
\qquad\text{on}\ \Gamma(t)\,.
\]
Since $\nabs\,\bar p_\Gamma$ is tangential, we obtain that
$\nabs\,\bar p_\Gamma=\vec 0$, and hence $\bar p_\Gamma$ is constant
on connected components of $\Gamma(t)$.
In addition, we obtain that $\varkappa\,\bar p_\Gamma = -[\bar p]_-^+$ on 
$\Gamma(t)$.
If $\varkappa$ is not constant on a connected component of 
$\Gamma(t)$, which is the case if it is not a sphere, 
then we have $\bar p_\Gamma = [\bar p]_-^+ = 0$ on this component, and so
$\bar p_+ = \bar p_-$ for the associated connected components of 
$\Omega_\pm(t)$.
Repeating this argument for all the connected components of $\Gamma(t)$
yields, if no connected component of $\Gamma(t)$ is a sphere, that
$p_\Gamma$ is unique on $\Gamma(t)$ and that $p$ is unique in $\Omega$
up to an additive constant. 
If $\varkappa$ is constant on a connected component of $\Gamma(t)$, however, 
i.e.\ if this component of $\Gamma(t)$ is a sphere, then on this component
$p_\Gamma$ is only unique up to an additive constant.

Taking the above into account we formulate the following LBB-type
condition. If $\Gamma(t)$ and $\partial\Omega$ are sufficiently smooth,
and provided that $\Gamma(t)$ does not contain a sphere, then there exists
a constant $C\in\bRplus$ such that
\begin{equation} \label{eq:LBBG}
\inf_{(\varphi,\eta) \in \widehat\pspace \times L^2(\Gamma(t))}
\sup_{\vec\xi \in {\mathbb V}_{\Gamma(t)}}
\frac{\left(\varphi, \nabla\,.\,\vec\xi\right)
+ \left\langle \eta, \nabs\,.\,\vec\xi \right\rangle_{\Gamma(t)}}
{\left(\left|\varphi\right|_{\Omega} 
+ \left| \eta \right|_{\Gamma(t)}\right)\left(\left\|\vec\xi\right\|_{1,\Omega}
+ \left\|\mat P_\Gamma\,\vec\xi_{\mid_{\Gamma(t)}} \right\|_{1,\Gamma(t)}
\right)} \geq C > 0\,.
\end{equation}
Here we have recalled \eqref{eq:VGamma} as well as the definitions of the
space $\widehat\pspace$ and the $H^1$--norm $\|\cdot\|_{1,\Omega}$
from \eqref{eq:LBB}. In addition, we let
$\|\cdot \|_{1,\Gamma(t)}^2 = |\cdot|_{\Gamma(t)}^2 + 
|\nabs\,\cdot|_{\Gamma(t)}^2$
define the $H^1$--norm on $\Gamma(t)$.
The LBB-type condition \eqref{eq:LBBG} can be deduced from the pressure 
reconstruction result in \cite{Lengeler15preprint}.

\subsection{Semidiscrete finite element approximation}
We now introduce a finite element version of the weak formulation of
(\ref{eq:weakGD}). 
We have discussed in previous sections the importance of the choice of the
discrete tangential velocity of $\Gamma^h(t)$. In general, the mesh
quality will deteriorate when using a velocity in the direction of the
discrete normal. Similarly, for many two-phase fluid flow problems, using
the tangential velocity induced by the surrounding fluid flow also leads to bad
meshes.
However, here it turns out that due to the approximation of the local
area conservation property, $\nabs\,.\,\vec u = 0$ on $\Gamma(t)$, the mesh
quality typically will remain very good during the evolution even if we choose
$\vec{\mathcal{V}} = \vec u$.

We recall the definitions of the semidiscrete finite element spaces in
\S\ref{subsubsec:sdStokes}, and in particular the space
\eqref{eq:pspacehT}. In addition, we define the space
\begin{align*}
\uspace^h_{\Gamma^h} & = \left\{
\vec\phi\in H^1(0,T; \uspace^h) : \right. \\ & \qquad \left.
\exists\, \vec\chi \in \VhTGhT \ \text{ with} \
\vec\chi(\cdot,t) =  \vec\pi_{\Gamma^h}
\left[\vec\phi_{\mid_{\Gamma^h(t)}}\right]\ 
\forall\ t\in[0,T]\right\}.
\end{align*}
Overall, we then obtain the following semidiscrete
continuous-in-time finite element approximation of (\ref{eq:weakGD}).
First, we will state it for the simpler situation when $\spont = \beta = 0$. 
To this end, we recall the scheme (\ref{eq:Dzwill1}).

Given the closed polyhedral hypersurface $\Gamma^h(0)$, find 
an evolving polyhedral hypersurface $\GhT$ with induced 
velocity $\vec{\mathcal{V}}^h \in \VhGhT$, 
$\vec\kappa^h\in \VhGhT$, 
$\vec U^h \in \uspace^h_{\Gamma^h}$, $P^h\in \widehat \pspace^h_T$, 
$P^h_\Gamma\in \WhGhT$ and $\vec F^h_\Gamma\in \VhGhT$
as follows.
For all $t\in (0,T]$, find 
\arxivyesno{$(\vec U^h(\cdot,t), P^h(\cdot,t), P^h_{\Gamma}(\cdot,t), 
\vec{\mathcal{V}}^h(\cdot,t),$ $\vec\kappa^h(\cdot,t), \vec F^h_\Gamma(\cdot,t))
\in \uspace^h \times \widehat\pspace^h(t) \times \Wht \times \Vht \times \Vht 
\times \Vht$}
{$(\vec U^h(\cdot,t), P^h(\cdot,t), P^h_{\Gamma}(\cdot,t), 
\vec{\mathcal{V}}^h(\cdot,t),\vec\kappa^h(\cdot,t), \vec F^h_\Gamma(\cdot,t))
\in \uspace^h \times \widehat\pspace^h(t) \times \Wht \times \Vht \times \Vht 
\times \Vht$}
such that
\begin{subequations} \label{eq:biosd}
\begin{align}
&2\left(\mu^h\,\mat D(\vec U^h), \mat D(\vec\xi) \right)
- \left(P^h, \nabla\,.\,\vec\xi\right)
+ \rho_\Gamma \left\langle \matpartxh\,\vec\pi_{\Gamma^h}\,\vec U^h, \vec\xi 
\right\rangle_{\Gamma^h(t)}^h \nonumber \\ & \qquad\qquad
+ 2\,\mu_\Gamma \left\langle \mat D_s (\vec\pi_{\Gamma^h}\,\vec U^h) , 
\mat D_s (\vec\pi_{\Gamma^h}\, \vec\xi ) \right\rangle_{\Gamma^h(t)}
- \left\langle P_\Gamma^h , 
\nabs\,.\,(\vec\pi_{\Gamma^h}\,\vec\xi) \right\rangle_{\Gamma^h(t)} \nonumber\\& \qquad
= \alpha \left\langle \vec F_\Gamma^h , \vec\xi \right\rangle_{\Gamma^h(t)}^h 
\quad\forall\ \vec\xi \in \uspace^h \,, \label{eq:biosda}\\
& \left(\nabla\,.\,\vec U^h, \varphi\right) = 0 
\quad\forall\ \varphi \in \widehat\pspace^h(t)\,,
\label{eq:biosdb} \\
& \left\langle \nabs\,.\,(\vec\pi_{\Gamma^h}\,\vec U^h), \eta 
\right\rangle_{\Gamma^h(t)}  = 0 
\quad\forall\ \eta \in \Wht\,,
\label{eq:biosdc} \\
& \left\langle \vec{\mathcal{V}}^h ,
\vec\chi \right\rangle_{\Gamma^h(t)}^h
= \left\langle \vec U^h, \vec\chi \right\rangle_{\Gamma^h(t)}^h
 \quad\forall\ \vec\chi \in \Vht\,,
\label{eq:biosdd}\\
& \left\langle \vec\kappa^h , \vec\eta \right\rangle_{\Gamma^h(t)}^{h}
+ \left\langle \nabs\,\vec\id, \nabs\,\vec\eta \right\rangle_{\Gamma^h(t)}
 = 0  \quad\forall\ \vec\eta \in \Vht\,,\label{eq:biosde} \\
& \left\langle \vec F_\Gamma^h, \vec\chi \right\rangle_{\Gamma^h(t)}^h =
\left\langle \nabs\,\vec{\kappa}^h , \nabs\,\vec\chi - 2\,\mat D_s(\vec\chi)
\right\rangle_{\Gamma^h(t)}
+ \left\langle \nabs\,.\,\vec{\kappa}^h , \nabs\,.\,\vec\chi
  \right\rangle_{\Gamma^h(t)} \nonumber \\ & \qquad\qquad\qquad
+\tfrac12 \left\langle |\vec{\kappa}^h|^2, 
 \nabs\,.\,\vec\chi \right\rangle_{\Gamma^h(t)}^{h}
\quad\forall\ \vec\chi \in \Vht\,. \label{eq:biosdf}
\end{align}
\end{subequations}
For this semidiscrete approximation we have the following stability and 
conservation results.

\begin{thm} \label{thm:stabGD} 
Let $(\GhT, \vec U^h, P^h, P^h_\Gamma, \vec\kappa^h, 
\vec F_\Gamma^h)$ be a solution to \eqref{eq:biosd},
and let \arxivyesno{\linebreak}{}%
$\vec\kappa^h \in \VhTGhT$. Then it holds that
\begin{align} \label{eq:bioen1}
& \tfrac12\,\ddt \left(\rho_\Gamma \left(\left|\vec U^h\right|_{\Gamma^h(t)}^h
\right)^2
+ \alpha \left(\left|\vec\kappa^h\right|_{\Gamma^h(t)}^{h}\right)^2
 \right) +
2\left( \mu^h\,\mat D(\vec U^h), \mat D(\vec U^h) \right)
\nonumber \\ & \hspace{1cm}
+ 2\,\mu_\Gamma \left\langle \mat D_s (\vec\pi_{\Gamma^h}\,\vec U^h) , 
\mat D_s (\vec\pi_{\Gamma^h}\, \vec U^h ) \right\rangle_{\Gamma^h(t)} = 0\,.
\end{align}
Moreover, it holds that
\begin{equation} \label{eq:bioen2}
\ddt\left\langle \phi^{\Gamma^h(t)}_k, 1 \right\rangle_{\Gamma^h(t)} = 0\,,
\qquad k = 1,\ldots,K\,,
\end{equation}
and hence that
\begin{equation} \label{eq:bioen3}
\ddt\left|\Gamma^h(t)\right| = 0 \,.
\end{equation}
\end{thm}
\begin{proof} 
Choosing $\vec\xi = \vec U^h(\cdot,t)$ in (\ref{eq:biosda}),
$\varphi=P^h(\cdot,t)$ in (\ref{eq:biosdb}) and $\eta = P^h_\Gamma(\cdot,t)$ in
(\ref{eq:biosdc}), we obtain
\begin{align} \label{eq:bioa1}
& 2\left(\mu^h\,\mat D(\vec U^h), \mat D(\vec U^h) \right)
+ \rho_\Gamma \left\langle \matpartxh\,\vec\pi_{\Gamma^h}\,\vec U^h, \vec U^h
\right\rangle_{\Gamma^h(t)}^h
\nonumber \\ & \hspace{1cm}
+ 2\,\mu_\Gamma \left\langle \mat D_s (\vec\pi_{\Gamma^h}\,\vec U^h) , 
\mat D_s (\vec\pi_{\Gamma^h}\, \vec U^h ) \right\rangle_{\Gamma^h(t)}
= \alpha \left\langle\vec F_\Gamma^h,\vec U^h \right\rangle_{\Gamma^h(t)}^h.
\end{align}
In addition, arguing as in the proof of Theorem~\ref{thm:Dzwill}
on (\ref{eq:biosde}) and \arxivyesno{}{\linebreak}%
(\ref{eq:biosdf}), recall (\ref{eq:Dzwill1}), yields
\begin{equation} \label{eq:bioa2}
\tfrac12\,\ddt \left(\left| \vec\kappa^h \right|^h_{\Gamma^h(t)} \right)^2
= -\left\langle \vec F^h_\Gamma,
\vec{\mathcal{V}}^h\right\rangle^h_{\Gamma^h(t)} = -\left\langle
\vec F^h_\Gamma, \vec U^h\right\rangle^h_{\Gamma^h(t)},
\end{equation}
where the last equality 
follows from choosing $\vec\chi = \vec F^h_\Gamma(\cdot,t)$ 
in (\ref{eq:biosdd}). 
Moreover, it follows from (\ref{eq:biosdd}) that 
\begin{equation} \label{eq:VhUh}
\vec{\mathcal{V}}^h(\cdot,t) 
= \vec\pi_{\Gamma^h}\left[ \vec U^h_{\mid_{\Gamma^h(t)}}\right].
\end{equation}
Noting this, Theorem \ref{thm:disctrans}\ref{item:disctransh}
and (\ref{eq:biosdc}) with $\eta = \pi_{\Gamma^h}\left[\left|\vec U^h_{\mid_{\Gamma^h(t)}}
\right|^2\right]$ yields
\begin{align} \label{eq:bioa3}
\tfrac12\,\ddt
\left\langle \vec U^h, \vec U^h \right\rangle_{\Gamma^h(t)}^h 
&= \tfrac12 \left\langle \matpartxh\,
\vec\pi_{\Gamma^h}\left[\left|\vec U^h\right|^2\right],1
\right\rangle_{\Gamma^h(t)}^h
+ \tfrac12 \left\langle \nabs\,.\,\vec{\mathcal{V}}^h, 
\left|\vec U^h\right|^2 \right\rangle_{\Gamma^h(t)}^h
\nonumber \\ & 
 = \left\langle \matpartxh\,\vec\pi_{\Gamma^h}\,\vec U^h, \vec U^h
\right\rangle_{\Gamma^h(t)}^h
+ \tfrac12 \left\langle \nabs\,.\,(\vec\pi_{\Gamma^h}\, \vec U^h), 
\left|\vec U^h\right|^2 \right\rangle_{\Gamma^h(t)}^h
\nonumber \\ & 
= \left\langle \matpartxh\,\vec\pi_{\Gamma^h}\,\vec U^h, \vec U^h
\right\rangle_{\Gamma^h(t)}^h .
\end{align}
Combining (\ref{eq:bioa1}), \eqref{eq:bioa2} and (\ref{eq:bioa3}) yields
(\ref{eq:bioen1}). 

Similarly to (\ref{eq:bioen1}), the identity \eqref{eq:bioen2} follows 
directly from Theorem~\ref{thm:disctrans}\ref{item:disctransh}, 
Remark \ref{rem:GhT}\ref{item:phihk}, \eqref{eq:VhUh} 
and choosing $\eta=\phi^{\Gamma^h(t)}_k$ in \arxivyesno{}{\linebreak}%
(\ref{eq:biosdc}). 
Finally, \eqref{eq:bioen3} follows by adding \eqref{eq:bioen2} 
for $k=1,\ldots,K$.
\end{proof}

\begin{rem}[Spontaneous curvature and ADE effects]
In the case that $\spont$ or $\beta$ are nonzero, we replace
\eqref{eq:biosdf} in the semidiscrete approximation \eqref{eq:biosd} with
\begin{subequations} \label{eq:biosd2}
\begin{align}
&\left\langle \vec\kappa^h + 
\left(\beta\,A^h(t)- \spont\right)\vec\nu^h - \vec Y^h, \vec\xi 
\right\rangle_{\Gamma^h(t)}^h
= 0 \quad\forall\ \vec\xi \in \Vht\,, \label{eq:biosdg} \\
& \left\langle \vec F_\Gamma^h, \vec\chi \right\rangle_{\Gamma^h(t)}^h
= \left\langle \nabs\,\vec Y^h, \nabs\,\vec\chi - 2\,\mat D_s(\vec\chi)
\right\rangle_{\Gamma^h(t)}
+ \left\langle\nabs\,.\,\vec Y^h,\nabs\,.\,\vec\chi \right\rangle_{\Gamma^h(t)}
\nonumber \\ &\qquad
- \left\langle 
\tfrac12\left|\vec\kappa^h - \spont\,\vec\nu^h\right|^2 -
\vec\kappa^h\,.\left(\vec Y^h - \beta\,A^h(t)\,\vec\nu^h\right),
\nabs\,.\,\vec\chi\right\rangle_{\Gamma^h(t)}^h
\nonumber \\ &\qquad
+ \left(\beta\,A^h(t) - \spont\right) 
\left\langle \vec\kappa^h, 
(\nabs\,\vec\chi)^\transT\, \vec\nu^h \right\rangle_{\Gamma^h(t)}^h
\quad\forall\ \vec\chi\in \Vht\,, \label{eq:biosdh}\\ 
&\qquad A^h(t) = \left\langle \vec\kappa^h , \vec\nu^h
\right\rangle_{\Gamma^h(t)}^h - M_0 \,, 
\label{eq:biosdi}
\end{align}
\end{subequations}
and seek in addition a $\vec Y^h\in\VhGhT$. 
Here we recall \eqref{eq:Dzasymdisc}.
Then a stability result as in {\rm Theorem~\ref{thm:stabGD}}, with 
$\tfrac12\left(\left|\vec\kappa^h\right|_{\Gamma^h(t)}^{h}\right)^2$
replaced by $E_{\spont,\beta}^h(\Gamma^h(t))$ can be shown for the 
system {\rm (\ref{eq:biosda}--e)}, \eqref{eq:biosd2}, 
on combining the proofs of {\rm Theorem~\ref{thm:stabGD}} and
{\rm Theorem~\ref{thm:Dzasymdisc}}; see also 
{\rm \citet[Theorem~4.2]{nsnsade}}.
\end{rem}

\begin{rem} \label{rem:biosd}
\rule{0pt}{0pt}
\begin{enumerate}
\item \label{item:biosdi}
The virtual element approach introduced in
{\rm \S\ref{subsubsec:DziukStokes}} can be used for the 
semidiscrete approximation \eqref{eq:biosd} and its generalization 
{\rm (\ref{eq:biosda}--e)}, \eqref{eq:biosd2}.
In this case we obtain, in addition to $\ddt\left|\Gamma^h(t)\right| = 0$, also
\begin{equation*}
\ddt \mathcal{L}^d(\Omega^h_-(t)) = 0\,,
\end{equation*}
which means that we obtain an approximation which conserves both the surface 
area and the enclosed volume, and also fulfills an energy identity.
\item The identity \eqref{eq:bioen2} ensures that the measure of the support of
each basis function on $\Gamma^h(t)$ is conserved. In the case of 
two space dimensions, and for the number of elements/vertices $J=K$ being odd, 
this is equivalent to each element maintaining its length. 
In particular, if $\Gamma^h(0) $ is equidistributed,
then $\Gamma^h(t)$ will remain equidistributed throughout.
For a slight modification of \eqref{eq:biosd} one can prove that
the measure of each element on $\Gamma^h(t)$ is conserved, i.e.\
that $\ddt\,\mathcal{H}^{d-1}\left(\sigma_j(t)\right) = 0$, 
$j=1,\ldots,J$, and hence \eqref{eq:bioen3}, for arbitrary $J\geq2$ and 
$d\geq2$. To achieve this, one needs to replace $\Wht$ in \eqref{eq:biosd} 
by $\Wcht$, recall {\rm Definition~\ref{def:Vh}\ref{item:Vch}}, i.e.\
piecewise constants are used for the trial space for $P^h_\Gamma(\cdot,t)$ 
and for the test space in \eqref{eq:biosdc}.
However, this constraint for $d\geq3$ can be too severe,
see {\rm \citet[Remark 4.1]{nsns}} for more details in the case $d=3$. 
\item 
The approximation \eqref{eq:biosd} and its generalization 
{\rm (\ref{eq:biosda}--e)}, \arxivyesno{}{\linebreak}%
\eqref{eq:biosd2}
are based on continuous piecewise linear approximations of the surface velocity
and surface pressure, and so are unlikely to satisfy a discrete version
of the LBB condition \eqref{eq:LBBG} 
with a constant $C$ independent of the mesh parameters. 
In fact, in practice it can lead to oscillatory surface pressure 
approximations, see {\rm \citet[Fig.\ 5]{nsns}}.
This can be avoided by using a continuous piecewise quadratic interpolation 
of the bulk velocity on the surface. This leads to better behaved 
approximations of the surface pressure. 
Although one can still prove \eqref{eq:bioen1} and its generalization
for the modified versions of \eqref{eq:biosd} and {\rm (\ref{eq:biosda}--e)},
\eqref{eq:biosd2} if $\rho_\Gamma=0$,      
one can no longer show \eqref{eq:bioen2} and \eqref{eq:bioen3}. 
This can lead to poor surface area conservation for $d=3$ in practice, see 
{\rm \citet[Remark 4.1]{nsns}} for more details.
\end{enumerate}
\end{rem}

Finally, we state a fully discrete equivalent of \eqref{eq:biosd}. 
Let the closed polyhedral hypersurface $\Gamma^0$ be an approximation of
$\Gamma(0)$, and let $\vec\kappa^0_{\Gamma^0} \in \Vhz$ be 
an approximation to its mean curvature vector. If $\rho_\Gamma>0$, let
$\vec U^0_{\Gamma^0} \in \Vhz$ be an approximation to 
$(\vec u_0)_{\mid_{\Gamma^0}}$. 
We also recall the time interval partitioning \eqref{eq:ttaum}. 
Then, for $m=0,\ldots, M-1$, find $\vec U^{m+1} \in \uspace^m$,
$P^{m+1} \in \widehat\pspace^m$, 
$P_\Gamma^{m+1} \in \Whm$,
$\vec X^{m+1}\in\Vhm$,
$\vec\kappa^{m+1}\in\Vhm$ and
$\vec F_\Gamma^{m+1} \in \Vhm$ such that
\begin{subequations} \label{eq:GD}
\begin{align}
&2\left(\mu^m\,\mat D(\vec U^{m+1}), \mat D(\vec\xi) \right)
- \left(P^{m+1}, \nabla\,.\,\vec\xi\right)
+ \rho_\Gamma \left\langle \frac{\vec U^{m+1} - \vec U^m_{\Gamma^m}}{\ttau_m}, 
\vec\xi \right\rangle_{\Gamma^m}^h \nonumber \\ & \qquad
+ 2\,\mu_\Gamma \left\langle \mat D_s (\vec\pi_{\Gamma^m}\,\vec U^{m+1}) , 
\mat D_s ( \vec\pi_{\Gamma^m}\,\vec\xi ) \right\rangle_{\Gamma^m}
- \left\langle P_\Gamma^{m+1} , 
\nabs\,.\,(\vec\pi_{\Gamma^m}\,\vec\xi) \right\rangle_{\Gamma^m}
\nonumber \\ & \qquad\qquad
= \alpha \left\langle \vec F_\Gamma^{m+1}, \vec\xi \right\rangle_{\Gamma^m}^h 
\quad\forall\ \vec\xi \in \uspace^m \,, \label{eq:GDa}\\
& \left(\nabla\,.\,\vec U^{m+1}, \varphi\right) = 0 
\quad\forall\ \varphi \in \widehat\pspace^m\,,
\label{eq:GDb} \\
& \left\langle \nabs\,.\,(\vec\pi_{\Gamma^m}\,\vec U^{m+1}), \eta 
\right\rangle_{\Gamma^m}  = 0 
\quad\forall\ \eta \in \Whm\,,
\label{eq:GDc} \\
& \left\langle \frac{\vec X^{m+1} - \vec\id}{\ttau_m} ,
\vec\chi \right\rangle_{\Gamma^m}^h
= \left\langle \vec U^{m+1}, \vec\chi \right\rangle_{\Gamma^m}^h
 \quad\forall\ \vec\chi \in \Vhm\,,
\label{eq:GDd} \\
& \left\langle \vec\kappa^{m+1} , \vec\eta \right\rangle_{\Gamma^m}^{h}
+ \left\langle \nabs\,\vec X^{m+1}, \nabs\,\vec\eta \right\rangle_{\Gamma^m}
 = 0  \quad\forall\ \vec\eta \in \Vhm\,,\label{eq:GDe} \\
& \left\langle \vec F_\Gamma^{m+1}, \vec\chi \right\rangle_{\Gamma^m}^h =
\left\langle \nabs\,\vec\kappa^{m+1} , \nabs\,\vec\chi \right\rangle_{\Gamma^m}
+ \left\langle \nabs\,.\,\vec\kappa^{m}_{\Gamma^m}, \nabs\,.\,\vec\chi 
\right\rangle_{\Gamma^m} \nonumber \\ & \ \
 +\tfrac12 \left\langle \left|\kappa^m_{\Gamma^m}\right|^2,
 \nabs\,.\,\vec\chi \right\rangle_{\Gamma^m}^{h}
-2 \left\langle \nabs\,\vec\kappa^{m}_{\Gamma^m}, 
\mat D_s(\vec\chi) \right\rangle_{\Gamma^m}
\quad\forall\ \vec\chi \in \Vhm\,, \label{eq:GDf}
\end{align}
\end{subequations}
and set $\Gamma^{m+1} = \vec X^{m+1}(\Gamma^m)$,
$\vec\kappa^{m+1}_{\Gamma^{m+1}} = \vec\kappa^{m+1}\circ (\vec
X^{m+1})^{-1}\in \Vhmp$ and 
\arxivyesno{$\vec U^{m+1}_{\Gamma^{m+1}} = $ \linebreak $
\vec\pi_{\Gamma^m}\left[\vec U^{m+1}_{\mid_{\Gamma^m}}\right] 
\circ (\vec X^{m+1})^{-1} \in \Vhmp$.}
{$\vec U^{m+1}_{\Gamma^{m+1}} =
\vec\pi_{\Gamma^m}\left[\vec U^{m+1}_{\mid_{\Gamma^m}}\right] 
\circ (\vec X^{m+1})^{-1} \in \Vhmp$.}

The cases in which one or both of $\spont$ and $\beta$ are different from
zero can be handled in a similar way. Moreover, on using the techniques
from \S\ref{subsec:NS}, the approximations introduced above can also be
generalized to the case of Navier--Stokes flow in the bulk. 
In addition, a fully discrete equivalent of the approach mentioned in
Remark~\ref{rem:biosd}\ref{item:biosdi} ensures good volume conservation
in practice.
We refer to \cite{nsns,nsnsade} for more details.
Moreover, on assuming a discrete version of the LBB condition \eqref{eq:LBBG},
one can prove that there exists a unique solution to \eqref{eq:GD}
and to the corresponding fully discrete approximation of (\ref{eq:biosda}--e),
\eqref{eq:biosd2}, as well as to their extensions discussed above.
The necessary techniques can be found in the proofs of 
\citet[Theorem~5.1]{nsns} and \citet[Theorem~5.1]{nsnsade}.
Finally, we observe that 
the linear systems resulting from \eqref{eq:GD} and its extensions
can be solved with the help of a Schur complement
approach, on combining the techniques presented in Remark~\ref{rem:Stokessolve}
and Remark~\ref{rem:GDsolve}. We refer to \citet[\S6]{nsns} 
and \citet[\S6]{nsnsade} for more details. 

In Figure~\ref{fig:nsns} we show a numerical simulation for an extension of 
the scheme \eqref{eq:GD} to Navier--Stokes flow in the bulk. 
The domain
$\Omega$ is chosen to have a constriction, and the chosen boundary conditions
model a Poiseuille-type flow.
\begin{figure}
\center
\newcommand\lwidth{0.24\textwidth} 
\includegraphics[angle=0,width=\lwidth]{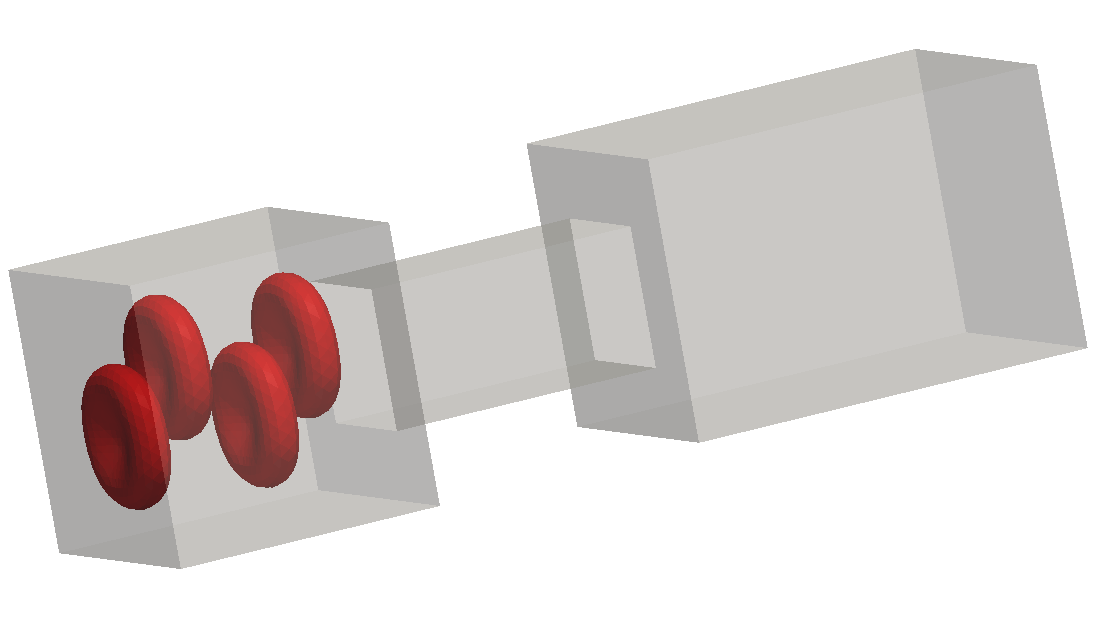}
\includegraphics[angle=0,width=\lwidth]{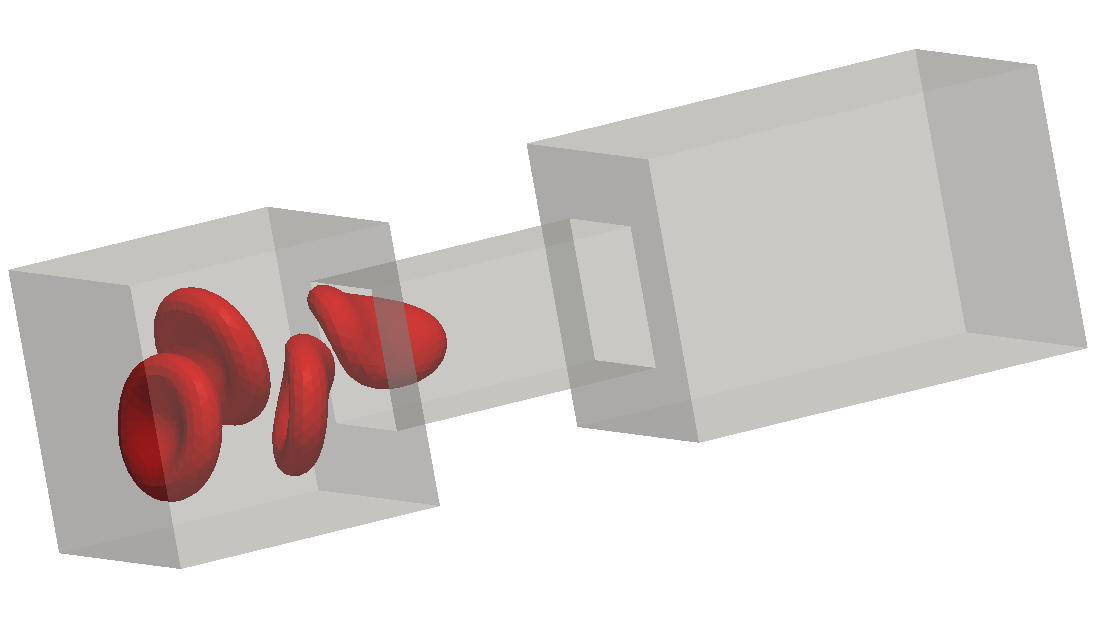}
\includegraphics[angle=0,width=\lwidth]{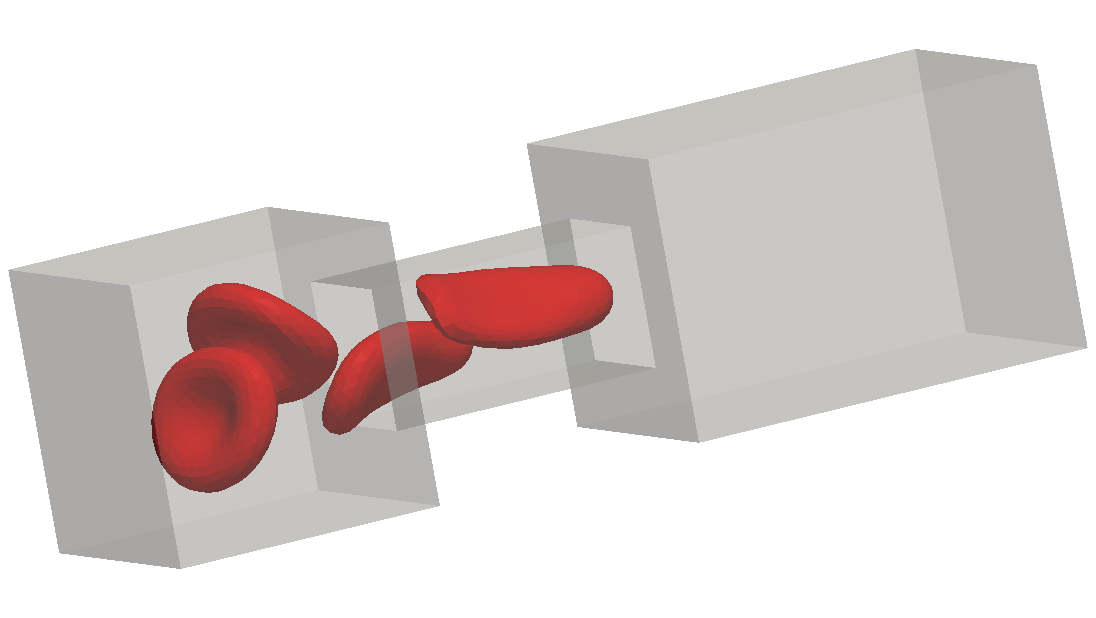}
\includegraphics[angle=0,width=\lwidth]{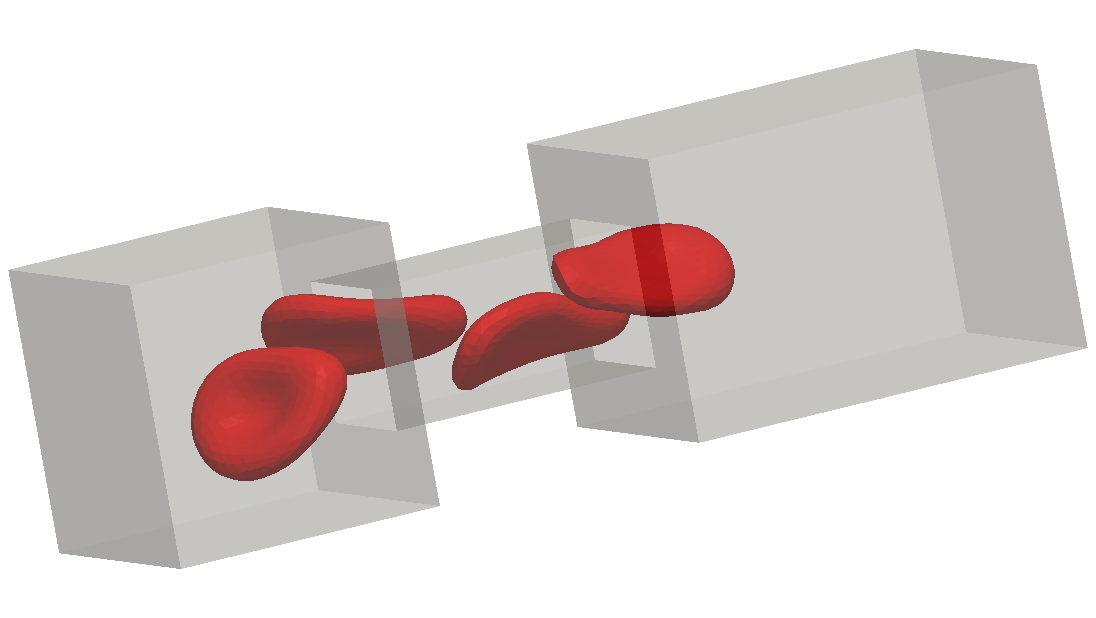}
\includegraphics[angle=0,width=\lwidth]{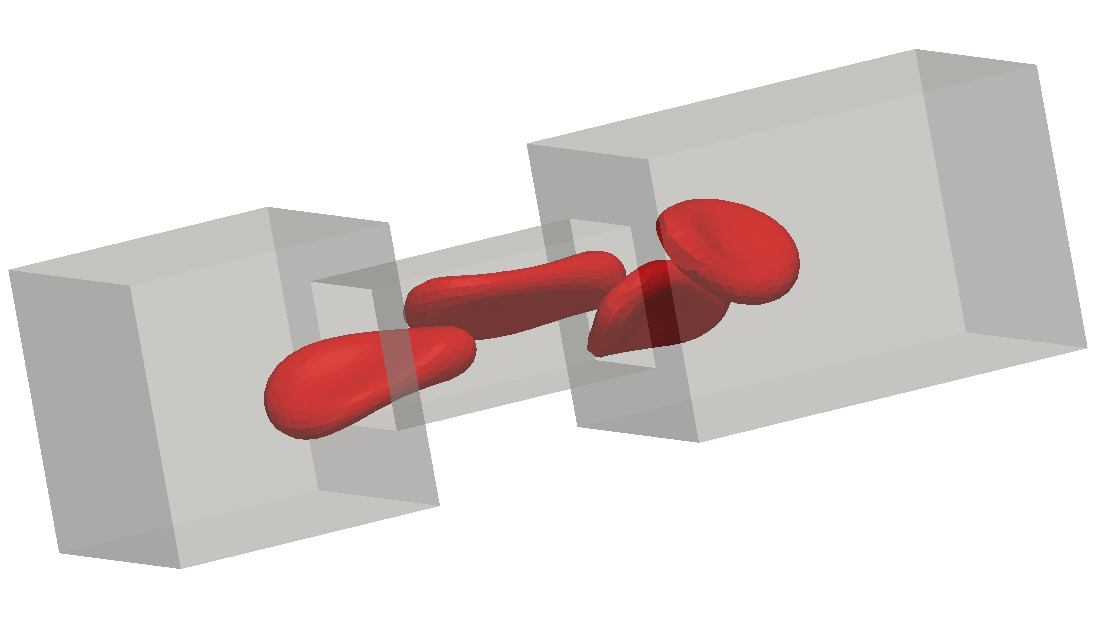}
\includegraphics[angle=0,width=\lwidth]{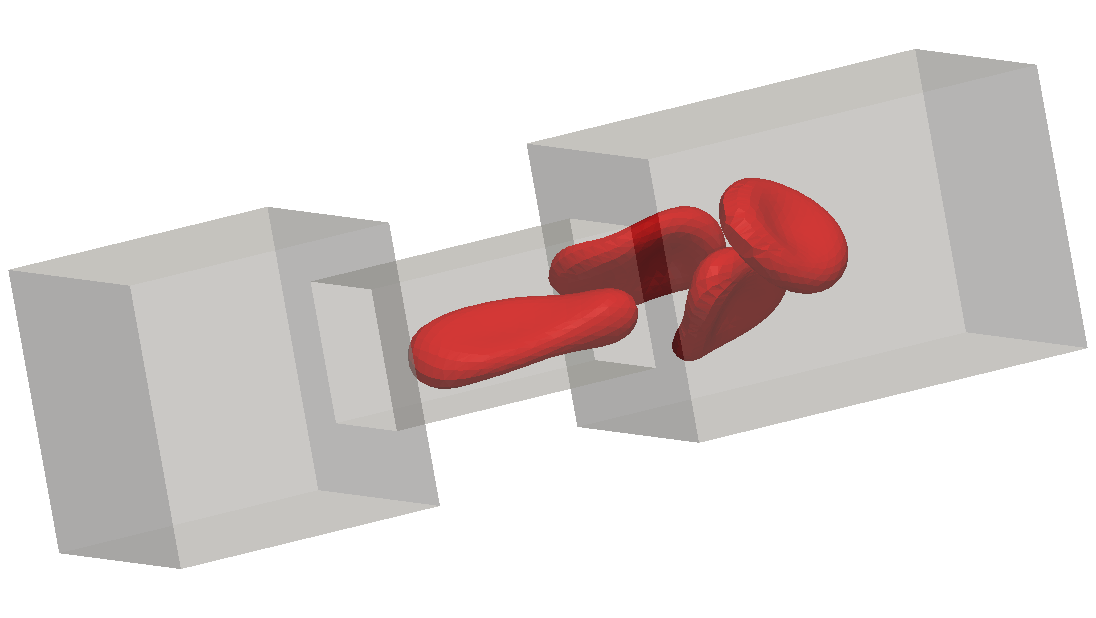}
\includegraphics[angle=0,width=\lwidth]{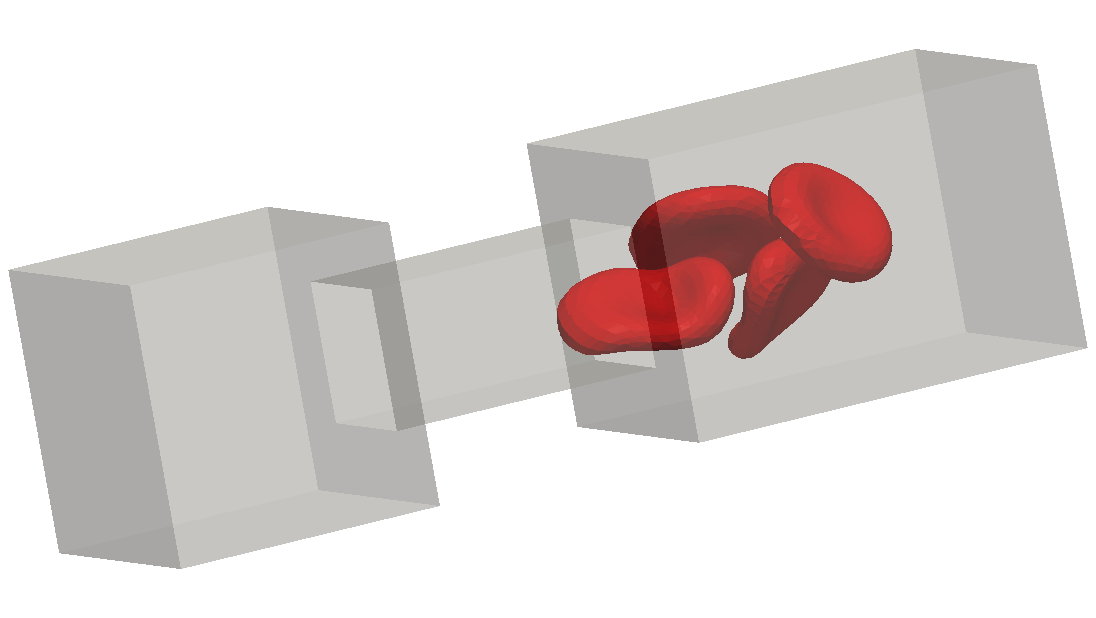}
\includegraphics[angle=0,width=\lwidth]{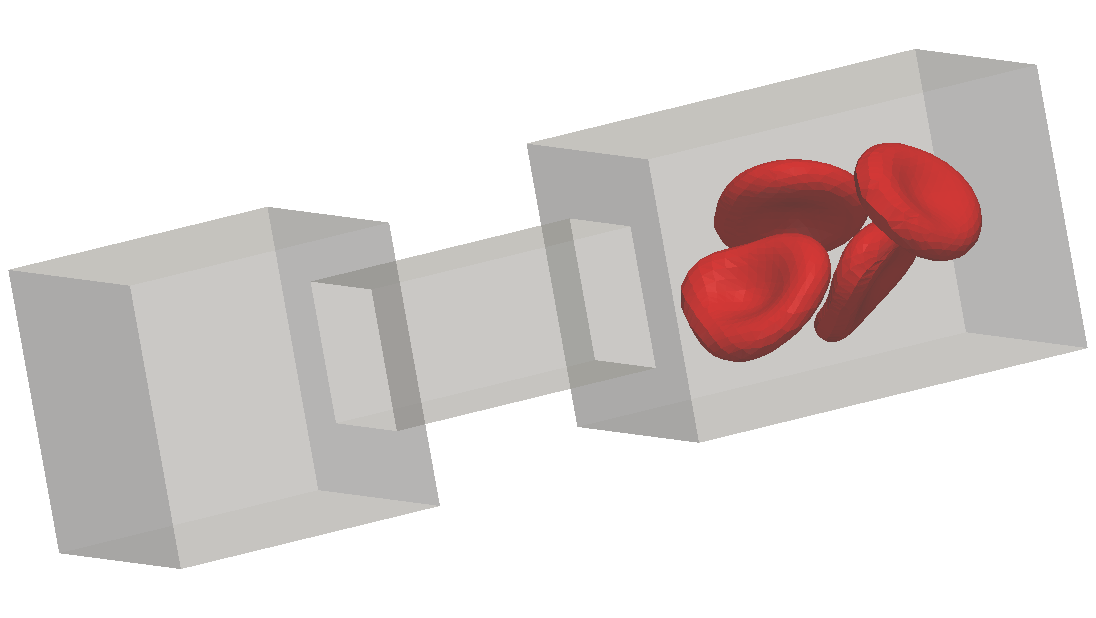}
\caption{Visualization of the numerical experiment shown in 
\citet[Fig.\ 12]{nsns}. The snapshots are taken at times $t=0,\ 0.5,\ 1,\
1.5,\ 2,\ 2.5,\ 3,\ 
4$.
}
\label{fig:nsns}
\end{figure}%

\subsection{Two-phase biomembranes} \label{subsec:nsns2phase}
It is also possible to consider two-phase biomembranes. In this
case we can introduce an order parameter $\mathfrak{c}$,
which takes the values $\pm 1$ in the two different
phases, and this parameter is related to the composition of a chemical
species within the membrane. On the surface we then use a phase field model
to approximate the interfacial energy by the Ginzburg--Landau functional
\begin{equation*}
\varsigma\,\int_\Gamma \tfrac12\,\epsilon\,|\nabs\,\mathfrak{c}|^2
+\epsilon^{-1}\,\Psi(\mathfrak{c}) \dH{d-1}\,.
\end{equation*}
Here $\Psi$ is a double well potential,
$\varsigma>0$ is related to the tension of the interface between 
the two phases on the membrane, often called line tension in the special case
$d=3$, and $\epsilon>0$ is related to the interfacial thickness 
of a diffuse layer between the two phases.

In the different phases $\alpha$, $\spont$ and $\alpha^G$, 
cf.\ (\ref{eq:CEHEn}), 
will take different values, and we will interpolate these values
obtaining functions $\alpha(\phasec)>0$, $\spont(\phasec)$ and
$\alpha^G(\phasec)$. The total energy, in the absence of ADE effects,
will hence have the form
\begin{equation} \label{eq:E}
E(\Gamma,\phasec) = \int_{\Gamma} b(\varkappa,\phasec) + 
\alpha^G (\phasec)\, \Gauss + \varsigma\,b_{GL}(\phasec) \dH{d-1}\,, 
\end{equation}
where
\begin{equation*} 
b(\varkappa,\phasec) = \tfrac12\,\alpha(\phasec) \,(\varkappa -
\spont(\phasec))^2
\quad\text{and}\quad
b_{GL}(\phasec) = \tfrac12\,\epsilon\,|\nabs\,\phasec|^2 + 
\epsilon^{-1}\,\Psi(\phasec)\,.
\end{equation*}
In \cite{nsns2phase} the present authors generalized the model for the dynamics
of fluidic biomembranes to the two-phase case in the absence of ADE effects. 
Let us state here some basic
ingredients. One now has to introduce an appropriate evolution law for
the species
concentration, $\phasec$, on the membrane. To this end, we considered
the following Cahn--Hilliard dynamics on $\Gamma(t)$
\[ 
\vartheta\, \matpartu\,\phasec = \Delta_s\,\chempot \,,\quad
\chempot = - \varsigma\,\epsilon\,\Delta_s\,\phasec +
\varsigma\,\epsilon^{-1}\,\Psi'(\phasec)
+ (\partial_\phasec\,b)(\varkappa,\phasec) 
+ (\alpha^G)'(\phasec)\,\Gauss\,, 
\]
where $\matpartu$ is the material time derivative as defined in 
\eqref{eq:materialder},
$\chempot$ denotes the chemical potential and $\vartheta \in \bRplus$ is a 
kinetic coefficient. We note here that $\chempot = \deldel{\phasec}\, 
E(\Gamma,\phasec)$ is
the first variation of the total energy with respect to $\phasec$.  In
addition, the force $\vec f_\Gamma$ in \eqref{eq:bio2} 
needs to be modified to take into account the concentration $\phasec$. 
The generalized force is given by minus the
first variation of the energy (\ref{eq:E}) with respect to $\Gamma$, i.e.\
\begin{align*} 
\vec f_\Gamma & = - \deldel{\Gamma}\,E(\Gamma,\phasec) \nonumber \\
& = \left[-\Delta_s\,[\alpha(\phasec)\,(\varkappa - \spont(\phasec))]
-\alpha(\phasec)\,(\varkappa-\spont(\phasec)) \,|\nabs\,\vec\nu|^2
+b(\varkappa,\phasec)\,\varkappa\right.\nonumber\\
&\qquad\left.- \nabs\,.\,([\varkappa\,\mat\Id + \nabs\,\vec\nu]\,\nabs\,\alpha^G(\phasec)) 
\right]\vec\nu
+ \left[(\partial_\phasec\,b)(\varkappa,\phasec) +
(\alpha^G)'(\phasec)\,\Gauss\right] \nabs\,\phasec \nonumber\\
&\qquad  + \varsigma\left[ b_{GL}(\phasec)\,
\varkappa\,\vec\nu 
+ \nabs\, b_{GL}(\phasec)
- \epsilon\,\nabs\,.\,( (\nabs\,\phasec)\otimes(\nabs\,\phasec)) \right],
\end{align*}
where we note that
\begin{equation*} 
(\partial_\phasec\,b)(\varkappa,\phasec)
= \tfrac12\,\alpha'(\phasec) \,(\varkappa - \spont(\phasec))^2 
- \alpha(\phasec)\,(\varkappa - \spont(\phasec))\,\spont'(\phasec)\,.
\end{equation*}
We observe that in contrast to situations where the energy density 
does not depend on a
species concentration, we now have tangential contributions to
$\vec f_\Gamma$, which gives rise to a Marangoni-type effect.
These equations now have to be coupled to (\ref{eq:NSNS1}) and \eqref{eq:bio}.
A stable semidiscrete formulation
of the resulting total system has been derived in \cite{nsns2phase}, where also
several numerical computations showing the influence of the occurrence of the 
two different phases on the evolution of the membrane are shown.

\subsection{Alternative numerical approaches}
\label{subsec:alternativebio}

Let us now mention other contributions that take
local incompressibility and/or fluid effects into account in the evolution of 
vesicles and membranes.
In \cite{BonitoNP11} a fluid-membrane system, in which forces resulting from 
the Willmore energy act on an interior flow, has been studied.
Local incompressibility conditions on the membrane have been addressed by
\cite{SalacM11,LaadhariSM14,GeraS18} within a level set context,
by \cite{JametM07,AlandELV14} with the help of a phase field approach
and by \cite{HuKL14,Heintz15} using an immersed boundary method.
Moreover, \cite{RahimiA12,RodriguesAMB15} presented numerical results
using a surface Stokes system
without taking the bulk fluid flow into account. There the volume
conservation is enforced by a global Lagrange multiplier.
\cite{ArroyoSH10preprint} simultaneously take surface and bulk
viscosity effects in the fluidic membrane evolution into account.
Their numerical results are in an axisymmetric setting.

\arxivyesno{
\section*{Acknowledgements}
The authors gratefully acknowledge the support 
of the Regensburger Universit\"atsstiftung Hans Vielberth.
}{
\begin{acks}
The authors gratefully acknowledge the support 
of the Regensburger Universit\"atsstiftung Hans Vielberth.
\end{acks}
\vspace*{-3pt}
}

\arxivyesno{
\renewcommand\refname{References\footnote{%
The items in this bibliography are sorted per the following interpretation of
the Harvard reference style. All items are sorted alphabetically by the first
author's name. For each first author, first appear all the items 
with one or two authors, sorted first alphabetically by the authors' names 
and then chronologically,
followed by all the items with more than two authors, which are arranged in
chronological order only.}}
}{}

\end{document}